\documentclass{amsart}

\usepackage{amssymb}
\usepackage{url}
\usepackage{hyperref}

\usepackage[all,cmtip]{xy} 

\usepackage{dsfont}

\newcommand{\I}{\mathds 1}
\newcommand{\C}{\mathbb C}
\newcommand{\Q}{\mathbb Q}
\newcommand{\R}{\mathbb R}
\newcommand{\Z}{\mathbb Z}

\newcommand{\bA}{\mathbb A}
\newcommand{\bG}{\mathbb G}
\newcommand{\bP}{\mathbb P}

\newcommand{\sA}{\mathcal A}
\newcommand{\sB}{\mathcal B}
\newcommand{\sC}{\mathcal C}
\newcommand{\sD}{\mathcal D}
\newcommand{\sE}{\mathcal E}
\newcommand{\sF}{\mathcal F}
\newcommand{\sG}{\mathcal G}
\newcommand{\sJ}{\mathcal J}
\newcommand{\sK}{\mathcal K}
\newcommand{\sL}{\mathcal L}
\newcommand{\sM}{\mathcal M}
\newcommand{\sO}{\mathcal O}
\newcommand{\sP}{\mathcal P}
\newcommand{\sR}{\mathcal R}
\newcommand{\sT}{\mathcal T}
\newcommand{\sV}{\mathcal V}
\newcommand{\sW}{\mathcal W}

\newcommand{\iso}{\xrightarrow{\sim}}

\newcommand{\nd}{\nobreakdash-\hspace{0pt}}

\renewcommand{\theenumi}{(\roman{enumi})}

\DeclareMathOperator{\Alt}{Alt}
\DeclareMathOperator{\Aut}{Aut}
\DeclareMathOperator{\card}{card}
\DeclareMathOperator*{\colim}{colim}
\DeclareMathOperator{\End}{End}
\DeclareMathOperator{\Ext}{Ext}
\DeclareMathOperator{\Gal}{Gal}
\DeclareMathOperator{\Hom}{Hom}
\DeclareMathOperator{\Id}{Id}
\DeclareMathOperator{\Iso}{Iso}
\DeclareMathOperator{\Ker}{Ker}
\DeclareMathOperator{\Lie}{Lie}
\DeclareMathOperator{\Mod}{Mod}
\DeclareMathOperator{\MOD}{MOD}
\DeclareMathOperator{\Pic}{Pic}
\DeclareMathOperator{\pr}{pr}
\DeclareMathOperator{\Rad}{Rad}
\DeclareMathOperator{\Rep}{Rep}
\DeclareMathOperator{\REP}{REP}
\DeclareMathOperator{\Spec}{Spec}
\DeclareMathOperator{\Sym}{Sym}
\DeclareMathOperator{\tr}{tr}

\theoremstyle{plain}
\newtheorem{thm}{Theorem}[section]
\newtheorem{cor}[thm]{Corollary}
\newtheorem{lem}[thm]{Lemma}
\newtheorem{prop}[thm]{Proposition}

\theoremstyle{definition}
\newtheorem{defn}[thm]{Definition}
\newtheorem{exmp}[thm]{Example}
\newtheorem{rem}[thm]{Remark}

\numberwithin{equation}{section}

\begin{document}

\title{Principal bundles under reductive groups}

\author{Peter O'Sullivan}
\address{Mathematical Sciences Institute \\
The Australian National University \\
Canberra ACT 2601, Australia}
\email{peter.osullivan@anu.edu.au}
\thanks{Supported by the Australian Research Council}



\date{}

\dedicatory{}

\begin{abstract}
Let $k$ be a field of characteristic $0$.
We consider principal bundles over a $k$\nd scheme with reductive structure group (not necessarily of finite type).
It is shown in particular that for $k$ algebraically closed there exists 
on any complete connected $k$\nd scheme a universal such bundle.
As a consequence, an explicit description of principal bundles with reductive structure group 
over smooth projective curves of genus $0$ or $1$
is obtained.
\end{abstract}

\maketitle


\section{Introduction}

In 1957 Grothendieck \cite{Gro57} showed that if $G$ is a complex Lie group with reductive Lie algebra,
then  principal $G$\nd bundles over the Riemann sphere are classified
by conjugacy classes of homomorphisms from the complex multiplicative group to $G$. 
This result is equivalent to a purely algebraic one with the complex numbers replaced by an algebraically closed
field $k$ of characteristic $0$,
the Riemann sphere by the complex projective line over $k$, 
and $G$ by a reductive algebraic group over $k$.
Modulo the classification of \emph{vector} bundles on the projective line, Grothendieck's theorem 
can be expressed as follows:
the functor on reductive algebraic groups over $k$ up to conjugacy that sends $G$ to the 
set of isomorphism classes of principal $G$\nd bundles over the projective line is representable.
In this paper it will be shown that analogous 
results hold over much more general base schemes, 
provided that we allow reductive groups which need not be of finite type as structure groups.

Let $k$ be a field and $G$ be an affine $k$\nd group, i.e.\  an affine group scheme over $k$.
Such a $G$ is the filtered limit of its affine quotient $k$\nd groups of finite type.
We define principal $G$\nd bundles over a $k$\nd scheme $X$ by requiring local triviality
for the $fpqc$ topology.
The set of isomorphism classes of such bundles will be written $H^1(X,G)$.
When $G$ is of finite type and $k$ is of characteristic $0$, 
local triviality for the $fpqc$ topology is equivalent to that for the \'etale topology.
We may regard $H^1(X,G)$ as a functor on the category of affine $k$\nd groups up
to conjugacy, in which a morphism from $G'$ to $G$ is an equivalence class of $k$\nd homomorphisms
from $G'$ to $G$ under the action by conjugation of $G(k)$.

For the rest of this section, $k$ will be a field of characteristic $0$.
By a \emph{reductive $k$\nd group} we mean an affine $k$\nd group for which each $k$\nd quotient of finite type
is a  (not necessarily connected) reductive algebraic group over $k$.
We now have the following result (cf. Corollary~\ref{c:cohomologyrep}\ref{i:gpdcohomologyrep}).

\begin{thm}\label{t:repr}
Let $X$ be a scheme over an algebraically closed field $k$ of characteristic $0$ for which $H^0(X,\sO_X)$ is 
a henselian local $k$\nd algebra with residue field $k$.
Then the functor $H^1(X,-)$ on the category of reductive $k$\nd groups up to conjugacy is representable.
\end{thm}

Explicitly, Theorem~\ref{t:repr} states that there exists a reductive $k$\nd group $G_0$ such that
we have a bijection
\begin{equation}\label{e:conjiso}
\Hom_k(G_0,G)/G(k) \iso H^1(X,G)
\end{equation}
which is natural in the reductive $k$\nd group $G$.
Such a $G_0$ is of course unique up to isomorphism.
Taking general linear groups for $G$ in \eqref{e:conjiso} gives a bijection from the set of isomorphism
classes of representations of $G_0$ to the set of isomorphism classes of vector bundles on $X$.
Since vector bundles in general have moduli while representations reductive groups of finite type do not,
the $k$\nd group $G_0$ thus cannot in general be of finite type.
It follows from \eqref{e:conjiso} that the largest profinite quotient of $G_0$ 
is the fundamental group of $X$.
The $k$\nd group $G_0$ itself may be regarded as an algebraic analogue of the loop space of a topological space.
Some condition on $X$ is necessary in Theorem~\ref{t:repr}
because \eqref{e:conjiso} implies for example that the Krull-Schmidt theorem must hold for vector bundles over $X$.

When $X$ is a point, or more generally the spectrum of a henselian local $k$\nd algebra with residue field $k$,
the universal reductive $k$\nd group $G_0$ is trivial.
When $X$ is the projective line, $G_0$
is $\bG_m$, and we recover Grothendieck's original result, 
in the slightly more general form where 
the reductive $k$\nd groups are not required to be of finite type (cf.\ Theorem~\ref{t:linerep}).

Suppose that $X$ an elliptic curve over $k$.
Then the universal reductive $k$\nd group $G_0$ is the product of $SL_2$
with a central extension of $k$\nd groups of multiplicative type.
To describe it explicitly,
write $D(M)$ for the $k$\nd group of multiplicative type with character group $M$.
For $n \ne 0$ the Weil pairing $e_n$ on the $n$\nd torsion $k$\nd subgroup ${}_nX$ of $X$ has values in
the $k$\nd group $D(\Z/n)$ of $n$th roots of unity, and passing to the limit gives a bilinear alternating
$k$\nd morphism
\[
e = \lim_{n \ne 0}e_n:\lim_{n \ne 0} {}_n X \times_k \lim_{n \ne 0} {}_n X \to \lim_{n \ne 0} D(\Z/n) = D(\Q/\Z).
\]
Further ${}_n X$ may be identified
with $D({}_nX(k))$ by assigning to the point $a$ of ${}_nX$ the point $e_n(-,a)$ of $D({}_nX(k))$,
so that there is a projection
\[
D(X(k)) \to D(\colim_{n \ne 0} {}_nX(k)) = \lim_{n \ne 0} {}_n X.
\]
We then have a bilinear alternating $k$\nd morphism
\[
u_X:D(X(k)) \times_k D(X(k)) \to D(\Q),
\]
which if we regard $D(\Q/\Z)$ as a $k$\nd subgroup of $D(\Q)$ is given by
\[
u_X(a,a') = \sqrt{e(\overline{a},\overline{a}{}')},
\]
where the bar denotes the image in $\lim_{n \ne 0} {}_n X$, 
and the square root in $D(\Q)$ is well-defined because $D(\Q)$ is uniquely divisible.
Now given commutative affine $k$\nd groups $G'$ and $G''$ and a bilinear alternating $k$\nd morphism 
\[
z:G' \times_k G' \to G'' ,
\]
we have a central extension $C_z$ of $G'$ by $G''$ with underlying $k$\nd scheme $G'' \times_k G'$ and product given by
\[
(g'',g')(h'',h') = (g''h''z(g',h'),g'h').
\]
In particular we have a central extension 
\[
1 \to D(\Q) \to C_{u_X} \to D(X(k)) \to 1.
\]
The universal reductive $k$\nd group $G_0$ of \eqref{e:conjiso} for the elliptic curve $X$ is then
\[
C_{u_X} \times_k SL_2.
\]
This can be deduced from Theorem~\ref{t:repr} using Atiyah's classification \cite{Ati57} of vector bundles on 
$X$ (cf.\ Theorem~\ref{t:ellrep}).

In general, an explicit determination of the $k$\nd group $G_0$ of \eqref{e:conjiso} does not seem possible.
It is however possible to describe appropriate quotients of $G_0$.
In particular, under mild finiteness conditions, the quotient of $G_0$ by the derived $k$\nd group of its 
identity component
is an extension of the fundamental group of $X$ by the 
protorus with group of characters the Picard group of the universal cover of $X$ (cf.\ Proposition~\ref{p:etunpull}).
When $X$ is an elliptic curve for example, this quotient is $C_{u_X}$, and the fundamental group is 
$\lim_{n \ne 0} {}_n X$.

Even when the $k$\nd group $G_0$ is not known explicitly, its existence has
strong consequences for principal bundles over $X$.
Suppose for example that $k$ and $X$ are as in Theorem~\ref{t:repr}, and let $G$ be a reductive $k$\nd group and $P$
be a principal $G$\nd bundle over $X$.
Then the set of pairs $(G',P')$ with $G'$ a reductive $k$\nd subgroup of $G$ and $P'$ a principal 
$G'$\nd subbundle of $P$ can be described as follows, starting from a minimal element $(G_1,P_1)$ of this set:
for any $(G',P')$ we have $P' = \theta(P_1gG')$ for some $k$\nd point 
$g$ of $G$ with $g^{-1}G_1g$ contained in $G'$ and 
$G$\nd automorphism $\theta$ of $P$ (cf.\ Corollary~\ref{c:minsubconjgpd}).
If further $X$ is complete and $G$ is of finite type then the $k$\nd group of $G$\nd automorphisms of $P$
exists and is affine of finite type and (cf.\ Lemma~\ref{l:minbggpd} and Example~\ref{e:GPAut}) 
it is the semidirect product of its unipotent radical and the centraliser of $G_1$ in $G$. 
Many other consequences are given in Sections~\ref{s:unmin}--\ref{s:pullback}.

The conclusion of Theorem~\ref{t:repr} no longer holds if the hypothesis 
that $k$ be algebraically closed is dropped, even for $X$ a point.
However if $X$ has a $k$\nd point $x$, and we write $\widetilde{H}^1(X,x,G)$ for the subset of $H^1(X,G)$ consisting 
of the classes of those principal $G$\nd bundles that have a $k$\nd point above $x$, then the following modified form of
Theorem~\ref{t:repr} holds (cf.\ Corollary~\ref{c:cohomologyrep}\ref{i:ptcohomologyrep}).

\begin{thm}\label{t:reprpt}
Let $X$ be a scheme over a field $k$ of characteristic $0$ for which $H^0(X,\sO_X)$ is 
a henselian local $k$\nd algebra.
Then for any $k$\nd point $x$ of $X$ the functor $\widetilde{H}^1(X,x,-)$ on the category of reductive $k$\nd groups up to 
conjugacy is representable.
\end{thm}

For $X$ a point, the reductive $k$\nd group that represents the functor of Theorem~\ref{t:reprpt} is again trivial,
and for $X$ the projective line, it is again $\bG_m$ (cf.\ Theorem~\ref{t:linerep}). 

For $X$ an elliptic curve, the reductive $k$\nd group that represents the functor of Theorem~\ref{t:reprpt}
is again $C_{u_X} \times_k SL_2$,
provided that in the definition of $C_{u_X}$ we replace $D(X(k))$ by the $k$\nd group of multiplicative
type $D(X(\overline{k}))$ with Galois module of characters $X(\overline{k})$,
where $\overline{k}$ is an algebraic closure of $k$ (cf.\ Theorem~\ref{t:ellrep}).

In the general case, where $k$ need not be algebraically closed and $X$ need not have a $k$\nd point, 
it is possible 
to proceed by replacing principal bundles with groupoids.
A groupoid over a $k$\nd scheme $X$ is a $k$\nd scheme $K$ with source and target 
$k$\nd morphisms $d_1$ and $d_0$ from $K$ to $X$,
together with an identity $X \to K$ and a composition 
\[
K \times_{{}^{d_1}X^{d_0}} K \to K
\]
which is associative and has inverses.
The target and source morphisms define on $K$ a structure of scheme over $X \times_k X$,
and its pullback along the diagonal is a group scheme $K^\mathrm{diag}$ over $X$.
We have a category of groupoids over $X$ up to conjugacy, where a morphism from $K'$ to $K$
is an equivalence class under conjugation by cross-sections of $K^\mathrm{diag}$.
Suppose that $X$ is non-empty.
Then a groupoid $K$ over $X$ is said to be transitive affine if it is affine and faithfully flat over $X \times_k X$.
Any principal $G$\nd bundle $P$ over $X$ gives rise to a transitive affine groupoid $\underline{\Iso}_G(P)$ over $X$,
whose points in a $k$\nd scheme $S$ above $(x_1,x_0)$ are the isomorphisms from $P_{x_0}$ to $P_{x_1}$ 
of principal $G$\nd bundles over $S$.
A transitive affine groupoid $K$ over $X$ will be called reductive if the fibres of $K^\mathrm{diag}$
are reductive.
Then we have the following result, which can be shown to contain Theorems~\ref{t:repr} and \ref{t:reprpt} 
(cf.\ Theorem~\ref{t:main}\ref{i:mainexist}).

\begin{thm}\label{t:init}
Let $X$ be a scheme over a field $k$ of characteristic $0$ for which $H^0(X,\sO_X)$ is 
a henselian local $k$\nd algebra with residue field $k$.
Then the category of reductive groupoids over $X$ up to conjugacy has an initial object.
\end{thm}

For the formulation of general theorems, it is simple and convenient to work with groupoids over $X$
as in Theorem~\ref{t:init},
but this does not directly give a description of principal bundles over $X$.
However an equivalent form of Theorem~\ref{t:init} analogous to Theorem~\ref{t:repr}
can be obtained by replacing the reductive $k$\nd groups of Theorem~\ref{t:repr} by reductive
groupoids over the spectrum of an algebraic closure of $k$.

Fix an algebraic closure $\overline{k}$ of $k$.
Groupoids over $\Spec(\overline{k})$ will also be called groupoids over $\overline{k}$.
A notion of principal bundle on a $k$\nd scheme $X$ under a transitive affine groupoid $F$ over $\overline{k}$ can be defined,
and the isomorphism classes of such bundles form a set $H^1(X,F)$.
For $F$ the constant groupoid over $\overline{k}$ defined 
by the affine $k$\nd group $G$, this set coincides with $H^1(X,G)$.
In general, in contrast to the case of affine $k$\nd groups, the set $H^1(X,F)$ for $F$ a transitive affine groupoid over 
$\overline{k}$ may be empty.
The following is an equivalent version of Theorem~\ref{t:init} (cf.\ Corollary~\ref{c:cohomologyrep}\ref{i:gpdcohomologyrep}).

\begin{thm}\label{t:reprgpd}
Let $X$ and $k$ be as in Theorem~\textnormal{\ref{t:init}}
and $\overline{k}$ be an algebraic closure of $k$.
Then the functor $H^1(X,-)$ on the category of reductive groupoids over $\overline{k}$ up to conjugacy
is representable.
\end{thm}

Theorem~\ref{t:reprgpd} leads to a description of principal bundles under a reductive $k$\nd group
over $X$, in the following way.
We consider pairs $(D,E)$ with $D$ an affine $\overline{k}$\nd group and $E$ an extension of topological groups
\[
1 \to D(\overline{k})_{\overline{k}} \to E \to \Gal(\overline{k}/k) \to 1,
\]
where the group $D(\overline{k})_{\overline{k}}$ of $\overline{k}$\nd points over $\overline{k}$
is given the pro-discrete topology defined by the quotients of $D$ of finite type,
and where the following condition is satisfied:
conjugation of $D(\overline{k})_{\overline{k}}$ by an element of $E$ above $\sigma$ in $\Gal(\overline{k}/k)$
arises from an automorphism of the scheme $D$ above the automorphism $\sigma^{-1}$ of $\overline{k}$. 
Such pairs will be called Galois extended $\overline{k}$\nd groups (cf.\ Definitions~\ref{d:Galext} and 
\ref{d:Galextloose}, and Proposition~\ref{p:Galextiso}).
There is an equivalence from the category of transitive affine groupoids over $\overline{k}$
to the category of Galois extended $\overline{k}$\nd groups, which sends $F$ to a pair $(D,E)$ with $D = F^\mathrm{diag}$
(cf.\ Proposition~\ref{p:grpdgalequ}).
The constant groupoid defined by $G$ is sent to
\[
(G_{\overline{k}},G(\overline{k}) \rtimes \Gal(\overline{k}/k)).
\]
For $X$ as in Theorem~\ref{t:init}, there then exists a pair $(D_0,E_0)$ with $D_0$ reductive such that we have a 
bijection
\begin{equation}\label{e:conjisopair}
\Hom((D_0,E_0),
(G_{\overline{k}},G(\overline{k}) \rtimes \Gal(\overline{k}/k)))/
G(\overline{k})
\iso H^1(X,G)
\end{equation}
which is natural in the reductive $k$\nd group $G$.
The push forward of $E_0$ obtained by factoring out the identity component of $D_0$ is 
the arithmetic fundamental group of $X$.
If $X$ is quasi-compact, the hypotheses of Theorem~\ref{t:repr} are satisfied with $X$ and $k$ replaced by $X_{\overline{k}}$
and $\overline{k}$, and $D_0$ is the $\overline{k}$\nd group that represents the functor $H^1(X_{\overline{k}},-)$ 
on reductive $\overline{k}$\nd groups up to conjugacy.
When $k = \overline{k}$ is algebraically closed, \eqref{e:conjisopair} reduces to \eqref{e:conjiso}.

When $X$ is a point, the $\overline{k}$\nd group $D_0$ of \eqref{e:conjisopair} is trivial,
and the set on the left of \eqref{e:conjisopair} is the continuous cohomology set $H^1(\Gal(\overline{k}/k),G(\overline{k}))$.

Let $X$ be a smooth projective curve over $k$ of genus $0$.
Then \eqref{e:conjisopair} holds with $D_0 = \bG_m{}_{\overline{k}}$, and $E_0$ the extension of $\Gal(\overline{k}/k)$
by $\bG_m(\overline{k})$ with class in in the Brauer group 
$H^2(\Gal(\overline{k}/k),\bG_m(\overline{k}))$ of $k$ that of the Severi--Brauer variety $X$ (cf.\ Theorem~\ref{t:gen0rep}).

Let $X$ be a smooth projective curve over $k$ of genus $1$.
Then $X$ is a principal homogeneous space under its Jacobian $X_0$.
For $n \ne 0$ there is a short exact sequence
\[
1 \to C_{u_{X_0}}(\overline{k}) \xrightarrow{n_{u_{X_0}}} C_{u_{X_0}}(\overline{k}) \to {}_nX_0(\overline{k}) \to 1
\]
of topological groups, with $n_{u_{X_0}}$ the $k$\nd homomorphism defined by the power $n^2$ on the factor $D(\Q)$ of $C_{u_{X_0}}$ and
the power $n$ on the factor $D(X_0(\overline{k}))$.
For some $n$, the class of $X$ in the Weil--Ch\^atelet group $H^1(\Gal(\overline{k}/k),X_0(\overline{k}))$
of $X_0$ is the image of an element $\alpha_n$ of $H^1(\Gal(\overline{k}/k),{}_nX_0(\overline{k}))$.
If $\varphi$ is a continuous $1$\nd cocycle of $\Gal(\overline{k}/k)$ with values in ${}_nX_0(\overline{k})$ with class 
$-\alpha_n$,
then \eqref{e:conjisopair} holds with $(D_0,E_0)$ the product
\[
((C_{u_{X_0}}){}_{\overline{k}},E) \times (SL_2{}_{\overline{k}},SL_2(\overline{k}) \rtimes \Gal(\overline{k}/k))
\]
where $E$ is the extension of $\Gal(\overline{k}/k)$ by $C_{u_{X_0}}(\overline{k})$ obtained by taking the semidirect product 
of $\Gal(\overline{k}/k)$ with the above short exact sequence and then pulling back along the continuous homomorphism
from $\Gal(\overline{k}/k)$ to ${}_nX_0(\overline{k}) \rtimes \Gal(\overline{k}/k)$ that sends $\sigma$ to 
$\varphi(\sigma)\sigma$
(cf.\ Theorem~\ref{t:gen1rep}).

The above theorems are proved
by considering the $k$\nd linear tensor category $\Mod(X)$ of vector bundles over $X$. 
For $G$ reductive, principal $G$\nd bundles over $X$ may be identified with $k$\nd linear tensor functors from the category
of representations of $G$ to $\Mod(X)$.
For $X$ as in the above theorems, $\Mod(X)$ has a unique maximal tensor ideal, and 
its quotient $\overline{\Mod(X)}$ by this ideal is a semisimple Tannakian category over $k$.
The universal reductive $k$\nd group of Theorem~\ref{t:repr} for example is then characterised by the property
that its category of representations is equivalent as a tensor category to $\overline{\Mod(X)}$.
The required universal property follows from  the splitting
theorem for tensor categories of Andr\'e and Kahn \cite{AndKah} and \cite{O}, 
in a slightly more general form proved in Section~\ref{s:splitting} below.

The version of the above theorems proved below is in fact a rather more general one.
We start with a groupoid $H$ (not necessarily transitive affine) over $X$, or more generally
what will be called a pregroupoid over $X$, and consider principal bundles over $X$ with an action of $H$,
or transitive affine groupoids over $X$ with a morphism from $H$.
The results as stated above are then included as the case where $H$ is the initial groupoid $X$ over $X$.
For appropriate $H$ (cf.\ Examples~\ref{ex:equivariant} and \ref{ex:connection}), we obtain analogous results for equivariant principal bundles,
or principal bundles equipped with a connection.
In this more general form, Theorem~\ref{t:init} contains the generalised
Jacobson--Morosov theorem of \cite{AndKah} and \cite{O10}, by taking for $X$ a point and for $H$ an affine $k$\nd group.

The paper is organised as follows.
Sections~\ref{s:grpdpre} to \ref{s:bunundergpd} give the basic definitions for groupoids, pregroupoids
and principal bundles.
It is technically convenient here to work in the category of schemes over a base.
In Sections~\ref{s:specfield} to \ref{s:grpdgalext}, we specialise to the case where the base is 
a field.
From Section~\ref{s:redsub} on, we assume that the base field has characteristic $0$.
The splitting theorem, in the form required, is proved in Section~\ref{s:splitting}.
The main theorem, the existence of universal reductive groupoids, is proved 
along with its corollaries in Section~\ref{s:unmin}.
These are applied to principal bundles in Section~\ref{s:appprinbun},
and further applications are given in Sections~\ref{s:gauge} and \ref{s:pullback}.
An explicit description of principal bundles under reductive groups over smooth projective curves of genus $0$ or $1$
is given in Sections~\ref{s:curves0} and \ref{s:curves1}.

\section{Groupoids and pregroupoids}\label{s:grpdpre}

This section contains the basic definitions and terminology for groupoids and pregroupoids, which will be used throughout
the paper.
For this foundational material, it is convenient to work initially in an arbitrary category with fibre products,
and to specialise later to the case of the category of schemes over a base.

Fix a category $\sC$ with fibre products.
Write $\Delta_2$ for the category with objects the ordered sets
$[0] = \{0\}$, $[1] = \{0,1\}$ and $[2] = \{0,1,2\}$, with morphisms 
the order-preserving maps between them.
A contravariant functor from  $\Delta_2$ to $\sC$ will be called
a \emph{pregroupoid in $\sC$},
and a natural transformation of such functors a \emph{morphism of pregroupoids in $\sC$}.
A pregroupoid in $\sC$ is thus the same as a $2$\nd truncated simplicial object in $\sC$.
The image of an object or morphism in $\Delta_2$ under a pregroupoid will usually be written
as a subscript.
A pregroupoid $H$ in $\sC$ thus consists of objects $H_{[0]}$, $H_{[1]}$, $H_{[2]}$ in $\sC$
and morphisms 
\[
\xymatrix{
H_{[0]} \ar[r] &  H_{[1]} \ar@<0.75ex>[l]  \ar@<-0.75ex>[l] \ar@<0.75ex>[r] \ar@<-0.75ex>[r]
& H_{[2]} \ar@<1.5ex>[l] \ar[l] \ar@<-1.5ex>[l] 
}
\]
satisfying the well-known compatibilities. 
As usual, we write
\[
d_i:H_{[n+1]} \to H_{[n]}
\]
for $H_{\delta_i}$ with $\delta_i:[n] \to [n+1]$ the injective map in $\Delta_2$ whose image omits $i$,
and
\[
s_i:H_{[n]} \to H_{[n+1]}
\]
for $H_{\sigma_i}$ with $\sigma_i:[n+1] \to [n]$ the surjective map in $\Delta_2$ for which 
$\sigma_i(i) = \sigma_i(i+1)$.
Occasionally it will be useful to write
\[
e_i:H_{[n]} \to H_{[0]}
\]
for $H_{\varepsilon_i}$ with $\varepsilon_i:[0] \to [n]$ the map whose image is $\{i\}$.
For $n = 1$ we have $e_0 = d_1$ and $e_1 = d_0$.
To every pregroupoid $H$ there is associated a pregroupoid 
\[
H^\mathrm{op}
\]
with $(H^\mathrm{op})_{[n]}$, $d_i:(H^\mathrm{op})_{[n+1]} \to (H^\mathrm{op})_{[n]}$ and 
$s_i:(H^\mathrm{op})_{[n]} \to (H^\mathrm{op})_{[n+1]}$ given
respectively by $H_{[n]}$, $d_{n+1-i}:H_{[n+1]} \to H_{[n]}$ and $s_{n-i}:H_{[n]} \to H_{[n+1]}$.

Recall that a groupoid is a category in which every morphism is an isomorphism.
We may regard a groupoid $K$ as a set $K_{[0]}$ of objects, a set $K_{[1]}$
of arrows, a source map $d_1$ and a target map $d_0$ from 
$K_{[1]}$ to $K_{[0]}$, an identity map $s_0$ from $K_{[0]}$ to $K_{[1]}$, where $s_0(u)$ has source
and target $u$, and a composition map 
\[
\circ:K_{[1]} \times_{{}^{d_1}{K_{[0]}}^{d_0}}K_{[1]} \to K_{[1]}
\] 
from the set of composable pairs of arrows
to $K_{[1]}$,
where $v \circ w$ has source that of $w$ and target that of $v$,
such that $\circ$ is associative with left and right identities given by $s_0$, 
and has inverses (necessarily unique).
An element $v$ of $K_{[1]}$ may be pictured as an arrow  
\[
\cdot \xleftarrow{v} \cdot
\]
pointing to the left, and a composable pair $(v,w)$ as a pair of arrows $\cdot \xleftarrow{v} \cdot \xleftarrow{w} \cdot$.
A morphism of groupoids is a functor between their underlying categories.
Explicitly, a morphism from $K$ to $K'$ consists of maps $K_{[0]} \to K'{}\!_{[0]}$ and
$K_{[1]} \to K'{}\!_{[1]}$ which commute with the source, target and composition maps.
Such a morphism necessarily commutes with the identity map and preserves inverses.

A \emph{groupoid $K$ in $\sC$} consists of objects $K_{[0]}$ and $K_{[1]}$ of objects and arrows in $\sC$,
source and target morphisms $d_1$ and $d_0$ from $K_{[1]}$ to $K_{[0]}$, an identity morphism $s_0$ from $K_{[0]}$ to $K_{[1]}$, 
and a composition morphism $\circ$ from $K_{[1]} \times_{{}^{d_1}{K_{[0]}}^{d_0}}K_{[1]}$ to $K_{[1]}$, such that the points in any object 
of $\sC$ form a groupoid.
A morphism $K \to K'$ of groupoids in $\sC$ is a pair of morphisms $K_{[0]} \to K'{}\!_{[0]}$
and $K_{[1]} \to K'{}\!_{[1]}$ which define on points in any object of $\sC$
a morphism of groupoids, or equivalently which commute with the morphisms
$d_0$, $d_1$, and $\circ$.

Any groupoid $K$ in $\sC$ may be regarded as a pregroupoid in $\sC$ by taking
\[
K_{[2]} = K_{[1]} \times_{{}^{d_1}{K_{[0]}}^{d_0}}K_{[1]}
\]
with $d_0$, $d_1$ and $d_2$ from $K_{[2]}$ to $K_{[1]}$ respectively the first projection,
$\circ$ and the second projection,
and $s_0$ and $s_1$ from $K_{[1]}$ to $K_{[2]}$ respectively 
$(1_{K_{[1]}},s_0 \circ d_1)$ and $(s_0 \circ d_0,1_{K_{[1]}})$.
A morphism $K \to K'$ of groupoids in $\sC$ then extends uniquely to a morphism of pregroupoids
in $\sC$, with $K_{[2]} \to K'{}\!_{[2]}$ the fibre product of morphisms $K_{[1]} \to K'{}\!_{[1]}$
over $K_{[0]} \to K'{}\!_{[0]}$.
Thus we obtain a fully faithful functor from groupoids in $\sC$ to pregroupoids in $\sC$.
Any pregroupoid in the essential image of this functor will also be called a groupoid,
so that the category of groupoids in $\sC$ may be regarded as a full subcategory of 
that of pregroupoids in $\sC$.

A morphism from a pregroupoid $H$ to a groupoid $K$ in $\sC$
is completely determined by its components at $[0]$ and $[1]$.
Explicitly, such a morphism is a pair of morphisms $f_0$ from $H_{[0]}$ to $K_{[0]}$
and $f_1$ from $H_{[1]}$ to $K_{[1]}$, commuting with $d_0$ and $d_1$, such that
\[
f_1(d_1(v)) = f_1(d_0(v)) \circ f_1(d_2(v)).
\]
for every point $v$ of $H_{[2]}$.
Such a pair necessarily commutes with $s_0$, because $s_0$ sends
a point $u$ of $K_{[0]}$ to the unique point $w$ of $K_{[1]}$
with $d_0(w) = u = d_1(w)$ such that $w = w \circ w$.

Let $X$ be an object of $\sC$.
By a \emph{pregroupoid over $X$} we mean a pregroupoid $H$ in $\sC$ with $H_{[0]} = X$.
A morphism $H' \to H$ of pregroupoids over $X$ is a morphism of pregroupoids from $H'$ to $H$
for which $H'{}\!_{[0]} \to H_{[0]}$ is the identity $1_X$.
The category of pregroupoids over $X$ is thus a subcategory, in general non-full, of the category
of pregroupoids in $\sC$.
It contains as a full subcategory the category of groupoids over $X$, i.e.\ of groupoids in $\sC$ with object of 
objects $X$.
A groupoid over $X$ will usually be denoted by its object of arrows.

The category of pregroupoids over $X$ has an initial object, 
the pregroupoid $H$ with $H_{[n]} = X$ for each $n$ and the $d_i$ and $s_i$ the identity.
This initial object is a groupoid over $X$, which we write simply as $X$.

Suppose that $\sC$ has a final object, and hence all finite limits.
Then the category of pregroupoids over $X$ has a final object, which we 
write as
\[
[X].
\] 
It is a groupoid over $X$, which has a unique point with given source and target.
Explicitly $[X]_{[n]}$ is the product $\prod_{i \in [n]}X = X^{n+1}$,
and $[X]_{\varphi}:[X]_{[m]} \to [X]_{[n]}$ for $\varphi:[n] \to [m]$ sends the point
$(x_j)_{j \in [m]}$ to the point $(x_{\varphi(i)})_{i \in [n]}$.
We write the points $(x_i)_{i \in [1]}$ of $[X]_{[1]}$ and $(x_i)_{i \in [2]}$ of $[X]_{[2]}$ respectively as
$(x_1,x_0)$ and $(x_2,x_1,x_0)$, numbered from right to left.
Thus for example the image in $[X]_{[1]}$ of a point $v$ of a groupoid $K$ over $X$ is
$(d_0(v),d_1(v))$ with the source $d_1(v)$ on the right and the target $d_0(v)$ on the left,
corresponding to the picture of arrows pointing to the left.

Let $X' \to X$ be a morphism in $\sC$.
Then for any pregroupoid $H$ over $X$ there is a pregroupoid $H'$ over $X'$, the \emph{pullback of $H$ along $X' \to X$},
which is universal among pregoupoids over $X'$ equipped with a morphism to $H$ with component $X' \to X$ at $[0]$.
Explicitly, $H'{}\!_{[n]}$ is the fibre product over $H_{[n]}$ of the pullbacks $e_i{}\!^*X'$ for $i \in [n]$,
so that its points are $(n+2)$\nd tuples $(v,(x'{}\!_i)_{i \in [n]})$ with $v$ a point of $H_{[n]}$ and
$x'{}\!_i$ a point of $X'$ above $e_i(v)$, while $H'{}\!_\varphi$ for $\varphi:[n] \to [m]$ sends
$(v,(x'{}\!_j)_{j \in [m]})$ to $(H_\varphi(v),(x'{}\!_{\varphi(i)})_{i \in [n]})$ and $H' \to H$ sends 
$(v,(x'{}\!_i)_{i \in [n]})$ to $v$.
If $H$ is a groupoid then $H'$ is a groupoid, with composition 
$(v,x'{}\!_2,x'{}\!_1) \circ (w,x'{}\!_1,x'{}\!_0) = (v \circ w,x'{}\!_2,x'{}\!_0)$.
When $\sC$ has a final object, the pullback of $H$ along $X' \to X$ is the fibre product
\[
H \times_{[X]} [X']
\]
of pregroupoids in $\sC$, with morphism to $H$ the projection.

Let $H$ be a pregroupoid over $X$.
By a \emph{pregroupoid over $H$} we mean a pregroupoid $H'$ over $X$ equipped with a morphism $H \to H'$
of pregroupoids over $X$.
A morphism of pregroupoids over $H$ is a morphism of their underlying pregroupoids over $X$
compatible with the morphisms from $H$.
When $H = X$, this notion reduces to the one of pregroupoid over $X$ already defined.

\begin{rem}
In what follows, we usually fix a pregroupoid $H$ in a category $\sC$ of relative schemes, and consider groupoids over $H$.
This is equivalent to equipping groupoids, or principal bundles which give rise to them,
with extra structure of a given type, such as an equivariant structure or a 
connection (Examples~\ref{ex:equivariant} and \ref{ex:connection} below).
We need to take for $H$ a pregroupoid, and not merely a groupoid, because for example
the $H$ that define connections are not in general groupoids.
Using the $2$\nd coskeleton functor, we may regard pregroupoids in $\sC$ as simplicial objects in $\sC$.
There would however be no gain in generality in taking for $H$ a simplicial object in $\sC$ rather than a pregroupoid in $\sC$,
because a morphism from such an $H$ to a groupoid $K$ in $\sC$ is the same as morphism from the $2$\nd truncation of $H$ to $K$.
\end{rem}

It will always be assumed in what follows that the fibre products of the form $Z \times_Z Z'$ and $Z' \times_Z Z$ are chosen to coincide with $Z'$, with projections to $Z'$ the identity.
If $H$ is a pregroupoid over $X$, it will be assumed, unless the contrary is indicated, 
that
\[
H_{[1]} \times_X X' = H_{[1]} \times_{{}^{d_1}X} X'
\]
for $H_{[1]}$ on the left of a fibre product over $X$, and that
\[
X' \times_X H_{[1]} = X' \times_{X^{d_0}} H_{[1]}
\]
for $H_{[1]}$ on the right of a fibre product over $X$.

Given a morphism $f:Z' \to Z$ in $\sC$, there is pullback functor $f^* = Z' \times_Z -$ from 
objects in $\sC$ over $Z$ to objects over $Z'$.
We have $f^*Z = Z'$, and $(1_Z)^*$ is the identity functor.
For every composable pair of morphisms $f$ and $f'$,
there is a pullback isomorphism from $(f \circ f')^*$ to $f'{}^*f^*$,
defined by the universal properties of $f^*$, $f'{}^*$ and $(f \circ f')^*$,
and satisfying the usual compatibilities.

Let $H$ be a pregroupoid over $X$ in $\sC$.
Given an object $X'$ in $\sC$ over $X = H_{[0]}$, an \emph{action of $H$ on $X'$} is an isomorphism
\[
\alpha:d_1{}\!^*X' \iso d_0{}\!^*X'
\]
over $H_{[1]}$ such that if for $i = 0,1,2$ we write $\alpha_i$ for the morphism
\begin{equation}\label{e:alphai}
(d_1 \circ d_i)^*X' \iso d_i{}\!^*d_1{}\!^*X' \xrightarrow{d_i{}\!^*\alpha} d_i{}\!^*d_0{}\!^*X' \iso (d_0 \circ d_i)^*X'
\end{equation}
of objects over $H_{[2]}$ defined using the pullback isomorphisms, then
\begin{equation}\label{e:alphacomp}
\alpha_1 = \alpha_0 \circ \alpha_2.
\end{equation}
Modulo the pullback isomorphisms, $s_0{}\!^*\alpha$ is then the identity of $X'$.

An object over $X$ in $\sC$ with an action of $H$ will be called an \emph{$H$\nd object}.
A morphism of $H$\nd objects from $X''$ with action $\alpha'$ to
$X'$ with action $\alpha$ is a morphism $f:X'' \to X'$ over $X$ such that 
$\alpha \circ d_1{}\!^*f = d_0{}\!^*f \circ \alpha'$.
The category of $H$\nd objects has finite limits.
If $H = X$ is the initial pregroupoid over $X$, then the category
of $H$\nd objects may be identified with the category of objects over $X$ in $\sC$.

A group object in the category of $H$\nd objects, or what is the same a group object over $X$ in $\sC$
equipped with an action of $H$ which is an isomorphism of group objects over $H_{[1]}$,
 will also be called an \emph{$H$\nd group}.

Given a pregroupoid $H'$ over $X'$ and a morphism $h:H' \to H$ of pregroupoids in $\sC$, 
there is a pullback functor $h^*$ from $H$\nd objects to $H'$\nd objects, which sends $Z$ to 
its pullback along $X' \to X$ with action of $H'$ given, modulo the pullback isomorphisms,
by the pullback along $H'{}\!_{[1]} \to H_{[1]}$ of the action of $H$ on $Z$.

An $H$\nd object is the same as an object $X'$ over $X$ in $\sC$ and a morphism
\[
\widetilde{\alpha}:H_{[1]} \times_X X' \to X'
\]
such that the square
\[
\xymatrix{
H_{[1]} \times_X X' \ar[d] \ar[r]^-{\widetilde{\alpha}} & X' \ar[d] \\
H_{[1]} \ar[r]^{d_0} & X
}
\]
with the left arrow the projection is cartesian,
and such that 
\[
\widetilde{\alpha}(d_1(w),x') = \widetilde{\alpha}(d_0(w),\widetilde{\alpha}(d_2(w),x'))
\]
for every point $w$ of $H_{[2]}$ and $x'$ of $X'$ above $e_0(w)$.
Indeed a morphism $\alpha:d_1{}\!^*X' \to d_0{}\!^*X'$ over $H_{[1]}$ with projection 
$\widetilde{\alpha}$ onto $X'$ is an isomorphism
if and only if the  square is cartesian, and \eqref{e:alphacomp} holds if and only if $\widetilde{\alpha}$
satisfies the associativity condition.
We also call $\widetilde{\alpha}$ the action of $H$ on $X'$.
A morphism of $H$\nd objects from $X''$ with action $\widetilde{\alpha}'$ to
$X'$ with action $\widetilde{\alpha}$ is then a morphism $f:X'' \to X'$ over $X$ such that 
$\widetilde{\alpha} \circ (H_{[1]} \times_X f) = f \circ \widetilde{\alpha}'$.
We often write $\widetilde{\alpha}(v,x')$ simply as $vx'$.

An $H^\mathrm{op}$\nd object will also be called a \emph{right $H$\nd object}.
A right $H$\nd object is the same as an object $X'$ over $X$ in $\sC$ and a morphism
\[
X' \times_X H_{[1]} \to X'
\]
satisfying conditions analogous to those satisfied by $\widetilde{\alpha}$ above.

Let $H$ be a pregroupoid over $X$ and $X'$ be an $H$\nd object.
Then with the $H_{[n]}$ regarded as objects over $X$ using $e_0$, 
there is a unique pregroupoid $H'$ over $X'$ with
\[
H'{}\!_{[n]} = H_{[n]} \times_X X'
\]
for $n = 0,1,2$, such that
\begin{enumerate}
\renewcommand{\theenumi}{(\arabic{enumi})}
\item
the projections $H'{}\!_{[n]} \to H_{[n]}$ define a morphism of pregroupoids $H' \to H$,
\item
$e_0:H'{}\!_{[n]} \to X'$ is the projection for each $n$,
\item
$d_0:H'{}\!_{[1]} \to X'$ is the action of $H$ on $X'$.
\end{enumerate}
We write this pregroupoid $H'$ as
\[
H \times_X X'.
\]
However $H \times_X X'$ is not the fibre product of pregroupoids in $\sC$,
and there is in general no morphism from $H$ to the initial pregroupoid $X$ over $X$.
If $H$ is a groupoid then $H \times_X X'$ is a groupoid, with composition $(v,wx') \circ (w,x') = (v \circ w,x')$.
An $(H \times_X X')$\nd object is the same as an $H$\nd object $X''$ equipped with a morphism of $H$\nd objects $X'' \to X'$.

We have a functor $H \times_X -$ from $H$\nd objects in $\sC$ to pregroupoids in $\sC$
equipped with a morphism to $H$, where $H \times_X f$ has component $H_{[n]} \times_X f$ at $[n]$.
It is fully faithful, with essential image those $H' \to H$ such that the square
\begin{equation}\label{e:Hcart}
\begin{gathered}
\xymatrix{
H'{}\!_{[m]} \ar[d] \ar[r]^{H'{}\!_{\varphi}} & H'{}\!_{[n]} \ar[d] \\
H_{[m]} \ar[r]^{H_{\varphi}} & H_{[n]}
}
\end{gathered}
\end{equation}
is cartesian for every $\varphi:[n] \to [m]$ in $\Delta_2$.

Let $K$ be a groupoid over $X$ and $J$ be a $K$\nd group.
The \emph{semidirect product}
\[
J \rtimes_X K
\]
\emph{of $K$ by $J$} is the groupoid over $X$ with object of arrows $J \times_X K$, source and target those of the projection onto $K$, identity defined by the identities of $J$ and $K$, and composition
defined by $(w,v) \circ (w',v') = (w(vw'),v \circ v')$.

\begin{rem}
In what follows, $\sC$ will be the category of schemes over some base $S$.
Any morphism $S \to S'$ defines by composition a functor, which preserves fibre products,
from  $\sC$ to the category $\sC'$ of schemes over $S'$.
A pregroupoid $H$ in $\sC$ may thus also be regarded as a pregroupoid in $\sC'$.
Most of the notions defined above are independent of whether $H$ is regarded as a pregroupoid in $\sC$ or $\sC'$,
as are the notion of $H$\nd module, and the pullback and push forward functors defined below.
The main exceptions are the final pregroupoid $[X]$ over $X$,
and the notion of transitive groupoid defined in Section~\ref{s:transgrpd} below.
\end{rem}

\emph{For the rest of this section we fix a scheme $S$, and pregroupoids and groupoids will, unless otherwise stated,
be in the category of schemes over $S$}.

Let $X$ be a scheme over $S$ and $H$ be a pregroupoid over $X$.
An $H$\nd object will also be called an \emph{$H$\nd scheme}.
An $H$\nd scheme will be called affine, finite, etc.\ if it is so as a scheme over $X$.
By an \emph{$H$\nd subscheme} of an $H$\nd scheme $X'$ we mean a subscheme $X''$ of $X'$ such that the action 
$d_1{}\!^*X' \iso d_0{}\!^*X'$ of $H$
on $X'$ restricts to an isomorphism $d_1{}\!^*X'' \iso d_0{}\!^*X''$.
This isomorphism is an action of $H$ on $X''$, and is the unique such action for which embedding 
is a morphism of $H$\nd schemes.
An $H$\nd subscheme of $X'$ is the same as an $(H \times_X X')$\nd subscheme of the final
$(H \times_X X')$\nd scheme $X'$.

Let $S'$ be a scheme over $S$.
Then the base extension functor $(-)_{S'}$ from schemes over $S$ to schemes over
$S'$ preserves finite limits.
To any pregroupoid (resp.\ groupoid) $H$ over $X$ there is associated a pregroupoid (resp.\ groupoid) $H_{S'}$ over $X_{S'}$ in the category of schemes over $S'$,
and to any $H$\nd scheme $X'$ an $H_{S'}$\nd scheme $X'{}\!_{S'}$.

Let $K$ be a groupoid over $X$.
By a \emph{subgroupoid of $K$} we mean a subscheme $K'$ of $K$ such the composite of any two composable
points of $K'$ lies in $K'$, the identities of $K$ lie in $K'$,
and the inverse of any point of $K'$ lies in $K'$. 
Such a $K'$ has a unique structure of groupoid over $X$ for which the embedding is a 
morphism of groupoids over $X$.
The equaliser of $d_0,d_1:K \to X$ is a subgroupoid
\[
K^\mathrm{diag}
\]
of $K$, and its identity and composition define on it a structure of group scheme over $X$.
Further $K^\mathrm{diag}$ has a structure of $K$\nd group with the point $v$ of $K$ acting as
$v \circ - \circ v^{-1}$.
If $X = S$, then $K = K^\mathrm{diag}$ is simply a group scheme over $X$.


Any cross-section $v$ of $K^\mathrm{diag}$ defines an inner automorphism
\[
v \circ - \circ v^{-1}:K = X \times_X K \times_X X \to K
\]
of the groupoid $K$ over $X$.
Let $H$ be pregroupoid over $X$.
Composition with the inner automorphisms defines an action $h \mapsto {}^vh$ of the group of cross-sections $v$ of 
$K^\mathrm{diag}$ on the set of morphisms $h$ from $H$ to $K$.
An orbit under this action will be called a morphism up to conjugacy from $H$ to $K$.
If $f:K \to K'$ is a morphism of groupoids over $X$,
then ${}^wf \circ {}^vh = {}^{wf(v)}(f \circ h)$.
We have a category of \emph{groupoids up to conjugacy over $H$}, with objects the groupoids over $H$, where a morphism
from $K'$ to $K$ is an orbit of the set of morphisms over $H$ from $K'$ to $K$ under the action of the group of
$H$\nd invariant cross-sections of $K^\mathrm{diag}$, i.e. those that fix the structural morphism $H \to K$ of $K$.

For any scheme $X$,
the category of $\sO_X$\nd modules has a structure of tensor category with the usual tensor product
$\sV \otimes_{\sO_X} \sW$ of $\sO_X$\nd modules $\sV$ and $\sW$.
We may assume the tensor product chosen so that the unit $\sO_X$ is strict.

Let $p:X' \to X$ be a morphism of schemes. 
The pullback functor $p^*$ from the category of $\sO_X$\nd modules to the category of $\sO_{X'}$\nd modules is left adjoint to the push forward $p_*$.
Write 
\[
\eta_p:\Id \to p_*p^*
\]
for the unit of this adjunction and
\[
\varepsilon_p:p^*p_* \to \Id
\]
for the counit.
We have canonical morphisms $\sO_X \to p_*\sO_{X'}$ and
\[
p_*\sV' \otimes_{\sO_X} p_*\sW' \to p_*(\sV' \otimes_{\sO_{X'}} \sW')
\]
defining a structure of lax (i.e.\ the defining morphisms need not be isomorphisms)
tensor functor on $p_*$, and canonical isomorphisms $\sO_{X'} \iso p^*\sO_X$ and
\[
p^*\sV \otimes_{\sO_{X'}} p^*\sW \iso p^*(\sV \otimes_{\sO_X} \sW)
\]
defining a structure of tensor functor on $p^*$, 
such that $\eta_p$ and $\varepsilon_p$ are compatible with the tensor products.
Replacing $p^*$ by a tensor isomorphic tensor functor, we may assume that $\sO_{X'} \iso p^*\sO_X$ is the identity. 
When $p$ and hence $p_*$ is the identity, we may take $p^*$ to be the identity.

With the usual composition of adjunctions, $p'{}^*p^*$ is left adjoint to $(p \circ p')_* = p_*p'{}\!_*$, so that there is a canonical isomorphism, the pullback isomorphism, between $(p \circ p')^*$ and $p'{}^*p^*$.
It is a tensor isomorphism, because it is defined using the units and counits for $p$, $p'$,
and $p \circ p'$.
There are similar tensor isomorphisms for any two decompositions of a given morphism
as the composite of a string of morphisms, and such isomorphisms satisfy the usual compatibilities.

Let $H$ be a pregroupoid over $X$.
By an \emph{action} of $H$ on an $\sO_X$\nd module $\sV$, we mean an isomorphism
\[
\alpha:d_1{}\!^*\sV \iso d_0{}\!^*\sV
\]
of $\sO_{H_{[1]}}$\nd modules such that if $\alpha_i$ is defined by \eqref{e:alphai} with $X'$ replaced by 
$\sV$, then \eqref{e:alphacomp} holds.
Modulo the pullback isomorphisms, $s_0{}\!^*(\alpha)$ is then the identity of $\sV$.

By an \emph{$H$\nd module} we mean a quasi-coherent $\sO_X$\nd module equipped with an action of $H$.
A morphism from $\sV'$ with action $\alpha'$ to $\sV$ with action $\alpha$ is a morphism $f:\sV' \to \sV$ of $\sO_X$\nd modules
such that $\alpha \circ d_1{}\!^*f = d_0{}\!^*f \circ \alpha'$.
We denote by
\[
\MOD_H(X)
\]
the category of $H$\nd modules.
When $H = X$, we may identify $\MOD_H(X)$ with the category $\MOD(X)$ of quasi-coherent $\sO_X$\nd modules.

Let $\sV$ and $\sW$ be $H$\nd modules.
Then the tensor product $\sV \otimes_{\sO_X} \sW$ of their underlying $\sO_X$\nd modules has a canonical structure of 
$H$\nd module, with the action given, modulo the isomorphisms defining the tensor structures on $d_1{}\!^*$
and $d_0{}\!^*$, by the tensor products of the actions on $\sV$ and $\sW$.
We thus obtain a tensor product on $\MOD_H(X)$, with unit $\sO_X$ strict, and associativity constraints and symmetries those of the underlying $\sO_X$\nd modules.

An algebra in the tensor category $\MOD_H(X)$ of a given type, for example unitary and associative, is the same as a quasi-coherent
$\sO_X$\nd algebra $\sR$ of the same type equipped with an isomorphism $\alpha:d_1{}\!^*\sR \iso d_0{}\!^*\sR$
of $\sO_{H_{[1]}}$\nd algebras which satisfies \eqref{e:alphacomp}.
Such an algebra will be called an \emph{$H$\nd algebra}.
If $\sR$ is a unitary associative $H$\nd algebra, a module over $\sR$ in $\MOD_H(X)$ is the same as a module $\sM$ over
the $\sO_X$\nd algebra $\sR$ equipped with a structure of $H$\nd module for which the actions of $H$ on $\sR$ and $\sM$
define an isomorphism from the $d_1{}\!^*\sR$\nd module $d_1{}\!^*\sM$ to the $d_0{}\!^*\sR$\nd module $d_0{}\!^*\sM$.
Such a module will be called an \emph{$(H,\sR)$\nd module}.
In what follows, algebras will be assumed to be unitary and associative, unless otherwise stated.

The anti-equivalence $\Spec$ from commutative quasi-coherent $\sO_X$\nd algebras to affine schemes over $X$ defines an anti-equivalence
from commutative $H$\nd algebras to affine $H$\nd schemes, where $\sR$ with action $\alpha$ of $H$ is sent to $\Spec(\sR)$ with
action given, modulo the canonical isomorphism between $\Spec(d_i{}\!^*\sR)$ and $d_i{}\!^*\Spec(\sR)$,
by $\Spec(\alpha^{-1})$.

Locally free $\sO_X$\nd modules of finite type will also be called vector bundles over $X$.
An $H$\nd module whose underlying $\sO_X$\nd module is a vector bundle will also be called a \emph{representation of $H$}.
The representations of $H$ then form a full subcategory 
\[
\Mod_H(X)
\] 
of $\MOD_H(X)$.
When $H = X$, we may identify $\Mod_H(X)$ with the category $\Mod(X)$ of vector bundles over $X$.
If $\sV$ is a representation of $H$, then the dual $\sV^\vee$ of its underlying vector bundle has a canonical structure
of representation of $H$, with action given, modulo the canonical isomorphism between $(d_i{}\!^*\sV)^\vee$
and $d_i{}\!^*(\sV^\vee)$, by the inverse of the transpose of the action of $H$ on $\sV$.
The usual canonical homomorphisms involving tensor products and duals are then morphisms of $H$\nd modules.

For every $H$\nd module $\sV$, we write
\[
H^0_H(X,\sV)
\]
for the subgroup of $H^0(X,\sV)$ consisting of those sections $v$ for which the action
of $H$ sends $d_1{}\!^*v$ to $d_0{}\!^*v$.
Then $H^0_H(X,\sO_X)$ is an $H^0(S,\sO_S)$\nd subalgebra of $H^0(X,\sO_X)$, and 
$H^0_H(X,\sV)$ is an $H^0_H(X,\sO_X)$\nd submodule of $H^0(X,\sV)$, which may be identified with
the hom-group $\Hom_H(\sO_X,\sV)$.
The endomorphism ring $\End_H(\sV)$ of $\sV$ is an $H^0_H(X,\sO_X)$\nd subalgebra 
of $\End_{\sO_X}(\sV)$.

Let $\sV$ be a vector bundle over $X$.
Arguing locally over $X$ shows that there is a groupoid 
\[
\underline{\Iso}_X(\sV)
\]
over $X$ whose points in a scheme $Z$ over $S$ above $(x_1,x_0)$ are the isomorphisms from $\sV_{x_0}$ to $\sV_{x_1}$
of vector bundles over $Z$.
A representation of $H$ is then a vector bundle $\sV$ over $X$ equipped with a structure on $\underline{\Iso}_X(\sV)$
of groupoid over $H$.

\begin{exmp}\label{ex:descent}
An $[X]$\nd module is a quasi-coherent $\sO_X$\nd module equipped with a descent datum from
$X$ to $S$, and an $[X]$\nd scheme is a scheme over $X$ equipped with a descent datum from
$X$ to $S$.
\end{exmp}

\begin{exmp}\label{ex:equivariant}
Let $G$ be a group scheme over the base scheme $S$.
We may regard $G$ as a groupoid over $S$,
and a $G$\nd scheme is then a $G$\nd scheme in the usual sense.
If $X$ is such a $G$\nd scheme, then $(G \times_S X)$\nd modules are the same as $G$\nd equivariant 
quasi-coherent $\sO_X$\nd modules, and $(G \times_S X)$\nd schemes are the same as $G$\nd equivariant
schemes over $X$.
\end{exmp}

\begin{exmp}\label{ex:connection}
Suppose that $X$ is smooth over $S$.
Write $X_{(2)}$ for the ``first infinitesimal neighbourhood of $X$ in $[X]$'',
i.e.\ for the pregroupoid over $X$ with $X_{(2)}{}_{[n]}$ the first infinitesimal neighbourhood
of the diagonal $X$ in the $(n+1)$th power $[X]_{[n]}$ of $X$ over $S$, and $d_i$ and $s_i$ induced by those of $[X]$.
Then $X_{(2)}{}_{[2]}$ is the amalgamated sum $X_{(2)}{}_{[1]} \amalg_X X_{(2)}{}_{[1]}$ with embeddings $s_0$ and $s_1$.
Pulling back along the embeddings shows that an isomorphism $d_1{}\!^*\sV \iso d_0{}\!^*\sV$
for a quasi-coherent $\sO_X$\nd module $\sV$ defines an action of $X_{(2)}$ on $\sV$ provided that its pullback
along the diagonal is the identity modulo the pullback isomorphisms.
An $X_{(2)}$\nd module is thus the same as a quasi-coherent $\sO_X$\nd module equipped with a connection over $S$.
There is a similar description for $X_{(2)}$\nd schemes.
In general, $X_{(2)}$ is not a groupoid.
\end{exmp}

Suppose given a commutative square
\begin{equation}\label{e:basechangesqu}
\begin{gathered}
\xymatrix{
X' \ar[d]_{p} & Y' \ar[d]^q \ar[l]_{u'} \\
X  & Y \ar[l]_{u}
}
\end{gathered}
\end{equation}
in the category of schemes.
It induces a square of categories and functors with $X$ for example replaced by the category of $\sO_X$\nd modules,
and $p$, $q$, $u$ and $u'$ by $p_*$, $q_*$, $u_*$ and $u'{}\!_*$.
By pasting the unit $\eta_{u'}$ on top of this last square and the counit $\varepsilon_u$ on the bottom,
we obtain the base change natural transformation \cite[I~9.3.1.1]{DG71}
\[
\beta = \beta_{p,u';u,q} = (\varepsilon_u q_*u'{}^*) \circ (u^*p_*\eta_{u'}):u^*p_* \to q_*u'{}^*.
\]
It is compatible with the tensor products, because $\eta_{u'}$ and $\varepsilon_u$ are.
Similarly, by pasting $\eta_{u'}$ on top, $\varepsilon_u$ on the bottom, $\eta_p$ on the left, and $\varepsilon_q$ on the right,
we obtain the pullback tensor isomorphism
\[
\gamma:q^*u^* \iso u'{}^*p^*.
\]
We then have commutative squares of functors and natural transformations
\begin{equation}\label{e:betagamma}
\begin{gathered}
\xymatrix{
u^* \ar[d]_{\eta_q u^*} \ar[r]^-{u^*\eta_p} & u^*p_*p^* \ar[d]^{\beta p^*} \\
q_*q^*u^* \ar[r]^{q_*\gamma} & q_*u'{}^*p^*
}
\qquad \qquad
\xymatrix{
q^*u^*p_* \ar[d]_{q^*\beta} \ar[r]^{\gamma p_*} & u'{}^*p^*p_* \ar[d]^{u'{}^*\varepsilon_p} \\
q^*q_*u'{}^* \ar[r]^-{\varepsilon_q u'{}^*} & u'{}^*
}
\end{gathered}
\end{equation}
where the left square commutes by the triangular identity for $\eta_q$ and $\varepsilon_q$,
and the right square by the triangular identity for $\eta_p$ and $\varepsilon_p$.
Given also $v'$, $v$ and $r$ with $q \circ v' = v \circ r$, we have a commutative diagram 
\begin{equation}\label{e:basechangetrans}
\begin{gathered}
\xymatrix@C+2em@R+0.25cm{
(u \circ v)^*p_* \ar[d]^{\wr} \ar[rr]^{\beta_{p,u' \circ v';u \circ v,r}} & & r_*(u' \circ v')^* \ar[d]^{\wr} \\
v^*u^*p_* \ar[r]^{v^*\beta_{p,u';u,q}}  & v^*q_*u'{}^* \ar[r]^{\beta_{q,v';v,r}u'{}^*} & r_*v'{}^*u'{}^* 
}
\end{gathered}
\end{equation}
with vertical arrows defined using the pullback isomorphisms.
This follows from the fact that, modulo the pullback isomorphisms, $\eta_{u' \circ v'}$ is given by appropriately
pasting $\eta_{u'}$ and $\eta_{v'}$, and $\varepsilon_{u \circ v}$ by appropriately pasting  $\varepsilon_u$
and $\varepsilon_v$.

Suppose that the square \eqref{e:basechangesqu} is cartesian.
If $p$ is affine, or if $p$ is quasi-compact and quasi-separated and $u$ is flat,
then $(\beta_{p,u';u,q})_{\sV'}$ is an isomorphism for every quasi-coherent $\sO_{X'}$\nd module $\sV'$ \cite[I~9.3.2]{DG71}.

Let $H'$ be a pregroupoid over $X'$ and $f:H' \to H$ be a morphism of pregroupoids.
Write $p:X' \to X$, $q:H'{}\!_{[1]} \to H_{[1]}$ and $r:H'{}\!_{[2]} \to H_{[2]}$
for the components of $f$ at $[0]$, $[1]$ and $[2]$,
and $d'{}\!_i$ instead of $d_i$ for the morphisms $H'{}\!_{[n+1]} \to H'{}\!_{[n]}$. 
To any action $\alpha$ of $H$ on an $\sO_X$\nd module $\sV$, there is associated an action of $H'$ on $p^*\sV$,
given, modulo the pullback isomorphism from $d'{}\!_i{}\!^*p^*$ to $q^*d_i{}\!^*$ for $i = 0,1$, by $q^*(\alpha)$.
Thus $f$ induces a pullback functor 
\begin{equation}\label{e:Hpull}
f^*:\MOD_H(X) \to \MOD_{H'}(X').
\end{equation}
It has a structure of tensor functor, defined using that on $p^*$.
Given also $f':H'' \to H'$ inducing $p':X'' \to X'$, we have a pullback isomorphism $(f \circ f')^* \iso f'{}^*f^*$,
with components those of $(p \circ p')^* \iso p'{}^*p^*$.
If $p = 1_X$, so that $H' \to H$ is a morphism of pregroupoids over $X$, then $f^*$ is simply
restriction along $H' \to H$, and is the identity on the underlying $\sO_X$\nd modules.

Suppose that either $p$ is affine, or that $p$ is quasi-compact and quasi-separated and the 
$d_i:H_{[n+1]} \to H_{[n]}$ are flat.
Suppose further that the squares formed by $p$, $d'{}\!_i$, $d_i$, $q$ for $i = 0,1$ and by $q$, $d'{}\!_j$, $d_j$, $r$  for $j = 0,1,2$
are cartesian.
Then the $(\beta_{p,d'{}\!_i;d_i,q})_{\sV'}$ and $(\beta_{q,d'{}\!_j;d_j,r})_{\sW'}$ are isomorphisms for $\sV'$ and $\sW'$ 
quasi-coherent.
To every $H'$\nd module $\sV'$ with action $\alpha'$ we associate a structure $\alpha$ of $H$\nd module on $p_*\sV'$
by requiring that the square
\begin{equation}\label{e:pushdef}
\begin{gathered}
\xymatrix{
d_1{}\!^*p_*\sV' \ar[d]_{(\beta_{p,d'{}\!_1;d_1,q})_{\sV'}} \ar[r]^{\alpha} 
&  d_0{}\!^*p_*\sV' \ar[d]^{(\beta_{p,d'{}\!_0;d_0,q})_{\sV'}} \\
q_*d'{}\!_1{}\!^*\sV' \ar[r]^{q_*(\alpha')} &  q_*d'{}\!_0{}\!^*\sV'
}
\end{gathered}
\end{equation}
commute:
the condition \eqref{e:alphacomp} holds because 
taking $u = d_i$, $v = d_j$, $u' = d'{}\!_i$, $v' = d'{}\!_j$ in \eqref{e:basechangetrans}
for $i = 0,1$ and using the naturality of $\beta_{q,d'{}\!_j;d_j,r}$ shows that for $j = 0,1,2$ there is a commutative square 
similar to \eqref{e:pushdef} with top arrow $\alpha_j$ and bottom arrow $r_*(\alpha'{}\!_j)$. 
It follows from the naturality of the $\beta$ that $p_*$ sends morphisms of $H'$\nd modules to morphisms of 
$H$\nd modules, and from the compatibility of the $\beta$ with tensor products that the lax tensor structure on $p_*$
is compatible with the $H$\nd module structures.
Thus we have a lax tensor functor
\begin{equation}\label{e:Hpush}
f_*:\MOD_{H'}(X') \to \MOD_H(X).
\end{equation}
which is $p_*$ on the underlying $\sO_{X'}$\nd modules.
Further $(\eta_p)_\sV$ is a morphism of $H$\nd modules for every $H$\nd module $\sV$ by \eqref{e:pushdef} with $\sV' = p^*\sV$,
the left square of \eqref{e:betagamma} with $u = d_i$ and $u' = d'{}\!_i$, and the naturality of $\eta_q$.
Similarly using the right square of \eqref{e:betagamma} shows that $(\varepsilon_p)_{\sV'}$ is a morphism of 
$H'$\nd modules for every $H'$\nd module $\sV'$.
Thus \eqref{e:Hpush} is right adjoint to \eqref{e:Hpull}, with the unit of the adjunction defined using $\eta_p$ and the counit
using $\varepsilon_p$.
Given also $f':H'' \to H'$ with $f'{}\!_*$ defined, we have $f_*f'{}\!_* = (f \circ f')_*$, because the isomorphism 
$f_*f'{}\!_*\iso (f \circ f')_*$ defined by $(f \circ f')^* \iso f'{}^*f^*$ 
using the units and counits is the identity on underlying $\sO_X$\nd modules.

Let $X'$ be an $H$\nd scheme with structural morphism $p$.
If $H' = H \times_X X'$ with $f:H' \to H$ given by the projections,
the squares \eqref{e:Hcart}, and in particular the squares $p$, $d'{}\!_i$, $d_i$, $q$ and $q$, $d'{}\!_j$, $d_j$, $r$, are cartesian.
We then usually write $p^*$ for $f^*$ of \eqref{e:Hpull} and, when the required conditions on $p$ and the $d_i$ hold,
$p_*$ for $f_*$ of \eqref{e:Hpush}.
Similarly if $p':X'' \to X'$ is an $H$\nd morphism, we usually write $p'{}^*$ for $(H \times_X p')^*$ and,
when the required conditions hold, $p'{}\!_*$ for $(H \times_X p')_*$.

Suppose that the above conditions on $p$ and the $d_i$ hold. 
Considering $H$\nd morphisms $\sO_X \to p_*\sV'$and using the above adjunction gives an equality
\begin{equation}\label{e:Hpushequ}
H^0_{H \times_X X'}(X',\sV') = H^0_H(X,p_*\sV')
\end{equation}
of subgroups of $H^0(X',\sV') = H^0(X,p_*\sV')$.

Suppose that $p$ is affine.
The lax tensor structure of $p_*$ defines on $p_*\sO_{X'}$ a structure of commutative $H$\nd algebra,
with underlying $\sO_X$\nd algebra the usual one.
Similarly there is a structure  of $(H,p_*\sO_{X'})$\nd module on $p_*\sV'$ for every  $(H \times_X X')$\nd module $\sV'$.
If $\sV$ is an $H$\nd module, then the $(H,p_*\sO_{X'})$\nd morphism corresponding
to the $H$\nd morphism $(\eta_p)_{\sV}$ is an isomorphism
\begin{equation}\label{e:pOXV}
p_*\sO_{X'} \otimes_{\sO_X} \sV \iso p_*p^*\sV,
\end{equation}
which reduces to the usual one on the underlying $p_*\sO_{X'}$\nd modules.
The functor from $(H \times_X X')$\nd modules to $(H,p_*\sO_{X'})$\nd modules induced by $p_*$
is fully faithful:
if $l:p_*\sV' \to p_*\sW'$ is an $(H,p_*\sO_{X'})$\nd morphism, then $l = p_*(l')$ for a unique
$\sO_{X'}$\nd morphism $l':\sV' \to \sW'$,
and $l'$ is an $(H \times_X X')$\nd morphism
because $(\varepsilon_p)_{\sV'}$ and $l' \circ (\varepsilon_p)_{\sV'} = (\varepsilon_p)_{\sW'} \circ p^*(l)$ are 
and $(\varepsilon_p)_{\sV'}$ is an epimorphism of $\sO_{X'}$\nd modules.

Let $\sR$ be a finite locally free commutative $\sO_X$\nd algebra.
The trace morphism
\[
\tr_{\sR}:\sR \to \sO_X
\]
of $\sO_X$\nd modules sends a section $s$ of $\sR$ to the trace of multiplication by $s$. 
It is compatible with pullback.
If $\sR$ is an $H$\nd algebra, then $\tr_{\sR}$ is an $H$\nd morphism.

Let $\sR$ be a finite \'etale $\sO_X$\nd algebra. 
Locally in the \'etale topology, $\sR$ is a finite product of copies of $\sO_X$.
Thus if $\sV$ a quasi-coherent $\sR$\nd module and $\sW$
a quasi-coherent $\sO_X$\nd module, then for every $\sO_X$\nd morphism $j:\sV \to \sW$ we have
\[
j = (\tr_{\sR} \otimes_{\sO_X} \sW) \circ l
\]
for a unique $\sR$\nd morphism $l:\sV \to \sR \otimes_{\sO_X} \sW$.
Further if we write $\mu$ for the action of $\sR$ on $\sV$ and $\tau$ for $l$ when $\sW = \sV$ and $j = 1_{\sV}$, then
\begin{equation}\label{e:VRVretr}
\sV \xrightarrow{\tau} \sR \otimes_{\sO_X} \sV \xrightarrow{\mu} \sV
\end{equation}
is the identity.
If $\sR$ is an $H$\nd algebra, $\sV$ is an $(H,\sR)$\nd module, $\sW$ is an $H$\nd module, and $j$ is an 
$H$\nd morphism, then $l$ is an $(H,\sR)$\nd morphism.

\begin{lem}\label{l:Rsummand}
Let $H$ be a pregroupoid over $X$ and $X'$ be a finite \'etale $H$\nd scheme with structural morphism $p:X' \to X$.
Then every $(H \times_X X')$\nd module $\sV'$ is a direct summand of $p^*p_*\sV'$.
\end{lem}

\begin{proof}
Since $p_*$ from $(H \times_X X')$\nd modules to $(H,p_*\sO_{X'})$\nd modules is fully faithful,
we may apply \eqref{e:pOXV} with $\sV = p_*\sV'$ and \eqref{e:VRVretr} with  $\sR = p_*\sO_{X'}$ and $\sV = p_*\sV'$.
\end{proof}

Since pullback preserves limits, the forgetful functor from $H$\nd schemes to schemes over $X$ creates limits.
In particular limits of affine $H$\nd schemes exist and are affine.
Similarly the forgetful functors from $\MOD_H(X)$ to $\MOD(X)$, and from  $H$\nd algebras to quasi-coherent 
$\sO_X$\nd algebras,
create colimits.
If the $d_i:H_{[1]} \to X$ are flat, and at least one of the $e_i:H_{[2]} \to X$ is flat, then the forgetful functor 
from $\MOD_H(X)$ to $\MOD(X)$ creates finite limits, and $\MOD_H(X)$ is abelian.

Call a morphism $h:\sV' \to \sV$ of quasi-coherent $\sO_X$\nd modules universally injective if $h|U \otimes_{\sO_U} \sW$ 
is injective for every open subscheme $U$ of $X$ and quasi-coherent $\sO_U$\nd module $\sW$. 
Any pullback of a universally injective morphism, and any faithfully flat morphism of commutative quasi-coherent 
$\sO_X$\nd algebras, 
is universally injective.

\begin{lem}\label{l:colimH0HV}
Let $H$ be a pregroupoid over $X$ and $(\sV_\lambda)_{\lambda \in \Lambda}$ be a filtered system of $H$\nd modules with colimit $\sV$.
Suppose that $X$ is quasi-compact and that $\sV_\lambda \to \sV$ is universally injective for each $\lambda \in \Lambda$.
Then the canonical map
\[
\colim_{\lambda \in \Lambda} H^0_H(X,\sV_\lambda) \to H^0_H(X,\sV)
\]
is bijective.
\end{lem}

\begin{proof}
It is enough to show that every element $v$ of $H^0_H(X,\sV)$ is the image of an element of some $H^0_H(X,\sV_\lambda)$.
If we regard sheaves over $X$ as local homeomorphisms with target $X$, we may identify $v$ with a cross-section of $\sV$,
and $(\sV_\lambda)$ with a filtered system of open subspaces of $\sV$ with union $\sV$.
Since $X$ is quasi-compact, we have $v^{-1}(\sV_\lambda) = X$ for some $\lambda$. 
Thus $v$ factors through $\sV_\lambda$, and hence is the image of a $v_\lambda$ in $H^0(X,\sV_\lambda)$.
That $v_\lambda$ lies in $H^0_H(X,\sV_\lambda)$, i.e.\ that the action of $H$ sends $d_1{}\!^*(v_\lambda)$ 
to $d_0{}\!^*(v_\lambda)$, follows from the fact that the action of $H$ sends $d_1{}\!^*(v)$ 
to $d_0{}\!^*(v)$,
because $\sV_\lambda \to \sV$ is an $H$\nd morphism and $d_0{}\!^*\sV_\lambda \to d_0{}\!^*\sV$ is injective.
\end{proof}

Given $H$\nd modules $\sV$ and $\sV'$, the assignment to every $H$\nd scheme $X'$ of the abelian group 
$\Hom_{H \times_X X'}(p^*\sV',p^*\sV)$ extends, using the pullback isomorphisms, 
to a functor from $H$\nd schemes to abelian groups.

\begin{lem}\label{l:colimHomR}
Let $H$ be a pregroupoid over $X$ and $(X_\lambda)_{\lambda \in \Lambda}$ be a filtered inverse system of affine $H$\nd schemes with limit $X'$.
Denote by $p:X' \to X$ and $p_\lambda:X_\lambda \to X$ the structural morphisms of $X'$ and $X_\lambda$.
Suppose that $X$ is quasi-compact and that each projection $X' \to X_\lambda$ is faithfully flat.
Then the canonical map
\[
\colim_{\lambda \in \Lambda}\Hom_{H \times_X X_\lambda}(p_\lambda{}\!^*\sV',p_\lambda{}\!^*\sV)
\to \Hom_{H \times_X X'}(p^*\sV',p^*\sV)
\]
is bijective for every $H$\nd module $\sV$ and representation $\sV'$ of $H$.
\end{lem}

\begin{proof}
Replacing $\sV$ by $\sV'{}^\vee \otimes_{\sO_X} \sV$  and $\sV'$ by $\sO_X$, we reduce to showing that
\[
\colim_{\lambda \in \Lambda}H^0_{H \times_X X_\lambda}(X_\lambda,p_\lambda{}\!^*\sV) \to H^0_{H \times_X X'}(X',p^*\sV)
\]
is bijective.
By \eqref{e:Hpushequ} and \eqref{e:pOXV}, this becomes
\[
\colim_{\lambda \in \Lambda}H^0_H(X,p_\lambda{}_*\sO_{X_\lambda} \otimes_{\sO_X} \sV) \to H^0_H(X,p_*\sO_{X'} \otimes_{\sO_X} \sV),
\]
which is bijective by Lemma~\ref{l:colimH0HV}.
\end{proof}

\begin{lem}\label{l:HRrepfp}
Let $H$ be a pregroupoid over $X$ and $(X_\lambda)_{\lambda \in \Lambda}$ be a filtered inverse system of affine $H$\nd schemes with limit $X'$.
Suppose that $X$ is quasi-compact and quasi-separated, that $H_{[1]}$ is quasi-compact,
and that each projection $\pr_\lambda:X ' \to X_\lambda$ is faithfully flat.
Then every representation of $H \times_X X'$ is isomorphic to $(\pr_\lambda)^*\sV$ for some $\lambda \in \Lambda$
and representation $\sV$ of $H \times_X X_\lambda$.
\end{lem}

\begin{proof}
Let $\sV'$ be a representation of $H \times_X X'$.
For some $\lambda_0$ and vector bundle $\sV_0$ over $X_{\lambda_0}$ there exists an $\sO_{X'}$\nd isomorphism 
\[
\iota:\sV' \iso (\pr_{\lambda_0})^*\sV_0
\]
\cite[IV~8.5.2(ii)]{EGA}.
By transport of structure, there is a unique action $\alpha'$ of $H \times_X X'$ on 
$(\pr_{\lambda_0})^*\sV_0$ for which 
$\iota$ is an $(H \times_X X')$\nd isomorphism.
The case $H = X$ of Lemma~\ref{l:colimHomR} with $X$, $(X_\lambda)$, $\sV$ and $\sV'$ replaced by $H_{[1]} \times_X X_{\lambda_0}$,
$(H_{[1]} \times_X X_{\lambda})$ for $\lambda \ge \lambda_0$, $d_0{}\!^*\sV_0$ and $d_1{}\!^*\sV_0$ shows that, 
modulo pullback isomorphisms, $\alpha'$ is of
the form $\pr_\lambda{}\!^*(\alpha)$ for some $\lambda \ge \lambda_0$ and 
$\sO_{(H_{[1]} \times_X X_\lambda)}$\nd morphism 
\[
\alpha:d_1{}\!^*\sV \to d_0{}\!^*\sV,
\]
where $\sV$ is the pullback of $\sV_0$ along $X_\lambda \to X_{\lambda_0}$.
Then $\alpha$ defines a structure of $(H \times_X X_\lambda)$\nd module on $\sV$ such that the pullback isomorphism
from $(\pr_{\lambda_0})^*\sV_0$ to $(\pr_\lambda)^*\sV$ is an $(H \times_X X')$\nd isomorphism:
the condition \eqref{e:alphacomp} for $\alpha$ follows from that for $\alpha'$ by faithful flatness of $\pr_\lambda$.
\end{proof}

\section{Transitive groupoids}\label{s:transgrpd}

This section contains the basic definitions and terminology for the $fpqc$ topology and for transitive affine
groupoids, which are closely related to principal bundles.
It is technically convenient to work here in the category of schemes over an 
arbitrary base, although for the applications the base will be the spectrum of a field.

\begin{defn}
A morphism of schemes $X' \to X$ will be called an \emph{$fpqc$ covering morphism} if every point of $X$ is contained
in an open subscheme $U$ with the following property:
there exists a faithfully flat quasi-compact morphism $U' \to U$ which factors through the restriction of
$X' \to X$ above $U$.
\end{defn}

The property of being $fpqc$ covering is local on the target.
Any pullback of an $fpqc$ covering morphism is $fpqc$ covering.
Any morphism $X' \to X$ through which an $fpqc$ covering morphism $X'' \to X$ factors is $fpqc$ covering.
A morphism $X' \to X$ with $X$ affine is $fpqc$ covering if and only if there exists a faithfully flat morphism $X'' \to X$
with $X''$ affine which factors through it.
The composite of two $fpqc$ covering morphisms is $fpqc$ covering. 
If the pullback of $X' \to X$ along an $fpqc$ covering morphism is $fpqc$ covering, 
then $X' \to X$ is $fpqc$ covering.

Let $X' \to X$ be an $fpqc$ covering morphism.
Then the pullback functors induced by $X' \to X$ from quasi-coherent $\sO_X$\nd modules to quasi-coherent $\sO_{X'}$\nd modules,
and from schemes over $X$ to schemes over $X'$, are both faithful.
A quasi-coherent $\sO_X$\nd module is a vector bundle if and only its pullback onto $X'$ is, 
and a scheme over $X$ is for example affine or faithfully flat if and only if its pullback onto $X'$ is.

Morphisms of any of the following three types are $fpqc$ covering:
\begin{enumerate}
\renewcommand{\theenumi}{(\Alph{enumi})}
\item\label{i:surjopen}
surjective morphisms $u$ whose source can be decomposed as a disjoint union of open subschemes $U$ with $u|U$ an open immersion,
\item\label{i:ffqc}
faithfully flat quasi-compact morphisms,
\item\label{i:retr}
retractions.
\end{enumerate}
For $p:X' \to X$ $fpqc$ covering, there is a composite $X'' \to X$ of a morphism of type 
\ref{i:surjopen} with one of type \ref{i:ffqc} which factors through $p$.
Taking the fibre product of $X'$ and $X''$ over $X$ we obtain a commutative square with one side $p$,
two sides the composite of a morphism of type \ref{i:surjopen} with one of type \ref{i:ffqc},
and one side of type \ref{i:retr}. 
Thus if $p$ is $fpqc$ is covering, then
\begin{equation}\label{e:fhh}
p \circ p' = p''
\end{equation}
for appropriate composites $p'$ and $p''$ of morphisms of type \ref{i:surjopen}, \ref{i:ffqc} or \ref{i:retr}.

Recall that $X' \to X$ is said to  be an effective epimorphism if it is the coequaliser of a pair of morphisms,
which may be taken to be the projections from $X' \times_X X'$.
Universal effective epimorphisms are stable under composition, and $p$ is an effective epimorphism 
if $p \circ p'$ is.
Any morphism of type \ref{i:surjopen}, \ref{i:ffqc} \cite[VIII~5.3]{SGA1} or \ref{i:retr},
and hence any $fpqc$ covering morphism, is a universal effective epimorphism.

For the rest of this section we fix a base scheme $S$, and $X$ will always be a scheme over $S$.
Unless otherwise stated, pregroupoids and groupoids will be in the category of schemes over $S$.
Thus $[X]$ denotes the groupoid over $X$ with $[X]_{[n]}$ the product of $n+1$ copies of $X$ over $S$.
The pullback of a pregroupoid $H$ over $X$ along $X' \to X$ is then the fibre product $H \times_{[X]} [X']$.

A groupoid $K$ over $X$ will be called affine if the morphism
\[
(d_0,d_1):K \to X \times_S X
\]
is affine.
It will be called \emph{transitive} if both the structural morphism $X \to S$ of $X$ and $(d_0,d_1)$ 
are $fpqc$ covering morphisms.

Let $K$ be a groupoid over $X$ and $h:Z \to K$ be a morphism over $S$ from a scheme $Z$ over $S$.
Write $K_i$ for the pullback of $K$ along $d_i \circ h:Z \to X$, and $q_i:K_i \to K$ for the projection. 
Then we have an isomorphism of groupoids over $Z$
\begin{equation}\label{e:Kconjiso}
\varphi_{K,h}:K_1 \iso K_0
\end{equation}
with $q_0(\varphi_{K,h}(v)) = h(d_0(v)) \circ q_1(v) \circ h(d_1(v))^{-1}$.
It is natural in $(K,h)$ for fixed $Z$.
For $Y$ a $K$\nd scheme, pullback of the action of $K$ on $Y$ along $h$ defines an isomorphism
\begin{equation}\label{e:conjisorep}
q_1{}\!^*Y \iso \varphi_{K,h}{}\!^*q_0{}\!^*Y
\end{equation}
of $K_1$\nd schemes which is natural in $Y$,
and similarly for $K$\nd modules.

Let $K$ be a transitive groupoid over $X$ and $p:X' \to X$ be a morphism
with $X' \to S$ $fpqc$ covering.
Then there exists an $h:Z\to K$ such that
\begin{equation}\label{e:dhba}
(d_0 \circ h,d_1 \circ h) = (c,p \circ c')
\end{equation}
with $c:Z \to X$ and $c':Z \to X'$ $fpqc$ covering:
take for $Z$ the fibre product of $K$ with $X \times_S X'$ over $X \times_S X$,
and for $h$ and $(c,c')$ the projections.

\begin{lem}\label{l:morphpull}
Let $K$ be a transitive groupoid over $X$ and $X' \to X$ be a morphism with $X' \to S$ $fpqc$ covering.
\begin{enumerate}
\item\label{i:morphpullgpd}
A morphism $f:K \to K'$ of groupoids over $X$ is an isomorphism  (resp.\ $fpqc$ covering, resp.\ faithfully flat,
resp.\ a closed immersion,
resp.\ of finite presentation, resp.\  affine, resp.\ finite \'etale) if and only if $f \times_{[X]} [X']$ is.
\item\label{i:morphpullsch}
A morphism $j:Y \to Y'$ of $K$\nd schemes  is an isomorphism  (resp.\ $fpqc$ covering, resp.\ faithfully flat, 
resp.\ a closed immersion,
resp.\ of finite presentation, resp.\  affine, resp.\ finite \'etale) if and only if $j \times_X X'$ is.
\end{enumerate}
\end{lem}

\begin{proof}
Write $p$ for $X' \to X$.
Both \ref{i:morphpullgpd} and \ref{i:morphpullsch} are clear when $p$ is $fpqc$ covering.
Let $h$ be such that \eqref{e:dhba} holds with $c$ and $c'$ $fpqc$ covering.
Then by naturality of \eqref{e:Kconjiso} applied to $f$, and of \eqref{e:conjisorep}
applied to $j$, we may replace $p$ by $c$.
\end{proof}

Since an $fpqc$ covering morphism is a universal effective epimorphism, we have descent
for morphisms of schemes along any such morphism. 
That we have descent for affine schemes or quasi-coherent sheaves along any $fpqc$ covering morphism 
will follow by Example~\ref{ex:descent} with $X'$ for $X$ from
Lemma~\ref{l:prereppull} below with $H = X = S$.

\begin{lem}\label{l:prereppull}
Let $H$ be a pregroupoid over $X$ and $H'$ be the pullback of $H$ along a morphism $X' \to X$.
Suppose that either of the following conditions holds:
\begin{enumerate}
\renewcommand{\theenumi}{(\alph{enumi})}
\item\label{i:prereppullcov}
$X' \to X$ is $fpqc$ covering;
\item\label{i:prereppulltrans}
$X' \to S$ is $fpqc$ covering and $H$ is a transitive groupoid over $X$.
\end{enumerate}
Then pullback along the projection $H' \to H$ defines an equivalence from the category of affine $H$\nd schemes (resp.\ $H$\nd modules, 
resp.\ representations of $H$) to the category of affine $H'$\nd schemes (resp.\ $H'$\nd modules, resp.\ representations of $H'$).
\end{lem}

\begin{proof}
Write $p$ for $X' \to X$.
We may suppose that \ref{i:prereppullcov} holds:
the case where \ref{i:prereppulltrans} holds reduces to that where \ref{i:prereppullcov} holds
by taking $h$ as in \eqref{e:dhba} with $c$ and $c'$ $fpqc$ covering, 
and using \eqref{e:conjisorep} or its analogue for $K$\nd modules.
By \eqref{e:fhh},
we may further suppose that $p$ is of type \ref{i:surjopen}, \ref{i:ffqc} or \ref{i:retr}.

Suppose first that $H = X$, so that $H$ and $H'$ may be regarded as groupoids in the category of schemes over $X$,
and it is to be shown that we have descent along $p$.
In case \ref{i:surjopen}, this follows from gluing along open subschemes.
In case \ref{i:ffqc}, it follows from \cite[VIII~1.1 and 2.1]{SGA1}.
For \ref{i:retr}, note that the functor defined by a right inverse $b$ to $p$ 
is quasi-inverse to that defined by $p$,
because by \eqref{e:conjisorep} with $X$, $X'$, $X'$, $H'$ and $(1_{X'},b \circ p)$ for $S$, $X$, $Z$, $K$ and $h$,
or its analogue for $K$\nd modules, the endofunctor defined by $b \circ p$ is an equivalence.

For $H$ arbitrary, 
the morphism $X \to H$ defines a cartesian square of pregroupoids
\[
\xymatrix{
X \times_{[X]} [X'] \ar[d] \ar[r] &   H' \ar[d] \\
X      \ar[r]    &  H
}
\]
with vertical arrows the projections.
Write $d'{}\!_i$ instead of $d_i$ for the morphisms $H'{}\!_{[n+1]} \to H'{}\!_{[n]}$,
and $q:H'{}\!_{[1]} \to H_{[1]}$ and $u:(X \times_{[X]} [X'])_{[1]} \to H'{}\!_{[1]}$
for the respective components of the right and top arrows at $[1]$.
We write the proof for the case of affine $H$\nd schemes: the other cases are similar. 
By assumption, $p$, and hence also $q$, is of type \ref{i:surjopen}, \ref{i:ffqc} or \ref{i:retr}.
By the case where $H = X$ with $q$ for $p$, we thus have descent for affine schemes along $q$.

Let $Y$ and $Y'$ be affine $H$\nd schemes.
Since $q$ is $fpqc$ covering, any morphism of $Y \to Y'$ of schemes over $X$ with $p^*Y \to p^*Y'$
a morphism of $H'$\nd schemes is a morphism of $H$\nd schemes.
By the case where $H = X$, the required full faithfulness thus follows from the square.

By the case where $H = X$, the pullback of any affine $H'$\nd scheme along the top arrow of the square
is isomorphic to the pullback of an affine scheme over $X$ along the left arrow.
For the required essential surjectivity, it is thus enough to prove the following:
given an affine scheme $Y$ over $X$ and a structure of $H'$\nd scheme
\[
\alpha':d'{}\!_1{}\!^*p^*Y \iso d'{}\!_0{}\!^*p^*Y
\]
on $p^*Y$ with $u^*(\alpha')$ the pullback isomorphism, 
there exists an isomorphism
\[
\alpha:d_1^{}\!^*Y \iso d_0^{}\!^*Y
\]
of schemes over $H_{[1]}$ such that $\alpha'$ coincides, modulo the pullback isomorphisms, with $q^*(\alpha)$.
Indeed since $H'{}\!_{[2]} \to H_{[2]}$ is $fpqc$ covering, the required condition \eqref{e:alphacomp} for $\alpha$ will follow from that for $\alpha'$.
By descent along $q$,
such an $\alpha$ will exist provided that
for every scheme $Z$ over $S$ and morphism $h:Z \to H_{[1]}$ over $S$ the following condition is satisfied:
if
\[
h' = (h,x'{}\!_1,x'{}\!_0):Z \to H'{}\!_{[1]}
\]
is a morphism over $S$ with $q \circ h' = h$, then the isomorphism
\[
\alpha'{}\!_{h'}:(d_1 \circ h)^*Y \iso (d_0 \circ h)^*Y
\]
defined modulo the pullback isomorphisms by $h'{}^*(\alpha')$ is independent of $h'$.

If $x'$ and $\widetilde{x}'$ are morphisms $Z \to X'$ over $S$ with $p \circ x' = x = p \circ \widetilde{x}'$, 
then $\alpha'{}\!_{(s_0 \circ x,\widetilde{x}',x')}$ is the identity, 
because by hypothesis $u^*(\alpha')$ is the pullback isomorphism.
Also if $j'$ is a morphism $Z \to H'{}\!_{[2]}$ over $S$, then by \eqref{e:alphacomp} with $\alpha'$ for $\alpha$
\[
\alpha'{}\!_{d_0 \circ j'} \circ \alpha'{}\!_{d_2 \circ j'} = \alpha'{}\!_{d_1 \circ j'}.
\]
Taking $(s_0 \circ h,x'{}\!_1,\widetilde{x}'{}\!_0,x'{}\!_0)$ for $j'$ shows that $\alpha'{}\!_{h'}$ 
does not depend on $x'{}\!_0$,
and taking $(s_1 \circ h,\widetilde{x}'{}\!_1,x'{}\!_1,x'{}\!_0)$ for $j'$ shows that it does not depend on $x'{}\!_1$.
\end{proof}

Let $H$ be a pregroupoid over $X$.
An $H$\nd scheme will be called \emph{constant} if it is isomorphic to the pullback $T \times_S X$ of a scheme $T$ over $S$ 
along $H \to S$.
Similarly we define for example constant $H$\nd groups or constant $H$\nd modules.
Suppose that $X \to S$ is $fpqc$ covering.
Then since by Lemma~\ref{l:prereppull} pullback along $[X] \to S$ defines an equivalence from affine schemes over $S$ 
to affine $[X]$\nd schemes, an affine $H$\nd scheme $X'$ is constant if and only if the action of $H$ on $X'$
factors through an action of $[X]$.
More precisely an isomorphism $X' \iso T \times_S X$ of affine $H$\nd schemes is the same up to isomorphism of schemes over $S$
as a factorisation of the action of $H$ on $X'$ through $[X]$.
When $H$ is a transitive groupoid, such a factorisation is unique if it exists.

A groupoid $K$ over $X$ will be called \emph{commutative} if the group scheme $K^\mathrm{diag}$ over $X$ is commutative.
If $K$ is a commutative transitive affine groupoid over $X$, then the $K$\nd group $K^\mathrm{diag}$ is constant.

A groupoid $K$ over $X$ will be called \emph{constant} if it is isomorphic to the pullback $G \times_S [X]$ of 
a group scheme $G$ over $S$ along $X \to S$.
An isomorphism $i:G \times_S [X] \iso K$ of groupoids over $X$ is the same as a pair $(j,\iota)$ with
$j:[X] \to K$ a structure of groupoid over $[X]$ on $K$ and $\iota:G \times_S X \iso j^*K^{\mathrm{diag}}$ an isomorphism of 
$[X]$\nd groups:
\[
i(g,x_1,x_0) = \iota(g,x_1) \circ j(x_1,x_0) = j(x_1,x_0) \circ \iota(g,x_0).
\] 
Suppose that $X \to S$ is $fpqc$ covering.
Then any constant groupoid over $X$ is transitive.
Further by Lemma~\ref{l:prereppull},
an affine groupoid $K$ over $X$ is constant if and only if it has a structure of groupoid over $[X]$.
Given such a structure, there exists, uniquely up to unique isomorphism, a pair consisting of an affine group
scheme $G$ over $S$ and an isomorphism $G \times_S [X] \iso K$ of groupoids over $[X]$.

Let $K$ be a groupoid over $X$.
If $Z$ is a $K$\nd scheme and $z:X \to Z$ is a cross-section of $Z$ over $X$, then the points $v$ of $K$
such that $v$ sends $z(d_1(v))$ to $z(d_0(v))$ are those of a subgroupoid of $K$,
which may be identified with the pullback of the groupoid $K \times_X Z$ over $Z$ along $z$. 
We call this subgroupoid the \emph{stabiliser} of $z$.

Let $K$ be a transitive groupoid over $X$ and $Y$ be a $K$\nd scheme.
By Lemma~\ref{l:morphpull}\ref{i:morphpullsch} with $Y$ for $X'$ and $X$ for $Y'$,
the morphism $Y \to X$ is $fpqc$ covering if and only if $Y \to S$ is.
The $K$\nd scheme $Y$ will be called \emph{transitive} if the groupoid $K \times_X Y$ over $Y$ is transitive,
and \emph{simply transitive} if further
\[
K \times_X Y = [Y].
\]
For any morphism $X' \to X$ with $X' \to S$ $fpqc$ covering, $Y$ is a transitive (resp.\ simply transitive) $K$\nd scheme
if and only if $Y\times_X X'$ is a transitive (resp.\ simply transitive) $(K \times_{[X]} [X'])$\nd scheme,
by Lemma~\ref{l:morphpull}\ref{i:morphpullsch} with $X$ for $Y'$ and Lemma~\ref{l:morphpull}\ref{i:morphpullgpd} with the projection $Y \times_X X' \to Y$ for $X' \to X$ and
$K \times_X Y \to [Y]$ for $f$.
If $Y$ is a transitive (resp.\ simply transitive) $K$\nd scheme then $Y_{S'}$ is a 
transitive (resp.\ simply transitive) $K_{S'}$\nd scheme for any $S'$ over $S$, and the converse holds when $S' \to S$
is $fpqc$ covering.

Let $K$ be a transitive affine groupoid over $X$, and $Y$ be a simply transitive $K$\nd scheme.
If $Y$ has a cross-section $y$, then for some affine group scheme $G$ over $S$ the pair $(K,Y)$ is
isomorphic to the pullback $(G \times_S [X],G \times_S X)$ along $X \to S$ of $(G,G)$ with $G$ acting on itself by translation,
because the stabiliser of $y$ is a subgroupoid $[X]$ of $K$ for which $y$ is an $[X]$\nd morphism.
Since the pullback of $Y \to X$ along itself has a cross-section, 
it follows for example that any simply transitive $K$\nd scheme is affine,
and that any morphism between simply transitive $K$\nd schemes is an isomorphism.

Any transitive affine groupoid $K$ over $X$ becomes constant after base change along an appropriate $fpqc$ covering morphism
$S' \to S$ and then pullback along an appropriate $fpqc$ covering morphism $X' \to X_{S'}$.
Indeed we may take $S' = X$, so that $X_{S'}$ over $S'$ has a cross-section $t$ and hence $K_{S'}$ has by pullback along $t$
using Lemma~\ref{l:prereppull} a simply transitive $K_{S'}$\nd scheme $Z$, 
and then $X' = Z$, so that $Z \times_{X_{S'}} X'$ over $X'$ has a cross-section.

\begin{lem}\label{l:isorep}
Let $K$ be a transitive affine groupoid over $X$,
$Y$ be simply transitive $K$\nd scheme, and $Z$ be an affine $K$\nd scheme with $Z \to S$ $fpqc$ covering.
Then the functor on the category of schemes over $S$ that sends $S'$ to the set of $K_{S'}$\nd morphisms from $Y_{S'}$ 
to $Z_{S'}$ is representable by a scheme which is affine over $S$ with structural morphism $fpqc$ covering.
\end{lem}

\begin{proof}
By Lemma~\ref{l:prereppull}, we reduce after replacing $X$ by $Y$ and $K$, $Y$ and $Z$ by their pullbacks onto
$Y$ to the case where $(K,Y)$ is $(G \times_S [X],G \times_S X)$ for some $G$, and then to the case where further 
$X = S$. 
Then $Z$ represents the required  functor.
\end{proof}

Let $K$ be a transitive groupoid over $X$ and $K'$ be a transitive subgroupoid of $K$ over $X$.
A pair $(Y,y)$ consisting of a transitive $K$\nd scheme $Y$ and a section $y$ of $Y$ over $X$ with
stabiliser $K'$ will be called a quotient of $K$ by $K'$.
If $Z$ is a $K$\nd scheme and $z$ is section of $Z$ over $X$ whose stabiliser contains $K'$, we have
a commutative square
\[
\xymatrix{
K \times_X K'  \ar[d] \ar[r] &  K \ar[d]^{p} \\
K    \ar[r]^p    &  Z
}
\]
with $p$ the restriction of the action of $K$ on $Z$ to the subscheme $K = K \times_X X$ of $K \times_X Z$ defined by $z$, 
and the top and left arrows the composition and the first projection.
When $(Z,z) = (Y,y)$, the square is cartesian, and $p$ is an $fpqc$ covering morphism
and hence an effective epimorphism, so that $p$ is the coequaliser of the top and left arrows.
Thus $(Y,y)$ is initial in the category of pairs $(Z,z)$.
In particular $(Y,y)$ is unique up to unique isomorphism.
We write $K/K'$ for $Y$ and call $y$ the base cross-section of $K/K'$.
When $X = S$, the notion of quotient $K/K'$ reduces to the usual one for group schemes over $S$.

The quotient $K/K'$ is preserved when it exists by base extension and pullback.
Conversely if the quotient of either the base extensions of $K$ and $K'$ along an $fpqc$ covering morphism 
or their pullbacks along a morphism $X' \to X$ with $X' \to S$ $fpqc$ covering exists and is affine,
then $K/K'$ exists and is affine, by descent for affine schemes along an $fpqc$ covering morphism and Lemma~\ref{l:prereppull}.

Let $K$ be a groupoid over $X$.
By the kernel of a morphism $K \to K_1$ of groupoids over $X$ we mean the kernel of the induced morphism on diagonals.
It is a $K$\nd subgroup of $K^\mathrm{diag}$.
By a quotient of $K$ by a $K$\nd subgroup $N$ of $K^\mathrm{diag}$ we mean a morphism 
$K \to K_1$ of groupoids over $X$ which is $fpqc$ covering
and has kernel $N$.
For any morphism $K \to K_1$ of groupoids over $X$ which is trivial on $N$ we have a commutative square
of groupoids over $X$
\[
\xymatrix{
N \rtimes_X K \ar[d] \ar[r] & K \ar[d] \\
K \ar[r] & K_1
}
\]
with top arrow the projection and left arrow given by $(n,v) \mapsto n \circ v$.
It is cartesian when $K \to K_1$ has kernel $N$.
A quotient of $K$ by $N$ is the coequaliser of the left and top arrows of the square,
and hence is universal among morphisms $K \to K_1$ trivial on $N$.
Such a quotient is thus unique up to unique isomorphism,
and will be written $K \to K/N$.
When it exists, a quotient by a $K$\nd subgroup of the diagonal is preserved by base extension and pullback.

\section{Torsors}

This section and the two following contain the foundational material that will be required on torsors
and principal bundles.
The notion of a torsor under an affine group scheme is an absolute one.
When working over a base scheme, principal bundles are the torsors under constant affine group schemes.

Let $X$ be a scheme and $J$ be an affine group scheme over $X$.
An action of $J^\mathrm{op}$ on a scheme $P$ over $X$ will also be called a right action of $J$ on $P$.
Thus a right action of $J$ on $P$ is a morphism
\begin{equation}\label{e:Jact}
P \times_X J \to P
\end{equation}
over $X$ which satisfies the usual unit and associativity conditions.
A scheme over $X$ equipped with a right action $J$ will be called a \emph{right $J$\nd scheme}.
A morphism of right $J$\nd schemes is a morphism of the underlying schemes over $X$
which is compatible with the action of $J$.
The composition of $J$ defines a structure of right $J$\nd scheme on $J$.
A right $J$\nd scheme will be called
\emph{trivial} if it is isomorphic to $J$.
Pullback along a morphism $X' \to X$ defines a functor $X' \times_X -$ from right $J$\nd schemes to right 
$(X' \times_X J)$\nd schemes.
By evaluation at the identity, a morphism $X' \times_X J \to X' \times_X P$ of right $(X' \times_X J)$\nd schemes 
may be identified with a morphism $X' \to P$ of schemes over $X$.
Taking $X' = P$, we have a morphism of right $(P \times_X J)$\nd schemes
\begin{equation}\label{e:PJmorph}
P \times_X J \to P \times_X P 
\end{equation}
for every right $J$\nd scheme $P$, corresponding to the identity of $P$, with first component the first projection and second component the action \eqref{e:Jact}.

Let $u:J \to J'$ is a morphism of affine group schemes over $X$.
Then we have a restriction functor along $u$ from right $J'$\nd schemes to right $J$\nd schemes.
If $P$ is a $J$\nd scheme and $P'$ is a $J'$\nd scheme, then a morphism $q:P \to P'$ of schemes over $X$
will be said to be \emph{compatible with $u$} if it is a morphism of $J$\nd schemes from $P$ to the restriction of $P'$
along $u$.
It is equivalent to require that the action of $J'$ on $P'$ composed with $q \times u$
 coincide with $q$ composed with the action of $J$ on $P$.
A morphism of right $J$\nd schemes is the same as a morphism of the underlying schemes over $X$ 
compatible with $1_J$.

A \emph{$J$\nd torsor} is a right $J$\nd scheme $P$ for which there exists an $fpqc$ covering morphism $X' \to X$
such that the right $(X' \times_X J)$\nd scheme $X' \times_X P$ is trivial.
Any morphism of $J$\nd torsors is an isomorphism.
A right $J$\nd scheme $P$ is a $J$\nd torsor if and only if $P \to X$ is $fpqc$ covering
and \eqref{e:PJmorph} is an isomorphism.
We write 
\[
H^1(X,J)
\]
for the set of isomorphism classes of $J$\nd torsors.
It is a pointed set with base point the class of the trivial $J$\nd torsors.

Define a category of \emph{affine torsors over $X$} as follows: an object is a pair $(J,P)$ with $J$ an affine group scheme over 
$X$ and $P$ a $J$\nd torsor, and a morphism from $(J,P)$ to $(J',P')$ is a pair $(u,q)$ with $u$ a morphism from $J$ to $J'$
and $q$ a morphism from $P$ to $P'$ over $X$ compatible with $u$.

\begin{lem}\label{l:push}
Let $J$ be an affine group scheme over $X$ and $P$ be a $J$\nd torsor.
Then the forgetful functor from the category of affine torsors over $X$ equipped with a morphism from $(J,P)$
to the category of affine group schemes over $X$ equipped with a morphism from $J$ is a surjective equivalence.
\end{lem}

\begin{proof}
To prove the full faithfulness, we reduce by pullback along an $fpqc$ covering morphism $p$ which trivialises $P$
and descent along $p$ to the case where $P = J$, which is clear.

By the full faithfulness, a morphism $(u,q)$ from $(J,P)$ with given $u$ is unique up to unique isomorphism if it exists.
To prove the surjectivity on objects,  we thus reduce again by descent for affine schemes to the case where $P = J$.
\end{proof}

Let $P$ be a $J$\nd torsor and $u:J \to J'$ be a morphism of group schemes over $X$.
By Lemma~\ref{l:push}, a pair consisting of a $J'$\nd torsor $P'$ and a morphism $P \to P'$ compatible with $u$ exists,
and is unique up to unique $J'$\nd isomorphism.
We say that $P'$ is the \emph{push forward of $P$ along $u$}.
The push forward of $P$ along an inner automorphism of $J$ is $P$ itself, because
if $j$ is a cross-section of $J$, then the action $P \to P$ of $j^{-1}$ is compatible with 
conjugation $J \to J$ by $j$.

We have a functor $H^0(X,-)$ from affine group schemes over $X$ to groups, defined by taking cross-sections.
By assigning to the morphism $u$ the map defined by push forward along $u$,
we also have a functor $H^1(X,-)$ from affine group schemes over $X$ to pointed sets.
It preserves finite products, so that $H^1(X,J)$ has a structure of abelian group when $J$ is commutative.
It also sends inner automorphisms of $J$ to the identity of $H^1(X,J)$.
Thus we may regard $H^1(X,-)$ as a functor
on the category of affine group schemes over $X$ up to conjugacy, where a morphism from $J'$ to $J$ is an orbit
in the set of morphisms $J' \to J$ of affine 
group schemes over $X$ under the action by composition of the group of inner automorphisms of $J$.
We have the usual exact cohomology sequence for $H^0(X,-)$ and $H^1(X,-)$, functorial in short exact sequences
of affine group schemes over $X$ 
\[
1 \to J'' \to J \to J' \to 1,
\]
which are defined by requiring $J \to J'$ to be $fpqc$ covering with kernel $J'' \to J$.

For the rest of this section we fix a base scheme $S$, and $X$ will be a scheme over $S$. 
Pregroupoids and groupoids will be in the category of schemes over $S$.

Let $H$ be a pregroupoid over $X$ and $J$ be an affine $H$\nd group.
By an \emph{$(H,J)$\nd scheme} we mean an $H$\nd scheme $P$ together with a structure of right $J$\nd scheme on $P$
for which the action \eqref{e:Jact} of $J$ on $P$ is a morphism of $H$\nd schemes.
An equivalent formulation of the last condition is that 
with $d_i{}\!^*P$ regarded as a right $d_i{}\!^*J$\nd scheme, the isomorphism $d_1{}\!^*P \iso d_0{}\!^*P$  over $H_{[1]}$
defining the $H$\nd structure of $P$ should be compatible with the isomorphism $d_1{}\!^*J \iso d_0{}\!^*J$ over $H_{[1]}$ 
defining the $H$\nd structure of $J$.
A morphism $P \to P'$ of $(H,J)$\nd schemes is a morphism of schemes over $X$ which is 
a morphism of $H$\nd schemes and of right $J$\nd schemes.
An $(H,J)$\nd scheme will be called an \emph{$(H,J)$\nd torsor} if its underlying right $J$\nd scheme is a $J$\nd torsor.
Any morphism of $(H,J)$\nd torsors is an isomorphism.
We write
\[
H^1_H(X,J)
\]
for the set of isomorphism classes of $(H,J)$\nd torsors.
It is a pointed set with base point the class of the trivial $(H,J)$\nd torsors, i.e.\ those $(H,J)$\nd isomorphic to $J$.

\begin{lem}\label{l:Hpush}
Let $H$ be a pregroupoid over $X$ and $u:J \to J'$ be a morphism of affine $H$\nd groups.
Let $P$ be an $(H,J)$\nd torsor, $P'$ be a $J'$\nd torsor,
and $q:P \to P'$ be a morphism over $X$ compatible with $u$.
Then there is a unique structure of $H$\nd scheme on $P'$ such that $P'$ is an $(H,J')$\nd torsor
and $q$ is a morphism of $H$\nd schemes.
\end{lem}

\begin{proof}
Apply Lemma~\ref{l:push} to the pullbacks of $P$ and $P'$ onto the $H_{[i]}$.
\end{proof}

Let $P$ be an $(H,J)$\nd torsor and $u:J \to J'$ be a morphism of affine $H$\nd groups.
By Lemmas~\ref{l:push} and \ref{l:Hpush}, a pair consisting of an $(H,J')$\nd torsor $P'$ and an $H$\nd morphism 
$P \to P'$ compatible with $u$ exists,
and is unique up to unique isomorphism of $(H,J)$\nd schemes.
We say that $P'$ is the \emph{push forward of $P$ along $u$}.

Call a cross-section $z$ of an $H$\nd scheme $Z$ $H$\nd invariant if the action of $H$ on $Z$ sends $d_1{}\!^*z$
to $d_0{}\!^*z$.
The inner automorphism induced by an $H$\nd invariant cross-section of an $H$\nd group is an automorphism of $H$\nd groups,
and will be called an $H$\nd inner automorphism.
The push forward of an $(H,J)$\nd torsor $P$ along an $H$\nd inner automorphism of $J$ is $P$ itself.

We have a functor $H^0_H(X,-)$ from affine $H$\nd groups to groups, defined by taking $H$\nd invariant cross-sections.
By assigning to the morphism $u$ the map defined by push forward along $u$,
we also have a functor $H^1_H(X,-)$ from affine $H$\nd groups to pointed sets.
It preserves finite products, so that $H^1_H(X,J)$ has a structure of abelian group when $J$ is commutative.
It also sends $H$\nd inner automorphisms of $J$ to the identity of $H^1(X,J)$.
Thus we may regard $H^1(X,-)$ as a functor
on the category of affine $H$\nd groups up to conjugacy, where a morphism from $J'$ to $J$ is an orbit 
under the action by composition of the group of $H$\nd inner automorphisms of $J$ on the set of morphisms $J' \to J$ of affine 
$H$\nd groups.
Again we have an exact cohomology sequence for $H^0_H(X,-)$ and $H^1_H(X,-)$, functorial in short exact sequences of 
affine $H$\nd groups.

When $H = X$, the notions of $(H,J)$\nd scheme and $(H,J)$\nd torsor reduce
to those of right $J$\nd scheme and $J$\nd torsor, and $H^i_H(X,-)$ coincides with $H^i(X,-)$.

\begin{lem}\label{l:schemetorspull}
Let $H$ and $X'$ be as in Lemma~\textnormal{\ref{l:prereppull}}.
Let $J$ be an affine $H$\nd group and $J'$ be its pullback along $X' \to X$.
Then an $(H,J)$\nd scheme is a $J$\nd torsor 
if and only if its pullback along $X' \to X$ is a $J'$\nd torsor.
\end{lem}

\begin{proof}
Let $P$ be an $(H,J)$\nd scheme.
It to be shown
that $P \to X$ is $fpqc$\nd covering if and only if its pullback along $X' \to X$ is,
and that \eqref{e:PJmorph} an isomorphism if and only if its pullback along $X' \to X$ is.
This is clear when \ref{i:prereppullcov} of Lemma~\ref{l:prereppull} holds, and when \ref{i:prereppulltrans}
of Lemma~\ref{l:prereppull} holds
it follows from Lemma~\ref{l:morphpull}\ref{i:morphpullsch} with $K = H$,
because \eqref{e:PJmorph} is a morphism of $H$\nd schemes.
\end{proof}

\begin{lem}\label{l:torspull}
Let $H$, $X'$ and $H'$ be as in Lemma~\textnormal{\ref{l:prereppull}}.
Let $J$ be an affine $H$\nd group and $J'$ be its pullback along $X' \to X$.
Then pullback along $X' \to X$ induces an equivalence from the category of $(H,J)$\nd torsors
to the category of $(H',J')$\nd torsors.
\end{lem}

\begin{proof}
By Lemma~\ref{l:prereppull}, pullback along $X' \to X$ induces an equivalence from
affine $(H,J)$\nd schemes to affine $(H',J')$\nd schemes.
It thus suffices to apply Lemma~\ref{l:schemetorspull}.
\end{proof}

\begin{rem}\label{r:simpletrantorsor}
Let $K$ be an affine groupoid over $X$.
Any $K$\nd scheme has a canonical structure of $(K,K^\mathrm{diag})$\nd scheme, with the right action 
of the point $v$ of $K^\mathrm{diag}$ defined as the action of the point $v^{-1}$ of $K$.
Suppose that $K$ is transitive.
Then a $K$\nd scheme is simply transitive if and only if it is a $K^\mathrm{diag}$\nd torsor.
This can be seen by reducing after base change and pullback using Lemma~\ref{l:schemetorspull} to the case where 
$X = S$.
\end{rem}

\section{Principal bundles}\label{s:prinbun}

This section contains foundational material on principal bundles, and their
connection with transitive affine groupoids.
We fix a base scheme $S$, and $X$ is a scheme over $S$ whose structural morphism is $fpqc$ covering.
Pregroupoids and groupoids are in the category of schemes over $S$, unless otherwise indicated.

Let $G$ be an affine group scheme over $S$.
By a right $G$\nd scheme over $X$ we mean a scheme $P$ over $X$ with a structure of right $G$\nd scheme
on $P$ over $S$ such that the action 
\begin{equation}\label{e:Gact}
P \times_S G \to P
\end{equation}
is a morphism over $X$.
A morphism of right $G$\nd schemes over $X$ is a morphism of right $G$\nd schemes which is a morphism over $X$.
A right $G$\nd scheme over $X$ is the same as a right $(G \times_S X)$\nd scheme.
It is also the same as a right $G$\nd scheme equipped with a morphism of right $G$\nd schemes to 
$X$ with the trivial right action of $G$. 
Pullback along a morphism $X' \to X$ defines a functor from right $G$\nd schemes over $X$ to right $G$\nd schemes over $X'$.

By a \emph{principal $G$\nd bundle over $X$} we mean a right $G$\nd scheme over $X$ which is a 
$(G \times_S X)$\nd torsor.
Any morphism of principal $G$\nd bundles over $X$ is an isomorphism.
A right $G$\nd scheme $P$ over $X$ is a principal $G$\nd bundle over $X$ if and only if $P \to X$
is $fpqc$ covering and the morphism from $P \times_S G$ to $P \times_X P$ over $P$ that sends $(z,g)$
to $(z,zg)$ is an isomorphism.
The pointed set $H^1(X,G \times_S X)$ of isomorphism classes of principal $G$\nd bundles over $X$ will usually be written simply as
\[
H^1(X,G),
\]
and the group $H^0(X,G \times_S X)$ of cross-sections of $G \times_S X$ as $H^0(X,G)$.

We have push forward of principal $G$\nd bundles over $X$ along a morphism $G \to G'$ of affine group schemes over $X$, defined by push forward of $G \times_S X$\nd torsors.
Using push forward we define a functor $H^1(X,-)$ on affine group schemes over $S$,
which factors through affine group schemes over $S$ up to conjugacy.

Let $H$ be a pregroupoid over $X$.
Given a scheme $P$ over $X$ equipped with a structure of $H$\nd scheme and a structure of 
right $G$\nd scheme over $X$,
it is equivalent to require that the action $d_1{}\!^*P \iso d_0{}\!^*P$ of $H$ on $P$ be a morphism of right $G$\nd schemes over 
$H_{[1]}$, 
or that \eqref{e:Gact} be a morphism of $H$\nd schemes for the action of $H$ on $P \times_S G$ through $P$, 
or that the two morphism
\[
H_{[1]} \times_X P \times_S G \to P
\]
defined by the actions by first factoring through $H_{[1]} \times_X P$ or $P \times_S G$ should coincide.
We then say that $P$ is an \emph{$(H,G)$\nd scheme}.
A morphism of $(H,G)$\nd schemes is a morphism of right $G$\nd schemes which is a morphism of $H$\nd schemes.
An $(H,G)$\nd scheme is the same as an $(H,G \times_S X)$\nd scheme with $G \times_S X$ the constant $H$\nd group.

By a \emph{principal $(H,G)$\nd bundle} we mean an $(H,G)$\nd scheme whose underlying right $G$\nd scheme
is a principal $G$\nd bundle over $X$. 
A principal $(H,G)$\nd bundle is the same as an $(H,G \times_S X)$\nd torsor.
The pointed set $H^1_H(X,G \times_S X)$ of isomorphism classes of principal $(H,G)$\nd bundles over $X$ 
will usually be written simply as
\[
H^1_H(X,G),
\]
and the group $H^0_H(X,G \times_S X)$ of $H$\nd invariant cross-sections of $G \times_S X$ as $H^0_H(X,G)$.

We have push forward of principal $(H,G)$\nd bundles along a morphism $G \to G'$ of affine group schemes over $S$, defined by push forward of $(H,G \times_S X)$\nd torsors.
Using push forward we define a functor $H^1_H(X,-)$ on affine group schemes over $S$,
which factors through affine group schemes over $S$ up to conjugacy.

Let $\sV$ be a vector bundle over $X$ of constant rank $n$.
Arguing locally over $X$ shows that the scheme
\[
\underline{\Iso}_X(\sO_X^n,\sV)
\]
over $X$ with points above $x$ the isomorphisms from $(\sO_X^n)_x$ to $\sV_x$ exists, and is a principal $GL_n$\nd bundle over $X$
for the action by composition.
If $\sV$ is a representation of $H$, then $\underline{\Iso}_X(\sO_X^n,\sV)$ is a principal $(H,GL_n)$\nd bundle.
We thus obtain a functor $\underline{\Iso}_X(\sO_X^n,-)$, compatible with pullback, 
from the category of representations of rank $n$ of $H$ and isomorphisms
between them to the category of principal $(H,GL_n)$\nd bundles.
It is a equivalence: for full faithfulness reduce to the case $H = X$ and then to the case of trivial vector bundles,
and for essentially surjectivity reduce by Lemma~\ref{l:prereppull} to the case of a trivial underlying $GL_n$\nd bundle.
Thus we have a bijection from the set of isomorphism classes of representations of $H$ of rank $n$ to
\begin{equation}\label{e:H1HXGLn}
H^1_H(X,GL_n)
\end{equation}
which sends the class of $\sV$ to the class of $\underline{\Iso}_X(\sO_X^n,\sV)$.

Let $P$ be a principal $G$\nd bundle over $X$.
Define the transitive affine groupoid 
\[
\underline{\Iso}_G(P)
\]
over $X$ of $G$\nd isomorphisms of $P$ as follows:
the points of $\underline{\Iso}_G(P)$ in $T$ above the point $(x_1,x_0)$ of $X \times_S X$ in $T$ are the isomorphisms
from $P_{x_0}$ to $P_{x_1}$ of principal $G$\nd bundles over $T$, and the identities and composition of $\underline{\Iso}_G(P)$
are those of isomorphisms of principal $G$\nd bundles.
The existence of $\underline{\Iso}_G(P)$ follows
from Lemma~\ref{l:isorep} with $X \times_S X$, $X \times_S X$ and $(G_{X \times_S X})^\mathrm{op}$
for $S$, $X$ and $K$, and the pullbacks of $P$ along the projections for $Y$ and $Z$.
Formation of $\underline{\Iso}_G(P)$ commutes with base change and pullback.
The $G$\nd automorphisms of $P$ are the cross-sections of $\underline{\Iso}_G(P)^\mathrm{diag}$.
If $H$ is a is a pregroupoid over $X$, then a structure of principal $(H,G)$\nd bundle on $P$ is the same as
a structure 
\begin{equation}\label{e:HIsoGP}
H \to \underline{\Iso}_G(P)
\end{equation}
of groupoid over $H$ on $\underline{\Iso}_G(P)$.
In particular, there is a canonical structure of principal $(\underline{\Iso}_G(P),G)$\nd bundle on $P$,
with \eqref{e:HIsoGP} the identity.
The morphism \eqref{e:HIsoGP} corresponding to a given structure of principal $(H,G)$\nd bundle on $P$ is then
the unique one along which the canonical action restricts to that of $H$.

Let $h:G \to G'$ be a morphism of affine group schemes
over $S$, $P$ be a principal $G$\nd bundle and $P'$ a principal $G'$\nd bundle over $X$,
and $q:P \to P'$ be a morphism over $X$ compatible with $h$.
By Lemma~\ref{l:Hpush} with $u = h \times_S X$ there is a unique morphism
\[
\underline{\Iso}_h(q):\underline{\Iso}_G(P) \to \underline{\Iso}_{G'}(P')
\]
of groupoids over $X$ such that, for the corresponding structure on $P'$ of principal $(\underline{\Iso}_G(P),G')$\nd bundle,
$q$ is an $\underline{\Iso}_G(P)$\nd morphism. 
If $H$ is a pregroupoid over $X$, $P$ is a principal $(H,G)$\nd bundle and $P'$ is a principal $(H,G')$ bundle, 
then by Lemma~\ref{l:Hpush} $q$ is a morphism of $H$\nd schemes if and only if $\underline{\Iso}_h(q)$ is 
a morphism of groupoids over $H$.

Define a category of \emph{affine principal bundles over $X$} as follows:
the objects are pairs $(G,P)$ with $G$ an affine group scheme over $S$ and $P$ a principal $G$\nd bundle over $X$,
and a morphism from $(G,P)$ to $(G',P')$ is a pair $(h,q)$ with $h:G' \to G'$ a morphism of group schemes over $S$
and $q:P \to P'$ a morphism over $X$ compatible with $h$.
Then $\underline{\Iso}_-(-)$ is a functor from this category to transitive affine groupoids over $X$.
It is compatible with base extension and pullback,
but in general is neither full, faithful, nor essentially surjective.

To each $g$ in $G(S)$ there is associated an inner automorphism of $(G,P)$,
which acts as conjugation by $g$ on $G$ and as $g^{-1}$ on $P$.
If $(h,q)$ is an inner automorphism of $(G,P)$ then $\underline{\Iso}_h(q)$ is the identity.
If $q$ is an automorphism of the principal $G$\nd bundle $P$,
then the automorphism $\underline{\Iso}_{1_G}(q)$ of $\underline{\Iso}_G(P)$ is conjugation by the cross-section
$q$ of its diagonal.

Let $K$  be a transitive affine groupoid over $X$.
Reducing by base change and Lemma~\ref{l:torspull} to the case where $K$ is constant shows that any simply transitive $K$\nd scheme
has a structure of principal $(K,G)$\nd bundle for some affine group scheme $G$ over $S$.
By the following lemma, $K$ is thus of the form $\underline{\Iso}_G(P)$ for some $G$ and $P$ if and only 
there exists a simply transitive $K$\nd scheme.

\begin{lem}\label{l:simpleprin}
Let $K$ be a transitive affine groupoid over $X$ and $P$ be a principal $(K,G)$\nd bundle.
Then $P$ is a simply transitive $K$\nd scheme if and only if the action morphism $K \to \underline{\Iso}_G(P)$
is an isomorphism.
\end{lem}

\begin{proof}
By base change along $P \to S$ and pullback, we reduce using Lemma~\ref{l:morphpull}\ref{i:morphpullgpd}
to the case where $X = S$ and $P = G$ has a cross-section, which is clear.
\end{proof}

Let $P$ be a principal $G$\nd bundle over $X$.
We have a morphism of groupoids
\[
f:X \times_{[X]} [P] \to G
\]
which sends $(x,z_1,z_0)$ to the unique point $g$ of $G$ with $z_0 = z_1g$.
By Lemma~\ref{l:prereppull} applied to $P \to X$, there exists for every $G$\nd module $\sV$
a pair, unique up to unique isomorphism, consisting of a quasi-coherent $\sO_X$\nd module
\[
P \times_S^G \sV
\]
and an isomorphism of $(X \times_{[X]} [P])$\nd modules
\[
i_{\sV}:f^*\sV \iso p^*(P \times_S^G \sV),
\]
where $p:X \times_{[X]} [P] \to X$ is the projection.
Thus $P \times_S^G \sV$ is the usual associated quasi-coherent $\sO_X$\nd module 
defined by descent from $P$ to $X$ by identifying the points $(zg,v)$ and $(z,gv)$ of the pullback of $\sV$ onto $P$.
It is functorial in $P$, $G$ and $\sV$, and its formation is compatible with base change and pullback.
We have $P \times_S^G \sO_S = \sO_X$, with $i_{\sO_S}$ the identity.

Given an affine $G$\nd scheme $Y$, there exists similarly a pair, unique up to unique isomorphism,
consisting of a scheme $P \times_S^G Y$ affine over $X$ and an isomorphism $i_Y$ of $(X \times_{[X]} [P])$\nd schemes
from $f^*Y$ to $p^*(P \times_S^G Y)$.
The diagram
\[
\xymatrix{
P \times_S G \times_S Y \ar@<1ex>[r] \ar@<-1ex>[r] & P \times_S Y \ar[r] & P \times_S^G Y
}
\]
with the right arrow defined by $i_Y$ and the left arrows by the actions of $G$ on $P$ and $Y$
is a coequaliser diagram which is preserved by pullback, because the right arrow 
is $fpqc$ covering and the square with two sides the right arrow and two sides the left arrows is cartesian. 
Again $P \times_S^G Y$ is functorial in $P$, $G$ and $Y$, and its formation is compatible with base change and pullback.

Let $G_0$ be an affine group subscheme of $G$ such that the quotient $G/G_0$ exists and is affine.
Applying $P \times_S^G -$ to $G \to G/G_0$ shows that the quotient $P/G_0$ of $P$ exists, i.e.\ that there 
exists an $fpqc$ covering morphism
\begin{equation}\label{e:Gzeroquot}
P \to P/G_0
\end{equation}
such that the square with two sides \eqref{e:Gzeroquot} and the other two the projection and the action 
from $P \times_S G_0$ to $P$ is cartesian.
Then $P$ is a principal $G_0$\nd bundle over $P/G_0$ with 
$\underline{\Iso}_{G_0}(P) = \underline{\Iso}_G(P) \times_X (P/G_0)$.

Suppose that $P$ has a structure of $(H,G)$\nd bundle for a pregroupoid $H$ over $X$.
Then $f$ factors through the morphism of pregroupoids 
\[
f_H:H \times_{[X]} [P] \to G 
\]
that sends $(h,z_1,z_0)$ to the unique point $g$ of $G$ with $hz_0 = z_1g$.
The action of $H$ on $P$ defines a structure of $H$\nd module on $P \times_S^G \sV$.
It is the unique such structure for which $i_{\sV}$ is an isomorphism of $(H \times_{[X]} [P])$\nd modules
\[
f_H{}\!^*\sV \iso p_H{}\!^*(P \times_S^G \sV),
\]
where $p_H:H \times_{[X]} [P] \to H$ is the projection.
We thus have a functor $P \times_S^G -$ from $G$\nd modules to $H$\nd modules,
and an isomorphism from $f_H{}\!^*$ to $p_H{}\!^*(P \times_S^G -)$.
If $H$ is transitive affine and $P$ is a simply transitive $H$\nd scheme, 
and in particular if $H = \underline{\Iso}_G(P)$, then $f_H$ induces
an isomorphism to $G \times_S [P]$, so that by Lemma~\ref{l:prereppull} $P \times_S^G -$ is an equivalence.
Restricting to representations gives a functor
\begin{equation}\label{e:assocvecbun}
P \times_S^G -:\Mod_G(S) \to \Mod_H(X)
\end{equation}
which is an equivalence when $H = \underline{\Iso}_G(P)$.

Similarly we have a functor $P \times_S -$ from affine $G$\nd schemes to affine $H$\nd schemes and an isomorphism
from $f_H{}\!^*$ to $p_H{}\!^*(P \times_S^G -)$, with $P \times_S^G -$ an equivalence when $H$ is transitive affine and
$P$ is a simply transitive $H$\nd scheme.

If $G'$ is an affine group scheme over $S$ and $Y$ is a principal $(G,G')$\nd bundle, then $P \times_S Y$
has a structure of principal $(H,G')$\nd bundle.
We then have a functor $- \times_S^G Y$ from principal $(H,G)$\nd bundles to principal $(H,G')$\nd bundles.
We have $P \times_S^G G = P$, with $i_G$ defined by the action of $G$ on $P$, and there is a natural isomorphism
\[
(P \times_S^G Y) \times_S^{G'} Y' \iso P \times_S^G (Y \times_S^{G'} Y').
\]
A principal $(G,G')$\nd bundle $Q$ is simply transitive as a $G$\nd scheme if and only it is a principal
$(G',G)$\nd bundle for the right action of $G$ and left action of $G'$ on $Q$ given by the inverse involutions.
When this is so, $- \times_S^G Q$ is an equivalence, with quasi-inverse $- \times_S^{G'} Q$ where $Q$ is regarded as a principal
$(G',G)$\nd bundle.

The $\underline{\Iso}_G(P)$\nd scheme $P \times_S^G Y$ is simply transitive 
if and only if $Y$ is a simply transitive $G$\nd scheme, because $f_{\underline{\Iso}_G(P)}$ induces an isomorphism to 
$G \times_S [P]$.
Since an isomorphism $\underline{\Iso}_G(P) \iso \underline{\Iso}_{G'}(P')$ of groupoids over $X$ is the same as a simply transitive action of 
$\underline{\Iso}_G(P)$ on the principal $G'$\nd bundle $P'$, it follows that $\underline{\Iso}_G(P)$ and 
$\underline{\Iso}_{G'}(P')$ are isomorphic over $H$ if and only if $P$ and $P'$ are twists of one another,
i.e.\ $P'$ is isomorphic as a principal $(H,G')$\nd bundle to $P \times_S ^G Q$ for a principal $(G,G')$\nd bundle $Q$
which is simply transitive as a $G$\nd scheme.

Given a section $x$ of $X$ over $S$, define as follows a category of \emph{affine principal bundles over $X$ 
pointed above $x$}.
The objects are triples $(G,P,z)$ with $(G,P)$ an affine principal bundle over $X$ and $z$ a section of $P$ over $S$ above $x$. 
A morphism from $(G,P,z)$ to $(G',P',z')$ is a morphism $(h,q)$ from $(G,P)$ to $(G',P')$ for which
$q(z) = z'$. 
We have a functor
\begin{equation}\label{e:GPpequiv}
(G,P,z) \mapsto \underline{\Iso}_G(P)
\end{equation}
from this category to the category of transitive affine groupoids over $X$.
It is an equivalence,
with quasi-inverse the functor 
\[
K \mapsto (K_{x,x},K_{-,x},1_x)
\]
where $K_{x,x}$ is the fibre of $K$ above
$(x,x)$, and $K_{-,x}$ is the fibre of $d_1$ above $x$, regarded as a principal $K_{x,x}$\nd bundle
with structural morphism $d_0$ and action of $K_{x,x}$ by composition:
we have a natural isomorphism
\begin{equation}\label{e:KKKiso}
K \iso \underline{\Iso}_{K_{x,x}}(K_{-,x}),
\end{equation}
given by the structure of principal $(K,K_{x,x})$\nd bundle on $K_{-,x}$,
as follows from Lemma~\ref{l:morphpull}\ref{i:morphpullgpd} by taking fibres above $(x,x)$,
while we have a natural isomorphism
\[
(\underline{\Iso}_G(P)_{x,x},\underline{\Iso}_G(P)_{-,x},1_x) \iso (G,P,z)
\]
given by evaluating at $z$.
If $H$ is a pregroupoid over $X$, we have a similar category with objects $(G,P,z)$
where $P$ is a principal $(H,G)$\nd bundle, 
and \eqref{e:GPpequiv} defines an equivalence from it to transitive affine groupoids over $H$.

Let $H$ be a pregroupoid over $X$.
Define as follows a category of \emph{affine principal bundles over $(H,x)$ up to conjugacy}.
The objects are those pairs $(G,P)$ with $G$ an affine group scheme over $S$ and $P$ a principal
$(H,G)$\nd bundle for which $P$ has a section over $S$ above $x$.
A morphism from $(G,P)$ to $(G',P')$ is a morphism $j:G \to G'$ up to conjugacy of affine groups schemes over $S$
such that the push forward of $P$ along any representative of $j$ is $(H,G')$\nd isomorphic to $P'$. 
We may regard $\underline{\Iso}_-(-)$ as a functor from this category to transitive affine groupoids over $H$ up to conjugacy by assigning to $j$ the class of 
$\underline{\Iso}_h(q)$ with $h$ a representative of $j$ and $q:P \to P'$ an $H$\nd morphism
compatible with $h$.

\begin{lem}\label{l:IsoHequivpt}
Let $H$ be a pregroupoid over $X$ and $x$ be a section of $X$ over $S$.
Then the functor $\underline{\Iso}_-(-)$ from the category of affine principal bundles over $(H,x)$ up to conjugacy to 
the category of transitive affine groupoids over $H$ up to conjugacy is an equivalence.
\end{lem}

\begin{proof}
The fullness and essential surjectivity follow from the equivalence \eqref{e:GPpequiv}.

Let $h_1$ and $h_2$ be morphisms from $G$ to $G'$ such that morphisms $q_1$ and $q_2$ exist from $P$ to $P'$
compatible respectively with $h_1$ and $h_2$ for which $\underline{\Iso}_{h_1}(q_1)$ is conjugate to 
$\underline{\Iso}_{h_1}(q_1)$ 
by a cross-section $t$ of $\underline{\Iso}_{G'}(P')^\mathrm{diag}$.
Then
\[
\underline{\Iso}_{h_2}(q_2) = \underline{\Iso}_{h_1}(t \circ q_1).
\]
Fix sections $z$ and $z'$ over $S$ of $P$ and $P'$ above $x$.
Then the composites of appropriate inner automorphisms of $(G',P')$
with $(h_2,q_2)$ and $(h_1,t \circ q_1)$ are morphisms from $(G,P,z)$ to $(G',P',z')$ whose images under 
\eqref{e:GPpequiv} coincide, and which thus themselves coincide.
Hence $h_1$ and $h_2$ are conjugate.
This proves the faithfulness.
\end{proof}

\begin{lem}\label{l:relgrpdequiv}
Let $H$, $X'$ and $H'$ be as in Lemma~\textnormal{\ref{l:prereppull}}.
Then pullback along $X' \to X$ induces an equivalence from the category of transitive affine groupoids over $H$ to the category of 
transitive affine groupoids over $H'$. 
\end{lem} 

\begin{proof}
To prove the full faithfulness, we reduce to the case where $X'$ has a section over $S$ 
by base change along an appropriate $S' \to S$ and descent along 
\[
X_{S'} \times_{S'} X_{S'} \to X \times_S X
\]
for morphisms.
Then by the equivalence \eqref{e:GPpequiv} for
principal bundles with an action of $H$, 
and the similar equivalence for $H'$,
the full faithfulness follows from Lemma~\ref{l:prereppull}.

By the full faithfulness, if $K'$ is a transitive affine groupoid over $H'$,
then a pair consisting of a transitive affine groupoid $K$ over $H$ and an isomorphism 
from $K'$ to $K \times_{[X]} [X']$ over $H'$ is unique up to unique isomorphism if it exists.
To prove the essential surjectivity, we thus reduce by base change and descent for affine schemes 
and for morphisms of schemes to the case where $X'$ has a section over $S$.
By the equivalence \eqref{e:GPpequiv}, 
it then suffices to apply Lemma~\ref{l:torspull} with $J$ constant.
\end{proof}

Let $K$ be a groupoid over $X$.
Call a morphism $K^\mathrm{diag} \to J$ of $K$\nd groups \emph{compatible with conjugation} if the two actions 
\[
K^\mathrm{diag} \times_X J \to J
\] 
of $K^\mathrm{diag}$ on $J$ given by restricting from $K$ to $K^\mathrm{diag}$ and by restricting the action by conjugation of $J$ along $K^\mathrm{diag} \to J$ coincide. 
Any morphism $K^\mathrm{diag} \to J$ of $K$\nd groups which is $fpqc$ covering is compatible with conjugation.
We obtain a functor from the category of groupoids over $K$ to the category of those $K$\nd morphisms from $K^\mathrm{diag}$
to a $K$\nd group which are compatible with conjugation,
by assigning to $K'$ the $K$\nd morphism
$K^\mathrm{diag} \to K'{}^\mathrm{diag}$ defined by restriction, with the action of $K$ on $K'{}^\mathrm{diag}$
that defined by restriction of the action of $K'$ by conjugation.
It is an equivalence when $X = S$.
Suppose that $K$ is transitive affine, and let $X' \to X$ be a morphism with $X' \to S$
$fpqc$ covering.
Then by Lemma~\ref{l:prereppull}, a morphism $K^\mathrm{diag} \to J$ of affine $K$\nd groups is compatible with conjugation 
if and only if the morphism of $(K \times_{[X]} [X'])$\nd groups obtained from it by pullback along $X' \to X$ is, because 
the two actions of $K^\mathrm{diag}$ on $J$ are both morphisms of $K$\nd schemes.

\begin{lem}\label{l:Kgrpdequiv}
Let $K$ be a transitive affine groupoid over $X$.
Then restriction to the diagonal induces an equivalence from the category of transitive affine
groupoids over $K$ to the category of those $K$\nd morphisms
from $K^\mathrm{diag}$ to an affine $K$\nd group which are compatible with conjugation.
\end{lem}

\begin{proof}
When $X$ has a section over $S$, it is enough to apply Lemmas~\ref{l:prereppull} and \ref{l:relgrpdequiv} with $X' = S$.
The full faithfulness follows in general by reducing after base change and descent to the case where $X$ has a section.
By the full faithfulness, if a morphism from $K^\mathrm{diag}$ of affine $K$\nd groups extends to a morphism from $K$ of transitive affine groupoids over $X$, 
it does so uniquely up to unique isomorphism.
The essential surjectivity thus also follows by reducing to the case where $X$ has a section.
\end{proof}

\section{Principal bundles under a groupoid}\label{s:bunundergpd}

In this section we consider a notion of principal bundle under a transitive affine groupoid.
We fix a base scheme $S$, and $X$ and $\overline{S}$ are schemes over $S$ whose structural morphisms are
$fpqc$ covering.
Pregroupoids and groupoids are in the category of schemes over $S$ unless otherwise indicated.

Let $F$ be a transitive affine groupoid over $\overline{S}$.
By a \emph{right $F$\nd scheme over $X$} we mean a scheme $P$ over $X \times_S \overline{S}$ together with
a structure of right $F$\nd scheme on $P$ such that the action 
\[
P \times_{\overline{S}} F \to P
\]
is a morphism over $X$.
A morphism of right $F$\nd schemes over $X$ is a morphism of right $F$\nd schemes which is a morphism over $X$.
A right $F$\nd scheme over $X$ is the same as a right $F_X$\nd scheme,
with $F_X$ the transitive affine groupoid in schemes over $X$ with underlying groupoid 
$X \times_S F$ in schemes over $S$.
Restriction to $F^\mathrm{diag}$ defines on any right $F$\nd scheme over $X$ an underlying structure of 
right $F^\mathrm{diag}$\nd scheme over $X \times_S \overline{S}$ in the category of schemes over $\overline{S}$.

By a \emph{principal $F$\nd bundle over $X$} we mean a right $F$\nd scheme over $X$ whose underlying
right $F^\mathrm{diag}$\nd scheme over $X \times_S \overline{S}$ in the category of schemes over $\overline{S}$ is a principal $F^\mathrm{diag}$\nd bundle over 
$X \times_S \overline{S}$ in the sense of Section~\ref{s:prinbun}.
When $\overline{S} = S$, so that $F$ is an affine group scheme over $S$, this notion of principal $F$\nd bundle
over $X$ reduces to that of Section~\ref{s:prinbun}.
Any morphism of principal $F$\nd bundles over $X$ is an isomorphism.
By Remark~\ref{r:simpletrantorsor},
a right $F$\nd scheme over $X$ is a principal $F$\nd bundle over $X$ if and only if it is simply transitive 
as an $(F_X)^\mathrm{op}$\nd scheme in the category of schemes over $X$.

Pullback along a morphism $X' \to X$ with $X' \to S$ $fpqc$ covering defines a functor from 
right $F$\nd schemes over $X$
to right $F$\nd schemes over $X'$, and from principal $F$\nd bundles over $X$ to principal $F$\nd bundles over $X'$. 

If $F'$ is the pullback of $F$ along a morphism $\overline{S}{}' \to \overline{S}$ with $\overline{S}{}' \to S$ $fpqc$ covering,
then pullback along $\overline{S}{}' \to \overline{S}$ defines a functor from right $F$\nd schemes over $X$ to right 
$F'$\nd schemes over $X$, and from principal $F$\nd bundles over $X$ to principal $F'$\nd bundles over $X$.

\begin{lem}\label{l:gpdprinGG'}
Let $F$ be a transitive affine groupoid over $\overline{S}$, and
$\overline{S}{}' \to \overline{S}$ be a morphism with $\overline{S}{}' \to S$ $fpqc$ covering.
Denote by $F'$ the pullback of $F$ along $\overline{S}{}' \to \overline{S}$.
Then pullback along $\overline{S}{}' \to \overline{S}$ induces an equivalence from the category of principal
$F$\nd bundles over $X$ (resp.\ right $F$\nd schemes over $X$) to the category of principal $F'$\nd bundles over $X$
(resp.\ right $F'$\nd schemes over $X$).
\end{lem}

\begin{proof}
By Lemma~\ref{l:prereppull} with $X$, $X \times_S \overline{S}$, $X \times_S \overline{S}{}'$ and $(F_X)^\mathrm{op}$
for $S$, $X$, $X'$ and $H$, pullback induces an equivalence from 
affine right $F_X$\nd schemes to affine right $F'{}\!_X$\nd schemes.
Since an $(F_X)^\mathrm{op}$\nd scheme 
is simply transitive if and only if its pullback is a simply
transitive $(F'{}\!_X)^\mathrm{op}$\nd scheme, the result follows.
\end{proof}

Define as follows a category of \emph{affine principal bundles over $(X,\overline{S})$}.
The objects are pairs $(F,P)$ with $F$ a transitive affine
groupoid  over $\overline{S}$ and $P$ a principal $F$\nd bundle over $X$.
A morphism from $(F,P)$ to $(F',P')$ 
is a pair $(h,q)$ with $h$ a morphism from $F$ to $F'$ over $\overline{S}$ and $q$ a morphism from $P$ to $P'$ 
over $X \times_S \overline{S}$ compatible with $h$,
in the sense that $q$ is a morphism of right $F$\nd schemes over $X$ from $P$ to the restriction of $P'$ along $h$,
or equivalently that the action of $F'$ composed  with $q \times h$ 
coincides with $q$ composed with the action of $F$.

\begin{lem}\label{l:gpdpush}
Let $(F,P)$ be an affine principal bundle over $(X,\overline{S})$.
Then the forgetful functor from the category of affine principal bundles over $(X,\overline{S})$
equipped with a morphism from $(F,P)$ to the category of transitive affine groupoids over $\overline{S}$
equipped with a morphism from $F$ is a surjective equivalence.
\end{lem}

\begin{proof}
By base change along $\overline{S} \to S$ and descent, we reduce first the full faithfulness
and then the surjectivity on objects to the case where $\overline{S}$ has a cross-section over $S$.
By Lemma~\ref{l:gpdprinGG'} with $\overline{S}{}' = S$,
we then reduce to the case where $\overline{S} = S$, which follows from Lemma~\ref{l:push}.
\end{proof}

Given a principal $F$\nd bundle $P$ over $X$ and a morphism $h:F \to F'$ of transitive affine groupoids
over $\overline{S}$, there exists by Lemma~\ref{l:gpdpush} a pair, unique up to unique $F'$\nd isomorphism,
consisting of a principal $F'$\nd bundle $P'$ and a morphism $P \to P'$ compatible with $h$.
We say that $P'$ is the \emph{push forward of $P$ along $h$}.
The push forward of $P$ along an inner automorphism of $F$ is $P$ itself: 
for every cross-section $v$ of $F^\mathrm{diag}$ we have an inner automorphism
of $(F,P)$  which acts by conjugation with $v$ on $F$ and as $v^{-1}$ on $P$.

Let $H$ be a pregroupoid over $X$.
Given a right $F$\nd scheme $P$ over $X$ equipped with a structure of $H$\nd scheme,
it is equivalent to require that the action $d_1{}\!^*P \iso d_0{}\!^*P$ of $H$ on $P$ be a morphism of right $F$\nd schemes over 
$H_{[1]}$, 
or that the two morphisms
\[
H_{[1]} \times_X P \times_{\overline{S}} F \to P
\]
defined by the actions by first factoring through $H_{[1]} \times_X P$ or $P \times_{\overline{S}} F$ should coincide.
We then say that $P$ is an $(H,F)$\nd scheme.
A morphism of $(H,F)$\nd schemes is a morphism of right $F$\nd schemes which is a morphism of $H$\nd schemes.
By a \emph{principal $(H,F)$\nd bundle} we mean an $(H,F)$\nd scheme which is a principal $F$\nd bundle over $X$.

\begin{lem}\label{l:gpdHpush}
Let $H$ be a pregroupoid over $X$ and $h:F \to F'$ be a morphism of transitive affine groupoids over $\overline{S}$.
Let $P$ be a principal $(H,F)$\nd bundle, $P'$ be a principal $F'$\nd bundle over $X$,
and $q:P \to P'$ be a morphism over $X \times_S \overline{S}$ compatible with $h$.
Then there is a unique structure of $H$\nd scheme on $P'$ such that $P'$ is a principal $(H,F')$\nd bundle
and $q$ is a morphism of $H$\nd schemes.
\end{lem}

\begin{proof}
Apply Lemma~\ref{l:gpdpush} to the pullbacks of $P$ and $P'$ onto the $H_{[i]}$.
\end{proof}

If $P$ is a principal $(H,F)$\nd bundle and $h:F \to F'$ is a morphism
of transitive affine groupoids over $\overline{S}$, then by Lemmas~\ref{l:gpdpush} and \ref{l:gpdHpush}
there exists a pair, unique up to unique $(H,F')$\nd isomorphism, consisting of a principal $(H,F')$\nd bundle
$P'$ over $X$ and a morphism $P \to P'$ of $H$ \nd schemes which is compatible with $h$. 
We say that $P'$ is the \emph{push forward of $P$ along $h:F \to F'$}.  

The set of isomorphism classes of principal $F$\nd bundles over $X$ will be written
\[
H^1(X,F)
\]
and the set of isomorphism classes of principal $(H,F)$\nd bundles as
\[
H^1_H(X,F).
\]
Using push forward, we have a functors $H^1(X,-)$ and $H^1_H(X,-)$ on transitive affine groupoids 
up to conjugacy over $\overline{S}$.
When $\overline{S} = S$, so that $F$ is an affine group scheme over $S$, the sets $H^1(X,F)$ and $H^1_H(X,F)$
coincide with those of Section~\ref{s:prinbun}.
For arbitrary $\overline{S}$ however, these sets do not come equipped with a base point, and may even be empty.

\begin{lem}\label{l:gpdprinHGG'}
Let $H$ be a pregroupoid over $X$ and $F$ be a transitive affine groupoid over $\overline{S}$.
Let $\overline{S}{}' \to \overline{S}$ be a morphism with $\overline{S}{}' \to S$ $fpqc$ covering
and $F'$ be the pullback of $F$ along $\overline{S}{}' \to \overline{S}$.
Then pullback along $\overline{S}{}' \to \overline{S}$ induces an equivalence from the category of principal
$(H,F)$\nd bundles to the category of principal $(H,F')$\nd bundles.
\end{lem}

\begin{proof}
Apply Lemma~\ref{l:gpdprinGG'} to pullbacks along the $H_{[i]} \to X$.
\end{proof}

\begin{lem}\label{l:gpdprinHH'G}
Let $H$, $X'$ and $H'$ be as in Lemma~\textnormal{\ref{l:prereppull}},
and $F$ be a transitive affine groupoid over $\overline{S}$.
Then pullback along $X' \to X$ induces an equivalence from the category of principal
$(H,F)$\nd bundles to the category of principal $(H',F)$\nd bundles.
\end{lem}

\begin{proof}
By base change along $\overline{S} \to S$ and descent, we reduce to the case where 
$\overline{S}$ has a section over $S$.
Applying Lemma~\ref{l:gpdprinHGG'} to pullback along such a section, we reduce further to the case where
$\overline{S} = S$.
Then $F$ is an affine group scheme over $S$, and it suffices to apply Lemma~\ref{l:torspull} with $J$ constant.
\end{proof}

Let $P$ be a principal $F$\nd bundle over $X$.
We define the transitive affine groupoid 
\[
\underline{\Iso}_F(P)
\]
over $X$ of $F$\nd isomorphisms of $P$ as follows:
the points of $\underline{\Iso}_F(P)$ in $T$ above the point $(x_1,x_0)$ of $X \times_S X$ in $T$ are the 
isomorphisms from $P_{x_0}$ to $P_{x_1}$ of principal $F$\nd bundles over $T$,
and the identities and composition of $\underline{\Iso}_F(P)$ are
those of isomorphisms of principal $F$\nd bundles.
The existence of $\underline{\Iso}_F(P)$ follows from Lemma~\ref{l:isorep} with $X \times_S X$, $X \times_S X \times_S \overline{S}$
and $(F_{X \times_S X})^\mathrm{op}$ for $S$, $X$ and $K$,
and the pullbacks of $P$ along the projections for $Y$ and $Z$.
A structure of principal $(H,F)$ bundle on $P$ is the same as
a structure 
\begin{equation}\label{e:gpdHIsoGP}
H \to \underline{\Iso}_F(P)
\end{equation}
of groupoid over $H$ on $\underline{\Iso}_F(P)$.
In particular, there is a canonical structure of principal $(\underline{\Iso}_F(P),F)$\nd bundle on $P$,
with \eqref{e:gpdHIsoGP} the identity.

Let $(h,q):(F,P) \to (F',P')$ be a morphism of affine principal bundles over $(X,\overline{S})$.
Then by Lemma~\ref{l:gpdHpush} there is a unique morphism
\[
\underline{\Iso}_h(q):\underline{\Iso}_F(P) \to \underline{\Iso}_{F'}(P')
\]
of groupoids over $X$ such that, for the corresponding structure on $P'$ of principal $(\underline{\Iso}_F(P),F')$\nd bundle,
$q$ is an $\underline{\Iso}_F(P)$\nd morphism.
If $P$ is a principal $(H,F)$\nd bundle and $P'$ is a principal $(H,F')$ bundle, 
then by Lemma~\ref{l:gpdHpush}, $q$ is a morphism of $H$\nd schemes if and only if $\underline{\Iso}_h(q)$ is a morphism of groupoids over $H$.

The assignments $(F,P) \mapsto \underline{\Iso}_F(P)$ and $(h,q) \mapsto \underline{\Iso}_h(q)$
define a functor $\underline{\Iso}_-(-)$ from affine principal bundles over $(X,\overline{S})$
to transitive affine groupoids over $X$.  
Its formation commutes with base change and pullback along morphisms $X' \to X$ or
$\overline{S}{}' \to \overline{S}$.
If $(h,q)$ is an inner automorphism of $(F,P)$ then $\underline{\Iso}_h(q)$ is the identity.
If $q$ is an automorphism of the principal $F$\nd bundle $P$,
then the automorphism $\underline{\Iso}_{1_F}(q)$ of $\underline{\Iso}_F(P)$ is conjugation by the cross-section
$q$ of its diagonal.
When $\overline{S} = S$, the functor $\underline{\Iso}_-(-)$ reduces to that of Section~\ref{s:prinbun}.

Let $K$ be a transitive affine groupoid over $X$.
The composition of $K$ defines a canonical structure of principal
$(K,K)$\nd bundle on $K$, with $\overline{S} = X$.
The action \eqref{e:gpdHIsoGP} then becomes an isomorphism
\begin{equation}\label{e:KKK}
K \iso \underline{\Iso}_K(K)
\end{equation}
as can be seen by reducing after base change and pullback to the case where $K$ is constant,
and then to the case where $X = S$.

Given a point $\overline{x}$ of $X$ in $\overline{S}$, define as follows a category of \emph{affine principal bundles over 
$(X,\overline{S})$ pointed above $\overline{x}$}.
The objects are triples $(F,P,\overline{z})$ with $(F,P)$ an affine principal bundle over $(X,\overline{S})$ and 
$\overline{z}$ a point of $P$ in $\overline{S}$ above $(\overline{x},1_{\overline{S}})$. 
A morphism from $(F,P,\overline{z})$ to $(F',P',\overline{z}{}')$ is a morphism $(h,q)$ from 
$(F,P)$ to $(F',P')$ for which
$q(\overline{z}) = \overline{z}{}'$. 
We have a functor
\begin{equation}\label{e:FPpequiv}
(F,P,\overline{z}) \mapsto \underline{\Iso}_F(P)
\end{equation}
from this category to the category of transitive affine groupoids over $X$.
It is an equivalence,
with quasi-inverse the functor 
\[
K \mapsto (K_{\overline{x} \times \overline{x}},K_{- \times \overline{x}},1_{\overline{x}})
\]
where $K_{\overline{x} \times \overline{x}}$ is the pullback of $K$ along $\overline{x}$, and
$K_{- \times \overline{x}}$ is the principal $K_{\overline{x} \times \overline{x}}$\nd bundle over $X$ obtained from the
principal $K$\nd bundle $K$ over $X$ by pullback along $\overline{x}$.
Indeed the canonical $(K,K)$\nd structure on $K$ defines by pullback a
$(K,K_{\overline{x} \times \overline{x}})$\nd structure on $K_{- \times \overline{x}}$,
and hence by \eqref{e:KKK} a natural isomorphism
\[
K \iso \underline{\Iso}_{K_{\overline{x} \times \overline{x}}}(K_{- \times \overline{x}}),
\]
while evaluating at $\overline{z}$ gives a natural isomorphism
\[
(\underline{\Iso}_F(P)_{\overline{x} \times \overline{x}},\underline{\Iso}_F(P)_{- \times \overline{x}},1_{\overline{x}}) 
\iso (F,P,\overline{z}).
\]
When $\overline{S} = S$, \eqref{e:FPpequiv} reduces to the equivalence \eqref{e:GPpequiv}.

\section{Transitive affine groupoids with base the spectrum of a field}\label{s:specfield}

\emph{In this section $k$ is a field and $X$ is a non-empty $k$\nd scheme}.

\medskip

From this section on we work in the category of schemes over a base field $k$.
In particular pregroupoids and groupoids will be understood to be in this category, unless otherwise indicated.

If $k'$ is an extension of $k$, it will often be convenient to refer for example to a groupoid over $\Spec(k')$ as a groupoid over $k'$.
Similarly we often write for example $H^1(k',G)$ for $H^1(\Spec(k'),G)$.

A $k$\nd scheme is non-empty if and only if it is faithfully flat over $k$ 
if and only if its structural morphism is $fpqc$ covering.

Recall that small limits of schemes affine over a scheme $S$ exist and are affine over $S$.
If $(f_\lambda)$ is a filtered inverse system of morphisms over $S$ 
of schemes affine over $S$, and if each $f_\lambda$ is faithfully flat, then $\lim_\lambda f_\lambda$ is faithfully flat:
the flatness is clear and for the surjectivity it suffices to note after passing to fibres that if $1 = 0$ in
a filtered colimit of rings $A_\lambda$, then $1 = 0$ in some $A_\lambda$.

Let $K$ be a transitive affine groupoid over $X$.
As in Section~\ref{s:transgrpd}, there exists an extension $k'$ and an $fpqc$ covering morphism $X' \to X_{k'}$
such that the pullback of $K_{k'}$ onto $X'$ is a constant groupoid in $k'$\nd schemes.
Thus $K$ is faithfully flat over $X \times_k X$.
Further any two fibres of the group scheme $K^\mathrm{diag}$ over $X$ in an extension of $k$ become isomorphic after
some extension of scalars.

The limit in the category of pregroupoids over $X$ of a filtered inverse system of transitive affine groupoids
over $X$ exists and is a transitive affine groupoid.

Let $K$ be a transitive affine groupoid over $X$.
Then any $K$\nd subgroup of a $K$\nd group is closed, i.e.\ the embedding is a closed immersion:
reduce using Lemma~\ref{l:morphpull}\ref{i:morphpullsch} to the case where $X$
is the spectrum of an extension $k'$ of $k$, which follows from the corresponding result for $k'$\nd groups
\cite[\href{https://stacks.math.columbia.edu/tag/047T}{Lemma 047T}]{stacks-project}.
Similarly every morphism of affine $K$\nd groups factors, necessarily uniquely up to unique isomorphism,
as a faithfully flat morphism followed by a closed immersion:
reduce using Lemmas~\ref{l:morphpull}\ref{i:morphpullsch} and \ref{l:prereppull}
to the case where $X = \Spec(k')$, which follows using the uniqueness from
the corresponding result for affine $k'$\nd groups \cite[14.1]{Wat}. 
In the same way a quotient of an affine $K$\nd group $J$ by a normal $K$\nd subgroup $N$, 
i.e.\ a $K$\nd group $J/N$ and a faithfully flat morphism 
\[
J \to J/N
\]
of $K$\nd groups with kernel $N$, exists and is affine, by \cite[14.1 and 16.3]{Wat}.  
If $J_1$ and $J_2$ are $K$\nd subgroups of an affine $K$\nd group $J$, and if $J_2$ normalises $J_1$, then we have a 
$K$\nd subgroup $J_1J_2$ of $J$ given by the image of the $K$\nd morphism from the semidirect product $J_1 \rtimes_X J_2$ to $J$.

By Lemma~\ref{l:Kgrpdequiv}, the quotient $K/N$ of a transitive affine groupoid $K$ over $X$ by a 
$K$\nd subgroup $N$ of $K^\mathrm{diag}$
exists and is transitive affine, and by Lemma~\ref{l:morphpull}\ref{i:morphpullgpd}
\[
K \to K/N
\]
is faithfully flat.
Any morphism of transitive affine groupoids over $X$ with trivial kernel is a closed immersion,
as follows after passing to a fibre from the case of  affine $k$\nd groups \cite[15.3]{Wat}.
Taking the quotient by the kernel thus gives a factorisation, necessarily unique up to unique isomorphism,
of an arbitrary morphism of transitive affine groupoids over $X$
as a faithfully flat morphism followed by a closed immersion.

Let $K$ be a transitive affine groupoid over $X$.
Then the $d_i:K \to X$ are faithfully flat.
It follows that the category $\MOD_K(X)$ is abelian and has small colimits, and that the forgetful
functor to $\MOD(X)$ is exact and preserves colimits.
Reducing using Lemma~\ref{l:prereppull} to the case where $X$ is the spectrum of a field shows further that the full
subcategory $\Mod_K(X)$ of $\MOD_K(X)$ is abelian, and that the embedding is exact.
Similarly, for $\sV$ in $\Mod_K(X)$ the functor $\Hom_K(\sV,-)$ on $\MOD_K(X)$ preserves filtered colimits. 
Every $K$\nd module is the filtered colimit of its $K$\nd submodules in $\Mod_K(X)$ \cite[3.9]{Del90}.

Let $H$ be a pregroupoid over $X$ and $\sV$ be an $H$\nd module.
Then for $V$ a $k$\nd vector space
we have a homomorphism of $k$\nd vector spaces
\begin{equation}\label{e:HsVtensV}
H^0_H(X,\sV) \otimes_k V \to H^0_H(X,\sV \otimes_k V),
\end{equation}
natural in $\sV$ and $V$, which sends $w \otimes v$ to the image of $w$ under
the morphism of $H$\nd modules from  $\sV$ to $\sV \otimes_k V$ defined by the embedding of $v$ into $V$.
It is an isomorphism when $V$ is finite-dimensional, or when $X$ is quasi-compact
(write $V$ as the filtered colimit of its finite-dimensional subspaces and use Lemma~\ref{l:colimH0HV}).

Let $k'$ be an extension of $k$.
We may regard $X_{k'}$ a constant $H$\nd scheme, and the underlying groupoid in $k$\nd schemes of $H_{k'}$ may
then be identified with $H \times_X X_{k'}$.
If $p:X_{k'} \to X$ is the projection,
we usually write $p^*\sV$ as $\sV_{k'}$.
When $V = k'$, \eqref{e:HsVtensV} becomes, by \eqref{e:Hpushequ} and \eqref{e:pOXV}, the canonical 
homomorphism
\begin{equation}\label{e:Hexthom}
H^0_H(X,\sV) \otimes_k k' \to H^0_{H_{k'}}(X_{k'},\sV_{k'})
\end{equation}
induced by $(-)_{k'}$.
If $k'$ is finite over $k$, or if $X$ is quasi-compact,
taking $\sW^\vee \otimes_{\sO_X} \sV$ for $\sV$ in \eqref{e:Hexthom} shows that $(-)_{k'}$ induces an isomorphism
\begin{equation}\label{e:Homextiso}
\Hom_H(\sW,\sV) \otimes_k k' \iso \Hom_{H_{k'}}(\sW_{k'},\sV_{k'})
\end{equation}
for every $H$\nd module $\sV$ and representation $\sW$ of $H$.

\begin{lem}\label{l:extquot}
Let $K$ be a transitive affine groupoid over $X$, and $k'$ be an extension of $k$.
Then every representation of $K_{k'}$ is a $K_{k'}$\nd submodule (resp.\ $K_{k'}$\nd quotient)
of $\sV_{k'}$ for some representation $\sV$ of $K$.
\end{lem}

\begin{proof}
If $\sV'$ is a representation of $K_{k'}$,
the counit $p^*p_*\sV' \to \sV'$
for $p:X_{k'} \to X$ is an epimorphism,
so that $\sV'$ is a $K_{k'}$\nd quotient of $(p_*\sV')_{k'}$.
Writing $p_*\sV'$ as the filtered colimit of its $K$\nd submodules of finite type gives the result
for $K_{k'}$\nd quotients.
The result for $K_{k'}$\nd submodules follows by taking duals.
\end{proof}

Let $K$ be a transitive affine groupoid over $X$.
A $K$\nd module $\sV$ will be called \emph{faithful} if distinct points of $K$ with the same source and 
target have distinct actions on $\sV$.
It is equivalent to require that points of $K^\mathrm{diag}$ distinct from the identity should act non-trivially on 
$\sV$.
More generally a family $(\sV_\gamma)$ of $K$\nd modules will be called faithful if distinct points of $K$ with the same source and target have distinct actions on some $\sV_\gamma$.
A family $(\sV_\gamma)$ of representations of $K$ is faithful if and only if the morphism
\[
K \to \prod_{\gamma} \underline{\Iso}_X(\sV_{\gamma})
\]
with components the actions is a closed immersion.

Every affine $k$\nd group $G$ has a faithful family of representations: 
the regular $G$\nd module $k[G]$ is the filtered colimit of its finite-dimensional $G$\nd submodules.
Thus by extension of scalars using Lemma~\ref{l:extquot} and pullback onto a point using 
Lemmas~\ref{l:morphpull}\ref{i:morphpullgpd} and \ref{l:prereppull}, every transitive affine groupoid over $X$
has a faithful family of representations.

\begin{lem}\label{l:subquot}
Let $K$ be a transitive affine groupoid over $X$ and $K'$ be a transitive affine subgroupoid of $K$.
Then every representation of $K'$ is a $K'$\nd subquotient of a representation of $K$.
\end{lem}

\begin{proof}
Let $\sV'$ be a representation of $K'$.
It is enough to show that $\sV'$ is a $K'$\nd subquotient of a $K$\nd module $\sV$, because if $\sV'$
is a $K'$\nd submodule of the $K'$\nd quotient $\overline{\sV}$ of $\sV$, then for a sufficiently large
$K$\nd submodule $\sV_0$ of $\sV$ of finite type, $\sV'$ is contained in the image of $\sV_0$ in $\overline{\sV}$.
If $k'$ is an extension of $k$ and $p:X_{k'} \to X$ is the projection, then 
the unit $\sV' \to p_*p^*\sV'$ is a monomorphism, so that $\sV'$ is a $K'$\nd submodule of $p_*\sV'{}\!_{k'}$.
Thus we may suppose after extension of scalars that $X$ has a $k$\nd point, and hence by Lemma~\ref{l:prereppull}
that $X = \Spec(k)$.
Then $K$ is an affine $k$\nd group $G$ and $K'$ is a $k$\nd subgroup $G'$, and the result is well known:
any $G'$\nd module $V'$ is a $G'$\nd submodule of $k[G'] \otimes_k V'{}\!_0$ 
with $V'{}\!_0$ the trivial $G'$\nd module on the underlying 
$k$\nd vector space of $V'$, while $k[G']$ is a $G'$\nd quotient of $k[G]$.
\end{proof}

A transitive affine groupoid $K$ over $X$ will be said to be \emph{of finite type} if it is of finite type 
as a scheme over $X \times_k X$.
By extension of scalars and pullback, it is equivalent to require that $K$ be
of finite presentation as a scheme over $X \times _k X$,
or again that some fibre of $K$ above the diagonal be of finite type.

Every point $\ne 1$ of a transitive affine groupoid $K$ over $X$ has image $\ne 1$
in $\underline{\Iso}(\sV)$ for some representation $\sV$ of $K$, and hence image $\ne 1$ in some
quotient of finite type of $K$.
The canonical morphism from $K$ to the limit of its quotients of finite type is thus a closed immersion.
It is also faithfully flat, and hence an isomorphism, because it is a filtered limit of faithfully flat morphisms.
Thus every transitive affine groupoid $K$ over $X$ is the filtered limit of its quotients of finite type.
Further every affine $K$\nd group $J$ is the filtered limit of its $K$\nd quotients of finite type,
as follows by considering the semidirect product $J \rtimes_X K$.

Let $K$ be a transitive affine groupoid over $X$.
Then by extension of scalars and pullback, any affine $K$\nd scheme of finite type is of finite presentation.
Pullback onto an affine scheme shows \cite[IV~8.8.2(i)]{EGA} that for $Z$ an affine $K$\nd scheme of finite type, 
the functor $\Hom_K(-,Z)$ sends filtered limits $\lim_\lambda Z_\lambda$ of affine $K$\nd schemes to colimits of sets,
and in particular that every morphism from $\lim_\lambda Z_\lambda$ to $Z$ factors through some $Z_\lambda$.
A similar result for affine $K$\nd groups follows.

Let $K$ be a transitive affine groupoid over $X$ of finite type.
Then reducing using Lemma~\ref{l:Kgrpdequiv} to a question about affine 
$(\lim_\lambda K_\lambda)$\nd groups shows that the functor $\Hom_X(-,K)$
sends filtered limits $\lim_\lambda K_\lambda$ of transitive affine groupoids over $X$ to colimits of sets,
and in particular that every morphism from $\lim_\lambda K_\lambda$ to $K$ factors through some $K_\lambda$.

\begin{lem}\label{l:finneutr}
Let $K$ be a transitive affine groupoid over $X$ of finite type.
Then for some finite extension $k'$ of $k$ there exists a simply transitive $K_{k'}$\nd scheme.
\end{lem}

\begin{proof}
By Lemma~\ref{l:prereppull}, we may assume that $X$ is affine.
Writing $X$ as the filtered limit of affine $k$\nd schemes of finite type then shows 
\cite[IV~8.8.2, 8.10.5(vi), 11.2.6(ii)]{EGA}
that there is a $k$\nd scheme $X_0$ of finite type such that $K$ is the pullback along a morphism $X \to X_0$
of a transitive affine groupoid $K_0$ over $X_0$ of finite type.
Replacing $X$ and $K$ by $X_0$ and $K_0$, we may further assume that $X$ is of finite type.
Then $X_{k'}$ has a $k'$\nd point over $k'$ for some finite extension $k'$ of $k$, and pullback along
its inclusion shows that $K_{k'}$ has a simply transitive $K_{k'}$\nd scheme.
\end{proof}

An affine $k$\nd group will be called  pro\'etale if each of its $k$\nd quotients of finite type is finite \'etale.
We write $G^\mathrm{con}$ for the identity component of an affine $k$\nd group $G$. 
The quotient $G/G^\mathrm{con}$ is a pro\'etale $k$\nd group \cite[6.7 and Ch~6, Ex~7]{Wat} which we write $G_\mathrm{\acute{e}t}$.

Let $K$ be a transitive affine groupoid over $X$.
It can be seen as follows that for every affine $K$\nd group $J$ over $X$ there is a unique normal $K$\nd subgroup
$J^\mathrm{con}$ of $J$ with fibres the identity components of the fibres of $J$. 
We reduce by Lemma~\ref{l:prereppull} to the case where $X$ is the spectrum of an extension $k'$ of $k$.
Then $J$ is an affine $k'$\nd group, and we take for $J^\mathrm{con}$ its identity component.
That $J^\mathrm{con}$ so defined is a $K$\nd subgroup of $J$ follows from the fact that
a morphism of group schemes $J_1 \to J_0$ where $J_1$ has connected fibres and $J_0$ is finite \'etale is trivial, because
its kernel is an open subscheme of $J_1$ containing every fibre.
Formation of $J^\mathrm{con}$ commutes with extension of scalars and pullback along morphisms
of transitive affine groupoids.

Recall \cite[10.1]{Wat} that the derived $k$\nd group of an affine $k$\nd group $G$ is the normal $k$\nd subgroup $G^\mathrm{der}$ defined as follows.
Write $f_n:G^{2n} \to G$ for the morphism of $k$\nd schemes defined by taking the product of $n$ commutators.
It is compatible with the actions by conjugation of $G$ on $G^{2n}$ and $G$.
Then $G^\mathrm{der}$ is the smallest closed $k$\nd subscheme of $G$ through which each $f_n$ factors.
Writing $f_n = \Spec(\varphi_n)$ and using the fact that
tensor products over $k$ preserve intersections of $k$\nd subspaces shows that 
$G^\mathrm{der}$ is a normal $k$\nd subgroup of $G$ whose formation commutes with extension of scalars.
The projection from $G$ onto $G/G^\mathrm{der}$ is universal among $k$\nd homomorphisms
from $G$ to a commutative affine $k$\nd group.
Thus $G^\mathrm{der}$ is the intersection of those normal $k$\nd subgroups $N$ of $G$ for which $G/N$ is commutative.

Let $K$ be a transitive affine groupoid over $X$ and $J$ be an affine $K$\nd group.
Write $J^\mathrm{der}$ for the intersection of those normal $K$\nd subgroups $N$ of $J$ for which
$J/N$ is commutative. 
If $X$ is the spectrum of an extension $k'$ of $k$, 
it follows from the fact that tensor products over $k'$ preserve intersections of $k'$\nd subspaces
that the derived $k'$\nd group of the $k'$\nd group $J$ 
is a $K$\nd subgroup, and hence coincides with $J^\mathrm{der}$.
Thus formation of $J^\mathrm{der}$ commutes with extension of scalars and pullback along morphisms
of transitive affine groupoids.

A groupoid $K$ over $X$ will be said to be \emph{finite \'etale} (resp.\ \emph{pro\'etale})
if it is transitive affine and the fibres of $K^\mathrm{diag}$ are finite \'etale (resp.\ pro\'etale).
Reducing to the case where $K$ is constant shows that $K$ is finite \'etale if and only if it is surjective and
finite \'etale as a scheme over $X \times_k X$.

Let $K$ be a transitive affine groupoid over $X$.
We also write $K^\mathrm{con}$ for the $K$\nd subgroup $(K^\mathrm{diag})^\mathrm{con}$ of $K^\mathrm{diag}$.
The quotient groupoid
\[
K_\mathrm{\acute{e}t} = K/K^\mathrm{con}
\]
is pro\'etale, and the projection $K \to K_\mathrm{\acute{e}t}$ is universal among morphisms from $K$ to a
pro\'etale groupoid.
Similarly we also write $K^\mathrm{der}$ for $(K^\mathrm{diag})^\mathrm{der}$.
The quotient
\[
K_\mathrm{ab} = K/K^\mathrm{der}
\]
is commutative, and the projection $K \to K_\mathrm{ab}$ is universal among morphisms from $K$ to a
commutative groupoid.
Formation of $K_\mathrm{\acute{e}t}$ and $K_\mathrm{ab}$ commutes with pullback and extension of scalars.

Let $G$ be a smooth affine $k$\nd group.
Then any principal $G$\nd bundle $P$ over $X$ is smooth.
Given $x \in X$, there thus exists a $z \in P$ above $x$ with residue field a finite separable extension of that of $x$
\cite[IV~17.15.10(iii)]{EGA} and a morphism $X' \to P$ with image containing $z$ such that $X' \to X$
is \'etale \cite[IV~17.11.4]{EGA}.
Thus $P$ is locally trivial in the \'etale topology.
Conversely any right $G$\nd scheme over $X$ which is locally trivial in the \'etale topology 
is a principal $G$\nd bundle.
Thus for $G$ smooth, principal $G$\nd bundles in the sense used here coincide with principal $G$\nd bundles
in the usual \'etale sense, and $H^1(X,G)$ coincides with the usual \'etale cohomology set.
This is so in particular when $k$ is of characteristic $0$ and $G$ is of finite type.

\begin{prop}\label{p:printriv}
Let $G$ be an affine $k$\nd group.
Suppose that $k$ is algebraically closed.
Then any principal $G$\nd bundle over $k$ is trivial.
\end{prop}

\begin{proof}
Let $P$ be a principal $G$\nd bundle over $k$.
Consider the set $\sP$ of pairs $(N,q)$ with $N$ a normal $k$\nd subgroup of $G$ and $q$ 
a $k$\nd morphism $P \to G/N$ to the trivial principal $(G/N)$\nd bundle compatible with the projection $G \to G/N$.
With $(N,q) \le (N',q')$ when $N' \subset N$ and $q$ factors through $q'$,
the set $\sP$ is inductively ordered.
Let $(N_0,q_0)$ be an element of $\sP$ with $N_0 \ne 1$.
We show that $(N_0,q_0)$ is not maximal in $\sP$.
Any maximal element of $\sP$ will then define a trivialisation of $P$.

Writing $G$ as the limit of its $k$\nd quotients of finite type shows that $G$ has a normal $k$\nd subgroup
$N$ with $G/N$ of finite type such that
\[
N_0 \cap N \ne N_0.
\]
Denote by $P'$ the push forward of $P$ 
along the projection from $G$ to $G/N$, and by $q$ the morphism $P \to P'$ compatible with the projection.
By Lemma~\ref{l:push},  the composite of the
projection $G/N_0 \to G/(N_0N)$ with $q_0$ factors as
\[
P \xrightarrow{q} P' \xrightarrow{r} G/(N_0N)
\] 
with $r$ compatible with the projection from $G/N$ onto $G/N_0N$.
The fibre of $r$ above the identity of $G/N_0N$ is non-empty and of finite type because $r$ is faithfully flat
and $P'$ is of finite type, and hence contains a $k$\nd point $z$.
Using $z$, we may identify $P'$ with $G/N$ in such a way that $r$ is the projection.
Then $(N_0 \cap N,q_1)$, with
\[
q_1 = (q_0,q):P \to G/(N_0 \cap N) = G/N_0 \times_{G/(N_0N)} G/N,
\]
lies in $\sP$ and is strictly greater than $(N_0,q_0)$.
\end{proof}

\begin{cor}\label{c:kgrpac}
Suppose that $k$ is algebraically closed.
\begin{enumerate}
\item\label{i:kgrpacsurj}
Any surjective $k$\nd homomorphism of affine $k$\nd groups induces a surjective homomorphism on groups of $k$\nd points.
\item\label{i:kgrpacdense}
The set of $k$\nd rational points in  any affine $k$\nd group is Zariski dense. 
\end{enumerate}
\end{cor}

\begin{proof}
\ref{i:kgrpacsurj}
To prove that a surjective homomorphism $h:G \to G'$ of affine $k$\nd groups
induces a surjective homomorphism on $k$\nd points, we may suppose that $G'$ is 
reduced and hence that $h$ is faithfully flat. 
The fibre of $h$ above a $k$\nd point of $G'$ is then a principal bundle under $\Ker h$, and has thus a $k$\nd point
by Proposition~\ref{p:printriv}.

\ref{i:kgrpacdense}
Any affine $k$\nd group $G$ is the limit $\lim_\lambda G_\lambda$ of its $k$\nd quotients of finite type.
The inverse images in $G$ of the open subsets $U_\lambda$ of the $G_\lambda$ form a base for the topology of $G$.
Since each non-empty $U_\lambda$ contains a $k$\nd point of $G_\lambda$, the required result follows from \ref{i:kgrpacsurj}.
\end{proof}

\begin{prop}\label{p:const}
Suppose that $k$ is algebraically closed.
Then every transitive affine groupoid over an algebraically closed extension of $k$ is constant.
\end{prop}

\begin{proof}
Let $K$ be a transitive affine groupoid over an algebraically closed extension $\overline{k}$ of $k$.
The set $\sK$ of transitive affine subgroupoids of $K$
is inductively ordered by reverse inclusion:
a chain $\sC$ in $\sK$ is bounded above by the filtered limit $\lim_{K' \in \sC^\mathrm{op}} K'$.
We show that any transitive affine groupoid $K_0 \ne [\overline{k}]$ over $\overline{k}$ 
contains a transitive affine subgroupoid $\ne K_0$.
It will follow that any maximal element of $\sK$ is a subgroupoid $[\overline{k}]$ of $K$, so that $K$ is constant.

Writing $K_0$ as the limit of its quotients of finite type shows that it has a quotient $K_1 \ne [\overline{k}]$ 
of finite type.
By Lemma~\ref{l:finneutr}, $K_1$ has a simply transitive $K_1$\nd scheme $Z$.
Since $Z$ is non-empty and of finite type over $\overline{k}$, it has a $\overline{k}$\nd point over $\overline{k}$.
Thus $K_1$ is constant, and hence has a subgroupoid $[\overline{k}]$.
Taking the inverse image of $[\overline{k}]$ under $K_0 \to K_1$ gives the required subgroupoid. 
\end{proof}

\begin{cor}\label{c:neutral}
Let $K$ be a transitive affine groupoid over $X$.
Suppose that $k$ is algebraically closed.
Then a simply transitive $K$\nd scheme exists, and is unique up to isomorphism. 
\end{cor}

\begin{proof}
By Lemma~\ref{l:prereppull}, we reduce first to the case where $X$ is the spectrum of an algebraically closed extension of $k$, 
and then using Proposition~\ref{p:const} to the case where $X = \Spec(k)$.
This case follows from Proposition~\ref{p:printriv}, by Remark~\ref{r:simpletrantorsor}. 
\end{proof}

Let $\overline{k}$ be an algebraically closed extension of $k$.
If $F$ is a transitive affine groupoid over $\overline{k}$, then any principal $F$\nd bundle over $\overline{k}$
has a cross-section above the diagonal, by Proposition~\ref{p:printriv} with $\overline{k}$ for $k$ and 
$F^\mathrm{diag}$ for $G$.
Given a transitive affine groupoid $F'$ over $\overline{k}$, a principal $(F',F)$\nd bundle $P$,
and a section $z$ of $P$ above the diagonal, 
there is a unique morphism $(h,q)$ from $(F',F')$ to $(F,P)$ of affine principal bundles over $(X,\overline{k})$
with $q$ a morphism of $F'$\nd schemes sending the identity of $F'$ to $z$.
Explicitly, for $v'$ a point of $F'$ above $(x_1,x_0)$,
\begin{equation}\label{e:qhdef}
q(v') = v'z(x_0,x_0) = z(x_1,x_1)h(v').
\end{equation}
With $\Hom^\mathrm{conj}_{\overline{k}}$ the set of morphisms over $\overline{k}$ up to conjugacy,
we thus have a bijection
\begin{equation}\label{e:FFrep}
\Hom^\mathrm{conj}_{\overline{k}}(F',F) \iso H^1_{F'}(\overline{k},F),
\end{equation}
natural in $F$ and $F'$,
which sends the conjugacy class of $h$ to the class of the push forward of the principal $(F',F')$\nd bundle $F'$
along $h$.
The bijection \eqref{e:FFrep} is compatible with pullback along $k$\nd homomorphisms $\overline{k} \to \overline{k}{}'$
of algebraically closed extensions of $k$.
Taking $F' = G \times_k [\overline{k}]$ with $G$ an affine $k$\nd group in \eqref{e:FFrep} and using 
Lemma~\ref{l:gpdprinHH'G}
gives a representation
\begin{equation}\label{e:HomconjHkG}
\Hom^\mathrm{conj}_{\overline{k}}(G \times_k [\overline{k}],F) \iso H^1_G(k,F)
\end{equation}
of the functor $H^1_G(k,-)$ on 
transitive affine groupoids over $\overline{k}$ up to conjugacy, with universal element the class of 
$G \times_k \overline{k}$.

Two constant affine groupoids $G \times_k [\overline{k}]$ and $G' \times_k [\overline{k}]$ over $\overline{k}$
are isomorphic if and only if $G$ and $G'$ are inner forms of one another,
i.e.\ if and only if a principal $(G,G')$\nd bundle exists which is a simply transitive $G$\nd scheme, or equivalently which is a principal $(G',G)$\nd bundle for the left and right actions of $G$ and $G'$ defined by the inverse involutions.
Explicitly, \eqref{e:qhdef} with $F =  G' \times_k [\overline{k}]$, $F' = G \times_k [\overline{k}]$, and $x_0 = x_1$ shows that
if \eqref{e:HomconjHkG} with $F = G' \times_k [\overline{k}]$ sends the class of $h$ to the class of $Q$ in 
$H^1_G(k,G') = H^1_G(k,G' \times_k [\overline{k}])$, then $h$ is an isomorphism if and only if 
$Q$ is a simply transitive $G$\nd scheme.

Suppose  that $X$ has a $k$\nd point.
Then for $F$ a transitive affine groupoid over $\overline{k}$, pulling back onto a $k$\nd point and taking $G = 1$ in
\eqref{e:HomconjHkG} shows that a principal $F$\nd bundle over $X$
exists if only if $F$ is constant.

\begin{cor}\label{c:gpdacequiv}
Let $\overline{k}$ and $\overline{k}{}'$ be algebraically closed extensions of $k$.
Then pullback along a $k$\nd homomorphism $\overline{k} \to \overline{k}{}'$ 
defines an equivalence from the category of transitive affine groupoids over $\overline{k}$
up to conjugacy to the category of transitive affine groupoids over $\overline{k}{}'$ up to conjugacy.
\end{cor}

\begin{proof}
The full faithfulness
follows from \eqref{e:FFrep} together with Lemmas~\ref{l:gpdprinHGG'} and \ref{l:gpdprinHH'G}.

Let $F'$ be a transitive affine groupoid over $\overline{k}{}'$.
The inverse image $\overline{k} \times_{[\overline{k}]} F'$ in $F'$ of the diagonal of $[\overline{k}]$ 
is a subgroupoid of $F'$
which is a transitive affine groupoid over $\overline{k}{}'$ in the category of $\overline{k}$\nd schemes.
By Proposition~\ref{p:const} with $\overline{k}$ for $k$, there is thus a morphism
from the final groupoid $\overline{k} \times_{[\overline{k}]} [\overline{k}{}']$ over $\overline{k}{}'$  
in the category of $\overline{k}$\nd schemes to $\overline{k} \times_{[\overline{k}]} F'$.
Composing with the embedding into $F'$, we obtain a structure
of groupoid over $\overline{k} \times_{[\overline{k}]} [\overline{k}{}']$ on $F'$.
Thus $F'$ lies in the essential image of the pullback functor from $\overline{k}$ to $\overline{k}{}'$,
by Lemma~\ref{l:relgrpdequiv} with $H = X = \Spec(\overline{k})$.
\end{proof}

By Lemma~\ref{l:gpdprinHGG'}, the study of the cohomology sets $H^1_H(X,F)$ for $F$ a transitive affine groupoid
over a non-empty $k$\nd scheme $\overline{S}$ reduces to the case where $\overline{S}$ is the spectrum
of an algebraically closed extension $\overline{k}$ of $k$.
Corollary~\ref{c:gpdacequiv} shows that the resulting theory is independent of the choice of $\overline{k}$.
When $k$ is algebraically closed for example, we may take $\overline{k} = k$, so that the groupoids $F$
are affine $k$\nd groups.

The above independence of the choice of $\overline{k}$ can be made more precise as follows.
Let $\overline{k}{}'$ be an algebraically closed extension of $k$ and $\overline{S}$ be a non-empty $k$\nd scheme.
Then any transitive affine groupoid $F$ over $\overline{S}$ has a $\overline{k}{}'$\nd point $v$ above a given 
$\overline{k}{}'$\nd point $(s_1,s_0)$ of $\overline{S} \times_k \overline{S}$, 
because the fibre $F_{s_1,s_0}$ of $F$ above $(s_1,s_0)$ is an $F_{s_0,s_0}$\nd torsor,
which is trivial by Proposition~\ref{p:printriv}.
Conjugation as in \eqref{e:Kconjiso} with such a $v$ defines an isomorphism 
from the pullback of $F$ along $s_0$ to its pullback along $s_1$, which up to conjugacy is independent of the choice of $v$.
In particular if we take $\overline{S} = \Spec(\overline{k})$ and write $[\sigma]^*$ 
for the functor of Corollary~\ref{c:gpdacequiv} with $\sigma$ for $\overline{k} \to \overline{k}{}'$,
we obtain for each pair $\sigma_0$ and $\sigma_1$ of $k$\nd homomorphisms from $\overline{k}$
to $\overline{k}{}'$ a canonical natural isomorphism 
\[
\theta_{\sigma_1,\sigma_0}:[\sigma_0]^* \iso [\sigma_1]^*.
\]
Further $\theta_{\sigma_0,\sigma_0}$ is the identity, we have 
$\theta_{\sigma_2,\sigma_1} \circ \theta_{\sigma_1,\sigma_0} = \theta_{\sigma_2,\sigma_0}$,
the $\theta$ are compatible, modulo the pullback isomorphisms, with composition with $[\tau]^*$ 
for $k$\nd homomorphisms $\tau$ to $\overline{k}$ or from $\overline{k}{}'$, and they are compatible 
with the natural bijections
\[
H^1_H(X,F) \iso H^1_H(X,[\sigma]^*F)
\]
given by Lemma~\ref{l:gpdprinHGG'}.
Note that given two algebraically closed extensions of $k$, there always exists a $k$\nd homomorphism
from at least one of them to the other.

Let $H$ be a pregroupoid over $X$ and $k'$ be an extension of $k$.
Define as follows a category of \emph{affine principal bundles over $(H,k')$ up to conjugacy}.
The objects are pairs $(F,P)$ with $F$ a transitive affine groupoid over $k'$ and $P$ a principal
$(H,F)$\nd bundle.
A morphism from $(F,P)$ to $(F',P')$ is a morphism $j:F \to F'$ of groupoids over 
$k'$ up to conjugacy such that the push forward of $P$ along any representative of $j$ is isomorphic to $P'$. 
We may regard $\underline{\Iso}_-(-)$ as a functor from this category to the category of transitive affine groupoids over $H$ up to conjugacy 
by assigning to $j$ the class of $\underline{\Iso}_h(q)$ with $h$ a representative of $j$ and $q:P \to P'$ an $H$\nd morphism
compatible with $h$.

\begin{cor}\label{c:IsoHequivac}
Let $H$ be a pregroupoid over $X$ and $\overline{k}$ be an algebraically closed extension of $k$.
Then the functor $\underline{\Iso}_-(-)$ from the category of affine principal bundles over $(H,\overline{k})$ up to conjugacy 
to the category of transitive affine groupoids over $H$ up to conjugacy is an equivalence.
\end{cor}

\begin{proof}
By Lemma~\ref{l:gpdprinHGG'} and Corollary~\ref{c:gpdacequiv}, me may after pulling back onto a sufficiently large
algebraically closed extension $\overline{k}{}'$ of $k$ and replacing $\overline{k}$ by $\overline{k}{}'$
suppose that $X$ has a $\overline{k}$\nd point $\overline{x}$.
By Proposition~\ref{p:printriv} with $G = F^\mathrm{diag}$, every principal $F$\nd bundle over $X$ has then a 
$\overline{k}$\nd point above $(\overline{x},1_{\overline{k}})$.
The fullness and essential surjectivity thus follow from the equivalence \eqref{e:FPpequiv}.

Let $h_1$ and $h_2$ be morphisms from $F$ to $F'$ such that morphisms $q_1$ and $q_2$ exist from $P$ to $P'$
compatible respectively with $h_1$ and $h_2$ for which $\underline{\Iso}_{h_2}(q_2)$ is conjugate to 
$\underline{\Iso}_{h_1}(q_1)$ 
by a cross-section $t$ of $\underline{\Iso}_{F'}(P')^\mathrm{diag}$.
Then 
\[
\underline{\Iso}_{h_2}(q_2) = \underline{\Iso}_{h_1}(t \circ q_1).
\]
Fix $k$\nd points $\overline{z}$ of $P$ and $\overline{z}{}'$ of $P'$ above $(\overline{x},1_{\overline{k}})$.
The composites of appropriate inner automorphisms of $(F',P')$
with $(h_2,q_2)$ and $(h_1,t \circ q_1)$ are morphisms from $(F,P,\overline{z})$ to $(F',P',\overline{z}{}')$ whose images under \eqref{e:FPpequiv} coincide, 
and which thus themselves coincide.
Hence $h_1$ and $h_2$ are conjugate.
This proves the faithfulness.
\end{proof}

\section{Fundamental groupoids}\label{s:fundgrpd}

\emph{In this section $k$ is a field, $X$ is a non-empty $k$\nd scheme, and $H$ is a pregroupoid over $X$.}

\medskip

In this section we prove the existence of initial objects in certain full subcategories of the category
of transitive affine groupoids over $H$. 
These initial objects are similar to the fundamental group schemes of Nori \cite{Nor82},
and are closely connected with the universal torsors of 
Colliot-Th\'el\`ene and Sansuc \cite{CS87}
in the case of groupoids of multiplicative type,
and with the theta groups of Mumford \cite{Mum66} in the case of groupoids of pro\'etale by multiplicative type.
They occur as quotients of the universal reductive groupoids of Section~\ref{s:unmin},
but their existence is much easier to prove, because they 
are unique up to unique isomorphism and not merely up to conjugacy.

An $H$\nd scheme $Y$ will be called \emph{$H$\nd connected} if it has no non-empty open and closed $H$\nd subscheme
$\ne Y$.
It is equivalent to require that $Y$ have no decomposition as the disjoint union
of two non-empty $H$\nd schemes.
An $H$\nd scheme $Y$ is $H$\nd connected if and only if $Y$ is $(H \times_X Y)$\nd connected
if and only if $H^0_{H \times_X Y}(Y,\sO_Y)$ has no idempotent $\ne 0,1$.
If $Y \to Y'$ is a surjective morphism of $H$\nd schemes and $Y$ is $H$\nd connected, then $Y'$ is $H$\nd connected.
If $k'$ is a purely inseparable extension of $k$, then $Y$ is $H$\nd connected if and only if $Y_{k'}$
is $H_{k'}$\nd connected.

For any integer $r \ge 0$, the open and closed subscheme of $X$ where a given representation of $H$
has rank $r$ is an $H$\nd subscheme of $X$.
Thus if $X$ is $H$\nd connected, then any representation of $H$ has constant rank.

An $H$\nd scheme $Y$ will be called \emph{geometrically $H$\nd connected} if $Y_{k'}$ is $H_{k'}$\nd connected for every finite 
separable extension $k'$ of $k$.
It is equivalent to require that $H^0_{H \times_X Y}(Y,\sO_Y)$
should contain no finite \'etale $k$\nd subalgebra other than $k$.
If $k'$ is an extension of $k$, and if either $k'$ is finite over $k$ or $Y$ is quasi-compact,
then $Y$ is geometrically $H$\nd connected if and only if $Y_{k'}$ is geometrically $H_{k'}$\nd connected,
because \eqref{e:Hexthom} with $Y$ for $X$ and $H \times_X Y$ for $H$ is an isomorphism.

Recall that a scheme over $X$ is finite \'etale if and only if it is \'etale locally over $X$ a 
finite disjoint union of copies of $X$.
Finite limits of finite \'etale schemes over $X$ are finite \'etale.
Any morphism of finite \'etale schemes over $X$ factors uniquely as a surjective finite \'etale morphism 
followed by the embedding of an open and closed subscheme.

\begin{defn}\label{d:Hproetale}
An $H$\nd scheme $X'$ will be called \emph{pro\'etale} if it is the limit of a filtered inverse system
$(X_\lambda)_{\lambda \in \Lambda}$ of finite \'etale $H$\nd schemes with each projection $X' \to X_\lambda$ faithfully flat.
\end{defn}

If $X'$ is the limit of a filtered inverse system $(X_\lambda)$ of finite \'etale schemes over $X$
with each $X' \to X_\lambda$ faithfully flat, and if $Y'$ is an open and closed subscheme of $X'$,
then arguing locally over $X$ shows that each $Y' \to X_\lambda$ factors uniquely as a faithfully flat morphism
$Y' \to Y_\lambda$ followed by the embedding of an open and closed subscheme,
and that $Y'$ is the limit of the system $(Y_\lambda)$.
It follows that any open and closed $H$\nd subscheme of a pro\'etale $H$\nd scheme is pro\'etale.

Suppose that $X$ is $H$\nd connected.
Then given $(X_\lambda)_{\lambda \in \Lambda}$ as in Definition~\ref{d:Hproetale} with limit $X'$, 
it can be seen as follows that the canonical map
\begin{equation}\label{e:colimfinet}
\colim_{\lambda \in \Lambda}\Hom_H(X_\lambda,Y) \to \Hom_H(X',Y)
\end{equation}
is bijective for every finite \'etale $H$\nd scheme $Y$.
It is enough to show that every $H$\nd morphism $f:X' \to Y$ factors as a morphism of schemes over $X$ through some $X_\lambda$.
Since $Y$ is of finite type, $f$ has such a factorisation on every affine open subscheme of $X$.
Consider the functor on
schemes over $X$ that sends $Z$ to the set of factorisations of $f \times_X Z$ through $(\pr_\lambda) \times_X Z$.
Restricting to affine open subschemes of $X$ and then replacing $X'$ by $X_{\lambda'}$ for $\lambda'$
sufficiently large and arguing \'etale locally shows that it is represented by an open and closed 
subscheme $U_\lambda$ of $X$.
Further $U_\lambda$ is an $H$\nd subscheme of $X$.
For $\lambda$ large, $U_\lambda \ne \emptyset$ and hence $U_\lambda = X$.

Let $Z$ be a finite \'etale scheme over $X$.
Then the groupoid 
\[
\underline{\Iso}_X(Z)
\]
over $X$, with points in $T$ above $(x_1,x_0)$ the set of isomorphisms
from $Z_{x_0}$ to $Z_{x_1}$ of schemes over $T$, exists, and it is finite \'etale
when $Z$ has constant rank over $X$.
An action of a pregroupoid over $X$ on $Z$ is the same as a morphism over $X$ to $\underline{\Iso}_X(Z)$.

Let $K$ be a transitive affine groupoid over $X$.
Then $X$ is geometrically $K$\nd connected.
For any finite \'etale $K$\nd scheme, and hence any pro\'etale $K$\nd scheme,
the action of $K$ factors uniquely through an action of $K_\mathrm{\acute{e}t}$.
If $Y = \Spec(\sR)$ is an affine $K$\nd scheme, then $\sR$ is the filtered colimit of its 
finitely generated $K$\nd subalgebras $\sR_0$, and $Y$ is pro\'etale if and only if each
$\sR_0$ is finite \'etale.
Thus by  Lemmas~\ref{l:morphpull}\ref{i:morphpullsch} and \ref{l:prereppull} 
$Y$ is pro\'etale as a $K$\nd scheme if and only if it is pro\'etale as an $X$\nd scheme,
and if $k'$ is an extension of $k$, then $Y$ is pro\'etale if and only if $Y_{k'}$ 
is pro\'etale.

\begin{lem}\label{l:pureinsext}
Let $k'$ be a purely inseparable extension of $k$.
Then extension of scalars from $k$ to $k'$ induces equivalences from the category of pro\'etale groupoids over $H$
to the category of pro\'etale groupoids over $H_{k'}$ and from the category of pro\'etale $H$\nd schemes to the
category of pro\'etale $H_{k'}$\nd schemes.
\end{lem}

\begin{proof}
By topological invariance of the \'etale site \cite[VIII~1.1]{SGA4}, extension of scalars induces for any $k$\nd scheme $Z$
an equivalence from finite \'etale schemes over $Z$ to finite \'etale schemes over $Z_{k'}$.
Reducing to the case where $Z$ is affine shows that a similar equivalence holds with ``finite \'etale schemes'' replaced by 
``filtered limits of finite \'etale schemes''.
The required equivalences follow.
\end{proof}

\begin{lem}\label{l:transproet}
Let $K$ be a transitive affine groupoid over $X$.
\begin{enumerate}
\item\label{i:transproetcon}
A pro\'etale $K$\nd scheme is transitive if and only if it is non-empty and geometrically $K$\nd connected.
\item\label{i:transproetact}
A transitive affine $K$\nd scheme is pro\'etale if and only if $K$ acts on it through $K_{\mathrm{\acute{e}t}}$.
\end{enumerate}
\end{lem}

\begin{proof}
\ref{i:transproetcon}
By Lemmas~\ref{l:morphpull}\ref{i:morphpullsch} and \ref{l:prereppull} we reduce to the case where $X$ is quasi-compact, and then by Lemma~\ref{l:colimHomR} with $\sV = \sV' = \sO_X$ to the case of a finite \'etale $K$\nd scheme $X'$.
After extension of scalars and pullback, we reduce further to the case where $X = \Spec(k)$ and
$X$ is a finite coproduct of copies of $X$,
which is clear.

\ref{i:transproetact} The ``only if'' has been noted above.
To prove the ``if'', we may suppose that $K = K_{\mathrm{\acute{e}t}}$ is pro\'etale.
By extension of scalars and pullback using Lemmas~\ref{l:morphpull}\ref{i:morphpullsch} and \ref{l:prereppull}, 
we reduce to the case where $X = \Spec(k)$ and the transitive $K$\nd scheme $X'$ 
has a $k$\nd point.
Then $K$ is a pro\'etale $k$\nd group $G$, and $X' = G/G'$ for a $k$\nd subgroup $G'$ of $G$.
Writing $G$ as the filtered limit of its finite \'etale quotients $G_\lambda$ shows that
$X' = \lim_\lambda G_\lambda/G'{}\!_\lambda$ with $G'{}\!_\lambda$ the image of $G'$ in $G_\lambda$.
\end{proof}

\begin{lem}\label{l:Cequiv}
Let $\sG$ be a full subcategory of the category of transitive affine groupoids over $H$ which is closed
under the formation of finite products, filtered limits, transitive affine subgroupoids over $H$,
and quotient groupoids.
Then the following conditions are equivalent:
\begin{enumerate}
\renewcommand{\theenumi}{(\alph{enumi})}
\item\label{i:Cinitial}
$\sG$ has an initial object;
\item\label{i:Cintersect}
for any $K$ of finite type in $\sG$, the intersection of two transitive affine subgroupoids over $H$ of $K$
is transitive affine;
\item\label{i:Cequaliser}
for any $K$ and $K'$ of finite type in $\sG$, the equaliser of two morphisms from $K$ to $K'$ over $H$
is transitive affine.
\end{enumerate}
\end{lem}

\begin{proof}
\ref{i:Cinitial} $\implies$ \ref{i:Cintersect}:
If $K_0$ is initial in $\sG$, then the intersection of two subgroupoids in $\sG$ of $K$
is transitive because $K_0 \to K$ factors through it.

\ref{i:Cintersect} $\implies$ \ref{i:Cequaliser}:
The equaliser of two morphisms from $K$ to $K'$ is embedded in $K \times_{[X]} K'$ as the intersection of their graphs.

\ref{i:Cequaliser} $\implies$ \ref{i:Cinitial}:
Assume \ref{i:Cequaliser}.
Denote by $\sG_0$ the full subcategory
of $\sG$ consisting of those $K$ in $\sG$ of finite type that have no transitive affine subgroupoid over $H$ other than $K$.
Given $K$ of finite type in $\sG$, there is at most one morphism from $K_0$ to $K$ in $\sG$ for any $K_0$ in $\sG_0$,
and there is a $K_0$ in $\sG_0$ for which such a morphism exists, because $K$ has a subgroupoid in $\sG_0$.
Now $\sG_0$ has a small skeleton, and by \ref{i:Cequaliser} its dual is filtered.
Thus $\lim_{K_0 \in \sG_0}K_0$ lies in $\sG$, and there exists a unique
morphism from it to any $K$ of finite type in $\sG$, and hence to any $K$ in $\sG$.
\end{proof}

Let $k'$ be a finite extension of $k$ and $G'$ be an affine $k'$\nd group.
Then the Weil restriction $R_{k'/k}G'$ exists and is an affine $k$\nd group.
Suppose that $G' \times_{k'} X_{k'}$ has a non-constant $H_{k'}$\nd invariant cross-section,
i.e.\ that
\begin{equation}\label{e:strincl}
H^0_{H_{k'}}(X_{k'},G') \supsetneqq G'(k')_{k'}.
\end{equation}
Then for any transitive affine groupoid $K$ over $X$ it can be seen as follows that there exist 
two distinct morphism from $K$ to $R_{k'/k}G' \times_k [X]$ over $H$, i.e.\ that
\begin{equation}\label{e:Homge}
\card(\Hom_H(K,(R_{k'/k}G') \times_k [X])) > 1.
\end{equation}
We may suppose that $K = [X]$.
By \eqref{e:strincl} and the universal property $R_{k'/k}G'$, there exists a non-constant $H$\nd invariant
cross-section $\alpha$ of the diagonal $(R_{k'/k}G') \times_k X$ of $(R_{k'/k}G') \times_k [X]$. 
The constant morphism $[X] \to (R_{k'/k}G') \times_k [X]$ and its conjugate by $\alpha$ are then 
distinct morphisms over $H$.

\begin{prop}\label{p:etfunexist}
The category of pro\'etale groupoids over $H$ has an initial object
if and only if $X$ is geometrically $H$\nd connected.
\end{prop}

\begin{proof}
Suppose that $X$ is geometrically $H$\nd connected.
We show that \ref{i:Cintersect} of Lemma~\ref{l:Cequiv} holds with $\sG$ the category of pro\'etale groupoids over $H$.
Let $K_1$ and $K_2$ be transitive affine subgroupoids over $H$ of $K$.
By extension of scalars and pullback onto a point,
the $K$\nd schemes $K/K_i$ exist and are finite \'etale.
The base cross-sections $s_i$ of the $K/K_i$ are $H$\nd morphisms.
Writing $K/K_1 \times_X K/K_2$ as the disjoint union of its $K$\nd connected open and closed $K$\nd subschemes $Z$ 
gives  for some $Z$ a factorisation 
\[
X \xrightarrow{s} Z \to K/K_1 \times_X K/K_2
\]
of the $H$\nd morphism $(s_1,s_2)$.
Further $Z$ is geometrically $K$\nd connected, and hence by Lemma~\ref{l:transproet}\ref{i:transproetcon} a transitive $K$\nd scheme, 
because $s$ induces a $k$\nd homomorphism from $H^0_{K \times_X Z}(Z,\sO_Z)$ to $H^0_H(X,\sO_X)$.
The stabiliser of $s$ is then a transitive subgroupoid over $H$ of $K$ contained in $K_1 \cap K_2$.
Hence $K_1 \cap K_2$ is transitive.

Conversely if $X$ is not geometrically $H$\nd connected,
there exists a finite separable extension $k'$ of $k$ such that \eqref{e:strincl} and hence \eqref{e:Homge} holds for 
$G' \ne 1$ finite discrete.
Thus the category of pro\'etale groupoids over $H$ cannot have an initial object.
\end{proof}

Recall that an affine $k$\nd group $G$ is said to be diagonalisable if every representation of $G$ (i.e.\ 
finite-dimensional $G$\nd module) is a direct sum of $1$\nd dimensional representations, and that $G$ is said to be 
of multiplicative type if the $k'$\nd group $G_{k'}$ is diagonalisable for some extension $k'$ of $k$.
The class of $k$\nd groups of multiplicative type is closed under the formation of $k$\nd subgroups, $k$\nd quotients,
and limits. 
If $k'$ is an extension of $k$, then $G$ is of multiplicative type if and only if $G_{k'}$ is.

By a \emph{groupoid of multiplicative type over $X$} we mean a transitive affine groupoid over $X$ with fibres
above the diagonal of multiplicative type.
The class of groupoids of multiplicative type over $X$ is closed under the formation of transitive affine
subgroupoids, quotient groupoids, products, and filtered limits.
If $k'$ is an extension of $k$, then a groupoid $K$ over $X$ is of multiplicative type if and only
if the groupoid $K_{k'}$ over $X_{k'}$ is.

Let $K$ be a groupoid of multiplicative type over $X$ and $\sV$ be a representation of $K$. 
Then there is an extension $k'$ of $k$ such that the representation $\sV_{k'}$ of $K_{k'}$
is a direct sum of representations of rank $1$, as follows by reducing after extension of scalars to the case where $X$ has a $k$\nd point.
Considering $\End_K(\sV)$ shows further that $k'$ may be taken to be finite separable.

Let $K$ be a transitive affine groupoid over $X$.
If $K'$ is a transitive affine subgroupoid of $K$ for which $K'{}^\mathrm{diag}$ is a $K$\nd subgroup of $K^\mathrm{diag}$,
then the image of $K'$ in $K/K'{}^\mathrm{diag}$ is a subgroupoid $[X]$ of $K/K'{}^\mathrm{diag}$ with inverse image
$K'$ in $K$.
When $K$ is commutative,
$K^\mathrm{diag}$ is a constant $K$\nd group, and $K'{}^\mathrm{diag}$ is a $K$\nd subgroup of it
for any transitive affine subgroupoid $K'$ of $K$.

\begin{prop}\label{p:multinit}
The category of groupoids of multiplicative type over $H$ has an initial object if and only if 
\begin{equation}\label{e:multinit}
H^0_{H_{k'}}(X_{k'},(\bG_m)_{k'}) = k'{}^*
\end{equation}
for every finite separable extension $k'$ of $k$.
\end{prop}

\begin{proof}
Suppose that \eqref{e:multinit} holds.
We show that \ref{i:Cintersect} of Lemma~\ref{l:Cequiv} is satisfied with $\sG$ the category of groupoids of 
multiplicative type over $H$.

We first show that for $K$ as in  \ref{i:Cintersect} of Lemma~\ref{l:Cequiv},
any two morphisms $[X] \to K$ over $H$ coincide.
To do this we may suppose after a finite separable extension of scalars
that $K = \underline{\Iso}_X(\sL)$ for a line bundle $\sL$ over $X$.
It is then to be shown that two actions of $[X]$ on $\sL$ coincide provided that their restrictions to $H$ coincide.
By Lemma~\ref{l:prereppull}, there is an $[X]$\nd isomorphism $\alpha$ between the representations of $[X]$
defined by the two actions.
By \eqref{e:multinit}, the $H$\nd automorphism underlying $\alpha$ is multiplication by an element of $k^*$.
The two actions of $[X]$ thus coincide as required.

Now let $K_1$ and $K_2$ be transitive affine subgroupoids over $H$ of $K$.
If we write $J_i$ for the $K$\nd subgroup $K_i{}\!^\mathrm{diag}$ of $K^\mathrm{diag}$, then
\[
r:K \to K/J_1 \times_{K/(J_1J_2)} K/J_2
\]
with components the projections is faithfully flat by Lemma~\ref{l:morphpull}\ref{i:morphpullgpd}.
The images of the $K_i$ in the $K/J_i$ are subgroupoids $[X]$ of the $K/J_i$ over $H$, which by the above with 
$K/(J_1J_2)$ for $K$ define a subgroupoid $[X]$ of the target of $r$.
The inverse image of $[X]$ under $r$ is $K_1 \cap K_2$, which is thus transitive.

Conversely if $k'$ is a finite separable extension of $k$ for which \eqref{e:multinit} does not hold,
then \eqref{e:strincl} and hence \eqref{e:Homge} holds with
$G' = (\bG_m)_{k'}$.
Thus the category of groupoids of multiplicative type over $H$
cannot have an initial object.
\end{proof}

An affine $k$\nd group will be said to be \emph{of pro\'etale by multiplicative type} if its identity component
is of multiplicative type.
The class of $k$\nd groups of pro\'etale by multiplicative type is closed under the formation of $k$\nd subgroups,
$k$\nd quotients, and limits.
If $k'$ is an extension of $k$, then $G$ is of pro\'etale by multiplicative type if and only if $G_{k'}$ is.

By a \emph{groupoid of pro\'etale by multiplicative type over $X$} we mean a transitive affine groupoid $K$ over $X$ with fibres above the diagonal of pro\'etale by multiplicative type.
The class of groupoids of pro\'etale by multiplicative type over $X$ is closed under the formation of transitive affine
subgroupoids, quotient groupoids, products, and filtered limits.
If $k'$ is an extension of $k$, then a groupoid $K$ over $X$ is of pro\'etale by multiplicative type if and only
if the groupoid $K_{k'}$ over $X_{k'}$ is.

Let $Y' \to Y$ be a finite \'etale morphism.
Since $Y'$ is locally on $Y$ in the \'etale topology a finite disjoint union of copies of $Y$,
the Weil restriction $R_{Y'/Y}Z$ of a scheme $Z$ affine over $Y'$ exists and is affine over $Y$.

\begin{prop}\label{p:etmultinit}
The category of groupoids of pro\'etale by multiplicative type over $H$ has an initial object
if and only if $X$ is geometrically $H$\nd connected and 
\begin{equation}\label{e:etmultinit}
H^0_{H_{k'} \times_{X_{k'}} X'}(X',(\bG_m)_{k'}) = k'{}^*
\end{equation}
for every finite separable extension $k'$ of $k$ and non-empty geometrically $H_{k'}$\nd connected 
finite \'etale $H_{k'}$\nd scheme $X'$.
\end{prop}

\begin{proof}
Suppose that $X$ is geometrically $H$\nd connected and \eqref{e:etmultinit} holds
for $k'$ and $X'$ as in the Proposition.
We show that \ref{i:Cequaliser} of Lemma~\ref{l:Cequiv} holds with $\sG$ the category of groupoids of pro\'etale by 
multiplicative type over $H$.
To show that the equaliser of morphisms $h_1$ and $h_2$ from $K$ to $K'$ over $H$ is transitive affine,
we may by Lemmas~\ref{l:finneutr}, \ref{l:pureinsext} and \ref{l:transproet}\ref{i:transproetact} 
suppose after a finite separable extension of scalars that 
there exists a simply transitive $K_\mathrm{\acute{e}t}$\nd scheme $Z$,
finite \'etale over $X$.
After a further finite separable extension of scalars we may suppose
that $Z$ has an $H$\nd connected component $X'$ which is geometrically $H$\nd connected.
Replacing $X$ and $H$ by $X'$ and $H \times_X X'$, and $Z$, $K$, $K'$ , $h_1$ and $h_2$ by their pullbacks along 
$X' \to X$, we may suppose that $Z$ has an $H$\nd section.
The embedding of its stabiliser is a morphism
\[
e:[X] \to K_\mathrm{\acute{e}t}
\]
of groupoids over $H$,
and $(h_1)_\mathrm{\acute{e}t} \circ e = (h_2)_\mathrm{\acute{e}t} \circ e$ by Proposition~\ref{p:etfunexist}.
Replacing $K$ and $K'$ by $K \times_{K_{\mathrm{\acute{e}t}}} [X]$ and 
$K' \times_{K'{}\!_{\mathrm{\acute{e}t}}} [X]$, we may suppose finally that $K$ and $K'$ are of 
multiplicative type.
The equaliser of $h_1$ and $h_2$ is then transitive by Proposition~\ref{p:multinit}. 

Conversely suppose that the category of groupoids of pro\'etale by multiplicative type over $H$ has an initial object $K$.
Then $K_\mathrm{\acute{e}t}$ is initial in the category of pro\.etale groupoids over $H$, so that
$X$ is geometrically $H$\nd connected by Proposition~\ref{p:etfunexist}.
Let $k'$ be a finite separable extension of $k$ and $X'$ be a non-empty geometrically connected finite \'etale $H_{k'}$\nd scheme.
As an $H$\nd scheme, $X'$ is finite \'etale, 
and the action $H \to \underline{\Iso}_X(X')$ of $H$ on $X'$ factors uniquely through an action of $K$.
Further $K$ acts on the Weil restriction
\[
J = R_{X'/X}(\bG_m \times_X X')
\]
through its action on $X'$.
As a $K_{k'}$\nd scheme, $X'$ is transitive, so that
\[
H^0_K(X,J) = H^0_{K \times_X X'}(X',\bG_m) = H^0_{K_{k'} \times_{X_{k'}} X'}(X',(\bG_m)_{k'}) = k'{}^*
\]
by transitivity in $k'$\nd schemes of $K_{k'} \times_{X_{k'}} X'$.
Suppose \eqref{e:etmultinit} does not hold.
Then 
\[
H^0_H(X,J) = H^0_{H \times_X X'}(X',\bG_m) = H^0_{H_{k'} \times_{X_{k'}} X'}(X',(\bG_m)_{k'}) \supsetneqq k'{}^*.
\]
The embedding $K \to J \rtimes_X K$ and its conjugate by an element of $H^0_H(X,J)$ outside $H^0_K(X,J)$ are
then distinct morphisms over $H$,
which is impossible because $K$ is initial and $J \rtimes_X K$ is of pro\'etale by multiplicative type.
Thus \eqref{e:etmultinit} holds.
\end{proof}

Let $X'$ be a finite \'etale scheme over $X$ of constant rank.
Then $\underline{\Iso}_X(X')$ is a finite \'etale groupoid over $X$.
Let $X_1$ be a closed subscheme of $X'$ which is \'etale of constant rank over $X$.
Then $X_1$ is also an open subscheme of $X'$, so that $X' = X_1 \amalg X_2$
with $X_2$ finite \'etale of constant rank over $X$.
The embeddings of $X_1$ and $X_2$ into $X'$ then define an open and closed immersion
\begin{equation}\label{e:isoembed}
\underline{\Iso}_X(X_1) \times_{[X]} \underline{\Iso}_X(X_2) \to
\underline{\Iso}_X(X')
\end{equation}
of finite \'etale groupoids over $X$.
If $X'$ has a structure of $H$\nd scheme, then $X_1$ is a closed $H$\nd subscheme of $X$ if and only the action
\[
H \to \underline{\Iso}_X(X')
\]
of $H$ on $X'$ factors through \eqref{e:isoembed}.

Let $K$ be a groupoid over $H$. 
Restricting the action from $K$ to $H$ defines a faithful forgetful functor $\omega$ from finite \'etale $K$\nd schemes
to finite \'etale $H$\nd schemes.
The functor $\omega$ admits transport of structure,
i.e.\ given an isomorphism $i$ with source $\omega(Z)$,
there is a unique isomorphism $j$ with source $Z$ such that $\omega(j) = i$.
Thus if $\omega$ is fully faithful (resp.\ essentially surjective) then it is injective (resp.\ surjective) on objects.
It follows that $\omega$ is an equivalence if and only if it is an isomorphism of categories.
Similar remarks apply much more generally, and in particular to the forgetful functor from pro\'etale $K$\nd schemes to pro\'etale
$H$\nd schemes.

\begin{lem}\label{l:funact}
Let $K$ be a pro\'etale groupoid over $H$.
Then the following conditions are equivalent:
\begin{enumerate}
\renewcommand{\theenumi}{(\alph{enumi})}
\item\label{i:funactff}
The forgetful functor from the category of finite \'etale $K$\nd schemes to the category of finite \'etale $H$\nd schemes  
is fully faithful;
\item\label{i:funactmin}
$X$ is geometrically $H$\nd connected, and $K$ has no pro\'etale subgroupoid over $H$ other than itself.
\end{enumerate}
\end{lem}

\begin{proof}
Assume \ref{i:funactff}.
Then any two actions $K$ on a finite \'etale scheme over $X$
extending a given action $H$ coincide.
Thus $X$ is geometrically $H$\nd connected:
otherwise there would be a finite separable extension $k'$ of $k$ such that \eqref{e:strincl}
and hence \eqref{e:Homge} holds for $G' \ne 1$ finite discrete,
and hence with $G = R_{k'/k}G'$ two distinct actions
\[
K \to \underline{\Iso}_G(G \times_k X) = G \times_k [X]
\]
of $K$ on $G \times_k X$ extending the constant action of $H$.
Let $K'$ be a pro\'etale subgroupoid of $K$ over $H$.
To prove that $K' = K$, we may suppose after replacing $K$ by a quotient that 
$K$ is of finite type.
Then the $K$\nd scheme $K/K'$ exists and is finite \'etale.
Its base cross-section 
\[
X \to K/K'
\]
is a morphism of $H$\nd schemes, and hence by \ref{i:funactff} of $K$\nd schemes. Thus $K' = K$.

Assume \ref{i:funactmin}. 
Then by Proposition~\ref{p:etfunexist}, the category of pro\'etale groupoids over $H$ has an initial object $K_0$.
The unique morphism from $K_0$ to $K$ over $H$ is faithfully flat.
To prove \ref{i:funactff}, we may thus suppose after replacing $K$ by $K_0$ that $K$ is initial in the category
of pro\'etale groupoids over $H$.
Since $X$ is geometrically $H$\nd connected,
any finite \'etale $H$\nd scheme has constant rank.
Any $H$\nd subscheme $X_1$ of a finite \'etale $K$\nd scheme $X'$ is thus a $K$\nd subscheme,
because the action of $K$ on $X'$ factors through \eqref{e:isoembed}.
Since an $H$\nd morphism $Y \to Z$ of $K$\nd schemes is a $K$\nd morphism
if its graph is a $K$\nd subscheme of $Y \times_X Z$, \ref{i:funactff} follows.
\end{proof}

\begin{prop}\label{p:funact}
Let $K$ be a pro\'etale groupoid over $H$.
Denote by $\omega$ the forgetful functor from the category of finite \'etale $K$\nd schemes to the category of finite \'etale 
$H$\nd schemes.
Then the following conditions are equivalent:
\begin{enumerate}
\renewcommand{\theenumi}{(\alph{enumi})}
\item\label{i:funactequ}
$\omega$ is an equivalence of categories;
\item\label{i:funactiso}
$\omega$ is an isomorphism of categories;
\item\label{i:funactinit}
$K$ is initial in the category of pro\'etale groupoids over $H$.
\end{enumerate}
\end{prop}

\begin{proof}
That \ref{i:funactequ} and \ref{i:funactiso} are equivalent has been seen above.
Assume \ref{i:funactinit}.
Then $X$ is geometrically $H$\nd connected by Proposition~\ref{p:etfunexist}.
Since $1_K$ must factor through any pro\'etale subgroupoid of $K$ over $H$, Lemma~\ref{l:funact}
shows that $\omega$ is fully faithful.
For any finite \'etale $H$\nd scheme $X'$, the action of $H \to \underline{\Iso}_X(X')$ of $H$ on $X'$ 
factors uniquely through an action of $K$.
Thus $\omega$ is bijective on objects.
This proves \ref{i:funactiso}.

Assume \ref{i:funactiso}.
To prove \ref{i:funactinit}, it is enough to show there is a unique morphism from $K$
to every finite \'etale groupoid $K'$ over $H$.
By Lemmas~\ref{l:finneutr}, \ref{l:pureinsext} and \ref{l:transproet}\ref{i:transproetact} there exists 
for some finite separable extension $k'$ of $k$ a simply transitive $K'{}\!_{k'}$\nd scheme $X'$,
finite \'etale over $X_{k'}$.
Restricting the action of $K'{}\!_{k'}$ to $K'$ gives an embedding
\[
K' \to \underline{\Iso}_X(X').
\]
Since $\omega$ is bijective on objects, there is a unique morphism $l:K \to \underline{\Iso}_X(X')$ over $H$.
By Lemma~\ref{l:funact} and Proposition~\ref{p:etfunexist} the category of pro\'etale groupoids over $H$ has an initial object $K_0$, 
and by Lemma~\ref{l:funact} 
the unique morphism $j:K_0 \to K$ over $H$ is faithfully flat.
Since $l \circ j$ factors through $K'$, so also does $l$.
\end{proof}

\begin{cor}\label{c:etfunact}
Let $K$ be a transitive affine groupoid over $H$ for which $K_\mathrm{\acute{e}t}$ is initial in 
the category of pro\'etale groupoids over $H$.
Then the forgetful functor from the category of pro\'etale $K$\nd schemes to the category of
pro\'etale $H$\nd schemes is an isomorphism of categories.
\end{cor}

\begin{proof}
By Lemma~\ref{l:transproet}, we may
replace $K$ by $K_\mathrm{\acute{e}t}$.
As seen above, it is enough to prove that the forgetful functor is an equivalence.
Since $X$ is $H$\nd connected by Proposition~\ref{p:etfunexist}, this follows from Proposition~\ref{p:funact} and \eqref{e:colimfinet}.
\end{proof}

Let $Y' \to Y$ be a finite \'etale morphism.
Since the Weil restriction functor $R_{Y'/Y}$ is right adjoint to the pullback functor from schemes affine over $Y$ 
to schemes affine over $Y'$,
it commutes with finite limits and is compatible with pullback along morphisms to $Y$.
Hence if $k'$ is a finite separable extension of $k$, extension of scalars from the category of 
affine groupoids over $X$ to the category of affine groupoids in $k'$\nd schemes over $X_{k'}$ has a right adjoint,
which we write as $R_{[X]_{k'}/[X]}$.

Let $k'$ be an algebraic extension of $k$.
Then $k'$ is the filtered colimit of its finite subextensions $k_\lambda$.
Let $K'$ be a transitive affine groupoid of finite type in $k'$\nd schemes over $X_{k'}$,
and suppose that $X$ is quasi-compact and quasi-separated and $H_{[1]}$ is quasi-compact.
It can be seen as follows that  $K'$
is isomorphic to $K_\mu \times_{k_\mu} k'$
for some $\mu$ and transitive affine groupoid of finite type $K_\mu$ in $k_\mu$\nd schemes over 
$X_{k_\mu}$, and for any such $\mu$ and $K_\mu$ 
that every morphism from $H_{k'}$ to $K_\mu \times_{k_\mu} k'$ over $X_{k'}$
arises from a morphism  from $H_{k_\nu}$ to $K_\mu \times_{k_\mu} k_\nu$ over $X_{k_\nu}$
for some $\nu \ge \mu$.
The scheme $K'$ over $X_{k'} \times_{k'} X_{k'}$ is of finite presentation, faithfully flat and affine,  
and hence \cite[IV~8.8.2(ii)]{EGA} arises for some $\mu$ by extension of
scalars from a scheme $K_\mu$ over $X_{k_\mu} \times_{k_\mu} X_{k_\mu}$
which is of finite presentation, faithfully flat and affine.
Reducing to the case where $X$ and $Z$ are affine 
shows that if $Z$ is a scheme over $X_{k_\mu}  \times_{k_\mu} X_{k_\mu} $ which is quasi-compact, 
then any morphism from
$Z \times_{k_\mu} k'$ to $K_{\mu} \times_{k_\mu} k'$ over $X_{k'} \times_{k'} X_{k'}$ arises from a morphism 
from $Z \times_{k_\mu} k_{\nu}$ to $K_{\mu} \times_{k_\mu} k_{\nu}$ over $X_{k_\nu} \times_{k_\nu} X_{k_\nu}$ for some 
$\nu \ge \mu$.
Taking $K_{\mu} \times_{X_{k_\mu}} K_{\mu}$ for $Z$ and replacing $\mu$ by $\nu$ shows that the 
composition of $K'$ is defined over $k_{\mu}$, and similarly for the identity.
Taking $(H_{[1]})_{k_{\mu}}$ for $Z$ gives the statement for $H$.

\begin{prop}\label{p:funscalext}
Let $K$ be a groupoid over $H$ which is pro\'etale (resp.\ of multiplicative type, resp. of pro\'etale by multiplicative type)
and $k'$ be a separable algebraic extension of $k$.
Suppose that either of the following conditions holds:
\begin{enumerate}
\renewcommand{\theenumi}{(\alph{enumi})}
\item\label{i:funscalextfin}
$k'$ is finite over $k$;
\item\label{i:funscalextqc}
$X$ is quasi-compact and quasi-separated and $H_{[1]}$
is quasi-compact.
\end{enumerate}
Then $K$ is initial in the category of groupoids over $H$ which are pro\'etale 
(resp.\ of multiplicative type, resp.\ of pro\'etale by multiplicative type)
if and only if $K_{k'}$ is initial in the category of groupoids over $H_{k'}$ which are 
pro\'etale (resp.\ of multiplicative type, resp.\ of pro\'etale by multiplicative type).
\end{prop}

\begin{proof}
We write the proof for the pro\'etale case: the other cases are almost identical.
Suppose that $K$ is initial in the category of pro\'etale groupoids over $H$.
Let $K'$ be a finite \'etale groupoid over $X_{k'}$.
To show that $K_{k'}$ is initial in the category of pro\'etale groupoids over $H_{k'}$, it is enough to show
that every morphism from $H_{k'}$ to $K'$ over $X_{k'}$ factors uniquely through $K_{k'}$.
If \ref{i:funscalextfin} holds, this follows from the fact that
every morphism from $H$ to $R_{[X]_{k'}/[X]}K'$ over $X$ factors uniquely through $K$.
Suppose that \ref{i:funscalextqc} holds, and write $k'$ as the filtered colimit 
of its finite subextensions $k_\lambda$.
Since $K'$ is as above isomorphic to $K_\mu \times_{k_\mu} k'$ for some 
$\mu$ and finite \'etale groupoid $K_\mu$ over $X_{k_\mu}$,
and since every morphism from $H_{k'}$ or $K_{k'}$ to $K_\mu \times_{k_\mu} k'$ over $X_{k'}$
arises from a morphism  from $H_{k_\nu}$ or $K_{k_\nu}$ to $K_\mu \times_{k_\mu} k_\mu$ over $X_{k_\nu}$
for some $\nu \ge \mu$, the required result follows from the finite case.

Conversely suppose that $K_{k'}$ is initial in the category of pro\'etale groupoids over $H_{k'}$.
Then by Lemma~\ref{l:Cequiv}, the category of pro\'etale groupoids over $H$ has an initial object $K_0$.
By the ``only if'', $K_0{}_{k'}$ is initial in the category of pro\'etale groupoids over $H_{k'}$.
The unique morphism $h:K_0 \to K$ over $H$ is thus an isomorphism because $h_{k'}$ is an isomorphism.
\end{proof}

Let $K$ be a groupoid over $X$ and $X'$ be a finite \'etale $K$\nd scheme.
Then the projection $K \times_X X' \to K$ is finite \'etale.
Let $\widetilde{K}$ be a groupoid over $X'$ equipped with an affine morphism
$\widetilde{K} \to K \times_X X'$ of groupoids over $X'$.
We define in the following way a structure of groupoid over $X$ on the Weil restriction 
$R_{K \times_X X'/K}\widetilde{K}$, 
such that 
\[
R_{K \times_X X'/K}\widetilde{K} \to K
\]
is a morphism of groupoids over $X$.
By definition, a point of $R_{K \times_X X'/K}\widetilde{K}$ in a $k$\nd scheme $S$ above the point $v$ of $K$ is an assignment $v'$
to every $S' \to S$ and point $x'$ of $X'$ in $S'$ above $d_1(v)$ of a point $v'(x')$ of $\widetilde{K}$ in $S'$ 
above $(v,x')$, which is compatible with morphisms $S'' \to S'$ over $S$.
Then $1_x(x') = 1_{\widetilde{K}}$, and
the composite of $w'$ above $w$ with $v'$ above $v$ is $w' \circ v'$ above $w \circ v$ with
\[
(w' \circ v')(x') = w'(vx') \circ v'(x').
\]
If $\widetilde{H}$ is a pregroupoid over $X$ equipped with a morphism $\widetilde{H} \to K$ over $X$,
and $\Hom_X^K$ for example denotes the set of morphisms of pregroupoids over $X$ compatible with the morphisms to $K$,
then we have a bijection
\begin{equation}\label{e:Weilgrpdpull}
\Hom_{X'}^{K \times_X X'}(\widetilde{H} \times_X X',\widetilde{K}) \iso \Hom_X^K(\widetilde{H},R_{K \times_X X'/K}\widetilde{K}),
\end{equation}
natural in $\widetilde{H}$, where $f:\widetilde{H} \times_X X' \to \widetilde{K}$ corresponds to 
$\widetilde{H} \to R_{K \times_X X'/K}\widetilde{K}$ which sends the point $h$ of $\widetilde{H}_{[1]}$ to the point
$x' \mapsto f(h,x')$ of $R_{K \times_X X'/K}\widetilde{K}$.

Suppose now that $K$ and $\widetilde{K}$ are transitive affine, that $X'$ is a transitive $K$\nd scheme, and that 
$\widetilde{K} \to K \times_X X'$ is faithfully flat.
Then $R_{K \times_X X'/K}\widetilde{K} \to K$ is faithfully flat, so that $R_{K \times_X X'/K}\widetilde{K}$
is a transitive affine groupoid over $X$.
Let $x$ be a point of $X$ in an algebraically closed extension $k'$ of $k$.
Then since $K^\mathrm{con}$ acts trivially on $X'$, the inverse image in 
$(R_{K \times_X X'/K}\widetilde{K})_{x,x}$ of the identity component of $K_{x,x}$ is a $k'$\nd subgroup of the 
product over $k'$ of the $\widetilde{K}_{x',x'}$ as $x'$ runs over the $k'$\nd points of $X'$ above $x$.
It follows that if $\widetilde{K}$ is pro\'etale, or of pro\'etale by multiplicative type, then so is
$R_{K \times_X X'/K}\widetilde{K}$.   

Let $X'$ be the limit of a filtered inverse system $(X_\lambda)$ of affine $H$\nd schemes,
with the projections $X' \to X_\lambda$ faithfully flat.
Suppose that $X$ is quasi-compact and quasi-separated and $H_{[1]}$ is quasi-compact.
Then by an argument similar to the one for extension of scalars preceding Proposition~\ref{p:funscalext},
every transitive affine groupoid of finite type over $X'$ is isomorphic to $K_\mu \times_{[X_\mu]} [X']$ for some 
$\mu$ and transitive affine groupoid of finite type $K_\mu$ over $X_\mu$,
and for any such $\mu$ and $K_\mu$,
every morphism from $H \times_X X'$ to $K_\mu \times_{[X_\mu]} [X']$ over $X'$ arises from a morphism
from $H \times_X X_\nu$ to $K_\mu \times_{[X_\mu]} [X_\nu]$ over $X_\nu$ for some $\nu \ge \mu$.

\begin{prop}\label{p:funpull}
Let $K$ be a groupoid over $H$ which is pro\'etale (resp.\ of pro\'etale by multiplicative type)
and $X'$ be a transitive pro\'etale $K$\nd scheme.
Suppose that either of the following conditions holds:
\begin{enumerate}
\renewcommand{\theenumi}{(\alph{enumi})}
\item\label{i:funpullfin}
$X'$ is finite over $X$;
\item\label{i:funpullqc}
$X$ is quasi-compact and quasi-separated and $H_{[1]}$
is quasi-compact.
\end{enumerate}
Then $K$ is initial in the category of groupoids over $H$ which are pro\'etale (resp.\ of pro\'etale by multiplicative type)
if and only if $K \times_X X'$ is initial in the category of groupoids over $H \times_X X'$ which are
pro\'etale (resp.\ of pro\'etale by multiplicative type).
\end{prop}

\begin{proof}
We write the proof for the pro\'etale case: the other case is almost identical.
Suppose that $K$ is initial in the category of pro\'etale groupoids over $H$.
Let $K'$ be a finite \'etale groupoid over $X'$.
To show that $K \times_X X'$ is initial in the category of pro\'etale groupoids over $H \times_X X'$, 
it is enough to show
that every morphism from $H \times_X X'$ to $K'$ over $X'$ factors uniquely through $K \times_X X'$.
Suppose that \ref{i:funpullfin} holds.
If $\widetilde{K}$ is a pro\'etale groupoid over $X'$ equipped with a faithfully flat morphism to $K \times_X X'$,
then by the universal property of $K$ and the naturality of \eqref{e:Weilgrpdpull} in $\widetilde{H}$, composition with
$H \times_X X'\to K \times_X X'$ induces a bijection
\[
\Hom_{X'}^{K \times_X X'}(K \times_X X',\widetilde{K}) \iso \Hom_{X'}^{K \times_X X'}(H \times_X X',\widetilde{K}).
\]
Taking $\widetilde{K} = K' \times_{[X']} (K \times_X X')$ with
morphism to $K \times_X X'$ the projection now gives what is required.
Suppose that \ref{i:funpullqc} holds, and write $X'$ as the filtered limit of finite \'etale $K$\nd schemes
$X_\lambda$ with faithfully flat projections.
Since $K'$ is as above isomorphic to $K_\mu \times_{[X_\mu]} [X']$ for some 
$\mu$ and finite \'etale groupoid $K_\mu$ over $X_\mu$,
and since every morphism from $H \times_X X'$ or $K \times_X X'$ to $K_\mu \times_{[X_\mu]} [X']$ over 
$X'$ arises from a morphism  from $H \times_X X_\nu$ or $K \times_X X_\nu$ to 
$K_\mu \times_{[X_\mu]} [X_\nu]$ over $X_\nu$
for some $\nu \ge \mu$, the required result follows from the finite case.

Conversely suppose $K \times_X X'$ is initial in the category of pro\'etale groupoids over $H \times_X X'$.
Then by Lemma~\ref{l:Cequiv} and pullback, the category of pro\'etale groupoids over $H$ has an initial object $K_0$,
and by Proposition~\ref{p:etfunexist} $X'$ is geometrically $H$\nd connected.
The unique $h:K_0 \to K$ over $H$ 
defines on $X'$ a structure of transitive $K_0$\nd scheme by Lemma~\ref{l:transproet}\ref{i:transproetcon},
with $K_0{} \times_X X'$ initial in the category of pro\'etale groupoids over 
$H \times_X X'$ by the ``only if''.
Thus $h$ is an isomorphism because $h \times_X X'$ is.
\end{proof}

We say that $X$ is \emph{geometrically simply $H$\nd connected} if every finite \'etale $H$\nd scheme is constant.
If $X$ is geometrically simply $H$\nd connected,
then it is geometrically $H$\nd connected.
By Lemma~\ref{l:funact} and Proposition~\ref{p:funact}, $X$ is geometrically simply $H$\nd connected if and only if 
$[X]$ is initial in the category of pro\'etale groupoids over $H$.

Suppose $X$ is geometrically $H$\nd connected.
We call a non-empty $H$\nd scheme $X'$ a \emph{geometric universal $H$\nd cover of $X$} if $X'$ is pro\'etale
and geometrically simply $(H \times_X X')$\nd connected.

Suppose $X$ is geometrically $H$\nd connected, 
and that $X$ is quasi-compact and quasi-separated and $H_{[1]}$ is quasi-compact.
By Proposition~\ref{p:etfunexist}, 
the category of pro\'etale groupoids over $H$ has an initial object $K$,
and by Corollary~\ref{c:etfunact}, any pro\'etale $H$\nd scheme $X'$ is the restriction of a unique
pro\'etale $K$\nd scheme.
By Proposition~\ref{p:funpull}, $X'$ is a geometric universal $H$\nd cover of $X$ if and only if
it is a simply transitive $K$\nd scheme.
If $k$ is separably closed, it follows by Corollary~\ref{c:neutral} and Lemmas~\ref{l:pureinsext} and
\ref{l:transproet}\ref{i:transproetact}
that a geometric universal $H$\nd cover of $X$ exists, and is unique up to (not necessarily unique) $H$\nd isomorphism.
Similarly by Lemma~\ref{l:prereppull}, if $x$ is a $k$\nd point of $X$, 
then a $k$\nd pointed geometric universal $H$\nd cover of $(X,x)$ exists, 
i.e.\ a  geometric universal $H$\nd cover equipped with a $k$\nd point above $x$.
By Corollary~\ref{c:etfunact} it is initial in the category of pairs $(X',x')$ with $X'$ a pro\'etale 
$H$\nd scheme and $x'$ a $k$\nd point of $X'$ above $x$.

The set of isomorphism classes of representation of rank $1$ of $H$ forms an abelian group under 
the tensor product, which we write as $\Pic_H(X)$. 
We have
\[
\Pic_H(X) = H^1_H(X,\bG_m),
\]
by the identification of isomorphism classes of representations of rank $n$ of $H$ with elements of \eqref{e:H1HXGLn}.
When $H = X$, this group is the usual Picard group $\Pic(X)$ of $X$, and when $X = \Spec(k)$ and $H$ is a $k$\nd group $G$,
it is the group 
\[
\Hom_k(G,\bG_m)
\]
of characters of $G$.

Let $K$ be a commutative transitive affine groupoid over $X$.
Then $K^\mathrm{diag}$ is a constant $K$\nd group:
\[
K^\mathrm{diag} = Z(K) \times_k X
\]
for a commutative affine $k$\nd group $Z(K)$.
In particular if $\sL$ is a line bundle over $X$,
\[
\underline{\Iso}_X(\sL)^\mathrm{diag} = \bG_m \times_k X
\]
with $\bG_m \times_k X$ acting homothetically on $\sL$.
By assigning to the class of $\rho:K \to \underline{\Iso}_X(\sL)$
the character $\chi$ of $Z(K)$ with $\chi \times_k X = \rho^\mathrm{diag}$, we obtain a homomorphism
\begin{equation}\label{e:Picchar}
\Pic_K(X) \to \Hom_k(Z(K),\bG_m).
\end{equation}
Two representations $\sL$ and $\sL'$ of rank $1$ of $K$ are isomorphic if they are isomorphic after some extension of scalars,
because by Lemma~\ref{l:prereppull} and \eqref{e:Homextiso}, formation of $\Hom_K(\sL,\sL')$ commutes
with extension of scalars.
Reducing after extension of scalars to the case where $X = \Spec(k)$ thus shows that \eqref{e:Picchar} is injective. 
When $k$ is separably closed, \eqref{e:Picchar} is an isomorphism:
to see that $\chi$ lies in its image, we may suppose after pushing $K$ forward along $\chi \times_k X$ using 
Lemma~\ref{l:Kgrpdequiv} that
$Z(K)$ is $\bG_m$, so that $K$ is of multiplicative type and has a faithful family of representations
of rank $1$, whose corresponding characters generate the character group of $Z(K)$.

Fix a separable closure $\overline{k}$ of $k$.
A Galois module is a (discrete) abelian group with a continuous action of the Galois group $\Gal(\overline{k}/k)$. 
To every affine $k$\nd group $G$ there is associated its Galois module of characters, consisting of the characters
of $G_{\overline{k}}$ over $\overline{k}$ with the Galois group acting through its action on $\overline{k}$.
By assigning to each $k$\nd group of multiplicative type its Galois module of characters, we obtain an equivalence
from $k$\nd groups of multiplicative type to Galois modules.
We write
\[
D(M)
\]
for the $k$\nd group of multiplicative type with Galois module of characters $M$.
It comes equipped with an isomorphism of Galois modules
\[
M \iso \Hom_{\overline{k}}(D(M)_{\overline{k}},\bG_m{}_{\overline{k}})
\]
which will be written as $\mu \mapsto \chi_\mu$.

Let $K$ be a transitive affine groupoid over $X$.
Then $\Gal(\overline{k}/k)$ acts on $\Pic_{K_{\overline{k}}}(X_{\overline{k}})$ through its action on $\overline{k}$.
The action is continuous, because by Lemmas~\ref{l:prereppull} and \ref{l:HRrepfp} and \eqref{e:Homextiso}, 
$\Pic_{K_{\overline{k}}}(X_{\overline{k}})$
is the filtered colimit of $\Pic_{K_{k'}}(X_{k'})$ as $k'$ runs over the finite Galois subextensions of $\overline{k}$.
Suppose that $K$ is of multiplicative type.
Then since \eqref{e:Picchar} with $X_{\overline{k}}$ for $X$ and $K_{\overline{k}}$ for $K$ is an isomorphism,
we have
\begin{equation}\label{e:diagPic}
K^\mathrm{diag} = D(\Pic_{K_{\overline{k}}}(X_{\overline{k}})) \times_k X
\end{equation}
where for $\mu$ the class of 
$\overline{\rho}:K_{\overline{k}} \to \underline{\Iso}_{X_{\overline{k}}}(\overline{\sL})$ 
we have $\chi_\mu \times_{\overline{k}} X_{\overline{k}} = \overline{\rho}^\mathrm{diag}$.

When it exists, we denote the initial object in the category of pro\'etale groupoids 
(resp.\ groupoids of multiplicative type,
resp.\ groupoids of pro\'etale by multiplicative type) over $H$ by $\pi_{H,\mathrm{\acute{e}t}}(X)$ 
(resp.\ $\pi_{H,\mathrm{mult}}(X)$, resp.\ $\pi_{H,\mathrm{\acute{e}tm}}(X)$).
When $\pi_{H,\mathrm{\acute{e}tm}}(X)$ exists, both  $\pi_{H,\mathrm{\acute{e}t}}(X)$ and $\pi_{H,\mathrm{mult}}(X)$ exist, by Lemma~\ref{l:Cequiv}.
We then have
\[
\pi_{H,\mathrm{\acute{e}t}}(X) = \pi_{H,\mathrm{\acute{e}tm}}(X)_\mathrm{\acute{e}t},
\]
and when $k$ is of characteristic $0$ we have further
\[
\pi_{H,\mathrm{mult}}(X) = \pi_{H,\mathrm{\acute{e}tm}}(X)_\mathrm{ab}.
\]
Suppose that $X$ is geometrically simply $H$\nd connected.
Then $\pi_{H,\mathrm{\acute{e}t}}(X) = [X]$,
and if either $\pi_{H,\mathrm{\acute{e}tm}}(X)$ or $\pi_{H,\mathrm{mult}}(X)$ exists then both exist,
and they coincide.

Suppose that $\pi_{H,\mathrm{mult}}(X)$ exists.
Then any action $H \to \underline{\Iso}_X(\sL)$ of $H$ on a line bundle $\sL$ extends uniquely to an action of 
$\pi_{H,\mathrm{mult}}(X)$ on $\sL$.
By transport of structure, we thus have an isomorphism
\begin{equation}\label{e:PicmultH}
\Pic_{\pi_{H,\mathrm{mult}}(X)}(X) \iso \Pic_H(X)
\end{equation}
defined by restriction from $\pi_{H,\mathrm{mult}}(X)$ to $H$. 

Suppose that $X$ is quasi-compact and quasi-separated and that $H_{[1]}$ is quasi-compact.
Then by Lemma~\ref{l:HRrepfp}, the action of $\Gal(\overline{k}/k)$ on
$\Pic_{H_{\overline{k}}}(X_{\overline{k}})$ through $\overline{k}$ defines on it a structure of
Galois module.
If $\pi_{H,\mathrm{mult}}(X)$ exists, then by Proposition~\ref{p:funscalext} and \eqref{e:PicmultH},
restriction to $H_{\overline{k}}$ induces an isomorphism of Galois modules from
$\Pic_{\pi_{H,\mathrm{mult}}(X)_{\overline{k}}}(X_{\overline{k}})$ to $\Pic_{H_{\overline{k}}}(X_{\overline{k}})$.
Thus by \eqref{e:diagPic}
\begin{equation}\label{e:multPic}
\pi_{H,\mathrm{mult}}(X)^\mathrm{diag} = D(\Pic_{H_{\overline{k}}}(X_{\overline{k}})) \times_k X
\end{equation}
where for $\mu$ the class of the restriction
of $\overline{\rho}:\pi_{H,\mathrm{mult}}(X)_{\overline{k}} \to \underline{\Iso}_{X_{\overline{k}}}(\overline{\sL})$ 
to $H_{\overline{k}}$ we have $\chi_\mu \times_{\overline{k}} X_{\overline{k}} = \overline{\rho}^\mathrm{diag}$.

\section{Groupoids and Galois extended groups}\label{s:grpdgalext}

\emph{In this section $k$ is a perfect field and $\overline{k}$ is an algebraic closure of $k$.}

\medskip

In this section we describe an equivalence between transitive affine groupoids over 
$\overline{k}$ and Galois extended $\overline{k}$\nd groups, which are affine
$\overline{k}$\nd groups $D$ equipped with extra structure.
When $k$ is of characteristic $0$, this extra structure is simply a topological extension $E$
of the Galois group $\Gal(\overline{k}/k)$ by the group $D(\overline{k})_{\overline{k}}$ of 
$\overline{k}$\nd points over $\overline{k}$.
This case will be used in Sections~\ref{s:curves0} and \ref{s:curves1} to classify principal bundles
under a reductive group over smooth projective curves of genus $0$ or $1$.

We begin, in this paragraph and the three following, with some preliminary remarks on local $k$\nd ringed spaces,
which are in fact valid over an arbitrary base field $k$.
The forgetful functor $Z \mapsto |Z|$ from the category of local $k$\nd ringed spaces to the category of topological
spaces has a fully faithful right adjoint $\Theta \mapsto \Theta_{/k}$,
where $\Theta_{/k}$ is the local $k$\nd ringed space with underlying topological space $\Theta$ and structure
sheaf the constant sheaf $k$.
The counit $|\Theta_{/k}| \to \Theta$ is the identity.
For any continuous map $\Theta' \to \Theta$ and morphism $Z \to \Theta_{/k}$ of local $k$\nd ringed spaces,
the fibre product $\Theta'{}\!_{/k} \times_{\Theta_{/k}} Z$ exists: its underlying topological space
is $\Theta' \times_{\Theta} |Z|$ and its structure sheaf is the inverse image of that of $Z$ along the projection.
Formation of $\Theta_{/k}$ is compatible with extension of scalars.
If $\Theta$ is discrete then $\Theta_{/k}$ is a disjoint union of copies of $\Spec(k)$, 
and if $\Theta$ is profinite
(i.e.\ compact totally disconnected) then $\Theta_{/k}$ is a profinite $k$\nd scheme.
In general, however, $\Theta_{/k}$ need not be a scheme, even if $\Theta$ is totally disconnected.

Let $M$ be a topological group.
Then $M_{/k}$ has a canonical structure of group object in the category of local $k$\nd ringed spaces.
If $M$ is discrete, then an action of $M$ on a local $k$\nd ringed space $Z$, i.e.\ a homomorphism from
$M$ to $\Aut_k(Z)$, is the same as an action $M_{/k} \times_k Z \to Z$
of $M_{/k}$ on $Z$.
In general, we say that an action of $M$ on $Z$ is \emph{continuous} if the action
\[
(M^d)_{/k}\times_k Z \to Z
\]
of the underlying discrete space $M^d$ of $M$ on $Z$ factors through the epimorphism 
\[
(M^d)_{/k} \times_k Z \to M_{/k} \times_k Z.
\]
A continuous action of $M$ on $Z$ is thus the same as an action
\[
M_{/k} \times_k Z \to Z
\] 
of $M_{/k}$ on $Z$.
A continuous right action of $M$ is defined similarly, and is the same as a continuous action of $M^\mathrm{op}$.

The right adjoint $\Theta \mapsto \Theta_{/k}$ to the forgetful functor from local $k$\nd ringed spaces 
to topological spaces has itself a right adjoint.
It sends $Z$ to $Z(k)$, equipped the topology, which we call the 
\emph{Krull topology}, 
with an open base formed by the sets of $k$\nd points in a given open subset $U$ of $Z$
at which given sections $f_1,f_2, \dots ,f_n$ of $\sO_Z$ above $U$ take given values 
$\alpha_1,\alpha_2, \dots ,\alpha_n$ in $k$.
The counit $X(k)_{/k} \to X$ sends a $k$\nd point to its underlying $k$\nd rational point,
and the unit $\Theta \to (\Theta_{/k})(k)$ is the homeomorphism sending a point to its corresponding $k$\nd point.
If $Z'$ is an open local ringed subspace of $Z$, then $Z'(k)$ is an open subset of $Z(k)$.
If $Z$ is a $k$\nd scheme which is locally of finite type, then $Z(k)$ is discrete.
If $Z$ is an affine $k$\nd scheme, then $Z$ as the filtered limit in the category of local $k$\nd ringed spaces
of affine $k$\nd schemes $Z_\lambda$ of finite type, 
and $Z(k)$ is the filtered limit of the discrete spaces $Z_\lambda(k)$.

More generally, let $k'$ be an extension of $k$.
Then the functor $\Theta \mapsto \Theta_{/k'}$ from topological spaces to local $k$\nd ringed spaces
has a right adjoint.
It sends $Z$ to $Z(k')$ equipped with the topology, which we again call the Krull topology,
with an open base formed by the sets of $k'$\nd points in a 
given open subset $U$ of $Z$
at which given sections $f_1,f_2, \dots ,f_n$ of $\sO_Z$ above $U$ take given values 
$\alpha'{}\!_1,\alpha'{}\!_2, \dots ,\alpha'{}\!_n$ in $k'$.
If $Z$ has a structure of local $k'$\nd ringed space, then the set $Z(k')_{k'}$ of $k'$\nd points over $k'$, 
equipped with the Krull topology in the sense of local $k'$\nd ringed spaces, is a subspace of $Z(k')$.
If $Z$ is a $k$\nd scheme, then the topological space $Z(k')$ coincides with $Z_{k'}(k')_{k'}$.

We write $\Gal(k'/k)$ for the Galois group, equipped with the Krull topology in the usual sense,
of a Galois extension $k'$ of $k$.
The group $\Gal(\overline{k}/k)$ acts on the set $Z(\overline{k})$ 
of $\overline{k}$\nd points of any local $k$\nd ringed space $Z$, 
with $\sigma$ in $\Gal(\overline{k}/k)$ sending $w$ in $Z(\overline{k})$ to 
\[
{}^\sigma w = w \circ \Spec(\sigma).
\]
This action is continuous for the Krull topology on $Z(\overline{k})$, because it sends basic open sets
to basic open sets, and every basic open set is fixed by an open subgroup of $\Gal(\overline{k}/k)$.
We have 
\[
f({}^\sigma w) = {}^\sigma (f(w))
\]
for any morphism $f$ of local $k$\nd ringed spaces.

We regard groupoids over $\overline{k}$ as $\overline{k}$\nd schemes using $d_0$.
In particular the final object
\[
[\overline{k}] = \Spec(\overline{k}) \times_k \Spec(\overline{k})
\]
in the category of groupoids over $\overline{k}$ is regarded as a $\overline{k}$\nd scheme using the first projection.
Then there is a unique isomorphism
\begin{equation}\label{e:GalGammaschiso}
\Gal(\overline{k}/k)_{/\overline{k}} \iso [\overline{k}]
\end{equation}
of $\overline{k}$\nd schemes which induces on $\overline{k}$\nd points over $\overline{k}$ a homeomorphism
\begin{equation}\label{e:GalGammaptiso}
\Gal(\overline{k}/k) \iso [\overline{k}](\overline{k})_{\overline{k}}
\end{equation}
sending $\sigma$ to $(1_{\Spec(\overline{k})},\Spec(\sigma))$.
This can be seen by taking the limit over finite Galois extensions $k_0 \subset \overline{k}$ of $k$
of isomorphisms of finite $\overline{k}$\nd schemes from $\Gal(k_0/k)_{/\overline{k}}$
to $\Spec(\overline{k}) \times_k \Spec(k_0)$.

Let $F$ be a transitive affine groupoid over $\overline{k}$.
Then the composite of the inverse of the homeomorphism \eqref{e:GalGammaptiso} with the continuous map $F(\overline{k})_{\overline{k}} \to [\overline{k}](\overline{k})_{\overline{k}}$
defined by $(d_0,d_1)$ is a continuous map
\begin{equation}\label{e:gammamap}
\gamma:F(\overline{k})_{\overline{k}} \to \Gal(\overline{k}/k)
\end{equation}
such that $d_1(u) = \Spec(\gamma(u))$ for every $u$ in $F(\overline{k})_{\overline{k}}$.
Let $u$ and $v$ be elements of $F(\overline{k})_{\overline{k}} \subset F(\overline{k})$.
Define the product $uv$ of $u$ and $v$ as
\begin{equation}\label{e:udotvdef}
uv = u \circ {}^{\gamma(u)} v,
\end{equation}
where $\circ$ denotes the composition of $F$.
With this product and the Krull topology, $F(\overline{k})_{\overline{k}}$ is a topological group,
and $\gamma$ is a continuous homomorphism.
To check for example the continuity of the product, note that the topological embedding 
$F(\overline{k})_{\overline{k}} \to F(\overline{k})$ composed with the product factors as
\[
F(\overline{k})_{\overline{k}} \times F(\overline{k})_{\overline{k}} \to 
F(\overline{k}) \times_{\Spec(\overline{k})(\overline{k})} F(\overline{k}) \to F(\overline{k}),
\]
where the first arrow sends $(u,v)$ to $(u,{}^{\gamma(u)}v)$ and is continuous because $\gamma$
and the action of $\Gal(\overline{k}/k)$ on $F(\overline{k})$ are continuous, and the second arrow
is the composition and is continuous because it is defined by a morphism of $k$\nd schemes.

It will be shown in Lemma~\ref{l:grpdretr} below that for any transitive affine groupoid $F$ over $\overline{k}$
there exists a section $t:[\overline{k}] \to F$
of the morphism of $\overline{k}$\nd schemes $F \to [\overline{k}]$.
By \eqref{e:GalGammaschiso}, any such $t$ factors uniquely through the counit
\begin{equation}\label{e:Kcounit}
(F(\overline{k})_{\overline{k}})_{/\overline{k}} \to F
\end{equation}
for $F$ in the category of local $\overline{k}$\nd ringed spaces.
Further any point of $F$ in a $k$\nd scheme can be written uniquely in the form 
\begin{equation}\label{e:dtu}
d \circ t(u)
\end{equation}
for points $d$ of $F^\mathrm{diag}$ and $u$ of $[\overline{k}]$ with $d_0(u) = d_1(d) = d_0(d)$.

We note that if $Z$ is a filtered limit of finite \'etale schemes $Z_\lambda$ over $\overline{k}$,
then any surjective morphism $Y \to Z$ of finite presentation is a retraction,
because \cite[IV~8.8.2(ii)]{EGA} it is defined over some $Z_\lambda$.

\begin{lem}\label{l:grpdretr}
Let $F$ be a transitive affine groupoid over $\overline{k}$.
Then the morphism of $k$\nd schemes $(d_0,d_1):F \to [\overline{k}]$ has a section $[\overline{k}] \to F$ whose restriction to the diagonal of $[\overline{k}]$ is the identity of $F$.
\end{lem} 

\begin{proof}
It is enough to show that a section $t:[\overline{k}] \to F$ exists: the section that sends 
$(u,v)$ to $t(u,v) \circ t(v,v)^{-1}$ is then the identity above the diagonal.

Consider the partially ordered set $\sP$ of pairs $(N,s)$ with $N$ an $F$\nd subgroup of $F^\mathrm{diag}$
and $s$ a section of the morphism of schemes $(d_0,d_1):F/N \to [\overline{k}]$, where $(N,s) \le (N',s')$
when $N' \subset N$ and $s'$ lifts $s$. 
Since any filtered limit of quotients of $F$ is a quotient of $F$, the set $\sP$ is inductively ordered.

Let $(N,s)$ be a maximal element of $\sP$.
It is enough to show that $N = 1$.
Let $N_0$ be an $F$\nd subgroup of $F^\mathrm{diag}$ with $F/N_0$ of finite type, and write $N'$ for $N \cap N_0$.
By Lemma~\ref{l:morphpull}\ref{i:morphpullgpd}, $F/N' \to F/N$ is of finite presentation.
Thus by \eqref{e:GalGammaschiso} and the remark above,
$F/N' \to F/N$ has a section above the image of $s$, which defines a lifting $s'$ of $s$.
Then $(N',s') \ge (N,s)$, so that $N' = N$.	
Thus $N \subset N_0$ for any $N_0$ with $F/N_0$ of finite type, so that $N = 1$ as required.								\end{proof}

Given $\overline{k}$\nd schemes $Z$ and $Z'$ and $\sigma$ in $\Gal(\overline{k}/k)$, define a \emph{$\sigma$\nd morphism
from $Z$ to $Z'$} as a $k$\nd morphism $f$ from $Z$ to $Z'$ such that the square formed by $f$, the structural morphisms 
of $Z$ and $Z'$, and $\Spec(\sigma^{-1})$, commutes.
Such a $\sigma$\nd morphism $f$ induces a map from $Z(\overline{k})_{\overline{k}}$ to $Z'(\overline{k})_{\overline{k}}$, 
which sends $z$ in $Z(\overline{k})_{\overline{k}}$ to the unique $z'$ in $Z'(\overline{k})_{\overline{k}}$ for which the square
formed by $f$, $z$, $z'$, and $\Spec(\sigma^{-1})$, commutes.
Equivalently, if we identify $Z(\overline{k})_{\overline{k}}$ and $Z'(\overline{k})_{\overline{k}}$ with subsets of the 
topological spaces $Z$ and $Z'$, the map is the one induced on these subsets.
The composite of a $\sigma'$\nd morphism from $Z'$ to $Z''$ with a $\sigma$\nd morphism from $Z$ to $Z'$ is a 
$(\sigma'\sigma)$\nd morphism from $Z$ to $Z''$.
If a $\sigma$\nd morphism is an isomorphism of $k$\nd schemes, its inverse is a $\sigma^{-1}$\nd morphism,
and we speak a $\sigma$\nd isomorphism, or a $\sigma$\nd automorphism when the source and target coincide.

Given $\overline{k}$\nd groups $D$ and $D'$, a $\sigma$\nd morphism $h:D \to D'$ of $\overline{k}$\nd schemes will be called
a \emph{$\sigma$\nd morphism of $\overline{k}$\nd groups} if the squares formed by $h \times_{\Spec(\sigma^{-1})} h$, 
the products
of $D$ and $D'$, and $h$ commutes.
For such an $h$, the map from $D(\overline{k})_{\overline{k}}$
to $D'(\overline{k})_{\overline{k}}$ induced by $h$ is a group homomorphism.

We may identify a $\sigma$\nd morphism from $Z$ to $Z'$ with a morphism of $\overline{k}$\nd schemes
\[
\Spec(\sigma)^*Z \to Z'
\]
from the pullback of $Z$ along $\Spec(\sigma)$, or what is the same the push forward of $Z$ along $\Spec(\sigma^{-1})$, to $Z'$.
A $\sigma$\nd morphism $D \to D'$ of $\overline{k}$\nd groups is then a $\overline{k}$\nd homomorphism
from $\Spec(\sigma)^*D$ to $D'$.

Denote by $\overline{\sT}_k$ the full subcategory of the category of local $k$\nd ringed spaces
consisting of those $Z$ for which $\sO_Z$ is $k$\nd isomorphic to the
constant $k$\nd sheaf $\overline{k}$.
The morphism $f^{-1}\sO_Z \to \sO_{Z'}$ induced on structure sheaves by any morphism $f:Z' \to Z$ in $\overline{\sT}_k$
is an isomorphism, because any $k$\nd homomorphism between algebraic closures of $k$ is an isomorphism.
It follows that pullbacks along any morphism in $\overline{\sT}_k$ exist in the category of local $k$\nd ringed spaces, 
and that $Z' \times_Z Z''$ lies in $\overline{\sT}_k$ when $Z$, $Z'$ and $Z''$ do.
Further binary products, and hence non-empty finite limits, of local $k$\nd ringed spaces in $\overline{\sT}_k$ exist
and lie in $\overline{\sT}_k$,
because by \eqref{e:GalGammaschiso} the product of $\Spec(\overline{k})$ with itself lies in $\overline{\sT}_k$.
Thus we have a category of groupoids in $\overline{\sT}_k$ over $\overline{k}$ with final object $[\overline{k}]$,
and left or right actions of such groupoids on a $\overline{k}$\nd scheme are defined as in 
Section~\ref{s:grpdpre}.

Let $E'$ and $E''$ be topological groups.
By an \emph{extension}
\[
1 \to E'' \to E \to E' \to 1
\] 
\emph{of $E'$ by $E''$} we mean a topological group $E$ together with continuous homomorphisms
$E \to E'$ and $E'' \to E$ such that $E'' \to E$ is a topological isomorphism onto the kernel of $E \to E'$ and
$E \to E'$ is locally on $E'$ a retraction of topological spaces.
It is equivalent to require that $E$ be a principal $E''$\nd bundle over $E'$ in the usual topological sense
for the action by right translation of $E''$ on $E$.
This notion will be used in what follows only when $E'$ is profinite, and in that case $E \to E'$ has necessarily 
a section globally when is has one locally on $E'$.

Let $E$ be an extension of the profinite group $\Gal(\overline{k}/k)$ by a topological group $E'$.
Considering first the case where $E = \Gal(\overline{k}/k)$ shows that there is a unique structure
of groupoid in $\overline{\sT}_k$ over $\overline{k}$ on $E_{/\overline{k}}$ such that $E_{/\overline{k}} \to [\overline{k}]$ 
is \eqref{e:GalGammaschiso} composed with the projection
$E_{/\overline{k}}$ to $\Gal(\overline{k}/k)_{/\overline{k}}$ and
the composition and identity have
underlying map of topological spaces the product and identity of $E$.
The diagonal of $E_{/\overline{k}}$ is $E'{}\!_{/\overline{k}}$.
We then have a fully faithful functor $E \mapsto E_{/\overline{k}}$ from topological extensions of 
$\Gal(\overline{k}/k)$ to groupoids in $\overline{\sT}_k$ over $\overline{k}$.
If $Z$ is a $\overline{k}$\nd scheme,
the projection of $E_{/\overline{k}}$ onto $E_{/k}$ defines an isomorphism
\begin{equation}\label{e:EkbarEkiso}
E_{/\overline{k}} \times_{\overline{k}} Z \iso E_{/k} \times_k Z
\end{equation}
of local $k$\nd ringed spaces.
Using this isomorphism, we may identify an action of the groupoid $E_{/\overline{k}}$ on $Z$ with a continuous action of 
the topological group $E$ on the underlying $k$\nd scheme of $Z$
such that $e$ in $E$ above $\sigma$ in $\Gal(\overline{k}/k)$ acts as a $\sigma$\nd automorphism of $Z$.

\begin{defn}\label{d:Galext}
By a \emph{Galois extended $\overline{k}$\nd group} $(D,E)$ we mean an affine $\overline{k}$\nd group $D$,
an extension of topological groups
\[
1 \to  D(\overline{k})_{\overline{k}} \to E \to \Gal(\overline{k}/k) \to 1,
\]
and a continuous action of $E$ on the underlying $k$\nd scheme of $D$ with $e$ in $E$ above $\sigma$ in 
$\Gal(\overline{k}/k)$
acting as a $\sigma$\nd automorphism $a_e$ of the $\overline{k}$\nd group $D$, such that:
\begin{enumerate}
\renewcommand{\theenumi}{(\alph{enumi})}
\item\label{i:DEgrpconj}
the automorphism of $D(\overline{k})_{\overline{k}}$ induced by $a_e$ is conjugation by $e$;
\item\label{i:DEkbargrpconj} 
for $e$ in $D(\overline{k})_{\overline{k}}$, the $\overline{k}$\nd automorphism $a_e$ of $D$ is conjugation by $e$. 
\end{enumerate}
\end{defn}

A morphism of Galois extended $\overline{k}$\nd groups from $(D,E)$ to $(D',E')$
is a pair $(h,l)$ with $h$ a $\overline{k}$\nd homomorphism from $D$ to $D'$ and $l$ a continuous homomorphism from $E$ to $E'$,
such that the diagram
\begin{equation}\label{e:DEhomdiagram}
\begin{gathered}
\xymatrix{
1 \ar[r] &  D(\overline{k})_{\overline{k}} \ar[d]^{h} \ar[r] & E \ar[d]^{l} \ar[r]
& \Gal(\overline{k}/k) \ar@{=}[d] \ar[r] & 1  \\
1 \ar[r] & D'(\overline{k})_{\overline{k}} \ar[r] & E' \ar[r] & \Gal(\overline{k}/k) \ar[r] & 1
}
\end{gathered}
\end{equation}
commutes, and such that $h$ is compatible with the action of $e$ on $D$ and $l(e)$ on $D'$ for any $e$ in $E$.

By \eqref{e:EkbarEkiso} with $Z = D$, a continuous action of $E$ on $D$ with $e$ above $\sigma$ acting 
as a $\sigma$\nd automorphism of the $\overline{k}$\nd group $D$
may be identified with an action
\begin{equation}\label{e:DEalphaact}
\alpha:E_{/\overline{k}} \times_{\overline{k}} D \to D
\end{equation}
of the groupoid $E_{/\overline{k}}$ in $\overline{\sT}_{\overline{k}}$ over $\overline{k}$ on the $\overline{k}$\nd group $D$.
If we write
\begin{equation}\label{e:Dcounit}
\varepsilon:(D(\overline{k})_{\overline{k}})_{/\overline{k}} \to D
\end{equation}
for the counit, \ref{i:DEgrpconj} of Definition~\ref{d:Galext} is then equivalent to the condition
\begin{equation}\label{e:DEactcompat}
\alpha(w,\varepsilon(v)) = \varepsilon(w \circ v \circ w^{-1})
\end{equation}
on points $w$ of $E_{/\overline{k}}$ and $v$ of its diagonal $(D(\overline{k})_{\overline{k}})_{/\overline{k}}$,
and \ref{i:DEkbargrpconj} of Definition~\ref{d:Galext} is equivalent to the condition
\begin{equation}\label{e:DEconjcompat}
\alpha(v,d) = \varepsilon(v)d\varepsilon(v)^{-1}
\end{equation}
on points $v$ of $(D(\overline{k})_{\overline{k}})_{/\overline{k}}$ and $d$ of $D$.
Given also a Galois extended $\overline{k}$\nd group $(D',E')$ with action $\alpha'$ and counit $\varepsilon'$,
a morphism of Galois extended $\overline{k}$\nd groups from $(D,E)$ to $(D',E')$ may be identified with a 
$\overline{k}$\nd homomorphism $h:D \to D'$ together with a morphism $\lambda:E_{/\overline{k}} \to E'{}\!_{/\overline{k}}$
of groupoids in $\overline{\sT}_k$ over $\overline{k}$, such that
\begin{equation}\label{e:DEmorphcounit}
h(\varepsilon(v)) = \varepsilon'(\lambda(v))
\end{equation}
for points $v$ of $(D(\overline{k})_{\overline{k}})_{/\overline{k}}$, and
\begin{equation}\label{e:DEmorphact}
h(\alpha(w,d)) = \alpha'(\lambda(w),h(d)).
\end{equation}
for points $w$ of $E_{/\overline{k}}$ and $d$ of $D$.

\begin{rem}
It has been shown \cite[Cor~5]{Gil} that limits exist in the category of local ringed spaces.
Thus we have a notion of groupoid in the category of local $k$\nd ringed spaces.
Conditions \eqref{e:DEactcompat} and \eqref{e:DEconjcompat} then state that the counit $\varepsilon$ is a morphism
of $E_{/\overline{k}}$\nd groups from $(E_{/\overline{k}})^\mathrm{diag}$ to $D$ 
which is compatible with conjugation in the sense of Lemma~\ref{l:Kgrpdequiv}.
From this point of view, a quasi-inverse for the functor of Proposition~\ref{p:grpdgalequ} below is a push forward 
along $\varepsilon$ of groupoids over $\overline{k}$.
Only groupoids in the category $\overline{\sT}_k$ above or in the category of $k$\nd schemes,
and the actions of these on $\overline{k}$\nd schemes, will be needed for what follows.
Thus the only limits actually needed are fibre products in $\overline{\sT}_k$ and in the category of $k$\nd schemes
and pullbacks of local $k$\nd ringed spaces along a morphism in $\overline{\sT}_k$, whose existence has been seen above.
\end{rem}

To every transitive affine groupoid $F$ over $\overline{k}$ is associated a Galois extended $\overline{k}$\nd group
$(F^\mathrm{diag},F(\overline{k})_{\overline{k}})$, where the product of $F(\overline{k})_{\overline{k}}$ is given by
\eqref{e:udotvdef}, the first arrow in
\[
1 \to F^\mathrm{diag}(\overline{k})_{\overline{k}} \to F(\overline{k})_{\overline{k}} \to \Gal(\overline{k}/k) \to 1
\]
is the embedding, the second is the homomorphism $\gamma$ of \eqref{e:gammamap}, 
and the action 
\[
\Spec(\gamma(v))^*F^\mathrm{diag} \iso F^\mathrm{diag}
\]
on $F^\mathrm{diag}$ of $v$ in $F(\overline{k})_{\overline{k}}$ is the action at $v$ of $F$ on $F^\mathrm{diag}$.
The first arrow is a topological isomorphism onto the kernel of $\gamma$ because it is defined
by the embedding of the fibre of $F$ above a $\overline{k}$\nd point of $[\overline{k}]$ over $\overline{k}$,
and $\gamma$ is a retraction of topological spaces by Lemma~\ref{l:grpdretr}.
The action of $F(\overline{k})_{\overline{k}}$ on $F^\mathrm{diag}$ is the action $\alpha$ of the form \eqref{e:DEalphaact} with $D = F^\mathrm{diag}$ and $E = F(\overline{k})_{\overline{k}}$, given by restricting the action of $F$ on $F^\mathrm{diag}$ along the counit \eqref{e:Kcounit}.
It satisfies \ref{i:DEgrpconj} of Definition~\ref{d:Galext} because \eqref{e:DEactcompat} holds for $\alpha$ by naturality of the counit and compatibility of \eqref{e:Kcounit} with composition, 
and \ref{i:DEkbargrpconj} because \eqref{e:DEconjcompat} is immediate.

We have a functor from the category of transitive affine groupoids over $\overline{k}$ to the category
of Galois extended $\overline{k}$\nd groups which sends $F$ to $(F^\mathrm{diag},F(\overline{k})_{\overline{k}})$
and $F \to F'$ to the pair $(h,l)$ with $h$ the restriction to the diagonals and $l$ the map induced on $\overline{k}$\nd points over $\overline{k}$.

\begin{prop}\label{p:grpdgalequ}
The functor from the category of transitive affine groupoids over $\overline{k}$ to
the category of Galois extended $\overline{k}$\nd groups that sends $F$ to $(F^\mathrm{diag},F(\overline{k})_{\overline{k}})$
is an equivalence of categories.
\end{prop}

\begin{proof}

Let $F$ and $F'$ be transitive affine groupoids over $\overline{k}$, and let $h:F^\mathrm{diag} \to F'{}^\mathrm{diag}$
and $\lambda:(F(\overline{k})_{\overline{k}})_{/\overline{k}} \to F'(\overline{k})_{\overline{k}})_{/\overline{k}}$ define a
morphism from
$(F^\mathrm{diag},F(\overline{k})_{\overline{k}})$ to $(F'{}^\mathrm{diag},F'(\overline{k})_{\overline{k}})$.
To prove the full faithfulness, it is to be shown that
there is a unique morphism $f:F \to F'$ of groupoids over $\overline{k}$ which induces $h$ and $\lambda$.

By Lemma~\ref{l:grpdretr}, $F \to [\overline{k}]$ has a section $t:[\overline{k}] \to F$
which is the identity above the diagonal.
Factor $t$ as $t_0:[\overline{k}] \to (F(\overline{k})_{\overline{k}})_{/\overline{k}}$ followed by
the counit \eqref{e:Kcounit},
and write $t'$ for $\lambda \circ t_0$ followed by the counit for $F'$.
By the factorisation \eqref{e:dtu}, an $f$ inducing $h$ and $\lambda$ is unique if it exists, because
it must send $d \circ t(u)$ to $h(d) \circ t'(u)$ for points $d$ of $F^\mathrm{diag}$ and $u$ of 
$[\overline{k}]$ with $d_0(u) = d_1(d) = d_0(d)$.
The morphism $f$ so defined preserves composition because
\[
h(t(u) \circ d \circ t(u)^{-1}) = t'(u) \circ h(d) \circ t'(u)^{-1}
\]
by \eqref{e:DEmorphact} and
\[
h(t(u) \circ t(u') \circ t(u \circ u')^{-1}) = t'(u) \circ t'(u') \circ t'(u \circ u')^{-1}
\]
by \eqref{e:DEmorphcounit}.
It is thus a morphism of groupoids over $\overline{k}$.

Let $(D,E)$ be a Galois extended $\overline{k}$\nd group.
To prove the essential surjectivity, 
it is to be shown that there exists a transitive affine groupoid $F$ over $\overline{k}$
with $(F^\mathrm{diag},F(\overline{k})_{\overline{k}})$ isomorphic to $(D,E)$.

The continuous map $E \to \Gal(\overline{k}/k)$ has a section 
which preserves the identity.
Applying $(-)_{/\overline{k}}$ and using \eqref{e:GalGammaschiso}, we obtain a section
\[
s:[\overline{k}] \to E_{/\overline{k}},
\]
which is the identity above the diagonal, to the morphism of $k$\nd schemes $E_{/\overline{k}} \to [\overline{k}]$.
With $\varepsilon$ the counit \eqref{e:Dcounit}, define a morphism 
$\delta:[\overline{k}] \times_{\overline{k}} [\overline{k}] \to D$ by
\[
\delta(u,u') = \varepsilon(s(u) \circ s(u') \circ s(u \circ u')^{-1}),
\]
and take
\[
F = D \times_{\overline{k}} [\overline{k}]
\]
with $F \to [\overline{k}]$ the projection, identity $(1,1)$ and composition $\circ:F \times_{\overline{k}} F \to F$
\[
(d,u) \circ (d',u') = (d\alpha(s(u),d')\delta(u,u'),u \circ u'),
\]
where $\alpha$ is the action \eqref{e:DEalphaact}.
The composition is associative because
\[
\delta(u,u')\delta(u \circ u',u'') =  \alpha(s(u),\delta(u',u''))\delta(u,u' \circ u'')
\]
by \eqref{e:DEactcompat}, and
\[
\delta(u,u')\alpha(s(u \circ u'),d'')\delta(u,u')^{-1} =  \alpha(s(u),\alpha(s(u'),d''))
\]
by \eqref{e:DEconjcompat} and the associativity of the action $\alpha$.
Also inverses exist:
\[
(d,u)^{-1} = (1,u)^{-1} \circ (d,1)^{-1} = (\varepsilon(s(u)^{-1} \circ s(u^{-1})^{-1}),u^{-1}) \circ (d^{-1},1).
\]
Thus $F$ is a transitive affine groupoid over $\overline{k}$.
Since \eqref{e:Kcounit} sends the point $(v,u)$ of
\[
(F(\overline{k})_{\overline{k}})_{/\overline{k}} = (D(\overline{k})_{\overline{k}})_{/\overline{k}} \times_{\overline{k}} [\overline{k}]
\]
to $(\varepsilon(v),u)$, the composition 
of $(F(\overline{k})_{\overline{k}})_{/\overline{k}}$ is given by
\[
(v,u) \circ (v',u') = (v \circ s(u) \circ v' \circ s(u') \circ s(u \circ u')^{-1},u \circ u').
\]
We thus have an isomorphism $(F(\overline{k})_{\overline{k}})_{/\overline{k}} \iso E_{/\overline{k}}$ sending
$(v,u)$ to $v \circ s(u)$.
Along with $F^\mathrm{diag} \iso D$ sending $(d,1)$ to $d$, this gives
an isomorphism from $(F^\mathrm{diag},F(\overline{k})_{\overline{k}})$ to $(D,E)$.
Indeed \eqref{e:DEmorphcounit} is clear, while \eqref{e:DEmorphact} holds for points $(v,1)$ of 
$(F(\overline{k})_{\overline{k}})_{/\overline{k}}$ by \eqref{e:DEconjcompat},
for points $(1,u)$ because the conjugate in $F$ of $(d,1)$ by $(1,u)$ is $(\alpha(s(u),d),1)$, 
and hence for any $(v,u) = (v,1) \circ (1,u)$ by associativity of the actions.
\end{proof}

\begin{rem}
The functor $D \mapsto D(\overline{k})_{\overline{k}}$ from the category pro\'etale $\overline{k}$\nd groups
to the category of profinite topological groups is an equivalence.
It follows that the forgetful functor $(D,E) \mapsto E$ from the category of Galois extended $\overline{k}$\nd groups $(D,E)$ 
with $D$ pro\'etale to the category of profinite extensions of $\Gal(\overline{k}/k)$ is an equivalence.
Thus Proposition~\ref{p:grpdgalequ} gives by restriction an equivalence $F \mapsto F(\overline{k})_{\overline{k}}$
from the category pro\'etale groupoids over $\overline{k}$ to the category of profinite extensions of $\Gal(\overline{k}/k)$.
More generally for arbitrary, not necessarily perfect $k$, if $k_s$ is the separable closure and $k'$ the perfect closure 
of $k$ in $\overline{k}$, then  Lemma~\ref{l:pureinsext} with $X = \Spec(k_s)$ and $H = X$ shows that $F \mapsto F(k_s)_{k_s}$
gives an equivalence from the category of pro\'etale groupoids over $k_s$ to the category of profinite extensions of $\Gal(k_s/k)$.
\end{rem}

\begin{defn}\label{d:Galextloose}
By a \emph{loosely Galois extended $\overline{k}$\nd group} we mean a pair $(D,E)$ with $D$ an affine $\overline{k}$\nd group
and $E$ a topological extension $\Gal(\overline{k}/k)$ by $D(\overline{k})_{\overline{k}}$,
such that conjugation of $D(\overline{k})_{\overline{k}}$ by any $e$ in $E$ above $\sigma$ in $\Gal(\overline{k}/k)$
is induced by a $\sigma$\nd automorphism of the $\overline{k}$\nd group $D$.
\end{defn}

A morphism of loosely Galois extended $\overline{k}$\nd groups from $(D,E)$ to $(D',E')$
is a pair $(h,l)$ with $h$ a $\overline{k}$\nd homomorphism from $D$ to $D'$ and $l$ a continuous homomorphism from $E$ to $E'$,
such that the diagram \eqref{e:DEhomdiagram} commutes.

If $D$ and $D'$ are affine $\overline{k}$\nd groups with $D$ reduced, then two $\sigma$\nd morphisms of $\overline{k}$\nd groups
$D \to D'$ which induce the same homomorphism $D(\overline{k})_{\overline{k}} \to D'(\overline{k})_{\overline{k}}$ coincide:
their equaliser is a closed subscheme of $D$ containing the subset $D(\overline{k})_{\overline{k}}$,
which is dense by Corollary~\ref{c:kgrpac}\ref{i:kgrpacdense}. 
It follows that for $(D,E)$ a loosely Galois extended $\overline{k}$\nd group with $D$ reduced, there is
for each $e$ in $E$ above $\sigma$ unique $\sigma$\nd automorphism $a_e$ of the $k$\nd group $D$
which induces conjugation by $e$ on $D(\overline{k})_{\overline{k}}$.
Explicitly,
\[
a_e({}^\sigma d) = ede^{-1}
\]
for each $d$ in $D(\overline{k})_{\overline{k}}$.
Further the $a_e$ satisfy \ref{i:DEgrpconj} and \ref{i:DEkbargrpconj} of Definition~\ref{d:Galext}.
Hence they define a structure of Galois extended $\overline{k}$\nd group on $(D,E)$ provided that the action $e \mapsto a_e$
is continuous.
Similarly if $(D,E)$ and $(D',E')$ are Galois extended $\overline{k}$\nd groups
and $D$ is reduced, then any morphism of loosely
Galois extended $\overline{k}$\nd groups from $(D,E)$ to $(D',E')$ is a morphism of Galois extended $\overline{k}$\nd groups.

By discarding the action of $E$ on $D$, we have a faithful forgetful functor 
from the category of Galois extended $\overline{k}$\nd groups
to the category of loosely Galois extended $\overline{k}$\nd groups.
By the above remarks it is fully faithful and injective on objects on the full subcategory
of those Galois extended $\overline{k}$\nd groups $(D,E)$ with $D$ reduced.

\begin{lem}\label{l:fgsubgrpdense}
Suppose that $k$ is algebraically closed of characteristic $0$.
Then for any affine $k$\nd group $G$ of finite type, there exists a finitely generated subgroup 
of the group of $k$\nd rational points of $G$ which is Zariski dense in $G$.
\end{lem}

\begin{proof}
The class $\sG$ of those affine $k$\nd groups $G$ of finite type for which the conclusion of the Lemma holds
contains $\bG_m$ and $\bG_a$.
Since the inverse image of a dense subset under an open map is dense, lifting a finite set of generators shows
that $\sG$ is closed under the formation of extensions.
A connected affine $k$\nd group of finite type thus lies in $\sG$ if it is soluble, and hence by the Bruhat decomposition
if it is reductive.
Since $\sG$ contains the finite $k$\nd groups, the result follows. 
\end{proof}

\begin{prop}\label{p:Galextiso}
Suppose that $k$ is of characteristic $0$.
Then the forgetful functor from the category of Galois extended $\overline{k}$\nd groups to the category
of loosely Galois extended $\overline{k}$\nd groups is an isomorphism of categories.
\end{prop}

\begin{proof}
Since any affine $\overline{k}$\nd group is reduced, it remains by the remarks following
Definition~\ref{d:Galextloose} only to show
that for any loosely Galois extended $\overline{k}$\nd group $(D,E)$, the action $e \mapsto a_e$ defined above
is continuous.
Fix a section $s:\Gal(\overline{k}/k) \to E$ with $s(1) = 1$ of the continuous map $E \to \Gal(\overline{k}/k)$.

Suppose first that $D$ is of finite type over $\overline{k}$.
Then $D(\overline{k})_{\overline{k}}$ is discrete, and
\[
D = D_0 \times_{k_0} \overline{k}
\]
for a finite subextension $k_0$ of $\overline{k}$ and affine $k_0$\nd group of finite type $D_0$.
By Lemma~\ref{l:fgsubgrpdense}, $D(\overline{k})_{\overline{k}}$ has a finitely generated subgroup 
$\Phi$ which is dense in $D$.
Increasing $k_0$ if necessary, we may suppose that a finite set of generators of $\Phi$, and hence
$\Phi$ itself, is contained in $D_0(k_0)_{k_0} \subset D(\overline{k})_{\overline{k}}$.
Similarly we may suppose that conjugation by $s(\sigma)$ for $\sigma$ in $\Gal(\overline{k}/k_0)$ fixes $\Phi$.
For $\sigma$ in $\Gal(\overline{k}/k_0)$ we then have
\[
a_{s(\sigma)} = D_0 \times_{k_0} \Spec(\sigma^{-1}),
\]
because the two sides are $\sigma$\nd automorphisms of $D$ which coincide on $\Phi$.
The restriction of the action of $E$ on $D$ to the inverse image $E_0$ of $\Gal(\overline{k}/k_0)$ in $E$ 
is thus continuous, because it is given by an action
\[
(E_0)_{/k} \times_k D \to (D(\overline{k})_{\overline{k}})_{/k} \times_k \Gal(\overline{k}/k_0)_{/k} \times_k D
\to (D(\overline{k})_{\overline{k}})_{/k} \times_k D \to D
\]
of $(E_0)_{/k}$, where the first arrow is defined by the continuous map that sends $e$ above $\sigma$ to 
$(es(\sigma)^{-1},\sigma)$, the second by the action of $\Gal(\overline{k}/k_0)$ on $D$ through its
action on $\Spec(\overline{k})$, and the third by the action of $D(\overline{k})_{\overline{k}}$ 
on $D$ by conjugation.
Since $E_0$ is an open subgroup of $E$, the action of $E$ on $D$ is also continuous.

To prove the continuity for arbitrary $D$, we show that
for every normal $\overline{k}$\nd subgroup $N$ of $D$ with $D/N$ of finite type,
there exists a normal $\overline{k}$\nd subgroup $N_0 \subset N$ of $D$ with $D/N_0$ of
finite type such that the closed subgroup $N_0(\overline{k})_{\overline{k}}$ of $E$ is normal.
Since $D(\overline{k})_{\overline{k}}/N_0(\overline{k})_{\overline{k}} = (D/N_0)(\overline{k})_{\overline{k}}$ by
Corollary~\ref{c:kgrpac}\ref{i:kgrpacsurj}, we have for such an $N_0$ a loosely Galois extended $\overline{k}$\nd group
$(D/N_0,E/N_0(\overline{k})_{\overline{k}})$.
The action of $E$ on $D$ will then be the limit of continuous actions of the $E/N_0(\overline{k})_{\overline{k}}$ 
on $D/N_0$, and hence will be continuous.

The assignment $(\sigma,u) \mapsto s(\sigma)us(\sigma)^{-1}$ defines a continuous map
\[
\Gal(\overline{k}/k) \times D(\overline{k})_{\overline{k}} \to D(\overline{k})_{\overline{k}}.
\]
Since $N(\overline{k})_{\overline{k}}$ is an open subgroup of $D(\overline{k})_{\overline{k}}$, its inverse image 
under this map is open, and hence contains $\Sigma \times N'(\overline{k})_{\overline{k}}$ for
some open normal subgroup $\Sigma$ of $\Gal(\overline{k}/k)$ and normal $\overline{k}$\nd subgroup $N'$ of $D$ with $D/N'$
of finite type.
Then if $\sigma_1,\sigma_2, \dots,\sigma_n$ are representatives for the cosets of $\Sigma$, we have
\[
eN(\overline{k})_{\overline{k}}e^{-1} \supset \bigcap_{i = 1}^n s(\sigma_i)N'(\overline{k})_{\overline{k}}s(\sigma_i)^{-1}
\]
for every $e$ in $E$.
Now
\[
N'' = \bigcap_{i = 1}^n a_{s(\sigma_i)}(N')
\]
is a normal $\overline{k}$\nd subgroup of $D$ with $D/N''$ of finite type,
and by Corollary~\ref{c:kgrpac}\ref{i:kgrpacdense} the inclusion $a_e(N) \supset N''$ holds for every $e$ in $E$, because 
it holds on $\overline{k}$\nd points over $\overline{k}$.
It thus suffices to take for $N_0$ the intersection of the $a_e(N)$ for $e$ in $E$.
\end{proof}

\begin{defn}\label{d:Galextact}
Let $(D,E)$ be a Galois extended $\overline{k}$\nd group.
By a \emph{right $(D,E)$\nd scheme} we mean a $\overline{k}$\nd scheme $Z$, together with a right action
\[
c:Z \times_{\overline{k}} D \to Z 
\]
of the $\overline{k}$\nd group $D$ on $Z$
and a continuous right action of $E$ on the underlying $k$\nd scheme of $Z$ with $e$ in $E$ above $\sigma$ in $\Gal(\overline{k}/k)$
acting as a $\sigma^{-1}$\nd automorphism  $b_e$ of $Z$, such that:
\begin{enumerate}
\renewcommand{\theenumi}{(\alph{enumi})}
\item\label{i:DEactDact}
$b_e \circ c = c \circ (b_e \times_{\Spec(\sigma)} a_{e^{-1}})$, where $a_e$ is the action of $e$ on $D$;
\item\label{i:DEactDpts} 
$b_e = c(-,e)$ for $e$ in $D(\overline{k})_{\overline{k}}$.
\end{enumerate}
\end{defn}

A morphism of right $(D,E)$\nd schemes is a morphism of $\overline{k}$\nd schemes which is compatible
with the actions of $D$ and $E$.

If $D$ is reduced and $Z$ is separated, then \ref{i:DEactDact} in Definition~\ref{d:Galextact}
is redundant:
the equaliser of the two morphisms of \ref{i:DEactDact} is a closed subscheme of $Z \times_{\overline{k}} D$,
and hence coincides with it, because by \ref{i:DEactDpts} in Definition~\ref{d:Galextact} and 
\ref{i:DEgrpconj} in Definition~\ref{d:Galext}, this equaliser contains the fibre above each point of 
the dense subset $D(\overline{k})_{\overline{k}}$ of $D$.

Just as a continuous action of $E$ on the underlying $k$\nd scheme of a $\overline{k}$\nd group $D$ for which $e$ above $\sigma$ 
acts as a $\sigma$\nd automorphism of $D$ may be identified with an
action \eqref{e:DEalphaact} of the groupoid $E_{/\overline{k}}$ on $D$,
a continuous right action of $E$ on the underlying $k$\nd scheme of a $\overline{k}$\nd scheme $Z$ for which 
$e$ above $\sigma$ acts as a $\sigma^{-1}$\nd automorphism of $Z$ may be identified with a right action
\begin{equation}\label{e:DEbetaact}
\beta:Z \times_{\overline{k}} E_{/\overline{k}} \to Z
\end{equation}
of $E_{/\overline{k}}$ on $Z$.
Then \ref{i:DEactDact} in Definition~\ref{d:Galextact} is equivalent the condition
\begin{equation}\label{e:DEactDact}
\beta(c(z,d),w) = c(\beta(z,w),\alpha(w^{-1},d))
\end{equation}
on points $z$, $d$ and $w$ of $Z$, $D$ and $E_{/\overline{k}}$, and \ref{i:DEactDpts} to the condition
\begin{equation}\label{e:DEactDpts}
\beta(z,v) = c(z,\varepsilon(v))
\end{equation}
on points $z$ of $Z$ and $v$ of $(D(\overline{k})_{\overline{k}})_{/\overline{k}}$, where
$\alpha$ is the action \eqref{e:DEalphaact} and $\varepsilon$ is the
counit \eqref{e:Dcounit}.
A morphism $Z \to Z'$ of right $(D,E)$\nd schemes is then a morphism $Z \to Z'$ of $\overline{k}$\nd schemes
compatible with the actions of $D$ and $E_{/\overline{k}}$.

Let $F$ be a transitive affine groupoid over $\overline{k}$.
If $Z$ is a right $F$\nd scheme,
we have a right action of $F^\mathrm{diag}$ on $Z$ by restriction, and a right action of $F(\overline{k})_{\overline{k}}$ on $Z$
with $w$ acting as the $\gamma(w)^{-1}$\nd automorphism of $Z$ defined by the action 
\[
Z \iso \Spec(\gamma(w))^*Z
\]
at $w$ of $F$ on $Z$, where $\gamma$ is \eqref{e:gammamap}.
The action of $F(\overline{k})_{\overline{k}}$ is continuous because it arises from an action of the form
\eqref{e:DEbetaact} with $E = F(\overline{k})_{\overline{k}}$, given by restricting the action of $F$ along the counit 
\eqref{e:Kcounit}.
Further \eqref{e:DEactDact} and hence \ref{i:DEactDact} of Definition~\ref{d:Galextact} is satisfied by associativity of the 
action of $F$ and the compatibility of the counit \eqref{e:Kcounit} with composition,
and \eqref{e:DEactDpts} and hence \ref{i:DEactDpts} is satisfied by naturality of the counit.
Thus we obtain a functor from the category of right $F$\nd schemes to the
category of right $(F^\mathrm{diag},F(\overline{k})_{\overline{k}})$\nd schemes.

\begin{prop}\label{p:grpdgalschequ}
Let $F$ be a transitive affine groupoid over $\overline{k}$.
Then passing from the action of $F$ to the actions of $F^\mathrm{diag}$ 
and $F(\overline{k})_{\overline{k}}$
defines an isomorphism from the category of right $F$\nd schemes to the category of
right $(F^\mathrm{diag},F(\overline{k})_{\overline{k}})$\nd schemes.
\end{prop}

\begin{proof}
By Lemma~\ref{l:grpdretr}, $F \to [\overline{k}]$ has a section $t:[\overline{k}] \to F$,
which is the identity above the diagonal.
It factors as $t_0:[\overline{k}] \to (F(\overline{k})_{\overline{k}})_{/\overline{k}}$ followed by
the counit \eqref{e:Kcounit}.

Let $Z$ and $Z'$ be right $F$\nd schemes.
By the factorisation \eqref{e:dtu}, a morphism $Z \to Z'$ of $\overline{k}$\nd schemes 
is compatible with the actions of $F$ on $Z$ and $Z'$ provided it is compatible with the actions of $F^\mathrm{diag}$ and
$(F(\overline{k})_{\overline{k}})_{/\overline{k}}$.
This proves the full faithfulness.

Let $Z$ be a right $(F^\mathrm{diag},F(\overline{k})_{\overline{k}})$\nd scheme.
By the factorisation \eqref{e:dtu}, a right $F$\nd scheme with underlying right
$(F^\mathrm{diag},F(\overline{k})_{\overline{k}})$\nd scheme $Z$ is unique if it exists,
because for points $d$ of $F^\mathrm{diag}$ and $u$ of $[\overline{k}]$ with $d_0(u) = d_1(d) = d_0(d)$, the point $d \circ t(u)$
of $F$ must act as
\[
z \mapsto \beta(c(z,d),t_0(u)),
\]
where $c$ is the action of $F^\mathrm{diag}$ and $\beta$ is the action of $(F(\overline{k})_{\overline{k}})_{/\overline{k}}$.
The action so defined is associative, because
\[
\beta(c(z,t(u) \circ d \circ t(u)^{-1}),t_0(u)) = c(\beta(z,t_0(u)),d)
\]
by \eqref{e:DEactDact}, and
\[
\beta(c(z,t(u) \circ t(u') \circ t(u \circ u')^{-1}),t_0(u \circ u')) 
=
\beta(z,t_0(u) \circ t_0(u'))
\]
by \eqref{e:DEactDpts}.
It thus defines a structure of right $F$\nd scheme on $Z$
with underlying right $(F^\mathrm{diag},F(\overline{k})_{\overline{k}})$\nd scheme the given one.
This proves the bijectivity on objects.
\end{proof}

\begin{rem}\label{r:Ksubgpdef}
Let $K$ be a transitive affine groupoid over a $k$\nd scheme $X$.
Proposition~\ref{p:grpdgalschequ} can sometimes be used to define $K$\nd subgroups of $K^\mathrm{diag}$.
Suppose for example given for every algebraically closed extension $\overline{k}{}'$ of $k$ and 
affine $\overline{k}{}'$\nd group
$G$ a normal $\overline{k}{}'$\nd subgroup $N(G)$ of $G$, stable under $\overline{k}{}'$\nd isomorphisms and pullback along $k$\nd homomorphisms $\overline{k}{}' \to \overline{k}{}''$.
Then there exists a unique $K$\nd subgroup of $K^\mathrm{diag}$ with fibre
$N((K^\mathrm{diag})_x)$ at $x$ for every $\overline{k}{}'$ and $x$ in $X(\overline{k}{}')$, 
and it is preserved by pullback and extension of scalars:
reduce by Lemma~\ref{l:prereppull} and Corollary~\ref{c:gpdacequiv} to 
the case where $X = \Spec(\overline{k})$, which follows from Proposition~\ref{p:grpdgalschequ}
with $F = K^\mathrm{op}$.
\end{rem}

Let $X$ be a $k$\nd scheme and $(D,E)$ be a Galois extended $\overline{k}$\nd group.
By a right $(D,E)$\nd scheme over $X$ we mean a scheme $P$ over $X \times_k \overline{k}$
and a structure of right $(D,E)$\nd scheme on $P$
such that the actions $P \times_{\overline{k}} D \to P$
of $D$ and $P \to P$ of any $e$ in $E$ are morphisms over $X$.
Given a $k$\nd morphism $X' \to X$, we have a pullback functor from right $(D,E)$\nd schemes
over $X$ to right $(D,E)$\nd schemes over $X'$.

Let $X$ be a non-empty $k$\nd scheme, $H$ be a pregroupoid over $X$ and $(D,E)$ be a Galois extended $\overline{k}$\nd group.
By a \emph{principal $(H,D,E)$\nd bundle} we mean a right $(D,E)$\nd scheme $P$ over $X$ whose
underlying right $D$\nd scheme is a principal $D$\nd bundle over $X \times_k \overline{k}$, together with
a structure of $H$\nd scheme on $P$ whose defining isomorphism $d_1{}\!^*P \iso d_0{}\!^*P$ is a morphism
of right $(D,E)$\nd schemes.
A morphism of principal $(H,D,E)$\nd bundles is a morphism of right $(D,E)$\nd schemes
over $X$ which is also an $H$\nd morphism.
Such a morphism is necessarily an isomorphism.

When $D$ is reduced, condition \ref{i:DEactDact} in the definition of the right $(D,E)$\nd scheme underlying a principal
$(H,D,E)$\nd bundle $P$ over $X$ is redundant: arguing locally over $X$, we reduce to the case where $P$ is
a separated $\overline{k}$\nd scheme, which has been seen.

\begin{cor}\label{c:grpdgalprin}
Let $X$ be a non-empty $k$\nd scheme, $H$ be a pregroupoid over $X$, and $F$ be a
transitive affine groupoid over $\overline{k}$.
Then passing from the action of $F$ to the actions of $F^\mathrm{diag}$ and $F(\overline{k})_{\overline{k}}$
defines an isomorphism from the category of principal $(H,F)$\nd bundles to the category of
principal $(H,F^\mathrm{diag},F(\overline{k})_{\overline{k}})$\nd bundles.
\end{cor}

\begin{proof}
By the factorisation \eqref{e:dtu}, a scheme over 
$X \times_k \overline{k}$ with a structure right $F$\nd scheme is a right $F$\nd scheme over $X$ if it is a right 
$(F^\mathrm{diag},F(\overline{k})_{\overline{k}})$\nd scheme over $X$.
It thus suffices to apply Proposition~\ref{p:grpdgalschequ} to right $F$\nd schemes over $X$ and their pullbacks
along $d_0, d_1:H_{[1]} \to X$.
\end{proof}

The existence and uniqueness up to unique isomorphism of a push forward of a principal $(H,D,E)$\nd bundle $P$ along\
a morphism $(h,l):(D,E) \to (D',E')$, i.e.\ a pair consisting of a principal $(H,D',E')$\nd bundle $P'$ and 
a morphism $P \to P'$ compatible with $h$, $l$, and the actions of $H$, $D$, $D'$, $E$ and $E'$,
follows from Propositions~\ref{p:grpdgalequ} and \ref{p:grpdgalschequ} and Corollary~\ref{c:grpdgalprin}, 
together with push forward for principal $(H,F)$\nd bundles.
There is thus a functor from the category of Galois extended $\overline{k}$\nd groups to the category of 
sets which sends $(D,E)$ to the set 
\[
H^1_H(X,D,E)
\]
of isomorphism classes of principal $(H,D,E)$\nd bundles over $X$, and which is defined on morphisms using push forward.
It factors through the category of Galois extended $\overline{k}$\nd groups up to conjugacy, where a morphism to $(D,E)$ is 
a conjugacy class under the action of $D(\overline{k})_{\overline{k}}$.
By Corollary~\ref{c:grpdgalprin}, passing from the action of $F$ to the actions of $F^\mathrm{diag}$ and 
$F(\overline{k})_{\overline{k}}$ defines a bijection
\begin{equation}\label{e:grpdgalH1}
H^1_H(X,F) \iso H^1_H(X,F^\mathrm{diag},F(\overline{k})_{\overline{k}})
\end{equation}
which is natural in the transitive affine groupoid $F$ over $\overline{k}$.

Let $G$ be an affine $k$\nd group.
Then $\Gal(\overline{k}/k)$ acts on $G(\overline{k})$ through $\overline{k}$, and
\[
((G \times_k [\overline{k}])^\mathrm{diag},(G \times_k [\overline{k}])(\overline{k})_{\overline{k}}) = 
(G_{\overline{k}},G(\overline{k}) \rtimes \Gal(\overline{k}/k)).
\]
Thus by Lemma~\ref{l:gpdprinHGG'} and Corollary~\ref{c:grpdgalprin} we have an equivalence $P \mapsto P_{\overline{k}}$
from principal $(H,G)$\nd bundles to principal 
$(H,G_{\overline{k}},G(\overline{k}) \rtimes \Gal(\overline{k}/k))$\nd bundles, where the map induced on 
$P_{\overline{k}}(\overline{k})_{\overline{k}} = P(\overline{k})$ by the $\sigma^{-1}$\nd automorphism of $P_{\overline{k}}$
defined by $g\sigma$ sends $z$ to ${}^{\sigma^{-1}}(zg)$.
We may identify a morphism
\[
(1_{\overline{k}},\Gal(\overline{k}/k)) \to (G_{\overline{k}},G(\overline{k}) \rtimes \Gal(\overline{k}/k))
\] 
with a continuous section of topological groups $s$ of
$G(\overline{k}) \rtimes \Gal(\overline{k}/k)$ over $\Gal(\overline{k}/k)$, or equivalently with the continuous 
$1$\nd cocycle $\varphi$
of $\Gal(\overline{k}/k)$ with values in $G(\overline{k})$ where $s(\sigma) = \varphi(\sigma)\sigma$.
By Proposition~\ref{p:grpdgalequ}, \eqref{e:HomconjHkG} with $1$ and  
$G \times_k [\overline{k}]$ for $G$ and $F$ thus gives a natural bijection 
\begin{equation}\label{e:Hbij}
H^1(\Gal(\overline{k}/k),G(\overline{k})) \iso H^1(k,G)
\end{equation}
from the set of $s$, or equivalently of $\varphi$, modulo conjugation by $G(\overline{k})$. 
Given $P$ with a $\overline{k}$\nd point $z$, the class of $P$ corresponds to the class of the unique $s$ with $z$ fixed by 
every $s(\sigma)$,
or equivalently to the unique $\varphi$ with ${}^\sigma z = z\varphi(\sigma)$ for every $\sigma$.

\section{Reductive subgroupoids}\label{s:redsub}

\emph{In this section $k$ is a field of characteristic $0$, $X$ is a non-empty $k$\nd scheme, and
$H$ is a pregroupoid over $X$.}

\medskip

For the remainder of the paper, we work over a field of characteristic $0$.

In this section we determine the conditions under which a transitive affine groupoid over $H$ contains
a reductive subgroupoid over $H$.
Every reductive subgroupoid of a transitive affine groupoid is contained in a Levi subgroupoid,
defined analogously to the Levi subgroups of an affine $k$\nd group (Corollary~\ref{c:redsubgrpd}), 
and any two Levi subgroupoids are conjugate (Corollary~\ref{c:redconj}).
The obstruction to the existence of a reductive subgroupoid over $H$ is an appropriate cohomology class (Corollary~\ref{c:Leviclass}), 
and there is an equivalent condition using representations (Theorem~\ref{t:requiv}).

Since $k$ is of characteristic $0$, every affine $k$\nd group is reduced, and every surjective homomorphism
of affine $k$\nd groups is faithfully flat.

An affine $k$\nd group $G$ will be called \emph{reductive} if its category of representations is 
semisimple abelian, i.e. if every short exact sequence of representations of $G$ splits.
It is equivalent to require that $\Hom_G(V,-)$ be exact for every representation $V$ of $G$,
or (since $\Hom_G(V,-)$ is $(V^\vee \otimes_k -)^G$) that $(-)^G$ be exact, 
or that $V^G \to k$ be surjective for every surjective homomorphism $V \to k$ of representations of $G$.
An affine $k$\nd group is reductive if and only if each of its quotients of finite type is reductive.

An affine $k$\nd group $G$ is called \emph{unipotent} if $V^G \ne 0$ for every representation $V \ne 0$ of $G$.
It is equivalent to require that every representation of $G$ be a successive extension of trivial representations $k$.
An affine $k$\nd group is unipotent if and only if each of its quotients of finite type is unipotent.
An affine $k$\nd group of finite type is unipotent if and only if it is a successive extension of additive groups $\bG_a$.
A commutative unipotent $k$\nd group of finite type is a product of groups $\bG_a$.

Any affine $k$\nd group $G$ has a unique normal unipotent $k$\nd subgroup, its \emph{unipotent radical} $R_u G$, 
with $G/R_uG$ reductive,
and any normal unipotent $k$\nd subgroup of $G$ is contained in $R_uG$.
This follows from the case where $G$ is of finite type.

Reductive $k$\nd groups are stable under the formation of normal $k$\nd subgroups,
$k$\nd quotients and  extensions,
and unipotent $k$\nd groups under formation of $k$\nd subgroups,
$k$\nd quotients and extensions. 
For $k'$ an extension of $k$, a $k$\nd group $G$ is reductive (resp.\ unipotent)
if and only if $G_{k'}$ is a reductive (resp.\ unipotent) $k'$\nd group.

A transitive affine groupoid $K$ over $X$ will be called \emph{reductive} if the fibres of $K^\mathrm{diag}$
are reductive.
Reducing after a finite extension of scalars using Lemmas~\ref{l:prereppull} and \ref{l:finneutr} and \eqref{e:Hpushequ}
to the case where $X = \Spec(k)$ shows that $K$ is reductive if and only if
$H^0_K(X,-)$ on $\Mod_K(X)$ is exact if and only if $\Mod_K(X)$ is semisimple.
When this is so, a Zorn's Lemma argument shows that every $K$\nd module is a coproduct of irreducible 
representations of $K$, so that $\MOD_K(X)$ is semisimple.

Let $K$ be a transitive affine groupoid over $X$.
An affine $K$\nd group will be called \emph{unipotent} if its fibres are unipotent.
By Remark~\ref{r:Ksubgpdef}, $K^\mathrm{diag}$ has a unique
$K$\nd subgroup with fibre $R_u((K^\mathrm{diag})_x)$ at any point $x$ of $X$ in an extension of $k$,
the \emph{unipotent radical $R_uK$ of $K$}.
It is a unipotent $K$\nd subgroup with $K/R_uK$ is reductive.

Let $\sV$ be a vector bundle over $X$.
Then the spectrum 
\[
\bA(\sV) = \Spec(\Sym (\sV^\vee))
\]
of the symmetric algebra of $\sV^\vee$ has a structure of commutative 
group scheme over $X$ with points in $Z$ the additive group of cross-sections of $\sV_Z$.
We then have a functor  $\bA$
from the category of vector bundles over $X$ to the category of commutative affine group schemes over $X$,
which can be seen as follows to be fully faithful.
Since the question is local, it is enough to show that every endomorphism of the group scheme
\[
\bA(\sO_X) = \bG_a \times_k X
\]
arises from a unique endomorphism of $\sO_X$.
The subgroup $\Z$ of $\bG_a(k)$ is Zariski dense in $\bG_a$.
Thus two endomorphisms of $\bG_a \times_k X$ coincide provided that they coincide along the cross section 
defined by $1 \in \bG_a(k)$,
because their equaliser is then a closed subscheme of $\bG_a \times_k X$ which contains the fibres
above $\Z \subset \bG_a(k)$.

More generally, we have a fully faithful functor $\bA$
from the category of representations of $H$ to the category of commutative affine $H$\nd groups.
If $K$ is a transitive affine groupoid over $X$, then reducing to the case where $X$ is the 
spectrum of a field shows that $\bA$ is an equivalence 
from the category of representations of $K$ to the category of commutative unipotent $K$\nd groups of finite type.

By an \emph{extension} of a representation $\sV'$ of $H$ by a representation $\sV''$ of $H$, we mean 
representation $\sE$ of $H$ together with morphisms of $H$\nd modules $\sV'' \to \sE$
and $\sE \to \sV'$ whose underlying morphisms of $\sO_X$\nd modules form a short exact sequence
\[
0 \to \sV'' \to \sE \to \sV' \to 0.
\]
We then have the usual push forward and pullback for extensions.
If $H$ is a transitive affine groupoid, so that $\Mod_H(X)$ is an abelian category, 
this notion of extension reduces to the usual one in an abelian category.
Write
\[
H^1_H(X,\sV)
\]
for the set of isomorphism classes of extensions of $\sO_X$ by the representation $\sV$ of $H$.
Using push forward of extensions, we may regard $H^1_H(X,-)$ as a functor from representations
of $H$ to abelian groups.
If $K$ is a reductive groupoid over $X$, then
\begin{equation}\label{e:HKXVtriv}
H^1_K(X,\sV) = 0
\end{equation}
for every representation $\sV$ of $K$.

Fix a representation $\sV$ of $H$.
Then $\bA(\sV)$ is a commutative affine $H$\nd group.
We have a functor from the category of extensions
of $\sO_X$ by $\sV$ to the category of $(H,\bA(\sV))$\nd torsors
which sends $\sE$ to the fibre of 
\[
\bA(\sE) \to \bA(\sO_X)
\]
above the cross-section defined by $1 \in \bG_a(k)$, equipped with the action of $\bA(\sV)$ by translation.
This functor is an equivalence:
for the full faithfulness reduce first to the case where $H = X$, and then arguing locally to the case where the
extensions split, and for the essential surjectivity reduce using Lemmas~\ref{l:prereppull} and \ref{l:torspull}
to the case where the underlying $\bA(\sV)$\nd torsor is trivial.
Thus we have an isomorphism
\begin{equation}\label{e:VViso}
H^1_H(X,\sV) \iso H^1_H(X,\bA(\sV))
\end{equation}
which is natural in $\sV$.

\begin{prop}\label{p:unipzero}
Let $K$ is a reductive groupoid over $X$ and $U$ be a unipotent $K$\nd group.
Then every $(K,U)$\nd torsor is trivial.
\end{prop}

\begin{proof}
Let $P$ be a $(K,U)$\nd torsor.
It is to be shown that $P$ has a $K$\nd invariant cross-section.
Consider the set $\sP$ of pairs $(N,s)$ with
$N$ a normal $K$\nd subgroup of $U$ and $s$ a $K$\nd invariant cross-section of the
push forward $P_N$ of $P$ along the projection $U \to U/N$.
Write $(N,s) \le (N',s')$ when $N$ contains $N'$ and $s$ is the image of $s'$
under the projection from $P_{N'}$ to $P_N$.
Then $\sP$ is inductively ordered.
Let $(N,s)$ be maximal element of $\sP$.

Suppose that $N \ne 1$.
Then for a sufficiently large $K$\nd quotient $\overline{U}$ of $U$ of finite type, the image $\overline{N}$
of $N$ under the projection $U \to \overline{U}$ is $\ne 1$.
Since a normal $K$\nd subgroup of $\overline{U}$ is the same as a $(\overline{U} \rtimes_X K)$\nd subgroup,
the derived series
\[
1 = \overline{U}{}^{(n)} \subset \dots \subset \overline{U}{}^{(1)} \subset \overline{U}{}^{(0)} = \overline{U}
\]
of $\overline{U}$, with $\overline{U}{}^{(i+1)} = \overline{U}{}^{(i)}{}^\mathrm{der}$,
consists of normal $K$\nd subgroups.
If $t$ is the largest $i$ for which $\overline{U}{}^{(i)}$ contains $\overline{N}$,
the intersection $N' \ne N$ of $N$ with the inverse image of $\overline{U}{}^{(t+1)}$
in $U$ is thus a normal $K$\nd subgroup of $U$, and $N/N'$ is a commutative unipotent 
$K$\nd group of finite type.
The inverse image of $s$ under the projection from $P_{N'}$ to $P_N$
is a $(K,N/N')$\nd torsor $Q$.
Since
\[
H^1_K(X,N/N') = 0
\]
by \eqref{e:HKXVtriv} and \eqref{e:VViso}, $Q$ has a $K$\nd invariant cross-section $s'$.
We may regard $s'$ as a section of $P_{N'}$ with image $s$ in $P_N$.
Then $(N',s')$ is strictly greater than $(N,s)$ in $\sP$, contradicting the maximality of $(N,s)$.
Thus $N = 1$, and $s$ is the required section of $P$.
\end{proof}

\begin{cor}\label{c:redconj}
Let $K$ and $K'$ be transitive affine groupoids over $X$,
and $f_1$ and $f_2$ be morphisms from $K'$ to $K$ over $X$.
Suppose that $K'$ is reductive, and that the composites of the projection from $K$ to $K/R_uK$
with $f_1$ and $f_2$ coincide.
Then $f_1$ and $f_2$ are conjugate by a cross-section of $R_uK$.
\end{cor}

\begin{proof}
Restriction along $f_1$ defines a structure of affine $K'$\nd group on $R_uK$.
We then have a $(K',R_uK)$\nd torsor with underlying $(R_uK)$\nd torsor $R_uK$
where the point $v$ of $K'$ sends the point 
$u$ of $R_uK$ to 
\[
f_2(v) \circ u \circ f_1(v)^{-1}.
\] 
By Proposition~\ref{p:unipzero}, it has a $K'$\nd invariant cross-section $u_1$.
Then $f_2$ is conjugate to $f_1$ by $u_1$.
\end{proof}

Let $K$ be a transitive affine groupoid over $X$.
By a \emph{Levi subgroupoid} of $K$ we mean a subgroupoid $L$ of $K$ such that the restriction 
to $L$ of the projection from $K$ onto $K/R_u K$ is an isomorphism.
We then have $K = R_uK \rtimes_X L$. 
A Levi subgroupoid of $K$ may be regarded as a right inverse to the projection of $K$ onto $K/R_uK$.
By Corollary~\ref{c:redconj}, any two Levi subgroupoids of $K$ are conjugate by a cross-section of $R_uK$.

\begin{lem}\label{l:Levitech}
Let $K$ be a transitive affine groupoid over $H$.
Suppose for every transitive affine subgroupoid $K'$ of $K$ over $H$ with $R_uK'$ contained in $R_uK$,
and quotient $K''$ of $K'$of finite type with $R_uK''$ commutative, that $K''$ has a Levi subgroupoid over $H$.
Then $K$ has a Levi subgroupoid over $H$.
\end{lem}

\begin{proof}
Consider the set $\sP$ of pairs $(N,L)$ with
$N$ a $K$\nd subgroup of $R_uK$ and $L$ a Levi subgroupoid of $K/N$ over $H$.
Write $(N,L) \le (N',L')$ when $N$ contains $N'$ and $L$ is the image of $L'$
under the projection from $K/N'$ to $K/N$.
Then $\sP$ is inductively ordered,
and hence has a maximal element $(N_0,L_0)$.

Suppose that $N_0 \ne 1$.
Then $N_0$ has a $K$\nd quotient $N_0/N_1 \ne 1$ of finite type.
Increasing $N_1$ if necessary, we may assume that $N_0/N_1$ is commutative.
Write $K'$ and $K_1$ for the respective inverse images of $L_0$ under the projections from $K$ and $K/N_1$ to $K/N_0$,
Then $R_uK' = N_0$ and $R_uK_1 = N_0/N_1$.
For a sufficiently large quotient $K''$ of $K_1$ of finite type, 
the projection from $K_1$ to $K''$ induces an isomorphism from $R_uK_1$ to $R_uK''$.
Since $K''$ is a quotient of $K'$, it has by hypothesis a Levi subgroupoid $L''$ over $H$.
The inverse image $L_1$ of $L''$ under the projection from $K_1$ to $K''$ is then a Levi subgroupoid of $K_1$,
and hence also of $K/N_1$, over $H$.
Then $(N_1,L_1)$ is strictly greater than $(N_0,L_0)$ in $\sP$, contradicting the maximality of $(N_0,L_0)$.
Thus $N_0 = 1$, and $L_0$ is a Levi subgroupoid of $K$.
\end{proof}

\begin{lem}\label{l:Levialg}
Any transitive affine groupoid over an algebraically closed extension of $k$ has a Levi subgroupoid.
\end{lem}

\begin{proof}
Let $K$ be a transitive affine groupoid over an algebraically closed extension $\overline{k}$ of $k$.
To prove that $K$ has a Levi subgroupoid, we may suppose by Lemma~\ref{l:Levitech} with $H = X$ that $K$ is of finite type
and $R_uK$ is commutative.
Then by Lemma~\ref{l:finneutr}, a simply transitive $K_{k'}$\nd scheme $Z$ exists for some finite Galois extension $k'$ of $k$.
Since $Z$ is of finite type over $\overline{k} \otimes_k k'$, it has a cross-section, so that $K_{k'}$ is constant.
By the classical Levi decomposition \cite[VIII.4.3]{Hoc81}, the set $\sL$ of Levi subgroupoids of $K_{k'}$
is thus non-empty.
The Galois group $\Gal(k'/k)$ acts on $\sL$ by its action on $K_{k'}$ through $k'$,
and the fixed points of $\sL$ are the Levi subgroupoids of $K_{k'}$ which are defined over $k$.

Since $R_uK$ is a commutative unipotent $K$\nd group of finite type, 
it is of the form $\bA(\sV)$ for a representation $\sV$ of $K$.
The action of $k$ on $\sV$ defines structure of $k$\nd vector space on the abelian group of cross-sections
of $R_uK_{k'}$ over $\overline{k} \otimes_k k'$, with the subgroup of $K_{k'}$\nd invariant cross-sections a $k$\nd vector subspace.
The quotient $k$\nd vector space $W$ then acts by conjugation on $\sL$, and by Corollary~\ref{c:redconj}
the action is simply transitive.
Further $\Gal(k'/k)$ acts $k$\nd linearly on $W$ through $k'$, and
\[
\sigma(wL) = \sigma(w)\sigma(L)
\]
for every $\sigma$ in $\Gal(k'/k)$, $w$ in $W$, and $L$ in $\sL$.
Choose an $L_0$ in $\sL$.
If we write $L''-L'$ for the unique $w$ in $W$ with $wL' = L''$,
there is a unique $L_1$ in $\sL$ for which
\[
\sum_{\sigma \in \Gal(k'/k)}(\sigma L_0 - L_1) = 0.
\]
Then $L_1$ is fixed by $\Gal(k'/k)$, 
and hence is defined over $k$.
\end{proof}

Let $K$ be a transitive affine groupoid over $X$.
To any $K$\nd scheme we may associate a $(K,R_uK)$\nd scheme with the same action of $K$,
where the right action of the point $u$ of $R_uK$ is defined as the left action of the point $u^{-1}$ of $K$.
We thus obtain a fully faithful functor from $K$\nd schemes to $(K,R_uK)$\nd schemes.
A $K$\nd scheme will be called a \emph{Levi $K$\nd scheme} if its associated $(K,R_uK)$\nd scheme is a $(K,R_uK)$\nd torsor.
If $X'$ is a non-empty scheme over $X$ and $K'$ is the pullback of $K$ onto $X'$, it follows from
Lemma~\ref{l:schemetorspull} that a $K$\nd scheme is a Levi $K$\nd scheme if and only if its pullback onto $X'$ 
is a Levi $K'$\nd scheme.
If $k'$ is an extension of $k$, then $Z$ is a Levi $K$\nd scheme if and only if $Z_{k'}$ is a Levi $K_{k'}$\nd scheme.
Any Levi $K$\nd scheme is transitive affine.

\begin{prop}\label{p:Levischeme}
Let $K$ be a transitive affine groupoid over $X$.
Then a Levi $K$\nd scheme exists, and is unique up to $K$\nd isomorphism.
If $Z$ is a Levi $K$\nd scheme, then the stabiliser of any cross-section of $Z$ is a Levi subgroupoid of $K$,
and any Levi subgroupoid of $K$ is the stabiliser of some cross-section of $Z$.
\end{prop}

\begin{proof}
A subgroupoid $L$ of $K$ is a Levi subgroupoid if and only if every point of $K$ can be written uniquely in the
form $u \circ l$ for points $u$ of $R_uK$ and $l$ of $L$.
It follows that the stabiliser of any cross-section of a Levi $K$\nd scheme is a Levi subgroupoid of $K$.
It also follows that for $L$ a Levi subgroupoid of $K$, the quotient $K/L$ exists  and is a Levi $K$\nd scheme:
we may take $K/L = R_uK$ with base cross-section the identity, where $u \circ l$ sends the point $u'$ of $R_uK$
to $u \circ l \circ u' \circ l^{-1}$.

It remains only to prove the existence and uniqueness statements.
By Lemma~\ref{l:prereppull}, we may suppose after pullback that $X$ is the spectrum of an algebraically closed extension of $k$.
Then $K$ has a Levi subgroupoid by Lemma~\ref{l:Levialg}, and the existence follows.
To show that two Levi $K$\nd schemes $Z_1$ and $Z_2$ are isomorphic, we may suppose after pullback that $Z_1$
and $Z_2$ have cross-sections, and hence are of the form $K/L_1$ and $K/L_2$ for Levi subgroupoids $L_1$ and $L_2$ of $K$.
By Corollary~\ref{c:redconj}, 
$L_2$ is the conjugate of $L_1$
by a cross-section $u$ of $R_uK$, so that
if $z_i$ is the base cross-section of $K/L_i$, the stabiliser of both $uz_1$ and $z_2$ is $L_2$,
and hence there is a unique isomorphism of $K$\nd schemes from $K/L_1$ to $K/L_2$ which sends $uz_1$ to $z_2$.
\end{proof}

Let $K$ be a transitive affine groupoid over $X$.
We define the \emph{Levi class} in 
\[
H^1_K(X,R_uK)
\]
as the class of the $(K,R_uK)$\nd torsor associated to any Levi $K$\nd scheme.
It is well-defined by Proposition~\ref{p:Levischeme}.
The conjugacy statement of the following Corollary also follows directly from Corollary~\ref{c:redconj}.

\begin{cor}\label{c:Leviclass}
Let $K$ be a transitive affine groupoid over $H$.
Then $K$ has a Levi subgroupoid over $H$ if and only if the Levi class in $H^1_K(X,R_uK)$ has image
in $H^1_H(X,R_uK)$ the base point.
When this is so, any two Levi subgroupoids of $K$ over $H$ are conjugate by an element of $H^0_H(X,R_uK)$.
\end{cor}

\begin{proof}
Immediate from Proposition~\ref{p:Levischeme}.
\end{proof}

The following Corollary shows that a transitive affine groupoid has a Levi subgroupoid if and only if it has a 
reductive subgroupoid.

\begin{cor}\label{c:redsubgrpd}
Any reductive subgroupoid of a transitive affine groupoid over $X$ is contained in a Levi subgroupoid.
\end{cor}

\begin{proof}
Take for $H$ in Corollary~\ref{c:Leviclass} a reductive subgroupoid of $K$, and apply Proposition~\ref{p:unipzero}.
\end{proof}

\begin{rem}\label{r:Levireduce}
Let $G$ be an affine $k$\nd group.
By Corollary~\ref{c:redsubgrpd}, $G$ has a Levi $k$\nd subgroup $G_0$,
and by Proposition~\ref{p:Levischeme}, the quotient $G/G_0$ exists and is affine.
If $P$ is a principal $G$\nd bundle over $X$, the quotient $P/G_0$ thus exists and is a Levi 
$\underline{\Iso}_G(P)$\nd scheme, with the projection \eqref{e:Gzeroquot} 
of $P$ onto $P/G_0$ a morphism of $\underline{\Iso}_G(P)$\nd schemes.
The inverse image in $P$ of a cross-section $z$ of $P/G_0$ is then a principal $G_0$\nd subbundle $P_0$ of $P$, with
the subgroupoid $\underline{\Iso}_{G_0}(P_0)$ of $\underline{\Iso}_G(P)$ the stabiliser of $z$.
By Proposition~\ref{p:Levischeme}, the Levi subgroupoids of $\underline{\Iso}_G(P)$ are thus those of the form 
$\underline{\Iso}_{G_0}(P_0)$ for a principal $G_0$\nd subbundle $P_0$ of $P$.
\end{rem}

\begin{lem}\label{l:affquotfilt}
Let $G$ be an affine $k$\nd group of finite type, $G'$ be a $k$\nd subgroup of $G$,
and $G_0$ be a normal $k$\nd subgroup of $G$.
Denote by $G_1$ the $k$\nd quotient $G/G_0$ and by $G'{}\!_1$ the image
of $G'$ under the projection $G \to G_1$.
Suppose that $G_0/(G_0 \cap G')$ and $G_1/G'{}\!_1$ are affine.
Then $G/G'$ is affine.
\end{lem}

\begin{proof}
Both $G/(G_0G')$ and $(G_0G')/G'$ are affine, because they are isomorphic to $G_1/G'{}\!_1$ and $G_0/(G_0 \cap G')$.
We have a cartesian square of $k$\nd schemes
\[
\xymatrix{
G/G'  \ar[d]  &  G \times_k ((G_0G')/G') \ar[d] \ar[l] \\
G/(G_0G')  &       G \ar[l]
}
\]
with the top arrow the morphism of $G$\nd schemes induced by the embedding of
$(G_0G')/G'$ into $G/G'$ and the other arrows the projections.
The left arrow is affine because the right arrow is affine.
Thus $G/G'$ is affine.
\end{proof}

\begin{lem}\label{l:affquot}
Let $K$ be a transitive affine groupoid over $X$ and $K'$ be a transitive affine subgroupoid
of $K$ over $X$ with $R_uK'$ contained in $R_uK$.
Then the $K$\nd scheme $K/K'$ exists and is affine over $X$.
\end{lem}

\begin{proof}
After extension of scalars and pullback we may suppose $X = \Spec(k)$, so that $K$ is an affine 
$k$\nd group $G$ and $K'$ is a $k$\nd subgroup $G'$.
Writing $G$ as the filtered limit of its $k$\nd quotients of finite type, we may suppose further that $G$
is of finite type. 
It is then enough to show that $R_uG/(R_uG \cap G')$ is affine, because by Matsushima's criterion \cite[I~2]{Lun73}, 
the hypotheses of Lemma~\ref{l:affquotfilt} are then satisfied with $G_0 = R_uG$.
Thus we may suppose finally that $G$ is unipotent,
when we argue by induction on its dimension:
if $G$ is of dimension $> 0$, then the hypotheses of Lemma~\ref{l:affquotfilt} are satisfied with 
$G_0$ the derived group of $G$,
because $G_0$ is of dimension strictly less than that of $G$ while $G_1$ is commutative.
\end{proof}

Let $\sA$ and $\sA'$ be abelian categories, $T:\sA \to \sA'$ be an exact functor,
and $T':\sA' \to \sA$ be a right adjoint to $T$.
Then there is a canonical homomorphism
\begin{equation}\label{e:exthom}
\Ext^1_\sA(M,T'(N)) \to \Ext^1_{\sA'}(T(M),N)
\end{equation}
given by applying $T$ to an extension of $M$ by $T'(N)$ and then pushing forward
along the counit $TT'(N) \to N$.
Suppose that $T'$ is also exact.
Then \eqref{e:exthom} is an isomorphism, with inverse given by applying $T'$ to
an extension of $T(M)$ by $N$ and then pulling back along the unit $M \to T'T(M)$.

\begin{lem}\label{l:Shapiro}
Let $K$and $K'$ be as in Lemma~\textnormal{\ref{l:affquot}}.
Then for every representation $\sV'$ of $K'$ and $v'$ in $H^1_{K'}(X,\sV')$ there exist a representation $\sV$ of $K$, 
a $v$ in $H^1_K(X,\sV)$, and a $K'$\nd morphism $f$ from $\sV$ to $\sV'$, such that $f$ sends $v$ to $v'$.
\end{lem}

\begin{proof} 
By Lemma~\ref{l:affquot}, the $K$\nd scheme $X' = K/K'$ exists and is affine over $X$.
The embedding of $K'$ into $K$ factors in the category of groupoids in $k$\nd schemes as
\[
(X,K') \to (X',K \times_X X') \to (X,K),
\]
where the first arrow is the embedding of the pullback along the base cross-section of $X'$ and
the second arrow is the projection.
If we write $\sA$, $\sA'$ and $\sA''$ for $\MOD_K(X)$, $\MOD_{K'}(X)$ and $\MOD_{K \times_X X'}(X')$,
and $p:X' \to X$ for the structural morphism,
then the restriction functor $T$ from $\sA$ to $\sA'$ factors up to isomorphism as  $p^*:\sA \to \sA''$ induced
by the second arrow followed by the equivalence $\sA'' \to \sA'$ induced by the first.
As in \eqref{e:Hpush}, $p^*$ has a right adjoint $p_*$, which is exact because $p$ is affine.
It follows that $T$ has an exact right adjoint $T'$.

Since both $T$ and $T'$ are exact, \eqref{e:exthom}
shows that restricting to $K'$ and pushing forward along the counit $TT'(\sV') \to \sV'$ gives an isomorphism
\[
\Ext^1_\sA(\sO_X,T'(\sV')) \iso \Ext^1_{\sA'}(\sO_X,\sV') = H^1_{K'}(X,\sV').
\]
Let $\sE$ be an extension of $\sO_X$ by $T'(\sV')$ in $\sA$ whose class corresponds under this isomorphism
to $v'$.
Writing $\sE$ as the filtered colimit of its $K$\nd submodules in $\Mod_K(X)$ shows that there is such 
a $K$\nd submodule $\sE_0$ such that the restriction to $\sE_0$ of $\sE \to \sO_X$ is non-zero, and hence an epimorphism.
Now take for $\sV$ the kernel of $\sE_0 \to \sO_X$, for $v$ the class of $\sE_0$, and for $f$ the restriction
of the counit to $T(\sV)$.
\end{proof}

\begin{thm}\label{t:requiv}
A transitive affine groupoid $K$ over $H$
has a reductive subgroupoid over $H$ if and only if restriction from $K$ to $H$ splits every
short exact sequence of representations of $K$.
\end{thm}

\begin{proof}
Since the category of representations of any reductive groupoid over $X$ is semisimple abelian,
the ``only if'' is clear.

Conversely, suppose that restriction to $H$ splits every short exact sequence of representations of $K$.
Let $K'$ and $K''$ be as in Lemma~\ref{l:Levitech}.
Then $R_uK'' = \bA(\sV')$ for a representation $\sV'$ of $K''$.
By Lemma~\ref{l:Levitech}, it will be enough to show that
\[
H^1_{K''}(X,\sV') \to H^1_H(X,\sV')
\]
induced by $H \to K''$ is $0$, because by \eqref{e:VViso} and Corollary~\ref{c:Leviclass},
$K''$ will then have a Levi subgroupoid over $H$.
Let $v''$ be an element of $H^1_{K''}(X,\sV')$.
If $v'$ is the image of $v''$ in $H^1_{K'}(X,\sV')$ under $K' \to K''$,
choose a $\sV$, $v$ and $f$ as in Lemma~\ref{l:Shapiro}.
By hypothesis, and the definition of the $H^1$ as groups of extensions, 
the image of $v$ in $H^1_H(X,\sV)$, and hence in $H^1_H(X,\sV')$, is $0$.
Thus the image of $v'$ and hence $v''$ in $H^1_H(X,\sV')$ is $0$. 
\end{proof}

Suppose that $X = \Spec(k)$ and that $H$ is an affine $k$\nd group.
Then a transitive affine groupoid $K$ over $H$ is the same as a $k$\nd homomorphism from $H$ to an affine $k$\nd group,
and $K$ has a reductive subgroupoid over $H$ if and only if the $k$\nd homomorphism is reductive 
in the sense of \cite[2.1.1]{O10}, i.e.\  factors through a reductive $k$\nd group.
Theorem~\ref{t:requiv} then reduces to the equivalence of (a) and (b) in \cite[2.1.6]{O10}, and Corollary~\ref{c:rext}
below to \cite[2.1.7]{O10}.

\begin{cor}\label{c:rext}
Let $K$ be a transitive affine groupoid over $H$
and $k'$ be an extension of $k$.
Then $K$ has a reductive subgroupoid over $H$ if and only if $K_{k'}$ has a
reductive subgroupoid over $H_{k'}$.
\end{cor}

\begin{proof}
The ``only if'' is immediate.
Conversely, suppose that $K_{k'}$ has a reductive subgroupoid over $H_{k'}$.
Let $\sE$ be an extension of a representation $\sV$ of $K$ by a representation $\sV'$.
Then there exists a morphism $g':\sV_{k'} \to \sE_{k'}$ of $H_{k'}$\nd modules right inverse to $\sE_{k'} \to \sV_{k'}$.
If $r:k' \to k$ is a $k$\nd linear map left inverse to the embedding $e:k \to k'$,
and we regard $g'$ as a morphism of $H$\nd modules, then
\[
\sV \xrightarrow{\sV \otimes_k e} \sV \otimes_k k' \xrightarrow{g'} \sE \otimes_k k' \xrightarrow{\sE \otimes_k r} \sE 
\]
is right inverse to $\sE \to \sV$.
Thus $K$ has a reductive subgroupoid over $H$ by Theorem~\ref{t:requiv}.
\end{proof}

\begin{cor}
Let $G$ be an affine $k$\nd group and $P$ be a principal $(H,G)$\nd bundle.
Then $P$ has a principal $(H,G_0)$\nd subbundle for some reductive $k$\nd subgroup $G_0$ of $G$ if and only if
the functor $P \times_k^G -$ from the category of representations of $G$ to the category of representations of $H$
splits every short exact sequence.
\end{cor}

\begin{proof}
As seen in \eqref{e:assocvecbun}, $P \times_k^G -$ factors as an equivalence from representations of $G$
to representations of $\underline{\Iso}_G(P)$ followed by restriction to $H$.
Since by Corollary~\ref{c:redsubgrpd} and Remark~\ref{r:Levireduce} a subbundle of the required form exists if and only if
$\underline{\Iso}_G(P)$ has a reductive subgroupoid over $H$, it suffices to apply Theorem~\ref{t:requiv}.
\end{proof}

\section{The splitting theorem}\label{s:splitting}

\emph{In this section $k$ is a field of characteristic $0$.}

\medskip

In this section we prove the splitting theorem in the form that will be required,
by suitably adapting the proof given in \cite{O11}.
The main idea there was to reduce a categorical problem about tensor categories to a geometrical problem
about actions of reductive groups on affine schemes.
Here the reduction to a geometrical problem applies almost unmodified,
and it is only necessary to adapt the proof of \cite[4.4.4]{O11}
to give the more general Lemma~\ref{l:Dextend} below.

We use the terminology of \cite[\S 2]{O11} for tensor categories.
Thus a \emph{$k$\nd pretensor category} is a $k$\nd linear category with a bilinear tensor product $\otimes$ and a unit $\I$,
together with associativity and commutativity constraints.
A \emph{$k$\nd tensor functor} between $k$\nd pretensor categories is a $k$\nd linear functor
together with constraints ensuring that the tensor product is preserved up to isomorphism,
and a \emph{tensor isomorphism} between $k$\nd tensor functors is a natural isomorphism compatible with the tensor products.
It is assumed that the units are strict and are strictly preserved by $k$\nd tensor functors.
A \emph{tensor ideal} in a $k$\nd pretensor category is an ideal which is stable under tensor product with
an arbitrary morphism.
The quotient by a tensor ideal is a $k$\nd pretensor category.

A \emph{$k$\nd tensor category} is a $k$\nd pretensor category which is pseudo-abelian, i.e. finite direct sums exist,
and every idempotent endomorphism has an image.
If $H$ is a pregroupoid over a $k$\nd scheme $X$, then $\MOD_H(X)$ with tensor product $\otimes_{\sO_X}$ and 
identity $\sO_X$ is a $k$\nd tensor category. 
Every $k$\nd pretensor category $\sC$ can be embedded as a full $k$\nd pretensor subcategory
in a $k$\nd tensor category $\sC'$, its pseudo-abelian hull,
characterised by the property that every object of $\sC'$ is a direct summand of a direct sum of objects of $\sC$.
Any $k$\nd tensor functor from $\sC$ to a $k$\nd tensor category factors uniquely up to tensor isomorphism 
through $\sC'$.

A \emph{dual} of an object $M$ of a $k$\nd pretensor category $\sC$ is an object $M^\vee$ of $\sC$
together with a unit $\I \to M^\vee \otimes M$ and a counit $M \otimes M^\vee \to \I$, satisfying identities 
analogous to those for an adjunction.
The usual properties of duals then hold, so that we have for example a natural isomorphism
\begin{equation}\label{e:dualiso}
\Hom_{\sC}(M,N) \iso \Hom_{\sC}(\I,M^\vee \otimes N).
\end{equation}
for $N$ in $\sC$.
When a dual of $M$ exists, it is unique up to unique isomorphism, and $M$ is then said to be \emph{dualisable}.
Dualisable objects are stable under direct sums, direct summands, and tensor products, and are preserved 
by $k$\nd tensor functors.
If $M$ is dualisable, there is defined for every endomorphism $f$ of $M$ a \emph{trace}
\[
\tr(f) \in \End_\sC(\I).
\] 
We have $\tr(f'' \circ f') = \tr(f' \circ f'')$ for $f':M \to N$ and $f'':N \to M$ with $M$ and $N$ dualisable.

A $k$\nd pretensor category $\sC$ is said to be \emph{rigid} if each object of $\sC$ is dualisable.
The full $k$\nd tensor subcategory $\Mod_H(X)$ of $\MOD_H(X)$ is rigid, 
with duals given by the usual duals of a vector bundle over $X$
and traces the usual ones.
By \eqref{e:dualiso}, an ideal $\sJ$ in a rigid $k$\nd pretensor category may be identified
with the subfunctor $\sJ(\I,-)$ of $\Hom(\I,-)$.

Let $\sC$ be a rigid $k$\nd pretensor category with $\End_\sC(\I)$ a local $k$\nd algebra.
Then $\sC$ has a unique maximal tensor ideal $\Rad(\sC)$.
A morphism $\I \to M$ lies in $\Rad(\sC)$ if and only if it has no left inverse.
In general, $f:M \to N$ lies in $\Rad(\sC)$ if and only if $\tr(g \circ f)$
lies in the maximal ideal of $\End_\sC(\I)$ for every $g:N \to M$.
We have $\Rad(\sC) = 0$ if and only if $\sC$ has no tensor ideals other than $0$ and $\sC$,
and $\End_\sC(\I)$ is then a field. 
In particular if $\Rad(\sC) = 0$ then every $k$\nd tensor functor $\sC \to \sC'$ with $\sC'$
non-zero is faithful.
If $\sC$ is semisimple abelian, then $\Rad(\sC) = 0$,
because $\I$ is indecomposable in $\sC$ and hence every non-zero $\I \to M$ has a left inverse.
In general, we write the quotient $k$\nd pretensor category of $\sC$ by $\Rad(\sC)$ as
\[
\overline{\sC} = \sC/\Rad(\sC).
\]
Then $\Rad(\overline{\sC}) = 0$, and $\End_{\overline{\sC}}(\I)$ is the residue field of $\End_\sC(\I)$.

An object $M$ of a $k$\nd tensor category is called \emph{positive} if $M$ is dualisable and some
exterior power (defined as the image of the antisymmetrising idempotent) of $M$ is $0$.
Positive objects are stable under direct sums, direct summands, and tensor products, and are preserved 
by $k$\nd tensor functors to a $k$\nd tensor category.

For a pregroupoid $H$ over $X$, the positive objects of $\Mod_H(X)$ are those of bounded rank as vector bundles over $X$.
In particular if $X$ is $H$\nd connected, then every object of $\Mod_H(X)$ is positive.

If $\sC$ is a $k$\nd tensor category which is essentially small (i.e.\ has a small skeleton)
and in which every object is positive, then for some affine $k$\nd scheme $X$ there
exists a faithful conservative functor from $\sC$ to $\Mod(X)$ \cite[4.1.3]{O11}.
Properties of positive objects in an arbitrary $k$\nd tensor category may be deduced from this result
by passing to an appropriate full $k$\nd tensor subcategory.  
For example the Cayley--Hamilton theorem holds for positive objects: if $M$ is dualisable with $(n+1)$th exterior power $0$,
then any endomorphism $f$ of $M$ is annulled by a monic polynomial of degree $n$,
with coefficients the traces of exterior powers of $f$ \cite[p.39]{O11}.
Similarly the trace of a nilpotent endomorphism of a positive object is nilpotent.
If $\End(\I)$ is indecomposable, then $\tr(1_M)$ for $M$ positive is the least integer $n \ge 0$ for which 
the $(n+1)$th exterior power of $M$ is $0$.

\begin{prop}\label{p:posfincons}
Let $\sC$ be a $k$\nd tensor category with every object positive
for which $\End_\sC(\I)$ is a local $k$\nd algebra.
\begin{enumerate}
\item\label{i:posfin}
The hom-spaces of $\overline{\sC}$ are finite-dimensional over $\End_{\overline{\sC}}(\I)$.
\item\label{i:poscons}
The projection from $\sC$ onto $\overline{\sC}$ reflects isomorphisms.
\end{enumerate}
\end{prop}

\begin{proof}
\ref{i:posfin} 
Every object in the pseudo-abelian hull $\widetilde{\sC}$ of $\overline{\sC}$ is positive.
By \eqref{e:dualiso} it is enough to prove that $\Hom_{\widetilde{\sC}}(\I,M)$ is finite-dimensional 
over $\End_{\widetilde{\sC}}(\I)$ for each $M$ in $\widetilde{\sC}$.
Since $\Rad(\widetilde{\sC}) = 0$,
every non-zero morphism from $\I$ to an object of $\widetilde{\sC}$ is the embedding of a direct summand.
By induction on $\tr(1_M)$, the dimension of $\Hom_{\widetilde{\sC}}(\I,M)$ over $\End_{\widetilde{\sC}}(\I)$ is thus bounded above by $\tr(1_M)$.

\ref{i:poscons}
It is enough to show that $1+f$ is invertible for every $f:M \to M$ in $\sC$ in the kernel of the projection.
Write $\mathfrak{o}$ for $\End_\sC(\I)$ and $\mathfrak{o}'$ for the commutative $\mathfrak{o}$\nd subalgebra
of $\End_\sC(M)$ generated by $f$.
By the Cayley--Hamilton Theorem, $f$ is annulled by a polynomial with leading coefficient $1$ and
other coefficients in the maximal ideal $\mathfrak{m}$ of $\mathfrak{o}$.
Thus $\mathfrak{o}'$ is finite over $\mathfrak{o}$, 
and $f$ is nilpotent in $\mathfrak{o}'/\mathfrak{m}\mathfrak{o}'$.
It follows that $1+f$ is invertible in $\mathfrak{o}'/\mathfrak{m}\mathfrak{o}'$, and hence in $\mathfrak{o}'$.
\end{proof}

A pseudo-abelian category with semisimple endomorphism rings is semisimple abelian,
because every object is a finite direct sum of indecomposable objects, and every non-zero morphism 
between indecomposable objects is an isomorphism, as follows by considering the endomorphism ring of
their direct sum. 
Further the Krull--Schmidt theorem holds for such a category, 
i.e.\ the commutative monoid of isomorphism classes of objects under direct sum is free.

Recall that a commutative local ring is said to be henselian if every finite commutative algebra over it
is a finite product of local rings.

\begin{prop}\label{p:posssks}
Let $\sC$ be a $k$\nd tensor category with every object positive
for which $\End_\sC(\I)$ is a henselian local $k$\nd algebra.
\begin{enumerate}
\item\label{i:posss}
$\overline{\sC}$ is semisimple abelian.
\item\label{i:posks}
Every object of $\sC$ is a finite direct sum of indecomposable objects, and such a decomposition is unique
up to isomorphism.
\end{enumerate}
\end{prop}

\begin{proof}
By Cayley--Hamilton, any endomorphism of $M$ in $\sC$ is finite over $\End_\sC(\I)$, and hence contained
in a commutative $k$\nd subalgebra of $\End_{\sC}(M)$ which is finite product of local $k$\nd algebras.
Idempotent endomorphisms can thus be lifted from $\overline{\sC}$ to $\sC$, so that $\overline{\sC}$ is pseudo-abelian.
Let $R$ be an endomorphism ring of $\overline{\sC}$.
By Proposition~\ref{p:posfincons}\ref{i:posfin}, $R$ is finite over the field $\End_{\overline{\sC}}(\I)$. 
Thus $R$ has an ideal $J$ consisting of nilpotent elements with $R/J$ semisimple.
For $f$ in $J$ we then have
\[
\tr(g \circ f) = 0
\]
for every $g$ in $R$.
Since $\Rad(\overline{\sC}) = 0$, this shows that $J = 0$ and $R$ is semisimple.
By the above remarks, \ref{i:posss} follows.
The Krull--Schmidt theorem for $\overline{\sC}$ also follows, 
and by Proposition~\ref{p:posfincons}\ref{i:poscons} this implies \ref{i:posks}.
\end{proof}

Let $G$ be an affine $k$\nd group.
The forgetful functor from $G$\nd modules to $k$\nd vector spaces
creates colimits and finite limits.
The category of $G$\nd modules also has arbitrary (small) limits, but in general these are
\emph{not} preserved by the forgetful functor.
If we denote by a subscript $0$ the underlying $k$\nd vector space of a $G$\nd module,
then the canonical $k$\nd linear map from $(\lim_\lambda V_\lambda)_0$ to $\lim_\lambda V_{\lambda0}$
is injective with image the filtered union of the subspaces $\varphi(V_0)$ for $V$ a $G$\nd module 
(which may be taken finite-dimensional) and $\varphi$ a $k$\nd linear map from $V_0$ to $\lim_\lambda V_{\lambda0}$
with each component a $G$\nd homomorphism.

By a $G$\nd algebra we mean a $G$\nd module $R$ equipped with a structure of $k$\nd algebra
for which the multiplication $R \otimes_k R \to R$ and identity $k \to R$ are $G$\nd homomorphisms.
Given a $G$\nd algebra $R$, a $(G,R)$\nd module is a module $M$ over the $k$\nd algebra $R$
such that the action $R \otimes_k M \to M$ of $R$ on $M$ is a $G$\nd homomorphism.
A $(G,R)$\nd submodule of $R$ will be called a $G$\nd ideal of $R$.
The forgetful functor from $G$\nd algebras to $G$\nd modules creates limits and filtered colimits,
and similarly for the forgetful functor from $(G,R)$\nd modules to $G$\nd modules.

A $G$\nd algebra will be called finitely generated if it is finitely generated
as a $k$\nd algebra.
A finitely generated $G$\nd algebra has a finite-dimensional $G$\nd submodule
which generates it as an algebra over $k$.
A  $(G,R)$\nd module will be called finitely generated if it is finitely generated as an $R$\nd module.
A finitely generated $(G,R)$\nd module has a finite-dimensional $G$\nd submodule which generates it over $R$.

Let $R$ be a finitely generated commutative $G$\nd algebra.
Then $\Hom_{G\textrm{-alg}}(R,-)$ preserves filtered colimits of commutative $G$\nd algebras.
This is clear when $R$ is the symmetric algebra $\Sym V$ on a finite-dimensional $G$\nd module $V$, because
then $\Hom_{G\textrm{-alg}}(R,-)$ is naturally isomorphic to $\Hom_G(V,-)$.
In general, there are finite-dimensional $G$\nd modules $V$ and $V'$ such that $R$ is the coequaliser in 
the category of commutative $G$\nd algebras of two morphisms from $\Sym V'$ to $\Sym V$,
and it suffices to note that in the category of sets finite limits commute with filtered colimits.

\begin{lem}\label{Gsubspace}
Let $G$ be an affine $k$-group and $W$ be a $G$\nd module.
Suppose that $k$ is algebraically closed.
Then every $G(k)$\nd subspace of $W$ is a $G$\nd submodule.
\end{lem}

\begin{proof}
Let $W_0$ be a $G(k)$\nd subspace of $W$.
Choosing a basis of $W$ containing a basis of $W_0$ and considering the matrix of the automorphism
of $W \otimes_k \sO_G$ defining the action of $G$ shows that $G$ has a closed subscheme $Z$
such that $g$ in $G(S)$ lies in $Z(S)$ if and only if the action of $g$ on $W \otimes_k \sO_S$
restricts to an automorphism of $W_0 \otimes_k \sO_S$.
Since  $G(k)$ is dense in $G$ by Corollary~\ref{c:kgrpac}\ref{i:kgrpacdense} and $G$ is reduced, $Z = G$.
\end{proof}

Let $G$ be a reductive $k$\nd group of finite type.
If $R$ is a finitely generated commutative $G$\nd algebra,
then $R^G$ is a finitely generated $k$\nd algebra \cite[II~Theorem~3.6]{ShaAlgIV},
and $M^G$ is a finitely generated $R^G$\nd module
for $M$ a finitely generated $(G,R)$\nd module \cite[II~Theorem~3.25]{ShaAlgIV}.
If $R$ is a commutative $G$\nd algebra which is finitely generated over $R^G$, it can be seen as follows
that $\Hom_G(V,M)$ is a finitely generated $R^G$\nd module for $V$ a finite-dimensional $G$\nd module and
$M$ a finitely generated $(G,R)$\nd module.
Since \eqref{e:dualiso} with $V$ for $M$ and $M$ for $N$ is an isomorphism of $R^G$\nd modules, we may suppose that $V = k$.
If $R_0$ is a finitely generated $G$\nd subalgebra of $R$ which generates $R$ over $R^G$, and $M_0$
is a finite-dimensional $G$\nd submodule of $M$ which generates $M$ as a $(G,R)$\nd module, it then
suffices to apply $(-)^G$ to 
\[
R^G \otimes_k R_0 \otimes_k M_0 \to M \to 0
\]
and use the fact that $(R_0 \otimes_k M_0)^G$ is finitely generated over $(R_0)^G$.

Let $G$ be a reductive $k$\nd group and $R$ be a commutative $G$\nd algebra for which $R^G$
is a local $k$\nd algebra.
Then $R$ has a unique maximal $G$\nd ideal, which contains the maximal ideal of $R^G$.
Indeed if $\sJ$ is the set of $G$\nd ideals $J \ne R$ of $R$ and $\mathfrak{m}$ is the maximal ideal of $R^G$, 
then applying $(-)^G$ shows that the canonical $G$\nd homomorphisms from $\coprod_{J \in \sJ} J$ and from
$\mathfrak{m} \otimes_k R$ to $R$ are not surjective.

For the proof of Lemma~\ref{l:complete} below we need the following fact from commutative
algebra  \cite[III, \S 3 prop.\~5 and IV, \S 1 prop.\~2, cor.\~2]{BAC-1}:
given an ideal $J$ in a noetherian commutative ring $R$ and a finitely generated $R$\nd module $M$,
we have 
\begin{equation}\label{e:capJM}
\bigcap_{n=1}^\infty J^rM = 0
\end{equation}
if and only if $J + \mathfrak{p} \ne R$
for every associated prime $\mathfrak{p}$ of $M$.

\begin{lem}\label{l:complete}
Let $G$ be a reductive $k$-group of finite type, $R$ be a commutative $G$-algebra,
$J \ne R$ be a $G$-ideal of $R$,
and $M$ be a finitely generated $(G,R)$-module.
Suppose that $R^G$ is a complete noetherian local $k$-algebra
with residue field $k$, and that $R$ is finitely generated
as an algebra over $R^G$.
Then $M$ is the limit in the category of $G$\nd modules of its quotients $M/J^nM$.
\end{lem}

\begin{proof}
We note that the conclusion of the lemma will hold
provided that for every finite-dimensional $G$\nd module $V$ the canonical homomorphism
\begin{equation}\label{e:Homlimiso}
\Hom_G (V,M) \rightarrow \lim_n \Hom_G (V,M/J^nM)
\end{equation} 
is bijective.
Indeed \eqref{e:Homlimiso} will then be bijective for arbitrary $V$, as follows by writing 
$V$ as the filtered colimit of its finite-dimensional $G$\nd submodules. 
We write $\mathfrak m$ for the maximal ideal of $R^G$.
The proof proceeds in three steps.

(1) Suppose that $R^G$ is finite over $k$.
Let $V$ be a finite-dimensional $G$\nd module.
Then $\Hom_G(V,M)$ is finite-dimensional over $k$ 
because it is as above a finitely generated $R^G$-module.
Hence \eqref{e:Homlimiso} is surjective because $\Hom_G(V,-)$ is exact.
To see that \eqref{e:Homlimiso} is injective,
it suffices to check that \eqref{e:capJM} holds.
To do this
we may after extending the scalars suppose
that $k$ is algebraically closed.
Let $\mathfrak{p}$ be an associated prime of $M$.
Then
\[
\mathfrak{p}_0 = \bigcap_{g \in G(k)} g\mathfrak p
\]
is stable under $G(k)$, and hence by Lemma~\ref{Gsubspace} is a $G$-ideal of $R$.
Since $R^G$ is local,
\[
R^G \rightarrow (R/J)^G \times (R/\mathfrak{p}_0)^G
\]
is not surjective. 
Thus
$R \rightarrow R/J \times R/\mathfrak{p}_0$
is not surjective, so that 
\[
J + \mathfrak{p}_0 \ne R.
\]
Since each $g\mathfrak p$ lies in the finite set of associated primes of $M$,
we therefore have $J + g\mathfrak{p} \ne R$ for some $g \in G(k)$.
Thus
\[
J + \mathfrak{p} = g^{-1}(J + g\mathfrak{p}) \ne R.
\]
Hence \eqref{e:capJM} holds as required.

(2) Suppose that $J = \mathfrak m R$.
Let $V$ be a finite-dimensional $G$\nd module.
Since $\Hom_G(V,-)$ preserves the exact sequence
\[
\mathfrak{m}^n \otimes_k M \rightarrow M \to M/\mathfrak{m}^nM \to 0,
\]
for each $n$ the projection $M \to M/\mathfrak{m}^nM$ induces an isomorphism
\[
\Hom_G (V,M)/\mathfrak{m}^n\Hom_G (V,M) \xrightarrow{\sim}
\Hom_G (V,M/{\mathfrak m}^nM).
\]
Thus \eqref{e:Homlimiso} is bijective,
because $\Hom_G (V,M)$ is as above a finitely generated $R^G$\nd module
and hence complete for the $\mathfrak m$-adic topology. 

(3) Now consider the general case. Write $R_r$ and $M_r$
for $R/\mathfrak{m}^rR$ and $M/\mathfrak{m}^rM$, and $J_r$
for the image of $J$ in $R_r$. 
Then $J_r \ne R_r$, because $J$ and $\mathfrak{m}^rR$ are contained in the unique maximal
$G$\nd ideal of $R$.
By (2), and by (1) with $R_r$ for $R$, $J_r$ for $J$ and $M_r$ for $M$, we have canonical isomorphisms
of $G$\nd modules
\begin{equation}\label{e:limriso}
M \xrightarrow{\sim} \lim_r M_r \iso \lim_r\lim_n M_r/(J_r)^nM_r.
\end{equation}
Also $M_{(n)} = M/J^nM$ is a finitely generated $R$-module
for each $n$, so that by (2) with $M_{(n)}$ for $M$, the second arrow of
\begin{equation}\label{e:limniso}
M \rightarrow \lim_n M_{(n)} \to \lim_n\lim_r M_{(n)}/\mathfrak{m}^rM_{(n)}
\end{equation}
is an isomorphism.
On the other hand,
by \eqref{e:limriso} and the canonical isomorphisms
\[
M_r/(J_r)^nM_r \xleftarrow{\sim} M/(J^nM +\mathfrak{m}^rM)
\xrightarrow{\sim} M_{(n)}/\mathfrak{m}^rM_{(n)},
\]
the composite of the two arrows \eqref{e:limniso} is an isomorphism.
Thus the first arrow of \eqref{e:limniso}
is an isomorphism, as required.
\end{proof}

Let $G$ be an affine $k$\nd group.
A commutative $G$\nd algebra $R$ will be called \emph{simple} if $R \ne 0$ and $R$ has no $G$\nd ideal
other than $0$ and $R$.
Suppose that $G$ is of finite type and $R$ is a simple commutative $G$\nd algebra with $R^G = k$.
Then by a theorem of Magid \cite[4.5]{Mag}, $\Spec(R)$ is a homogeneous (i.e.\ transitive) $G$\nd scheme,
and in particular $R$ is a finitely generated $k$\nd algebra \cite[4.6]{Mag}.

Let $A$ be the henselisation at a $k$\nd point of $k$\nd scheme of finite type, and
$F$ be a functor from commutative $A$\nd algebras to sets.
Then if $\widehat{A}$ is the completion of $A$, it follows from Artin's approximation theorem \cite[1.12]{Art69}
that $F(A)$ is non-empty provided that $F(\widehat{A})$ is non-empty and $F$ preserves filtered colimits.

The following lemma reduces when $R^G = k$ to \cite[4.4.4]{O11}.
Since some of the steps in the proof of \cite[4.4.4]{O11} apply essentially unmodified
to the proof of Lemma~\ref{l:Dextend}, we simply refer at the relevant places to \cite{O11}.
Note that what are here called reductive $k$\nd groups (resp.\ reductive $k$\nd groups of finite type)
were called proreductive $k$\nd groups (resp.\ reductive $k$\nd groups) in \cite{O11}.

\begin{lem}\label{l:Dextend}
Let $G$ be a reductive $k$-group, $R$ be a commutative $G$-algebra, $D$ be a $G$-subalgebra of $R$, and 
$p:R \rightarrow \overline{R}$ be the projection onto a simple quotient $G$-algebra of $R$.
Suppose that $R^G$ is a henselian local $k$\nd algebra with residue field $k$, and that the restriction
of $p$ to $D$ is injective.
Then $R$ has a $G$-subalgebra $D'$ containing $D$
such that the restriction of $p$ to $D'$ is an isomorphism.
\end{lem}

\begin{proof}
Write $J$ for the $G$\nd ideal $\Ker p$ of $R$ and $A$ for the $k$\nd algebra $R^G$.
Since $\overline{R}$ is a simple $G$\nd algebra and $A$ is local, $J$ is as above the unique maximal 
$G$\nd ideal of $R$,
and it contains the maximal ideal of $A$.
Hence $\overline{R}{}^G = k$, because $A \to \overline{R}{}^G$ is surjective.
By the theorem of Magid,  $\overline{R}$ will thus be a finitely generated $G$\nd algebra
provided that $G$ is of finite type.
We consider successively the following cases:
\begin{enumerate}
\renewcommand{\theenumi}{(\arabic{enumi})}
\item
$G$ is of finite type, $J^2 = 0$, and $D$ is a
finitely generated $k$\nd algebra
\item
$G$ is of finite type, $A$ is a complete
noetherian local $k$\nd algebra, $R$ is a finitely generated
$A$-algebra, and $D$ is a finitely generated $k$\nd algebra
\item
$G$ is of finite type, $A$ is the henselisation
at a $k$-point of a $k$-scheme of finite type,
$R$ is a finitely generated $A$\nd algebra,
and $D$ is a finitely generated $k$\nd algebra
\item
$G$ is of finite type and $D$ is a finitely generated $k$\nd algebra
\item
the general case.
\end{enumerate}
When $A = k$, the cases (1), (4) and (5) above are respectively the cases (1), (2) and (3) of 
\cite[4.4.4]{O11}.

\quad (1) 
Step (1) of the proof of \cite[4.4.4]{O11} applies unmodified.

\quad (2) 
By Lemma~\ref{l:complete}, $R$ is the limit in the category of
$G$\nd algebras of its quotients $R/J^n$.
If $D_n$ is the image of $D$ in $R/J^n$, it is thus enough to show
that any $G$\nd subalgebra $D'{}\!_n \supset D_n$ of $R/J^n$ with
$D'{}\!_n \to \overline{R}$ an isomorphism can be lifted to
a $G$\nd subalgebra $D'{}\!_{n+1} \supset D_{n+1}$ of $R/J^{n+1}$ with
$D'{}\!_{n+1} \to \overline{R}$ an isomorphism.
To do this, apply (1) with the inverse
image of $D'{}\!_n$ in $R/J^{n+1}$ for $R$ and $D_{n+1}$ for $D$.

\quad (3) 
Given a commutative $A$\nd algebra $B$, denote by $F(B)$ the set of those $G$\nd algebra
homomorphisms 
\[
a:\overline{R} \to R \otimes_A B
\]
such that the restrictions to $D$ of $a \circ p$
and the canonical homomorphism $R \to R \otimes_A B$ coincide,
and such that $(p \otimes_A B) \circ a$ is the canonical homomorphism
$\overline{R} \to \overline{R} \otimes_A B$.
Then $B \mapsto F(B)$ may be regarded as a functor
from commutative $A$\nd algebras to sets.
It is enough to show that $F(A)$ is non-empty, because with the identification $R = R \otimes_A A$
we may take $D' = a(\overline{R})$ for any $a$ in $F(A)$.
By Artin's approximation theorem, it will suffice
to show that $F$ commutes with filtered colimits and that
$F(\widehat{A})$ is non-empty, where $\widehat{A}$
is the completion of $A$.

Since $\overline{R}$ is a finitely generated $G$\nd algebra, $\Hom_{G\textrm{-alg}}(\overline{R},-)$
preserves filtered colimits of commutative $G$\nd algebras.
Thus if $B$ is the filtered colimit of commutative $A$\nd algebras $B_\lambda$,
then the canonical map
\[
\colim_\lambda \Hom_{G\textrm{-alg}}(\overline{R}, R \otimes_A B_\lambda) \to 
\Hom_{G\textrm{-alg}}(\overline{R},R \otimes_A B)
\]
is bijective.
Since $D$ and $\overline{R}$ are finitely generated $k$\nd algebras,
any element of $\Hom_{G\textrm{-alg}}(\overline{R}, R \otimes_A B_\lambda)$
whose image in $\Hom_{G\textrm{-alg}}(\overline{R}, R \otimes_A B)$ lies in $F(B)$
has image in $\Hom_{G\textrm{-alg}}(\overline{R}, R \otimes_A B_{\lambda'})$
which lies in $F(B_{\lambda'})$ for $\lambda'  \ge \lambda$ sufficiently large.
Thus $F$ preserve filtered colimits.

The restriction $A \to \overline{R}$ of $p$ to $A$ factors through
the augmentation $A \to k$.
Thus $p$ and $\widehat{A} \to k \to \overline{R}$ define a homomorphism of $G$\nd algebras 
\[
\widehat{p}:\widehat{R} = R \otimes_A  \widehat{A} \to \overline{R},
\]
which factors as $p \otimes_A  \widehat{A}$ followed by an isomorphism.
The canonical homomorphism from $R$ to $\widehat{R}$
embeds $D$ as a $G$\nd subalgebra of $\widehat{R}$
to which the restriction of $\widehat{p}$ is injective.
Applying $(-)^G$ to the projection of
$R \otimes_k \widehat{A}$ onto $\widehat{R}$
shows that $\widehat{R}^G$ is the image of
$\widehat{A}$ in $\widehat{R}$.
Since $\widehat{p}$ is surjective,
we may thus apply (2) with $\widehat{R}$ for $R$ and $\widehat{p}$ for $p$ to obtain
a $G$\nd subalgebra $D'$ of $\widehat{R}$ such that the restriction of $\widehat{p}$
to $D'$ is an isomorphism.
The composite of the embedding of $D'$ with the inverse of this isomorphism is then an element of $F(\widehat{A})$.
Thus $F(\widehat{A})$ is non-empty.

\quad (4) 
Write $R_0$ for the $G$-subalgebra of $R$
generated by $D$ and a lifting to $R$ of a finite set
of generators of $\overline{R}$.
Then $R_0$ is a finitely generated $G$-algebra,
so that $(R_0)^G$ is a finitely generated $k$-algebra.
Since $A$ is henselian, the embedding of $(R_0)^G$ into $A$ extends to a homomorphism to $A$
from the henselisation  of $(R_0)^G$ at the kernel of $(R_0)^G \to A \to k$.
Its image is a $k$-subalgebra $A_0$ of $A$
containing $(R_0)^G$ which is the henselisation at a $k$\nd point of a $k$\nd scheme of finite type.
If $R_1$ is the $G$-subalgebra of $R$ generated by
$A_0$ and $R_0$,
then the restriction $p_1$ of $p$ to $R_1$ is surjective, $R_1$ contains $D$,
and $R_1$ is a finitely generated $A_0$\nd algebra.
Applying $(-)^G$ to the canonical $G$\nd homomorphism from $R_0 \otimes_k A_0$ to $R_1$ 
shows further that  $(R_1)^G = A_0$, 
Thus by (3) with $R_1$ and $p_1$ for $R$ and $p$, there is a $G$-subalgebra
$D' \supset D$ of $R_1 \subset R$
such that the restriction of $p$ to $D'$ is an isomorphism.

\quad (5)
Step (3) of the proof of  \cite[4.4.4]{O11} applies with
the references to (2) replaced by references to (4), and the second paragraph
replaced by the following one:

Let $B \supset D_\mu$ be a $G$\nd subalgebra of $R$ for which $B^{H_\mu}$ contains a simple $G/H_\mu$\nd subalgebra
containing $D_\mu$.
Then we have a unique factorisation
\[
\Spec(B) \xrightarrow{f_{\mu,B}} X_\mu \to \Spec(D_\mu)
\]
of the $G$\nd morphism $\Spec(B) \to \Spec(D_\mu)$.
By (4) with $G/H_\mu$, $R^{H_\mu}$ and $D_\mu$ for $G$, $R$ and $D$, the condition on $B$ is satisfied with $B = R$.
\end{proof}

Let $\sC_1$ and $\sC_2$ be $k$\nd pretensor categories.
Then there is a $k$\nd pretensor category
$\sC_1 \otimes_k \sC_2$
with set of objects the product of the sets
of objects of $\sC_1$ and $\sC_2$, hom-spaces the tensor products over $k$ of those of $\sC_1$ and $\sC_2$,
and composition, identities, tensor products, associativities and symmetries defined component-wise.
Given $k$\nd tensor functors $T_i:\sC_i \to \sC$ for $i=1,2$, we then have $T_i = TI_i$ for $i = 1,2$ for
a $k$\nd tensor functor $T:\sC_1 \otimes_k \sC_2 \to \sC$, where $I_1$ is $(\I,-)$ and $I_2$ is $(-,\I)$ \cite[p.9]{O11}.
If $\sC_i$ is rigid for $i = 1,2$, then $\sC_1 \otimes_k \sC_2$ is rigid, because each of its objects
is of the form $I_1(M_1) \otimes I_2(M_2)$.
If further $\End_{\sC_i}(\I) = k$ and $\Rad(\sC_i) = 0$ for $i = 1,2$, then $\End_{\sC_1 \otimes_k \sC_2}(\I) = k$
and $\Rad(\sC_1 \otimes_k \sC_2) = 0$.

The following lemma is a modified form of \cite[4.4.5]{O11}. 
It was there assumed that $\End_\sC(\I) = k$, but there was a less strict condition
than positivity of the objects, and the result was proved in a slightly sharper form with ``right quasi-inverse''
replaced by ``right inverse''.
The proof given there also applies here, with the modifications indicated below.
Note that what are here written $\MOD_G(k)$ and $\Mod_G(k)$ were written $\REP_k(G)$ and $\Rep_k(G)$ in 
\cite{O11}.
By a right quasi-inverse to a $k$\nd tensor functor is meant a right inverse up to tensor isomorphism.

\begin{lem}\label{l:liftsplit}
Let $\sC$ and $\sD$ be essentially small $k$\nd tensor categories with every object positive. 
Suppose that $\End_\sC(\I)$ is a henselian local $k$\nd algebra with residue field $k$.
Then any lifting $\sD \to \sC$ along the projection $Q:\sC \to \overline{\sC}$ of a faithful
$k$\nd tensor functor $\sD \to \overline{\sC}$
factors up to tensor isomorphism through some right quasi-inverse to $Q$.
\end{lem}

\begin{proof}
The first three paragraphs of the proof of \cite[Lemma~4.4.5]{O11} concern the reduction to the 
quasi-inverse and positive cases, and are irrelevant here.
If $E:\sD \to \sC$ is the lifting, the remaining two paragraphs may be applied here, 
provided that $P$ is replaced by $Q$,
and the following modifications are made.
In the fourth paragraph, replace the second-last sentence by ``Since each object of $\sB$ and $\sD$ is positive,
so also is each object of $\widetilde{\sD}$''.
In the fifth paragraph, replace ``$\Rep_k(G)$'' by $\Mod_G(k)$'', the ``$= k$'' in ``$\End_{\sF_R}(\I) = k$''
and ``$R^G = \Hom_G(k,R) = k$'' by ``is a henselian local $k$\nd algebra with residue field $k$'',
and the reference to Lemma~4.4.4 by a reference to Lemma~\ref{l:Dextend}.
\end{proof}

\begin{thm}\label{t:splitting}
Let $\sC$ be an essentially small $k$\nd tensor category with every object positive.
Suppose that $\End_\sC(\I)$ is a henselian local $k$\nd algebra with residue field $k$.
Then the projection $Q:\sC \to \overline{\sC}$ has a right quasi-inverse.
If $U$ is such a right quasi-inverse, and $\sD$ is an essentially small rigid $k$\nd tensor 
category,
then any $k$\nd tensor functor $T:\sD \to \sC$ with $QT$ faithful 
factors up to tensor isomorphism through $U$,
and such a factorisation is unique up to tensor isomorphism.
\end{thm}

\begin{proof}
A right quasi-inverse to $Q$ exists by Lemma~\ref{l:liftsplit} with $\sD = \Mod(k)$.

Let $U_1$ and $U_2$ be right quasi-inverses to $Q$.
The pseudo-abelian hull $\sE$ of $\overline{\sC} \otimes_k \overline{\sC}$ is a rigid $k$\nd tensor category
with $\End_{\sE}(\I) = k$ and $\Rad(\sE) = 0$.
There is a $k$\nd tensor functor $\overline{\sC} \otimes_k \overline{\sC} \to \sC$
through which both $U_1$ and $U_2$ factor, and hence a $k$\nd tensor functor $E:\sE \to \sC$
through which both $U_1$ and $U_2$ factor up to tensor isomorphism.
Further $QE$ is faithful because $\Rad(\sE) = 0$, so that every object of $\sE$ is positive.
By Lemma~\ref{l:liftsplit} there is a right quasi-inverse $W$ to $Q$
through which $E$, and hence each of $U_1$ and $U_2$, factors up to tensor isomorphism.
Composing with $Q$ shows that $U_1$ and $U_2$ are tensor isomorphic to $W$, and hence to each other.

Every object of $\sD$ is positive, because $QT$ is faithful.
Since a right quasi-inverse to $Q$ is unique up to tensor isomorphism,
the existence of the required factorisation of $T$ thus follows from Lemma~\ref{l:liftsplit}.
Its uniqueness up to tensor isomorphism follows by composing with $Q$. 
\end{proof}

\section{Universal and minimal reductive groupoids}\label{s:unmin}

\emph{In this section $k$ is a field of characteristic $0$, $X$ is a non-empty $k$\nd scheme, and
$H$ is a pregroupoid over $X$.}

\medskip

In this section we prove that the category of reductive groupoids over
$H$ up to conjugacy has an initial object provided that $H^0_H(X,\sO_X)$ is a henselian local $k$\nd algebra
with residue field $k$.
As well as the classification theorems of Sections~\ref{s:curves0} and \ref{s:curves1}, 
this result has many consequences which will be developed in this and the next three sections.
A particularly important application is to those reductive groupoids over $H$ that do not
strictly contain a reductive subgroupoid over $H$, and to the principal bundles that give rise to them.

The main result, Theorem~\ref{t:main}, is proved by combining the well-known dictionary between 
transitive affine groupoids and fibre functors on Tannakian categories \cite[1.12]{Del90} with 
Theorem~\ref{t:splitting}.
We begin by recalling some basic definitions and results for Tannakian categories.

Let  $h:H' \rightarrow H$ be a morphism of pregroupoids over $X$.
Then the $k$\nd tensor functor $h^*$ from $\Mod_H(X)$ to $\Mod_{H'}(X)$ induced by $h$ is strict,
i.e.\ strictly preserves the tensor product.
We have $(h \circ h')^* = h'{}^* h^*$.
If $H' = X$, then $h^*$ may be identified with the forgetful $k$\nd tensor functor, which we write
\[
\omega_H:\Mod_H(X) \rightarrow \Mod(X).
\] 
We have $\omega_H = \omega_{H'}h^*$.
If $K$ is a groupoid over $H$, we also write
\[
\omega_{K/H}:\Mod_K(X) \to \Mod_H(X)
\]
for $f^*$ with $f:H \to K$ the structural morphism of $K$.

The functor $\omega_H$ admits transport of structure,
i.e.\ given a $\sV$ in $\Mod_H(X)$ and an isomorphism $i_0$ in $\Mod(X)$ with source $\omega_H\sV$,
there is a unique isomorphism $i$ in $\Mod_H(X)$ with source $\sV$ for which $\omega_H(i) = i_0$.
Thus given a $k$\nd tensor functor $T$ to $\Mod_H(X)$ and a tensor isomorphism $\theta_0$ with
source $\omega_HT$, there is a unique tensor isomorphism $\theta$ with source $T$ for which $\omega_H\theta = \theta_0$.
Similarly for $\omega_{K/H}$.

By assigning to each point of $H_{[1]}$ in a scheme $Z = (Z,z_1,z_0)$ over $X \times_k X$ its action
on the underlying $\sO_X$\nd module of a representation of $H$, we obtain a map 
\[
\eta_{H,Z}:H_{[1]}(Z)_{X \times_k X} \to \Iso^\otimes(z_0{}\!^*\omega_H,z_1{}\!^*\omega_H)
\]
which is natural in $Z$, where $\Iso^\otimes$ denotes the set of tensor isomorphisms.
The map $\eta_{H,Z}$ is also natural in $H$, and is compatible with composition, in the sense that
\[
\eta_{H,Z_1}(d_1(v)) = \eta_{H,Z_0}(d_0(v)) \circ \eta_{H,Z_2}(d_2(v))
\]
for $v$ in $H_{[2]}(Z)$, where  $Z_i = (Z,d_0(d_i(v)),d_1(d_i(v)))$.
The image of $1_{H_{[1]}}$ under $\eta_{H,H_{[1]}}$ is the tautological tensor isomorphism
\[
\tau_H:d_1^*\omega_H \iso d_0^*\omega_H
\]
with components given by the action of $H$.
If $K$ is a transitive affine groupoid over $X$, then $\eta_{K,Z}$ is bijective for every $Z$ \cite[1.12(iii)]{Del90}.

A \emph{Tannakian category over $k$} is an essentially small abelian rigid $k$\nd tensor category $\sD$ with 
$\End(\I) = k$
for which the following condition holds:
for some extension $k'$ of $k$, there exists a faithful $k$\nd tensor functor from $\sD$ to $\Mod(k')$.
The last condition is equivalent \cite[7.1]{Del90} to the following one:
each object of $\sD$ is positive in the sense of Section~\ref{s:splitting}. 
If $K$ is a transitive affine groupoid over $X$, then $\Mod_K(X)$ is a Tannakian category over $k$, which is semisimple
if and only if $K$ is reductive.

Let $\sD$ be a Tannakian category over $k$.
A \emph{fibre functor of $\sD$ over $X$} is a $k$\nd tensor functor from $\sD$ to $\Mod(X)$ 
which preserves short exact sequences, 
where a short exact sequence in $\Mod(X)$ is defined as one which is exact in $\MOD(X)$.
When $\sD$ is semisimple, every $k$\nd tensor functor from $\sD$ to $\Mod(X)$ is a fibre functor.
Any fibre functor of $\sD$ over $X$ factors \cite[1.12(i),(ii)]{Del90} as
\begin{equation}\label{e:fibrefac}
\sD \to \Mod_K(X) \xrightarrow{\omega_K} \Mod(X)
\end{equation}
with the first arrow an equivalence and $K$ a transitive affine groupoid over $X$.

\begin{lem}\label{l:Tf*}
Let $K$ be a transitive affine groupoid over $X$ and 
$T$ be a $k$\nd tensor functor from $\Mod_K(X)$ to $\Mod_H(X)$ with $\omega_H T = \omega_K$ (resp. with\ $\omega_H T$
tensor isomorphic to $\omega_K$).
Then there exists a unique morphism $f$
(resp.\ a morphism $f$) from $H$ to $K$ over $X$ with $T = f^*$ (resp.with \ $T$ tensor isomorphic to $f^*$).
\end{lem}

\begin{proof}
By transport of structure for $\omega_H$, we may suppose after replacing if necessary $T$ by a tensor isomorphic 
functor that $\omega_H T = \omega_K$.
Any $k$\nd tensor functor $T$ with $\omega_H T = \omega_K$ is strict, and the identity
on underlying $\sO_X$\nd modules, and hence completely determined by $\tau_H T$.
Thus, we have $T = f^*$ for $f:H \to K$ over $X$ if and only 
if for each $Z = (Z,z_1,z_0)$ the square 
\[
\xymatrix{
H_{[1]}(Z)_{X \times_k X} \ar[d]_{\eta_{H,Z}}    \ar[r]^{f} &   K(Z)_{X \times_k X} \ar[d]_{\eta_{K,Z}}^{\wr}  \\
\Iso^\otimes(z_0{}\!^*\omega_H,z_1{}\!^*\omega_H) \ar[r]^{-T} &  \Iso^\otimes(z_0{}\!^*\omega_K,z_1{}\!^*\omega_K)
}
\]
commutes: the ``only if'' follows by naturality of $\eta_{H,Z}$ in $H$,
and the ``if'' by taking $Z = H_{[1]}$ and evaluating at $1_{H_{[1]}}$.
Since the bottom three arrows are natural in $Z$ and compatible with composition,
the existence and uniqueness of $f$ follow.
\end{proof}

Let $\sD$ be a Tannakian category over $k$ and $T:\sD \to \Mod_H(X)$ be a $k$\nd tensor functor with $\omega_HT$ a fibre functor.
It can be seen as follows that $T$ factors as 
\begin{equation}\label{e:fibrefacH}
\sD \to \Mod_K(X) \xrightarrow{\omega_{K/H}} \Mod_H(X)
\end{equation}
with the first arrow an equivalence and $K$ a transitive affine groupoid over $H$.
By transport of structure for $\omega_{K/H}$, it suffices to give such a factorisation up to tensor isomorphism.
By \eqref{e:fibrefac}, $\omega_HT = \omega_KI$ for some $K$ over $X$ and equivalence $I$.
If $I'$ is quasi-inverse to $I$, then $TI'$ is by Lemma~\ref{l:Tf*} tensor isomorphic
to $\omega_{K/H}$ for some structure of groupoid over $H$ on $K$.
Thus $T$ is tensor isomorphic to $\omega_{K/H}I$.

Let $K$ be a groupoid over $X$
and $v$ be a cross-section of $K^\mathrm{diag}$.
If $f$ is a morphism from $H$ to $K$ over $X$ and ${}^v\! f$ is its conjugate by $v$, then
for each representation $\sV$ of $K$, the automorphism of its underlying $\sO_X$\nd module
induced by $v$ is an $H$\nd isomorphism from $f^*\sV$ to $({}^v\! f)^*\sV$.
Thus $v$ induces a tensor isomorphism from $f^*$ to $({}^v\! f)^*$ with component at a representation of $K$
the action of $v$ on its underlying $\sO_X$\nd module.

\begin{lem}\label{l:tensisoconj}
Let $f_1$ and $f_2$ be morphisms from $H$ to a transitive affine groupoid $K$ over $X$.
Then any tensor isomorphism from $f_1^*$  to $f_2^*$
is induced by a cross-section $v$ of $K^\mathrm{diag}$ with $f_2 = {}^v\! f_1$,
and such a $v$ is unique.
\end{lem}

\begin{proof}
Let $\theta:f_1^* \iso f_2^*$ be a tensor isomorphism.
Then there is a unique cross-section $v$ of $K^\mathrm{diag}$ with $\eta_{K,X}(v) = \omega_H\theta$.
If $\theta':f_1^* \iso ({}^v\! f_1)^*$ is the tensor isomorphism induced by $v$, then
$\omega_H \theta' = \omega_H \theta$,
so that $\theta' = \theta$ by transport of structure for $\omega_H$.
Thus $f_2^* = ({}^v\! f_1)^*$, and hence $f_2 = {}^v\! f_1$ by the uniqueness of Lemma~\ref{l:Tf*}.
\end{proof}

Suppose that $H^0_H(X,\sO_X)$ is a henselian local $k$\nd algebra with residue field $k'$.
Then by Proposition~\ref{p:posssks}, the Krull--Schmidt theorem holds for
$\Mod_H(X)$, and the quotient $\overline{\Mod_H(X)}$ by its unique maximal tensor ideal is a semisimple abelian 
$k$\nd tensor category with $\End(\I) = k'$ and every object positive.
In particular if $k' = k$, then $\overline{\Mod_H(X)}$ is a semisimple Tannakian category over $k$.

\begin{defn}
Let $K$ be a reductive groupoid over $H$.
We say that $K$ is a \emph{universal reductive groupoid over $H$}, or is \emph{universally reductive over $H$},
if for every reductive groupoid $K'$ over $H$
there exists a morphism from $K$ to $K'$ over $H$, and if such a morphism is unique up to
conjugacy.
We say that $K$ is a \emph{minimal reductive groupoid over $H$}, or is
\emph{minimally reductive over $H$}, if it does not strictly contain a reductive
subgroupoid over $H$.
\end{defn}

A reductive groupoid over $H$ is universally reductive if and only if it is initial in the category of reductive
groupoids over $H$ up to conjugacy.
A universal reductive groupoid over $H$ is thus unique up to isomorphism if it exists. 
Any universal reductive groupoid over $H$ is minimally reductive over $H$.
If $f:K \to K'$ is a morphism of reductive groupoids over $H$ and $K$ is minimally reductive,
then $K'$ is minimally reductive if and only if $f$ is surjective.
By Zorn's lemma, every reductive groupoid over $H$ contains a minimal reductive groupoid over $H$.

The existence of a universal reductive groupoid over $H$ does not imply that $X$ is $H$\nd connected:
if $k$ is algebraically closed, $H = X$ and $X$ is the disjoint union of copies of $\Spec(k)$, then by
Proposition~\ref{p:printriv} $[X]$ is universally reductive over $H$.

\begin{prop}\label{p:unconjbij}
Let $K$ be a universal reductive groupoid over $H$.
\begin{enumerate}
\item\label{i:unconj}
Any two morphisms over $H$ from $K$ to a transitive affine groupoid $K'$ over $H$ are conjugate.
\item\label{i:unbij}
$\omega_{K/H}$ induces a bijection from isomorphism classes of representations of $K$
to isomorphism classes of representations of constant rank of $H$.
\end{enumerate}
\end{prop}

\begin{proof}
\ref{i:unconj}
Let $f_1$ and $f_2$ be morphisms from $K$ to $K'$ over $H$.
By Corollary~\ref{c:redsubgrpd}, $f_1(K)$ and $f_2(K)$ are contained in Levi subgroupoids $K_1$ and $K_2$ of $K'$,
and by Corollary~\ref{c:Leviclass}, $K_1$ and $K_2$ are conjugate over $H$.
The result follows.

\ref{i:unbij}
Any structure $H \to \underline{\Iso}_X(\sV)$
of representation of $H$ on a vector bundle $\sV$ of constant rank over $X$ factors, 
uniquely up to conjugacy, through $H \to K$.
The result thus follows by transport of structure for $\omega_{K/H}$.
\end{proof}

Let $f:K' \to K$ be a morphism of transitive affine groupoids over $X$.
Then by Lemmas~\ref{l:Tf*} and \ref{l:tensisoconj}, $f$ is an isomorphism if and only if $f^*$ is an equivalence.
Suppose that $K'$ is reductive.
Then $f$ is surjective if and only if $f^*$ is fully faithful:
to see the ``if'', we may suppose that $f$ is the embedding of a subgroupoid,
so that by Lemma~\ref{l:subquot} every representation of $K'$ is a direct summand of a representation of $K$
and hence $f^*$ is an equivalence.

When $H^0_H(X,\sO_X)$ is a local $k$\nd algebra, it follows from Theorem~\ref{t:requiv} that a transitive affine groupoid 
$K$ over $H$ has a reductive subgroupoid over $H$ if and only if the projection from
$\Mod_H(X)$ to $\overline{\Mod_H(X)}$ composed with $\omega_{K/H}$ is faithful. 
Theorem~\ref{t:main}\ref{i:maincrit} below gives similar conditions, under a more stringent hypothesis on $H^0_H(X,\sO_X)$,
for $K$ to be universally or minimally reductive over $H$.

\begin{thm}\label{t:main}
Suppose that $H^0_H(X,\sO_X)$ is a henselian local $k$\nd algebra
with residue field $k$.
\begin{enumerate}
\item\label{i:mainexist}
A universal reductive groupoid over $H$ exists.
\item\label{i:maincrit}
A transitive affine groupoid $K$ over $H$ is universally (resp.\ minimally) reductive over $H$ if and only if
the projection from $\Mod_H(X)$ to $\overline{\Mod_H(X)}$ composed with $\omega_{K/H}$
is an equivalence of categories (resp.\ fully faithful).
\end{enumerate}
\end{thm}

\begin{proof}
Consider the category $\sT_H$ whose objects are pairs $(\sD,T)$ with $\sD$ a semisimple Tannakian 
category over $k$ and $T$ a $k$\nd tensor functor from $\sD$ to $\Mod_H(X)$, where a morphism
from $(\sD,T)$ to $(\sD',T')$ is a tensor isomorphism class of $k$\nd tensor functors $E:\sD \to \sD'$ with
$T'E$ tensor isomorphic to $T$.
If $\sR_H$ is the category of reductive groupoids over $H$ up to conjugacy, we have a contravariant functor
\[
\Phi:\sR_H \to \sT_H
\]
which sends $K$ to $(\Mod_K(X),\omega_{K/H})$ and the class of $f$ to the class of $f^*$. 

By \eqref{e:fibrefacH}, $\Phi$ is essentially surjective.
By Lemma~\ref{l:Tf*} with $H = K'$, any morphism $\Phi(K) \to \Phi(K')$ in $\sT_H$ is the class of $j^*$
for some morphism $j:K' \to K$ of groupoids over $X$.
Hence $\Phi$ is fully faithful by Lemma~\ref{l:tensisoconj}.
Thus if $K$ is a reductive groupoid over $H$,
then $K$ is universally reductive $\iff$
$K$ is initial in $\sR_H$ $\iff$ $\Phi(K)$ is final in $\sT_H$.
By Theorem~\ref{t:splitting} with $\sC = \Mod_H(X)$, the category $\sT_H$ has a final object, and if $Q$ is the projection,
then $(\sD,T)$ is final $\iff$ $QT$ is an equivalence.
Hence \ref{i:mainexist} and the criterion of \ref{i:maincrit} for universality.

To prove the criterion of \ref{i:maincrit} for a minimality we may suppose that $K$ is reductive.
By \ref{i:mainexist}, there is a morphism $f:K' \to K$ over $H$ with $K'$ universally reductive over $H$.
Then $K$ minimally reductive over $H$
$\iff$ $f$ surjective $\iff$
$f^*$ fully faithful ($K'$ is reductive) $\iff$
$Q\omega_{K'/H}f^* = Q\omega_{K/H}$ fully faithful.
\end{proof}

\begin{exmp}
Suppose that $X = \Spec(k)$ and $H$ is an affine $k$\nd group.
Then a universal reductive groupoid over $H$ is the same
as a pro-reductive envelope of $H$ in the sense of \cite[19.3.2]{AndKah} or a universal 
reductive homomorphism with source $H$ in the sense of \cite[2.1.1]{O10}.
Theorem~\ref{t:main}\ref{i:mainexist} then reduces to \cite[19.3.1]{AndKah} or \cite[2.2.8]{O10}.
Similarly, the equivalence of \ref{i:unequivun} and \ref{i:unequivbij} in Corollary~\ref{c:unequiv} below
reduces to \cite[19.3.4]{AndKah} and Corollary~\ref{c:unequiv} to \cite[2.2.11]{O10},
Corollary~\ref{c:minequiv} reduces to the equivalence of (a), (b) and (f) in \cite[2.3.4]{O10},
Corollary~\ref{c:minhomconj}\ref{i:minhomconj} reduces to \cite[2.3.1]{O10} and 
Corollary~\ref{c:minhomconj}\ref{i:minhomconjcoh} is contained in the implication (a) $\implies$ (c)
of \cite[2.3.4]{O10}, Corollary~\ref{c:minsubconj} reduces to \cite[20.1.3.a)]{AndKah} or \cite[2.3.2]{O10},
Corollary~\ref{c:minscalext} to \cite[2.3.5]{O10}, and Corollary~\ref{c:unscalext} to \cite[19.6.1]{AndKah}
or \cite[2.2.16]{O10}.
\end{exmp}

If $X$ is $H$\nd connected, then by Proposition~\ref{p:unconjbij}\ref{i:unbij}
a necessary condition for a universal reductive
groupoid over $H$ to exist is that the Krull-Schmidt theorem should hold for representations of $H$.
Example~\ref{ex:local} below shows that this condition is not sufficient,
and why the condition on $H^0_H(X,\sO_X)$ in Theorem~\ref{t:main} is reasonable.

\begin{exmp}\label{ex:local}
Suppose that $X$ is a local $k$\nd scheme.
If a universal reductive groupoid $K_0$ over $X$ exists, then $K_0 = [X]$, because any vector bundle $\sV$
over $X$ is trivial, so that any morphism from $K_0$ to $\underline{\Iso}_X(\sV)$ over $X$ factors through $[X]$.
It follows that a universal reductive groupoid over $X$ exists if and only if both the following conditions hold:
\begin{enumerate}
\renewcommand{\theenumi}{(\alph{enumi})}
\item\label{i:redconst}
every reductive groupoid over $X$ is constant;
\item\label{i:pulltriv}
a principal bundle over $k$ under a reductive $k$\nd group is trivial if its pullback onto $X$ is trivial.
\end{enumerate}
Indeed the existence of a morphism over $X$ from $[X]$ to every
reductive groupoid over $X$ is equivalent to \ref{i:redconst}, and
the uniqueness up to conjugacy is equivalent to \ref{i:pulltriv}.
If there exists a non-constant finite \'etale scheme $X'$ over $X$,
then \ref{i:redconst} does not hold because $\underline{\Iso}_X(X')$ is not constant.
Suppose now that $X$ is the spectrum of a henselian local $k$\nd algebra with residue field $k'$.
Then $X$ can be given a structure of $k'$\nd scheme.
Thus if $k' \ne k$ is a Galois extension of $k$, then \ref{i:pulltriv} does not hold for the
principal $\Gal(k'/k)_{/k}$\nd bundle $\Spec(k')$ over $k$.
Also if $K$ is a reductive groupoid over $X$,
then its restriction $K'$ above $X \times_{k'} X$ is constant in $k'$\nd schemes by 
Theorem~\ref{t:main}\ref{i:mainexist} and \ref{i:redconst}. 
Thus \ref{i:redconst} and \ref{i:pulltriv} hold if and only if they hold with $X$
replaced by $\Spec(k')$.
Indeed if  \ref{i:redconst} holds with $\Spec(k')$ for $X$, then pulling back onto the closed point $x$
shows that any $K$ has a simply transitive $K$\nd scheme $Z$ with a section above $x$,
and since $Z$ has a structure of $K'$\nd scheme, it is constant in $k'$\nd schemes and 
hence has a section above $X$.
In particular if both $k$ and $k'$ are algebraically closed, then \ref{i:redconst} and \ref{i:pulltriv}
hold by Propositions~\ref{p:const} and \ref{p:printriv}.
\end{exmp}

Call an object of a $k$\nd tensor category \emph{trivial} if it is 
a direct sum of copies of $\I$.
Let $T:\sD \to \sD'$ be a faithful $k$\nd tensor functor between Tannakian categories over $k$.
Suppose that $\sD'$ is semisimple.
Then $T$ is fully faithful if and only if it preserves non-trivial indecomposables:
the ``only if'' is clear, and for the ``if'' note that by \eqref{e:dualiso},
$T$ is fully faithful provided that it induces a bijection
\[
\Hom_\sD(\I,M) \iso \Hom_{\sD'}(\I,T(M))
\]
for $M$ indecomposable in $\sD$.

\begin{cor}\label{c:unequiv}
Let $K$ be a transitive affine groupoid over $H$,
and suppose that $H^0_H(X,\sO_X)$ is as in Theorem~\textup{\ref{t:main}}. 
Then the following conditions are equivalent:
\begin{enumerate}
\renewcommand{\theenumi}{(\alph{enumi})}
\item\label{i:unequivun}
$K$ is universally reductive over $H$;
\item\label{i:unequivbij}
$K$ has a reductive subgroupoid over $H$, and
$\omega_{K/H}$ induces a bijection on isomorphism classes of representations;
\item\label{i:unequivmin}
$K$ is minimally reductive over $H$, and every representation of $H$ is a direct summand of a representation of $K$.
\end{enumerate}
\end{cor}

\begin{proof}
\ref{i:unequivun} $\implies$ \ref{i:unequivbij} is a particular case of Proposition~\ref{p:unconjbij}\ref{i:unbij}.

\ref{i:unequivbij} $\implies$ \ref{i:unequivmin}:
Assume \ref{i:unequivbij}.
Then if $Q:\Mod_H(X) \to \overline{\Mod_H(X)}$ is the projection,
$Q\omega_{K/H}$ is faithful.
Since $Q\omega_{K/H}$ preserves non-trivial indecomposables by Proposition~\ref{p:posfincons}\ref{i:poscons},
it is fully faithful by the above criterion.
Thus \ref{i:unequivmin} holds by Theorem~\ref{t:main}\ref{i:maincrit}.

\ref{i:unequivmin} $\implies$ \ref{i:unequivun} follows from Theorem~\ref{t:main}\ref{i:maincrit}. 
\end{proof}

Suppose that $H^0_H(X,\sO_X)$ is a local $k$\nd algebra with residue field $k$.
Then given an $H$\nd module $\sV$, we have a pairing
\begin{equation}\label{e:Hpairing}
\Hom_H(\sV,\sO_X) \otimes_k H^0_H(X,\sV) \to H^0_H(X,\sO_X) \to k
\end{equation}
where the first arrow sends $u \otimes v$ to $u(v)$
and the second arrow is the projection onto the residue field.
We write 
\[
{}^\mathrm{rad}H^0_H(X,\sV) \subset H^0_H(X,\sV)
\]
for the kernel on the right of this pairing.
It is an $H^0_H(X,\sO_X)$\nd submodule of $H^0_H(X,\sV)$ containing $\mathfrak{m}H^0_H(X,\sV)$,
where $\mathfrak{m}$ is the maximal ideal of $H^0_H(X,\sO_X)$.
To any morphism $\sV \to \sV'$ of $H$\nd modules
is associated a commutative square
\begin{equation}\label{e:Hpairingsquare}
\begin{gathered}
\xymatrix{
\Hom_H(\sV',\sO_X) \otimes_k H^0_H(X,\sV)  \ar[d] \ar[r] &  \Hom_H(\sV,\sO_X) \otimes_k H^0_H(X,\sV) \ar[d] \\
\Hom_H(\sV',\sO_X) \otimes_k H^0_H(X,\sV')  \ar[r] &         k
}
\end{gathered}
\end{equation}
of $k$\nd vector spaces.
Thus ${}^\mathrm{rad}H^0_H(X,-)$ is a subfunctor of $H^0_H(X,-)$ on $H$\nd modules.
If $\sV$ is a representation of $H$, then 
${}^\mathrm{rad}H^0_H(X,\sV)$ is the subspace of $\Hom_H(\sO_X,\sV) = H^0_H(X,\sV)$ annihilated 
by the projection from $\Mod_H(X)$ to $\overline{\Mod_H(X)}$.

\begin{cor}\label{c:minequiv}
Let $K$ be a transitive affine groupoid over $H$,
and suppose that $H^0_H(X,\sO_X)$ is as in Theorem~\textup{\ref{t:main}}.
Then the following conditions are equivalent:
\begin{enumerate}
\renewcommand{\theenumi}{(\alph{enumi})}
\item\label{i:minequivmin}
$K$ is minimally reductive over $H$;
\item\label{i:minequivrad}
$H^0_H(X,\sV) = {}^\mathrm{rad}H^0_H(X,\sV) \oplus H^0_K(X,\sV)$ for every representation $\sV$ of $K$;
\item\label{i:minequivindec}
$K$ has a reductive subgroupoid over $H$, and every non-trivial indecomposable representation of $K$ 
is a non-trivial indecomposable representation of $H$.
\end{enumerate}
\end{cor}

\begin{proof}
If $Q$ is the projection onto $\overline{\Mod_H(X)}$,
each of \ref{i:minequivmin}, \ref{i:minequivrad}, and \ref{i:minequivindec} is equivalent to the condition
that $Q\omega_{K/H}$ be fully faithful: for \ref{i:minequivmin} by Theorem~\ref{t:main}\ref{i:maincrit},
for \ref{i:minequivrad} because \ref{i:minequivrad} is equivalent to requiring that $Q\omega_{K/H}$ induce
an isomorphism on the $\Hom_K(\sO_X,\sV)$, and for \ref{i:minequivindec} by the criterion preceding 
Corollary~\ref{c:unequiv}.  
\end{proof}

\begin{exmp}\label{ex:tenspower}
Corollary~\ref{c:minequiv} can be used to decompose the tensor powers of a vector bundle
into its indecomposable summands by reducing to a question about representations of reductive groups.
The following is a simple example, which will be used in Section~\ref{s:curves1}.
Suppose that $k$ is algebraically closed,  that $H^0_H(X,\sO_X)$ is as in Theorem~\ref{t:main}, 
and that $X$ is geometrically simply $H$\nd connected.
Then every symmetric power of an indecomposable representation $\sV$ of $H$ of rank $2$ is indecomposable.
To see this, let $K$ be a minimal reductive subgroupoid over $H$ of $\underline{\Iso}_X(\sV)$.
By Lemma~\ref{l:funact} with $K_{\mathrm{\acute{e}t}}$ for $K$ and Corollary~\ref{c:minequiv}, 
we may replace $H$ by $K$ and hence suppose that $H$ is a reductive groupoid.
By Lemma~\ref{l:prereppull}
and Proposition~\ref{p:const}, we may suppose further that $H$ is constant, and hence finally that $X = \Spec(k)$
and $H$ is a reductive $k$\nd subgroup of $GL(\sV)$.
Then $H$ is non-commutative, because otherwise $\sV$ would be decomposable.
Also $H$ is connected, because any finite quotient $\ne 1$ would define a non-constant finite 
\'etale $H$\nd scheme.
Thus $H$ contains $SL(\sV)$, and the result follows.
\end{exmp}

\begin{cor}\label{c:minhomconj}
Let $K$ be a minimal reductive groupoid over $H$,
and suppose that $H^0_H(X,\sO_X)$ is as in Theorem~\textup{\ref{t:main}}.
\begin{enumerate}
\item\label{i:minhomconj}
Any two morphisms over $H$ from $K$ to a transitive affine groupoid over $H$ are conjugate.
\item\label{i:minhomconjcoh}
Any two representations of $K$ which are $H$\nd isomorphic are $K$\nd isomorphic.
\end{enumerate}
\end{cor}

\begin{proof}
A universal reductive groupoid $K'$ over $H$ exists by Theorem~\ref{t:main}\ref{i:mainexist},
and there is a surjective morphism $h:K' \to K$ over $H$.
After composing with $h$, it suffices to apply Proposition~\ref{p:unconjbij}.
\end{proof}

\begin{exmp}
Suppose that $k$ is algebraically closed.
By a theorem of Malcev \cite[\S 3 Theorem~2]{Mal}, the map from $\Hom^\mathrm{conj}_k(SL_2,G)$ to 
$\Hom^\mathrm{conj}_k(\bG_m,G)$
defined by an embedding $\bG_m \to SL_2$ is injective for any affine $k$\nd group $G$ of finite type.
That the same holds for any affine $k$\nd group $G$ can be seen using Proposition~\ref{p:printriv}.
Thus if $X = \Spec(k)$, $H = \bG_m$ and $K = SL_2$, then 
both \ref{i:minhomconj} and (taking $G = GL_n$) 
\ref{i:minhomconjcoh} of Corollary~\ref{c:minhomconj} hold, but $K$ is not minimally reductive over $H$.
\end{exmp}

\begin{cor}\label{c:minsubconj}
let $K$ be a transitive affine groupoid over $H$,
and suppose that $H^0_H(X,\sO_X)$ is as in Theorem~\textup{\ref{t:main}}.
Then any two minimal reductive subgroupoids of $K$ over $H$ are conjugate by an element of $H^0_H(X,K^\mathrm{diag})$.
\end{cor}

\begin{proof}
By Theorem~\ref{t:main}\ref{i:mainexist}, a universal reductive groupoid $K'$ over $H$ exists.
Any minimal reductive subgroupoid of $K$ over $H$ is the image of a morphism $K' \to K$ over $H$.
By Proposition~\ref{p:unconjbij}\ref{i:unconj},
any two morphisms $K' \to K$ over $H$ are conjugate by a cross section of $K^\mathrm{diag}$, which necessarily
lies in $H^0_H(X,K^\mathrm{diag})$.
\end{proof}

\begin{exmp}
That some condition on $X$ and $H$ is necessary for
Corollaries~\ref{c:minhomconj} and \ref{c:minsubconj} to hold can be seen as follows.
Suppose that $H = X$, that $\Pic(X) = 0$, and that there exists a non-constant
principal $\mu_p$\nd bundle $P$ over $X$ for some prime $p$ (e.g.\ $X$ is the punctured affine line).
Then the embedding and the trivial $k$\nd homomorphisms $\mu_p \to \bG_m$ extend to morphisms
$(\mu_p,P) \to (\bG_m,\bG_m \times_k X)$,   
and $K = \underline{\Iso}_{\mu_p}(P)$ is minimally reductive.
Applying $\underline{\Iso}_-(-)$ gives morphisms $f_1$ and $f_2$ from $K$ to 
$K' = \underline{\Iso}_{\bG_m}(\bG_m \times_k X)$
where $f_1$ is an embedding and $f_2$ factors through $[X]$.
If the hypothesis on $H^0_H(X,\sO_X)$ is dropped, then the $f_i$ give a counterexample to 
Corollary~\ref{c:minhomconj}\ref{i:minhomconj}, the $f_i(K)$ to 
Corollary~\ref{c:minsubconj} with $K'$ for $K$, and the $f_i{}\!^*\sV$ for any faithful
representation $\sV$ of $K'$ to Corollary~\ref{c:minhomconj}\ref{i:minhomconjcoh}.
\end{exmp}

For every $H$\nd module $\sV$ we have a pairing
\[
\Hom_H(\sV,\sO_X) \otimes_k H^0_H(X,\sV) \to H^0_H(X,\sO_X)
\]
which sends $u \otimes v$ to $u(v)$.
Its image is an ideal 
\[
I_H(\sV) \subset H^0_H(X,\sO_X).
\]
We have $I_H(\sV) + I_H(\sW) = I_H(\sV \oplus \sW)$ for every $\sV$ and $\sW$.

Let $K$ be a reductive groupoid over $H$.
We write 
\[
I_H(K) = \bigcup_{\sV \in \Mod_K(X), \; H^0_K(X,\sV) = 0}I_H(\sV) \; \subset H^0_H(X,\sO_X)
\]
for the filtered union of the $I_H(\sV)$ where $\sV$ runs over those representations of $K$
with no non-zero trivial direct summand.
It is an ideal of $H^0_H(X,\sO_X)$.

If $M$ and $N$ are objects in a pseudo-abelian category, then the composition
\begin{equation}\label{e:HomMN}
\Hom(M,N) \otimes \Hom(N,M) \to \End(N)
\end{equation}
is surjective if and only if $N$ is a direct summand of $M^n$ for some $n$.

\begin{lem}\label{l:Kideal}
Let $K$ be a reductive groupoid over $H$.
Then the following conditions are equivalent:
\begin{enumerate}
\renewcommand{\theenumi}{(\alph{enumi})}
\item
$I_H(K) \ne H^0_H(X,\sO_X)$;
\item
for every representation of $K$ with no non-zero trivial direct summand, 
the underlying representation of $H$ has no non-zero trivial direct summand.
\end{enumerate}
\end{lem}

\begin{proof}
Taking for $M$ in \eqref{e:HomMN} the representation $\sV$ of $H$ and for $N$ the trivial representation $\sO_X$ of $H$ shows that
$I_H(\sV) = H^0_H(X,\sO_X)$ if and only if $\sV^n$ has the direct summand $\sO_X$ for some $n$.  
The result follows.
\end{proof}

Let $k'$ be an extension of $k$ contained in $H^0(X,\sO_X)$.
Then we may regard $X$ as a $k'$\nd scheme.
A pregroupoid $H$ over $X$ is then a pregroupoid in $k'$\nd schemes if and only if $k'$ lies in $H^0_H(X,\sO_X)$.
Suppose that this condition holds, and let $K$ be a reductive groupoid over $H$.
The restriction $K'$ of $K$ above the subscheme $X \times_{k'} X$ of $X \times_k X$ is a reductive groupoid
in $k'$\nd schemes over $H$.
We have
\begin{equation}\label{e:IHKK'}
I_H(K) = I_H(K').
\end{equation}
Indeed $K'$ coincides, as a groupoid in $k'$\nd schemes over $X$, with the pullback of $K_{k'}$ along the canonical morphism
of $k'$\nd schemes $X \to X_{k'}$,
and since $X \to X_{k'}$ defines by Lemma~\ref{l:prereppull} an equivalence from 
$\Mod_{K_{k'}}(X_{k'})$ to $\Mod_{K'}(X)$,
the representations of $K'$ with no non-zero trivial direct summand 
are by Lemma~\ref{l:extquot} the direct summands of representations of $K$ with no non-zero trivial direct summand.

Let $k'$ be an extension of $k$.
Suppose that either $k'$ is finite over $k$ or that $X$ is quasi-compact.
By \eqref{e:Homextiso}, the embedding of $H^0_H(X,\sO_X)$ induces an isomorphism
\begin{equation}\label{e:H0Hscalext}
H^0_H(X,\sO_X)_{k'} \iso H^0_{H_{k'}}(X_{k'},\sO_{X_{k'}}),
\end{equation}
and $I_{H_{k'}}(\sV_{k'})$ is the image of $I_H(\sV)_{k'}$ under \eqref{e:H0Hscalext} for every
representation $\sV$ of $H$.

\begin{lem}\label{l:IHKscalext}
Let $K$ be a reductive groupoid over $H$ and $k'$ be an extension of $k$.
Suppose that ether $k'$ is finite over $k$ or that $X$ is quasi-compact.
Then $I_{H_{k'}}(K_{k'})$ is the image of $I_H(K)_{k'}$ under \eqref{e:H0Hscalext}.
\end{lem}

\begin{proof}
By Lemma~\ref{l:extquot}, every representation of $K_{k'}$ is a direct summand of $\sV_{k'}$
for some representation $\sV$ of $K$.
\end{proof}

\begin{cor}\label{c:IHKmin}
Let $K$ be a reductive groupoid over $H$.
If $I_H(K) \ne H^0_H(X,\sO_X)$, then $K$ is minimally reductive over $H$.
The converse holds when $H^0_H(X,\sO_X)$ is as in Theorem~\textnormal{\ref{t:main}}.
\end{cor}

\begin{proof}
Suppose that $K$ contains a reductive subgroupoid $K' \ne K$ over $H$.
Then by the remarks following Proposition~\ref{p:unconjbij},
the functor from $\Mod_K(X)$ to $\Mod_{K'}(X)$ defined by the embedding is not fully faithful.
Thus by \eqref{e:dualiso} there is a representation $\sV$ of $K$ with no non-zero trivial direct
summand such that the underlying representation of $K'$, and hence the underlying representation of $H$,
has a non-zero trivial direct summand.
This implies $I_H(K) = H^0_H(X,\sO_X)$, by Lemma~\ref{l:Kideal}.

Suppose that $H^0_H(X,\sO_X)$ is as in Theorem~\textnormal{\ref{t:main}} and that $K$ is minimally reductive over $H$.
Since ${}^\mathrm{rad}H^0_H(X,\sV)$ is the kernel on the right of \eqref{e:Hpairing},
it follows from Corollary~\ref{c:minequiv}\ref{i:minequivrad} that
$I_H(\sV) \ne H^0_H(X,\sO_X)$ for every representation $\sV$ of $K$
with no non-zero trivial direct summand.
Thus $I_H(K) \ne H^0_H(X,\sO_X)$.
\end{proof}

\begin{exmp}
Suppose that $X$ is affine.
Then $I_X(\sV) = H^0(X,\sO_X)$ for every $\sV$ in $\Mod(X)$ which is nowhere $0$.
Thus $I_X(K) = H^0(X,\sO_X)$ for every reductive groupoid $K \ne [X]$ over $X$.
Suppose that a non-constant reductive groupoid over $X$ exists,
as for example when there is a non-free $\sV$ in $\Mod(X)$ of constant rank.
Then a minimal reductive groupoid $\ne [X]$ over $X$ exists.
This shows that some condition on $X$ and $H$ is necessary for the final statement of Corollary~\ref{c:IHKmin}.
\end{exmp}

\begin{rem}\label{r:hensres}
Suppose that $H^0_H(X,\sO_X)$ is a henselian local $k$\nd algebra, and let $k'$ be a 
maximal extension of $k$ in $H^0_H(X,\sO_X)$.
Then $k'$ is a lifting of the residue field.
Let $K$ be a reductive groupoid over $H$ and $K'$ be the restriction of $K$ above the subscheme 
$X \times_{k'} X$ of $X \times_k X$.
We may regard $K'$ as a reductive groupoid in $k'$\nd schemes over $H$.
Then by \eqref{e:IHKK'}
and Corollary~\ref{c:IHKmin}, the conditions
\begin{enumerate}
\renewcommand{\theenumi}{(\alph{enumi})}
\item\label{i:hensresIH}
$I_H(K) \ne H^0_H(X,\sO_X)$
\item\label{i:hensresmin}
$K'$ is minimally reductive over $H$
\end{enumerate}
are equivalent.
These conditions are in general strictly stronger than the one that $K$ be minimally reductive over $H$.
The following are two examples of this, with $X = \Spec(k')$ and $H = X$.
Suppose first that $k$ is algebraically closed but that $k'$ is not.
Then there is a non-constant principal $G$\nd bundle $P$ over $X$ for some reductive $k$\nd group $G$,
with for example $P$ the spectrum of a non-trivial Galois extension of
$k'$ and $G$ the Galois group.
The reductive groupoid $\underline{\Iso}_G(P)$ over $X$ contains a minimal reductive subgroupoid $K \ne [X]$.
On the other hand reductive groupoids over the $k'$\nd scheme $X$ are just reductive $k'$\nd groups,
so that $K' \ne 1$ is not minimally reductive.
For the second example, suppose that $k'$ is an algebraic closure $\overline{k}$ of $k$,
and that the Brauer group of $k$ has an element $\alpha$ of order $n \ne 1$.
Then $\alpha$ is an element of the subgroup $H^2(\Gal(\overline{k}/k),\mu_n(\overline{k}))$ of the Brauer group.
If $E$ is an extension of $\Gal(\overline{k}/k)$ by $\mu_n(\overline{k})$ with class $\alpha$,
then $(\mu_n{}_{\overline{k}},E)$ is a Galois extended $\overline{k}$\nd group.
The groupoid $K$ over $X$ corresponding to it by Proposition~\ref{p:grpdgalequ} is minimally reductive, 
but $K'$, and by Lemma~\ref{l:finneutr} even $K_{k'}$, is not.
\end{rem}

We may regard $H$ as a groupoid in $S$\nd schemes, where
\[
S = \Spec(A) = \Spec(H^0_H(X,\sO_X)).
\]
If $S'$ is a scheme over $S$ and $X' = X \times_S S'$ is the constant $H$\nd scheme defined by $S'$,
then we have a canonical isomorphism
\[
H' = H \times_S S' \iso H \times_X X'
\]
of pregroupoids over $X'$.
Suppose that $S' = \Spec(A')$ is affine and flat over $S$, and that $X$ is quasi-compact and quasi-separated
and $H_{[1]}$ is quasi-compact.
Then for any representation $\sV$ of $H$ with pullback $\sV'$ onto $X'$, the canonical homomorphism
from $H^0_H(X,\sV)$ to $H^0_{H'}(X',\sV')$ extends to an isomorphism
\[
H^0_H(X,\sV) \otimes_A A' \iso H^0_{H'}(X',\sV').
\]
This follows from the fact that $H^0_H(X,\sV)$ for example is the equaliser of two morphisms from
$H^0(X,\sV)$ to $H^0(H_{[1]},d_0{}\!^*\sV)$, while the base change homomorphism along $S' \to S$ is
an isomorphism for $X$ and $\sV$ and injective for $H_{[1]}$ and $d_0{}\!^*\sV$.
In particular we have a canonical isomorphism
\[
A' \iso H^0_{H'}(X',\sO_{X'}),
\]
which induces for every reductive groupoid $K$ over $H$ an isomorphism
\[
I_H(K) \otimes_A A' \iso I_{H'}(K \times_{[X]} [X']).
\]
This can be used as follows to describe the support of the subscheme of $S$ defined by $I_H(K)$.
Let $s$ be a point of $S$. 
Write $S_{(s)}$ for the spectrum of the henselisation of $A$ at $s$ and $X_{(s)}$ for $X \times_S S_{(s)}$, 
and let $k_s$ be a maximal extension of $k$ in this henselisation.
Then by the equivalence of \ref{i:hensresIH} and \ref{i:hensresmin} of Remark~\ref{r:hensres}, 
$s$ lies in the support of $\Spec(A/I_H(K))$ if and only the groupoid in $k_s$\nd schemes
over
\[
H \times_S S_{(s)} = H \times_X X_{(s)}
\]
given by restricting above $X_{(s)} \times_{k_s} X_{(s)}$ the pullback of $K$ along  $X_{(s)} \to X$
is minimally reductive.

\begin{cor}\label{c:minscalext}
Let $K$ be a reductive groupoid over $H$ and $k'$ be an extension of $k$.
Suppose that $H^0_H(X,\sO_X)$ is as in Theorem~\textup{\ref{t:main}}, 
and that either $k'$ is finite over $k$ or that $X$ is quasi-compact.
Then $K$ is minimally reductive over $H$ if and only if $K_{k'}$ is minimally reductive over $H_{k'}$.
\end{cor}

\begin{proof}
The ``if'' is immediate.
The ``only if'' follows from Lemma~\ref{l:IHKscalext} and Corollary~\ref{c:IHKmin}.
\end{proof}

\begin{exmp}
That some condition on $X$ and $H$ is needed in Corollary~\ref{c:minscalext} can be seen by the example
of Remark~\ref{r:hensres} with $H = X$ and $X$ the spectrum of an algebraic closure of $k$. 
The following is an example with $H = X$ and $X$ geometrically connected.
Suppose that $X$ is the affine plane, and that $k$ is not algebraically closed.
Then by \cite[Theorem~B]{Rag89}, there exists a non-constant principal $G_1$\nd bundle $P_1$ over $X$
for some reductive $k$\nd group $G_1$ of finite type.
The reductive groupoid $\underline{\Iso}_{G_1}(P_1)$ contains a minimal reductive subgroupoid $K \ne [X]$.
By the equivalence \eqref{e:GPpequiv}, $K$ is of the form $\underline{\Iso}_G(P)$ for some 
reductive $k$\nd group $G$ of finite type. 
If $\overline{k}$ is an algebraically closed extension of $k$,
then $P_{\overline{k}}$ is trivial, by \cite{Rag78}.
Thus $K_{\overline{k}}$ is not minimally reductive.
\end{exmp}

Note that the condition on $H^0_H(X,\sO_X)$ in Theorem~\ref{t:main} is not in general preserved by extension of scalars.
However by \eqref{e:H0Hscalext} it is preserved by algebraic extension of scalars, provided that either
the extension is finite or $X$ is quasi-compact.

\begin{cor}\label{c:unscalext}
Let $K$ be a reductive groupoid over $H$ and $k'$ be an \emph{algebraic} extension of $k$.
Suppose that $H^0_H(X,\sO_X)$ is as in Theorem~\textup{\ref{t:main}}, 
and that either of the following conditions holds:
\begin{enumerate}
\renewcommand{\theenumi}{(\alph{enumi})}
\item\label{i:unscalextfin}
$k'$ is finite over $k$;
\item\label{i:unscalextalg}
$X$ is quasi-compact and quasi-separated and $H_{[1]}$
is quasi-compact.
\end{enumerate}
Then $K$ is universally reductive over $H$ if and only if $K_{k'}$ is universally reductive over $H_{k'}$.
\end{cor}

\begin{proof}
For the ``only if'', it suffices by Corollaries~\ref{c:unequiv} and \ref{c:minscalext}
to show that every representation
of $H$ is a direct summand of a representation of $K$ only if every representation of $H_{k'}$ is a 
direct summand of a representation of $K_{k'}$.
This follows from Lemma~\ref{l:Rsummand} with $X' = X_{k'}$
if \ref{i:unscalextfin} holds, together with Lemma~\ref{l:HRrepfp} if \ref{i:unscalextalg} holds.

By Theorem~\ref{t:main}, a universal reductive groupoid $K_0$ over $H$ exists, and by the ``only if'',
$(K_0)_{k'}$ is universally reductive over $H_{k'}$.
The ``if'' follows,
because $K_0 \to K$ is an isomorphism if $(K_0)_{k'} \to K_{k'}$ is.
\end{proof}

\section{Applications to principal bundles}\label{s:appprinbun}

\emph{In this section $k$ is a field of characteristic $0$, $\overline{k}$ is an algebraically closed extension of $k$,
$X$ is a non-empty $k$\nd scheme, and $H$ is a pregroupoid over $X$.}

\medskip

In this section we derive from Theorem~\ref{t:main} and its corollaries results applicable to principal bundles. 
These follow almost immediately from the corresponding results for groupoids, by the equivalences of Lemma~\ref{l:IsoHequivpt}
and Corollary~\ref{c:IsoHequivac}.

Suppose that $X$ has a $k$\nd point $x$, and let $G$ be an affine $k$\nd group.
If $Q$ is a principal $G$\nd bundle over $k$, then $Q$ is a principal 
$(G',G)$\nd bundle with $G'$ the inner form $\underline{\Iso}_G(Q)$ of $G$, and $- \times_k^{G'} Q$ defines an equivalence
from the category of principal $(H,G')$\nd bundles with fibre above $x$ isomorphic to $G'$ to the category of principal
$(H,G)$\nd bundles with fibre above $x$ isomorphic to $Q$.
To describe principal $(H,G)$\nd bundles we thus reduce, modulo a description of principal $G$\nd bundles over $k$,
to the case of those which are trivial above $x$.  
We write 
\[
\widetilde{H}^1_H(X,x,G)
\]
for the pointed subset of $H^1_H(X,G)$ consisting of the classes of those principal $(H,G)$\nd bundles
which have a $k$\nd point above $x$.
We have a short exact sequence of pointed sets
\[
1 \to \widetilde{H}^1_H(X,x,G) \to H^1_H(X,G) \to H^1(k,G) \to 1
\]
which is natural in $H$ and $G$, where the second arrow,
defined by taking the fibre at $x$,
is surjective because above the class of a principal $G$\nd bundle over $k$ lies the
class of a constant $(H,G)$\nd bundle.

In the general case, where $X$ need not have a $k$\nd point, we consider principal $(H,F)$\nd bundles over $X$ 
for $F$ a transitive affine groupoid over the algebraically closed extension $\overline{k}$ of $k$.
By Lemma~\ref{l:gpdprinHGG'}, principal $(H,G)$\nd bundles for $G$ an affine $k$\nd group are included as the case where $F$ is the constant groupoid $G \times_k {[\overline{k}]}$.
Affine $k$\nd groups which are inner forms of one another are then treated together, because their associated groupoids
over $\overline{k}$ are isomorphic.
However subgroupoids of constant groupoids over $\overline{k}$ need not be constant,
and the universal groupoids over $\overline{k}$ that arise need not be constant.
If $X$ has a $k$\nd point, a principal $F$\nd bundle over $X$ can exist only for $F$ constant, 
but even so the functor from affine $k$\nd groups up to conjugacy to affine groupoids over $\overline{k}$ up to conjugacy,
while faithful, need not be full.

When $k$ is algebraically closed, we may take $\overline{k} = k$, 
so that the results for groupoids over $\overline{k}$ reduce to results for $k$\nd groups.

The groupoid $\underline{\Iso}_G(P)$ over $X$ is reductive if and only if the affine $k$\nd group $G$ is reductive,
and similarly $\underline{\Iso}_F(P)$ is reductive  if and only if the transitive affine groupoid $F$ over an extension of $k$ is reductive.

\begin{lem}\label{l:cohreppt}
Let $x$ be a $k$\nd point of $X$.
\begin{enumerate}
\item\label{i:unrepablept}
The functor $\widetilde{H}^1_H(X,x,-)$ on the category of reductive $k$\nd groups up to conjugacy is representable
if and only if a universal reductive groupoid over $H$ exists.
\item\label{i:unreppt}
The functor of \textnormal{\ref{i:unrepablept}} is represented by $G$ and the class of $P$ if and only if $\underline{\Iso}_G(P)$
is universally reductive over $H$.
\end{enumerate}
\end{lem}

\begin{proof}
Immediate from Lemma~\ref{l:IsoHequivpt}.
\end{proof}

\begin{lem}\label{l:cohrepgpd}
\begin{enumerate}
\item\label{i:unrepablegpd}
The functor $H^1_H(X,-)$ on the category of reductive groupoids over $\overline{k}$ up to conjugacy is representable
if and only if a universal reductive groupoid over $H$ exists.
\item\label{i:unrepgpd}
The functor of \textnormal{\ref{i:unrepablegpd}} is represented by $F$ and the class of $P$ if and only if 
$\underline{\Iso}_F(P)$ is universally reductive over $H$.
\end{enumerate}
\end{lem}

\begin{proof}
Immediate from Corollary~\ref{c:IsoHequivac}
\end{proof}

\begin{cor}\label{c:cohomologyrep}
Suppose that $H^0_H(X,\sO_X)$ is as in Theorem~\textup{\ref{t:main}}.
\begin{enumerate}
\item\label{i:ptcohomologyrep}
For every $k$-point $x$ of $X$, the functor $\widetilde{H}^1_H(X,x,-)$ on the category of reductive $k$-groups up to 
conjugacy is representable.
\item\label{i:gpdcohomologyrep}
The functor $H^1_H(X,-)$ on the category of reductive groupoids over $\overline{k}$ up to conjugacy is representable.
\end{enumerate}
\end{cor}

\begin{proof}
Immediate from Theorem~\ref{t:main}\ref{i:mainexist} and Lemmas~\ref{l:cohreppt}\ref{i:unrepablept} and 
\ref{l:cohrepgpd}\ref{i:unrepablegpd}.
\end{proof}

\begin{exmp}\label{ex:affline}
The following is an example with $H = X$ and $X$ geometrically connected and of finite type over $k$
where a universal reductive groupoid over $H$ exists but the condition on $H^0_H(X,\sO_X)$ in Theorem~\ref{t:main}
does not hold. 
Suppose that $k$ is algebraically closed and that $X$ is the affine line.
Then \cite[4.3 and 4.4]{Ram83} for $G$ an affine $k$\nd group of finite
type, every principal $G$\nd bundle over $X$ is trivial.
A Zorn's lemma argument similar to the proof of Proposition~\ref{p:printriv} shows that the same
is true for an arbitrary affine $k$\nd group $G$.
The functor $H^1(X,-)$ on reductive $k$\nd groups up to conjugacy is thus represented by the trivial $k$\nd group $1$.
Hence $[X]$ is a universal reductive groupoid over $X$, by Lemma~\ref{l:cohrepgpd}\ref{i:unrepgpd} with $\overline{k} = k$.
\end{exmp}

\begin{exmp}
The following is an example, with $H = X$, where the $k$\nd algebra $H^0_H(X,\sO_X)$ is the same as in Example~\ref{ex:affline}
but a universal reductive groupoid over $H$ does not exist.
Suppose that $k$ is algebraically closed.
Let $Z$ be a non-empty reduced connected projective $k$\nd scheme with a fixed-point-free involution $e$, 
such as an elliptic curve with $e$ translation by a point of order $2$.
Now take
\[
X = (Z \times_k \bG_m)/(\Z/2),
\]
where the action of $\Z/2$ is the product of the action by $e$ on $Z$ and by the inverse involution on $\bG_m$.
Then
\[
H^0(X,\sO_X) = H^0(\bG_m,\sO_{\bG_m})/(\Z/2)
\]
is a polynomial $k$\nd algebra in one indeterminate, as in Example~\ref{ex:affline}.
We show that there exist a reductive $k$\nd group $G$, a principal $G$\nd bundle $P$, and for each $n \ge 1$ 
a finite \'etale $k$\nd subgroup $G_n$ of $G$ of rank $2n$,
such that $P$ has a connected principal $G_n$\nd subbundle $P_n$ for each $n$. 
Since the structure groups of such $P_n$ cannot be reduced, it will follow that $H^1(X,-)$ on reductive 
$k$\nd groups up to conjugacy is not representable,
so that by Lemma~\ref{l:cohrepgpd}\ref{i:unrepablegpd} a universal reductive groupoid over $X$ does not exist. 
For $G$ take $\bG_m \rtimes_k \Z/2$ with $\Z/2$ acting on $\bG_m$ 
as the inverse involution, and for $G_n$ the $k$\nd subgroup $\mu_n \rtimes_k \Z/2$.
For $P$ take $Z \times_k \bG_m \times_k \bG_m$, regarded as a scheme over $X$ by projection onto the first two factors
and then onto $X$.
The $k$\nd subgroup $\bG_m$ of $G$ acts on $P$ by translation on the last factor $\bG_m$, and the $k$\nd subgroup $\Z/2$
as the involution of $P$ that sends $(z,g,g')$ to $(e(z),g^{-1},g'{}^{-1})$.
The $G_n$\nd subbundle $P_n$ of $P$ is that with points $(z,g^n,g)$.
\end{exmp}

Let $G$ be an affine $k$\nd group.
By Lemma~\ref{l:IsoHequivpt}, if a principal $(H,G)$\nd bundle with a $k$\nd point $z$
has a principal $(H,G')$\nd subbundle for a reductive $k$\nd subgroup $G' \ne G$ of $G$, 
it also has for some reductive $k$\nd subgroup $G'' \ne G$ of $G$ a principal $(H,G'')$\nd subbundle containing $z$.

\begin{lem}\label{l:minbgpt}
Let $G$ be a reductive $k$\nd group and 
$P$ be a principal $(H,G)$\nd bundle which has a $k$\nd point.
Then $\underline{\Iso}_G(P)$ is minimally reductive over $H$ if and only if
$P$ has no principal $(H,G')$\nd subbundle for any reductive $k$\nd subgroup $G' \ne G$ of $G$.
\end{lem}

\begin{proof}
Immediate from Lemma~\ref{l:IsoHequivpt}.
\end{proof}

Unless $k$ is algebraically closed,
Lemma~\ref{l:minbgpt} does not hold without the hypothesis that $P$ has a $k$\nd point,
even if $H = X = \Spec(k)$.

\begin{lem}\label{l:minbggpd}
Let $F$ be a reductive groupoid over $\overline{k}$ and $P$ be a principal $(H,F)$\nd bundle.
Then $\underline{\Iso}_F(P)$ is minimally reductive over $H$ if and only if $P$
has no principal $(H,F')$\nd subbundle for any reductive subgroupoid $F' \ne F$ over $\overline{k}$ of $F$.
\end{lem}

\begin{proof}
Immediate from Corollary~\ref{c:IsoHequivac}.
\end{proof}

\begin{cor}\label{c:minhomconjcohpt}
Let $G$ be a reductive $k$\nd group 
and $P$ be a principal $(H,G)$-bundle
with no principal $(H,G_0)$\nd subbundle for any reductive $k$\nd subgroup $G_0 \ne G$ of $G$.
Suppose that $H^0_H(X,\sO_X)$ is as in Theorem~\textup{\ref{t:main}},
and that $P$ has a $k$\nd point.
Then the push forwards of $P$ along two $k$\nd homomorphisms $h_1$ and $h_2$ from $G$ to an affine $k$\nd group $G'$
are $(H,G')$\nd isomorphic if and only if $h_1$ and $h_2$ are conjugate.
\end{cor}

\begin{proof}
Immediate from Lemmas~\ref{l:IsoHequivpt} and \ref{l:minbgpt} and Corollary~\ref{c:minhomconj}\ref{i:minhomconj}.
\end{proof}

\begin{cor}

Let $F$ be a reductive groupoid over $\overline{k}$, 
and $P$ be a principal $(H,F)$\nd bundle with no principal $(H,F_0)$\nd subbundle for any reductive subgroupoid $F_0 \ne F$
over $\overline{k}$ of $F$.
Suppose that $H^0_H(X,\sO_X)$ is as in Theorem~\textup{\ref{t:main}}.
Then the push forwards of $P$ along two morphisms $h_1$ and $h_2$ from $F$ to a transitive affine groupoid $F'$
over $\overline{k}$ are $(H,F')$\nd isomorphic if and only if $h_1$ and $h_2$ are conjugate.
\end{cor}

\begin{proof}
Immediate from Lemma~\ref{l:minbggpd} and Corollaries~\ref{c:IsoHequivac} and \ref{c:minhomconj}\ref{i:minhomconj}.
\end{proof}

Let $G$ be an affine $k$\nd group and $P$ be a principal $(H,G)$\nd bundle, 
Recall that \eqref{e:assocvecbun} with $S = \Spec(k)$ is the usual associated vector bundle functor
$P \times_k^G -$ from representations of $G$ to representations of $H$, with $H$ acting through $P$.

\begin{cor}\label{c:unequivpt}
Let $G$ be a affine $k$\nd group, $x$ be a $k$\nd point of $X$, and $P$ be a principal $(H,G)$-bundle
with a $k$\nd point above $x$. 
Suppose that $H^0_H(X,\sO_X)$ is as in Theorem~\textup{\ref{t:main}}.
Then the following conditions are equivalent:
\begin{enumerate}
\renewcommand{\theenumi}{(\alph{enumi})}
\item\label{i:unequivunpt}
$G$ is reductive, and represents the functor $\widetilde{H}^1_H(X,x,-)$ on the category of reductive $k$\nd groups up to conjugacy
with universal element the class of $P$;
\item\label{i:unequivbijpt}
$P$ has a principal $(H,G_0)$\nd subbundle for some reductive $k$\nd subgroup $G_0$ of $G$,
and $P \times^G_k -$ induces a bijection from isomorphism classes of representations of $G$
to isomorphism classes of representations of $H$;
\item\label{i:unequivminpt}
$G$ is reductive, $P$ has no principal $(H,G')$\nd subbundle for any reductive $k$\nd subgroup $G' \ne G$ of $G$,
and every representation of $H$ is a direct summand of $P \times^G_k V$ for some representation $V$ of $G$. 
\end{enumerate}
\end{cor}

\begin{proof}
It is enough to show that with $K = \underline{\Iso}_G(P)$, each of \ref{i:unequivunpt}, \ref{i:unequivbijpt} and \ref{i:unequivminpt}
is equivalent to the corresponding condition of Corollary~\ref{c:unequiv}. 
For \ref{i:unequivunpt} this follows from Lemma~\ref{l:cohreppt}\ref{i:unreppt}, and for \ref{i:unequivbijpt} and \ref{i:unequivminpt} from the fact that \eqref{e:assocvecbun} with $H = \underline{\Iso}_G(P)$ is an equivalence,
together with Lemma~\ref{l:IsoHequivpt} for \ref{i:unequivbijpt} and Lemma~\ref{l:minbgpt} for \ref{i:unequivminpt}.
\end{proof}

Let $F$ be a transitive affine groupoid over an extension $k'$ of $k$, 
and $P$ be a principal $(H,F)$\nd bundle.
Just as for the case $k' = k$ of affine $k$\nd groups, we may define a functor
$P \times_{k'}^F -$ from representations of $F$ to representations of $H$,
which is an equivalence when $H =\underline{\Iso}_F(P)$. 
When $k' = \overline{k}$, an analogue of Corollary~\ref{c:unequivpt} for principal $(H,F)$\nd bundles can then be deduced from 
Corollary~\ref{c:unequiv} in the same way as above, and similarly for Corollary~\ref{c:minequivpt} below.

Let $G$ be an affine $k$\nd group, $P$ be a principal $(H,G)$\nd bundle, and $V$ be a representation of $G$.
Since $P \times_k^G k = \sO_X$, 
we may by functoriality identify
the space $\Hom_G(k,V) = V^G$ of invariants of $G$
with a subspace of $H^0_H(X,P \times^G_k V)$.
This subspace $V^G$ coincides, since \eqref{e:assocvecbun} with $H =\underline{\Iso}_G(P)$ is an equivalence,
with $H^0_{\underline{\Iso}_G(P)}(X,P \times^G_k V)$.

\begin{cor}\label{c:minequivpt}
Let $G$ be an affine $k$\nd group and $P$ be a principal $(H,G)$-bundle
which has a $k$\nd point. 
Suppose that $H^0_H(X,\sO_X)$ is as in Theorem~\textup{\ref{t:main}}.
Then the following conditions are equivalent:
\begin{enumerate}
\renewcommand{\theenumi}{(\alph{enumi})}
\item\label{i:minequivminpt}
$G$ is reductive, and $P$ has no principal $(H,G')$\nd subbundle for any reductive $k$\nd subgroup $G' \ne G$ of $G$;
\item\label{i:minequivradpt}
$H^0_H(X,P \times^G_k V) = {}^\mathrm{rad}H^0_H(X,P \times^G_k V) \oplus V^G$ 
for every representation $V$ of $G$;
\item\label{i:minequivindecpt}
$P$ has a principal $(H,G_0)$\nd subbundle for some reductive $k$\nd subgroup $G_0$ of $G$,
and $P \times^G_k V$ is a non-trivial indecomposable representation of $H$ 
for every non-trivial indecomposable representation $V$ of $G$.
\end{enumerate}
\end{cor}

\begin{proof}
It is enough to show that with $K = \underline{\Iso}_G(P)$, each of \ref{i:minequivminpt}, \ref{i:minequivradpt} and \ref{i:minequivindecpt}
is equivalent to the corresponding condition of Corollary~\ref{c:minequiv}. 
For \ref{i:minequivminpt} this follows from Lemma~\ref{l:minbgpt}, and for \ref{i:minequivradpt} and \ref{i:minequivindecpt} 
from the fact that \eqref{e:assocvecbun} with $H = \underline{\Iso}_G(P)$  is an equivalence,
together with Lemma~\ref{l:IsoHequivpt} for \ref{i:minequivindecpt}.
\end{proof}

Let $G$ be a affine $k$\nd group, $P$ be a principal $(H,G)$\nd bundle, $G_1$ be a 
$k$\nd subgroup of $G$, and $P_1$ be a principal $(H,G_1)$\nd subbundle of $P$. 
We can form other principal subbundles of $P$ starting from $P_1$ using the following operations of
conjugation, extension and gauge transformation.
The conjugate of $P_1$ by $g$ in $G(k)$ is the principal $(H,g^{-1}G_1g)$\nd subbundle
\[
P_1g
\]
of $P$ defined as the image of $P_1$ under right translation $P \to P$ by $g$.
It is the push forward of $P_1$ along $G_1 \to g^{-1}G_1g$ defined by $g_1 \mapsto g^{-1}g_1g$.
The extension of $P_1$ to a $k$\nd subgroup $G_2$ of $G$ containing $G_1$ is the principal
$(H,G_2)$\nd subbundle
\[
P_1G_2
\]
of $P$ defined as the image of
$P_1 \times_k G_2$ under the action of $G$ on $P$.
It is the push forward of $P_1$ along $G_1 \to G_2$.
The gauge transform of $P_1$ by the $(H,G)$\nd automorphism $\theta$ of $P$ is the principal $(H,G_1)$\nd subbundle
\[
\theta P_1
\] 
of $P$ defined as the image of $P_1$ under $\theta$.
Two principal $(H,G_1)$\nd subbundles of $P$ are gauge transforms of each other if and only if they 
are $(H,G_1)$\nd isomorphic, because such an $(H,G_1)$\nd isomorphism extends to an $(H,G)$\nd automorphism of $P$.

Let $G'$ be a $k$\nd subgroup of $G$.
Then the principal $(H,G')$\nd subbundles of $P$ obtained from $P_1$ by iterating 
the above three operations are those of the form
\[
\theta P_1 g G'
\]
for $g$ in $G(k)$ with $g^{-1}G_1g$ contained in $G'$.
A principal $(H,G')$\nd subbundle $P'$ of $P$ can be written in this form for a given $g$ if and only if
conjugation
\[
\alpha_g:G_1 \to G'
\]
by $g^{-1}$ sends the class of $P_1$ in $H^1_H(X,G_1)$ to that of  $P'$ in $H^1_H(X,G')$.
Since $\alpha_{\widetilde{g}}$ and $\alpha_g$ are equal if and only if $\widetilde{g} = g'g$ for $g'$ in 
$G(k)$ centralising $G_1$, they are conjugate if and only if
$\widetilde{g} = g'gg''$ for $g'$ in $G(k)$ centralising $G_1$ and $g''$ in $G'(k)$.

\begin{cor}\label{c:minsubconjgp}
Let $G$ be an affine $k$-group and $P$ be a principal $(H,G)$-bundle.
Let $G_1$ be a reductive $k$\nd subgroup of $G$
and $P_1$ be a principal $(H,G_1)$\nd subbundle of $P$
with no principal $(H,G_2)$\nd
subbundle for any reductive $k$\nd subgroup $G_2 \ne G_1$ of $G_1$.
Let $G'$ be a $k$\nd subgroup of $G$
and $P'$ be a principal $(H,G')$\nd subbundle of $P$
which has a principal $(H,G'')$\nd
subbundle for some reductive $k$\nd subgroup $G''$ of $G'$.
Suppose that $H^0_H(X,\sO_X)$ is as in Theorem~\textup{\ref{t:main}},
and that $X$ has a $k$\nd point above which both $P_1$ and $P'$ have $k$\nd points.
\begin{enumerate}
\item\label{i:minsubconjgpexist}
There is a $k$-point $g$ of $G$ with $g^{-1}G_1g$ contained in $G'$ for which
$P'$ is the image of $P_1gG'$ under some $(H,G)$\nd automorphism of $P$.
\item\label{i:minsubconjgpun}
A $g$ as in \textnormal{\ref{i:minsubconjgpexist}} is unique up to multiplication on the left by a $k$\nd point of 
the centraliser of $G_1$ in $G$
and on the right by a $k$\nd point of $G'$. 
\end{enumerate}
\end{cor}

\begin{proof}
Let  $x$ be a $k$\nd point of $X$ above which $P_1$ and $P'$ have $k$\nd points.
It is enough to show that there is a $g$, unique as in \ref{i:minsubconjgpun}, for which the push forward of $P_1$
along  $\alpha_g$ above is isomorphic to $P'$.
The uniqueness is clear by Corollary~\ref{c:minhomconjcohpt}.

Let $G_0$ represent the functor of Corollary~\ref{c:cohomologyrep}\ref{i:ptcohomologyrep}, 
and let $h_1:G_0 \to G_1$ and $h':G_0 \to G'$ send the universal element to the 
classes of $P_1$ and $P'$.
Then $h_1$ is surjective, and
the composites of the embeddings of $G_1$ and $G'$ into $G$ with $h_1$ and $h'$ are conjugate.
Thus $h' = \alpha_g \circ h_1$ for some $g$.
The existence follows. 
\end{proof}

Given a transitive affine groupoid $F$ over an extension $k'$ of $k$, a transitive affine subgroupoid $F_1$
of $F$ over $k'$, a principal $(H,F)$\nd bundle $P$, and a principal $(H,F_1)$\nd subbundle $P_1$ of $P$,
we define similarly to the above subbundles $P_1v$, $P_1F_2$ and $\theta P_1$ of $P$ for any $v$ in $H^0(k',F^\mathrm{diag})$, transitive affine subgroupoid 
$F_2$ of $F$ containing $F_1$, and $(H,F)$\nd automorphism $\theta$ of $P$.
The proof of Corollary~\ref{c:minsubconjgpd} is then similar to that of Corollary~\ref{c:minsubconjgp}, and is omitted.

\begin{cor}\label{c:minsubconjgpd}
Let $F$ be a transitive affine groupoid over $\overline{k}$ and 
$P$ be a principal $(H,F)$-bundle.
Let $F_1$ be a reductive subgroupoid of $F$ over $\overline{k}$
and $P_1$ be a principal $(H,F_1)$\nd subbundle of $P$
with no principal $(H,F_2)$\nd
subbundle for any reductive subgroupoid $F_2 \ne F_1$ of $F_1$ over $\overline{k}$.
Let $F'$ be a transitive affine subgroupoid of $F$ over $\overline{k}$
and $P'$ be a principal $(H,F')$\nd subbundle of $P$
which has a principal $(H,F'')$\nd
subbundle for some reductive subgroupoid $F''$ of $F'$ over $\overline{k}$.
Suppose that $H^0_H(X,\sO_X)$ is as in Theorem~\textup{\ref{t:main}}.
\begin{enumerate}
\item\label{i:minsubconjgpdexist}
There is a $v$ in $H^0(\overline{k},F^\mathrm{diag})$ with $v^{-1} \circ F_1 \circ v$ contained in $F'$ for which
$P'$ is the image of $P_1vF'$ under some $(H,F)$\nd automorphism of $P$.
\item
A $v$ as in \textnormal{\ref{i:minsubconjgpdexist}} is unique up to composition on the left by an element
of $H^0_{F_1}(\overline{k},F^\mathrm{diag})$
and on the right by an element of $H^0(\overline{k},F'{}^\mathrm{diag})$. 
\end{enumerate}
\end{cor}

Let $G$ be a reductive $k$\nd group and $P$ be a principal $(H,G)$\nd bundle.
With $I_H(\sV)$ as in Section~\ref{s:unmin}, write $I_{H,G}(P)$ for the union of the
ideals $I_H(P \times^G_k V)$ of $H^0_H(X,\sO_X)$, where $V$ runs over those representations of $G$ with $V^G = 0$.
It is an ideal of $H^0_H(X,\sO_X)$, and since \eqref{e:assocvecbun} with $H = \underline{\Iso}_F(P)$ is an equivalence, 
\[
I_{H,G}(P) = I_H(\underline{\Iso}_G(P)).
\]
The properties of $I_H(K)$ proved in Section~\ref{s:unmin} thus imply corresponding properties for $I_{H,G}(P)$.
In particular by Lemma~\ref{l:IHKscalext}, formation of $I_{H,G}(P)$ is compatible with finite extension of scalars,
and arbitrary extension of scalars when $X$ is quasi-compact.
If $H$ is as in Theorem~\ref{t:main} and $P$ has a $k$\nd point, then by Corollary~\ref{c:IHKmin} 
and Lemma~\ref{l:minbgpt},
$I_{H,G}(P) \ne H^0_H(X,\sO_X)$ if and only if $P$ has no principal $(H,G')$\nd subbundle 
for any reductive $k$\nd subgroup $G' \ne G$ of $G$.
Similar results hold with $G$ replaced by a reductive groupoid $F$ over $\overline{k}$, and $P \times^G_k -$ replaced by
$P \times^F_{\overline{k}} -$.

Using Lemmas~\ref{l:cohreppt}\ref{i:unreppt}, \ref{l:cohrepgpd}\ref{i:unrepgpd}, 
\ref{l:minbgpt} and \ref{l:minbggpd}, we obtain analogues for principal bundles
of  Corollaries~\ref{c:minscalext} and \ref{c:unscalext}.

Suppose $\overline{k}$ is an algebraic closure of $k$.
By Proposition~\ref{p:Galextiso}, Galois extended $\overline{k}$\nd groups as in Definition~\ref{d:Galext} are the same as loosely
Galois extended $\overline{k}$\nd groups as in of Definition~\ref{d:Galextloose}.
Using Proposition~\ref{p:grpdgalequ}, Corollary~\ref{c:grpdgalprin} and \eqref{e:grpdgalH1}, 
the results of this section on principal bundles under a groupoid may also be formulated in terms of principal bundles under a Galois extended $\overline{k}$\nd group.
We state explicitly the following form of Corollary~\ref{c:unscalext}, which is used in Sections~\ref{s:curves0}
and \ref{s:curves1}.
A Galois extended $\overline{k}$\nd group $(D,E)$ will be called \emph{reductive} if the $\overline{k}$\nd group $D$
is reductive.

\begin{cor}\label{c:unscalextac}
Let $\overline{k}$ be an algebraic closure of $k$
and $(D,E)$ be a reductive Galois extended $\overline{k}$\nd group.
Let $\alpha$ be an element of $H^1_H(X,D,E)$
and $\overline{\alpha}$ be the image of $\alpha$ in $H^1_{H_{\overline{k}}}(X_{\overline{k}},D)$.
Suppose that $H^0_H(X,\sO_X)$ is as in Theorem~\textup{\ref{t:main}}, and that
$X$ is quasi-compact and quasi-separated and $H_{[1]}$ is quasi-compact.
Then $(D,E)$ represents the functor $H^1_H(X,-,-)$ on the category of reductive Galois extended $\overline{k}$\nd groups up to conjugacy 
with universal element $\alpha$ 
if and only if $D$ represents the functor $H^1_{H_{\overline{k}}}(X_{\overline{k}},-)$ on the category of
reductive $\overline{k}$\nd groups up to conjugacy with universal element $\overline{\alpha}$.
\end{cor}

\begin{proof}
By Proposition~\ref{p:grpdgalequ}, we may suppose that $(D,E) = (F^\mathrm{diag},F(\overline{k})_{\overline{k}})$
for a reductive groupoid $F$ over $\overline{k}$.
By \eqref{e:grpdgalH1}, it is then enough to show that if $\beta$ in $H^1_H(X,F)$ has image $\overline{\beta}$ in 
$H^1_{H_{\overline{k}}}(X_{\overline{k}},F^\mathrm{diag})$, then $F$ and $\beta$ represent $H^1_H(X,-)$ 
on reductive groupoids over $\overline{k}$
up to conjugacy if and only if $F^\mathrm{diag}$ and $\overline{\beta}$ represent
$H^1_{H_{\overline{k}}}(X_{\overline{k}},-)$ on reductive
$\overline{k}$\nd groups up to conjugacy.
This is clear from Corollary~\ref{c:unscalext} and Lemma~\ref{l:cohrepgpd}\ref{i:unrepgpd}, because if $\beta$ is
the class of the principal $(H,F)$\nd bundle $P$, then $\overline{\beta}$ is the class of 
the underlying principal $(H_{\overline{k}},F^\mathrm{diag})$\nd bundle of $P$, while
$\underline{\Iso}_{F^\mathrm{diag}}(P) = \underline{\Iso}_{F_{\overline{k}}}(P_{\overline{k}}) = \underline{\Iso}_F(P)_{\overline{k}}$
by Lemma~\ref{l:gpdprinGG'}.
\end{proof}

\section{Gauge groups}\label{s:gauge}

\emph{In this section $k$ is a field of characteristic $0$, $X$ is a non-empty $k$\nd scheme,
and $H$ is a pregroupoid over $X$.}

\medskip

In this section we consider the $k$\nd groups of cross-sections $\underline{H}^0_H(X,J)$ of certain
affine $H$\nd groups $J$.
Such $k$\nd groups exist for example when $X$ is proper over $k$ and $J = K^\mathrm{diag}$
for a transitive affine groupoid over $H$.
In general, $\underline{H}^0_H(X,K^\mathrm{diag})$ is not reductive, even if $K$ is.
However by Proposition~\ref{p:LieRu} it is reductive when both $H$ and $K$ are reductive groupoids,
and by Theorem~\ref{t:semidirect} 
$\underline{H}^0_{K_0}(X,K^\mathrm{diag})$ is a Levi $k$\nd subgroup of 
$\underline{H}^0_H(X,K^\mathrm{diag})$ when $K$
is reductive and $K_0$ is a minimal reductive subgroupoid of $K$ over $H$.
For simplicity, the main results of this section will be given only for $H$\nd groups of finite type.
The general case is not much more difficult.

Let $\sV$ be an $H$\nd module and $V$ be a $k$\nd vector space.
The canonical homomorphism \eqref{e:HsVtensV} is always injective, and is an isomorphism if and only if
every element of its target lies in $H^0_H(X,\sV \otimes_k V_0)$ for some finite-dimensional
$k$\nd vector subspace $V_0$ of $V$.
It is an isomorphism if either $V$ is finite-dimensional over $k$ or $X$ is quasi-compact,
and hence, by restricting to an open subscheme of $X$, if $H$ is transitive affine.
If $\sV' \to \sV$ is an injective morphism of $H$\nd modules and if \eqref{e:HsVtensV} is an isomorphism, 
then \eqref{e:HsVtensV} with $\sV$ replaced by $\sV'$ is an isomorphism.

\begin{defn}
Let $K$ be a transitive affine groupoid over $H$.
We say that $H$ is \emph{$K$\nd finite} if $H^0_H(X,\sV)$ is finite-dimensional over $k$ for every
representation $\sV$ of $K$ and \eqref{e:HsVtensV} is an isomorphism for every 
representation $\sV$ of $K$ and $k$\nd vector space $V$.
\end{defn}

If $X$ is proper over $k$, or if $H$ is a transitive affine groupoid over $X$, then $H$ is $K$\nd finite
for every transitive affine groupoid $K$ over $H$.

Suppose that $H$ is $K$\nd finite.
Then for any extension $k'$ of $k$, \eqref{e:HsVtensV} with $V = k'$ and hence \eqref{e:Hexthom} is an isomorphism, 
so that formation of $H^0_H(X,\sV)$ for $\sV$ a representation of $K$ commutes with extension of scalars.
Thus by Lemma~\ref{l:extquot}, $H_{k'}$ is $K_{k'}$\nd finite for every extension $k'$ of $k$.

Let $K$ and $K'$ be transitive affine groupoids over $H$.
If there exists a morphism of groupoids from $K$ to $K'$ over $H$,
then $H$ is $K'$\nd finite when it is $K$\nd finite. 
If $K$ is a subgroupoid over $H$ of $K'$, and $K$ has a reductive subgroupoid over $H$,
then $H$ is $K$\nd finite when it is $K'$\nd finite: we may suppose $K$ reductive, and
by Lemma~\ref{l:subquot} every representation
of $K$ is then a direct summand of a representation of $K'$.

Let $Z$ be an $H$\nd scheme.
If the functor $\Hom_H(- \times_k X,Z)$ on $k$\nd schemes is representable, we write
\[
\underline{H}^0_H(X,Z)
\]
for the representing $k$\nd scheme.
It comes equipped with a universal morphism from the constant $H$\nd scheme $\underline{H}^0_H(X,Z) \times_k X$ to $Z$.
We have a functor $\underline{H}^0_H(X,-)$ to the category of $k$\nd schemes from the category of those 
$H$\nd schemes $Z$ for which 
$\underline{H}^0_H(X,Z)$ exists.
This last category is closed under the formation of any limits of $H$\nd schemes which exist, and
$\underline{H}^0_H(X,-)$ preserves such limits.
Also $\underline{H}^0_H(X,-)$ is compatible with extension of scalars, and sends monomorphisms of $H$\nd schemes
to monomorphisms of $k$\nd schemes.
Any morphism $H \to H'$ of pregroupoids over $X$ induces a monomorphism of $k$\nd schemes from $\underline{H}^0_{H'}(X,Z)$
to $\underline{H}^0_H(X,Z)$, when both exist.
If $H = X$, then $\underline{H}^0_H(X,Z)$ is the Weil restriction $R_{X/k}Z$.

\begin{prop}\label{p:sectexist}
Let $K$ be a transitive affine groupoid over $H$ and $Z$ be an affine $K$\nd scheme.
Suppose that $H$ is $K$\nd finite.
Then the $k$\nd scheme $\underline{H}^0_H(X,Z)$ exists and is affine.
It is of finite type over $k$ if $Z$ is of finite type over $X$. 
\end{prop}

\begin{proof}
Since $\Hom_H(- \times_k X,Z)$ on $k$\nd schemes is a sheaf for the Zariski topology and
$Z$ is $\Spec(\sR)$ for a commutative $K$\nd algebra $\sR$, it is enough to show that
\[
\Hom_{H\textrm{-alg}}(\sR,\sO_X \otimes_k -)
\]
on commutative $k$\nd algebras is represented by a commutative $k$\nd algebra $R$ which is of finite type if $\sR$ is.
Since $\sR$ is the coequaliser in the category
of commutative $K$\nd algebras of two morphisms from $\Sym \sV_1$ to $\Sym \sV_2$ for $K$\nd modules $\sV_i$, 
and since the $\sV_i$ are filtered colimits of representations of $K$, we may suppose that 
$\sR$ is $\Sym(\sV)$ for a representation $\sV$ of $K$.
It is then enough to show that
\[
\Hom_H(\sV,\sO_X \otimes_k -)
\]
on $k$\nd vector spaces is represented by a finite-dimensional vector space $W$, because we may take $R = \Sym W$.
By the $K$\nd finiteness of $H$, we have natural isomorphisms 
\[
\Hom_H(\sV,\sO_X \otimes_k V) \iso H^0_H(X,\sV^\vee \otimes_k V) \iso H^0_H(X,\sV^\vee) \otimes_k V
\]
with $H^0_H(X,\sV^\vee)$ finite-dimensional.
Thus we may take $W = H^0_H(X,\sV^\vee)^\vee$.
\end{proof}

Recall \cite[IV 2, Cor~4.5]{DeGa70} that the functor that assigns to an affine $k$\nd group of finite type its Lie algebra
induces an equivalence from the category of unipotent $k$\nd groups of finite type to the category of finite-dimensional
nilpotent Lie algebras over $k$.
If $M$ is an affine $k$\nd group of finite type with Lie algebra $\mathfrak{m}$, and $\mathfrak{m}'$ is a Lie subalgebra 
of $\mathfrak{m}$, then a connected $k$\nd subgroup $M'$ of $M$ with Lie algebra $\mathfrak{m}'$ is unique if exists, because
the intersection of two such subgroups also has Lie algebra $\mathfrak{m}'$.
Further if $\mathfrak{m}'$ is an $M$\nd submodule of $\mathfrak{m}$ for the adjoint action of $M$, 
then $M'$ is a normal $k$\nd subgroup of $M$ if it exists, because it is stable under conjugation by points of $M$ in an algebraic
closure of $k$.

Let $M$ be an affine $k$\nd group of finite type with Lie algebra $\mathfrak{m}$.
To every representation $V$ of $M$ there is associated a symmetric pairing on $\mathfrak{m}$, the trace pairing
\[
a \otimes b \mapsto \tr(\sigma(a) \circ \sigma(b))
\]
with $\sigma:\mathfrak{m} \to \End(V)$ the differential of the action of $M$ on $V$.
It is a homomorphism of $M$\nd modules.
If $V$ is faithful, it can be seen as follows that the trace pairing has kernel the Lie algebra of the unipotent
radical $R_uM$ of $M$.
Since $V$ has a filtration on whose steps $R_uM$ acts trivially, we may suppose after passing to the associated graded 
that $M$ is reductive.
We may also suppose $k$ is algebraically closed and $M$ is connected.
Next we may suppose $M$ is either a torus or simple,
and finally that $M$ is a torus, when the trace pairing arises from a positive definite
pairing over $\Q$.

Let $J$ be an $H$\nd group for which $\underline{H}^0_H(X,J)$ exists.
Then $\underline{H}^0_H(X,J)$ has a canonical structure of group scheme over $k$, and the universal morphism
\[
u:\underline{H}^0_H(X,J) \times_k X \to J
\]
is a morphism of $H$\nd groups.
Suppose that $J$ is smooth over $X$ and  $\underline{H}^0_H(X,J)$ is affine and of finite type over $k$.
Then the Lie algebra $\Lie(J)$ of $J$ is an
$H$\nd algebra (not necessarily unitary or associative), and taking points in the $k$\nd algebra
of dual numbers shows that
\[
\Lie(\underline{H}^0_H(X,J)) = H^0_H(X,\Lie(J)).
\]
The differential of $u$ is then the  homomorphism of Lie $H$\nd algebras
\[
H^0_H(X,\Lie(J)) \otimes_k \sO_X \to \Lie(J)
\]
that sends $\alpha \otimes 1$ to $\alpha$.

Suppose further that $H^0_H(X,\sO_X)$ is a local $k$\nd algebra with residue field $k$.
Recall that ${}^\mathrm{rad}H^0_H(X,\sV)$ denotes the kernel on the right of the pairing \eqref{e:Hpairing}.
If $\sR$ is an $H$\nd algebra (not necessarily unitary or associative),
then $H^0_H(X,\sR)$ is an $H^0_H(X,\sO_X)$\nd algebra, and ${}^\mathrm{rad}H^0_H(X,\sR)$ is a two-sided ideal of $H^0_H(X,\sR)$. 
Indeed if we write $\langle w, a \rangle$ for the image of $w \otimes a$ under \eqref{e:Hpairing},
then 
\[
\langle w(-a), a' \rangle = \langle w, a'a \rangle  = \langle w(a'-), a \rangle
\]
for every $a$ and $a'$ in $H^0_H(X,\sR)$ and $w:\sR \to \sO_X$.
In particular
\[
{}^\mathrm{rad}H^0_H(X,\Lie(J))
\]
is an ideal of $H^0_H(X,\Lie(J))$.
It is functorial in $J$,
and it is stable under the adjoint action of 
$\underline{H}^0_H(X,J)$, as can be seen starting from the fact that any $k$\nd point
of $\underline{H}^0_H(X,J)$ induces by conjugation an $H$\nd morphism $J \to J$.
We write
\[
{}^\mathrm{rad}\underline{H}^0_H(X,J)
\]
for the connected $k$\nd subgroup of $\underline{H}^0_H(X,J)$ with Lie algebra ${}^\mathrm{rad}H^0_H(X,\Lie(J))$, when it exists.
It is a normal $k$\nd subgroup of $\underline{H}^0_H(X,J)$ which is functorial in $J$.

By Proposition~\ref{p:sectexist}, the above conditions on $H$ and $J$ are satisfied in particular when $X$ is 
geometrically $H$\nd connected and $J$ is a $K$\nd group of finite type for some transitive affine groupoid $K$ over $H$ such that 
$H$ is $K$\nd finite.
In this case $H^0_H(X,\sO_X)$ is a finite local $k$\nd algebra with residue field $k$.

\begin{prop}\label{p:LieRu}
Let $K$ be a transitive affine groupoid over $H$ and $J$ be an affine $K$\nd group of finite type.
Suppose that $X$ is geometrically $H$\nd connected and $H$ is $K$\nd finite. 
Then ${}^\mathrm{rad}\underline{H}^0_H(X,J)$ exists, and is contained in $R_u\underline{H}^0_H(X,J)$.
If the fibres of $J$ are reductive, then ${}^\mathrm{rad}\underline{H}^0_H(X,J)$ coincides with $R_u\underline{H}^0_H(X,J)$.
\end{prop}

\begin{proof}
Since $R_u\underline{H}^0_H(X,J)$ is unipotent and hence connected, it is enough by the above remarks to prove that we
have an inclusion
\[
{}^\mathrm{rad}H^0_H(X,\Lie(J)) \subset \Lie(R_u\underline{H}^0_H(X,J))
\]
of Lie algebras, with equality if the fibres of $J$ are reductive.

Extending the scalars if necessary, we may assume that $k$ is algebraically closed
and that $X$ has a $k$\nd point $x$.
Since $J$ is of finite type, $K$ acts on $J$ through a quotient $K_1$ of finite type.
Replacing $K$ by $K_1$, we may assume that $K$ is of finite type.
Then $J \rtimes_X K$ is a transitive affine
groupoid of finite type over $X$,
and has thus a faithful representation $\sV$.
It defines a faithful representation of $J$ on $\sV$ which is compatible with the actions of $K$ on $J$ and $\sV$.

The differential of the action of $J$ on $\sV$ is an embedding 
\[
\rho:\Lie(J) \to \underline{\End}_{\sO_X}(\sV).
\]
of representations of $K$.
We then have a homomorphism
\begin{equation}\label{e:Liedual}
\Lie(J) \to \Lie(J)^\vee
\end{equation}
of representations of $K$
which sends the section $\alpha$ of $\Lie(J)$ to $\tr(\rho(\alpha) \circ \rho(-))$.
Taking the fibre at $x$ and using the trace pairing characterisation of $\Lie(R_uJ_x)$
shows that \eqref{e:Liedual} is an isomorphism if and only if the fibres of $J$ are reductive.

We have a commutative square
\[
\xymatrix{
H^0_H(X,\Lie(J)^\vee) \otimes_k H^0_H(X,\Lie(J)) \ar[r] &  k   \\
H^0_H(X,\Lie(J)) \otimes_k H^0_H(X,\Lie(J)) \ar[u] \ar[r] & H^0_H(X,\sO_X) \ar[u]
}
\]
where the top arrow, modulo the identification of $H^0_H(X,(-)^\vee)$ with $\Hom_H(-,\sO_X)$, 
is of the form \eqref{e:Hpairing}, 
the left arrow is defined using \eqref{e:Liedual},
the bottom arrow sends $\alpha \otimes \beta$ to $\tr(\rho(\alpha) \circ \rho(\beta))$, 
and the right arrow is the projection onto the residue field.
Write 
\[
M = \underline{H}^0_H(X,J).
\] 
Then $\Lie(M)= H^0_H(X,\Lie(J))$.
We show that the pairing $\pi$ on $H^0_H(X,\Lie(J))$ defined by the square has kernel $\Lie(R_uM)$.
The required results will follow, because the left arrow of the square is an isomorphism 
if the fibres of $J$ are reductive.

The right arrow of the square is the unique homomorphism of $k$\nd algebras from $H^0_H(X,\sO_X)$ to $k$,
and hence coincides with evaluation at $x$.
Thus
\[
\pi(\alpha \otimes \beta) = \tr(\rho_x(\alpha_x) \circ \rho_x(\beta_x)).
\]
Since the differential of the fibre
\[
u_x:M \to J_x
\]
above $x$ of the universal morphism $u:M \times_k X \to J$ sends $\alpha$ to $\alpha_x$, it follows that
$\pi$ is the trace pairing associated to the action of $M$ on $\sV_x$ through $u_x$.
To show that $\pi$ has kernel $\Lie(R_uM)$, it is thus enough by the trace pairing characterisation of $\Lie(R_u(M/\Ker u_x))$
to show that $\Ker u_x$ is unipotent.

Let $e:\mu_n \to M$ be a $k$\nd homomorphism.
The action of $\mu_n \times_k X$ on $\sV$ through $u \circ (e \times_k X)$
commutes with the action of $H$.
Thus the decomposition of $\sV$ as $\bigoplus_{i \in \Z/n} \sV_i$, where $\mu_n \times_k X$ 
acts on $\sV_i$ as the $i$th 
power character, is a decomposition of $H$\nd modules.
Since $X$ is geometrically $H$\nd connected, the $\sV_i$ have constant rank.
Suppose that $e(\mu_n)$ lies in $\Ker u_x$.
Then for $i \ne 0$
we have $(\sV_i)_x = 0$ and hence $\sV_i = 0$, so that $e(\mu_n) \times_k X$ lies in $\Ker u$.
Thus $e(\mu_n) = 1$, by the universal property of $M$.
Hence $\Ker u_x$ is unipotent.
\end{proof}

Let $H$, $K$ and $J$ be as in Proposition~\ref{p:LieRu}.
Then the canonical morphism from $\underline{H}^0_K(X,J)$ to $\underline{H}^0_H(X,J)$ is a monomorphism
of $k$\nd schemes and hence the embedding of a $k$\nd subgroup.
Suppose now that $K$ is minimally reductive over $H$.
Then taking $\sV = \Lie(J)$ in Corollary~\ref{c:minequiv}\ref{i:minequivrad} shows that the Lie algebra of 
$\underline{H}^0_H(X,J)$ is the semidirect product of the Lie algebras of $\underline{H}^0_K(X,J)$ 
and ${}^\mathrm{rad}\underline{H}^0_H(X,J)$. 
The more precise result that we have a semidirect product decomposition for the $k$\nd groups
themselves is proved in Theorem~\ref{t:semidirect} below.
In particular, it follows from Proposition~\ref{p:LieRu} and Theorem~\ref{t:semidirect} 
that if the fibres of $J$ are reductive then $\underline{H}^0_K(X,J)$
is a Levi $k$\nd subgroup of $\underline{H}^0_H(X,J)$.
The proof of Theorem~\ref{t:semidirect} is based on the characterisation of the $k$\nd points of 
${}^\mathrm{rad}\underline{H}^0_H(X,J)$ given by Lemma~\ref{l:radKerAut} below in the case where
$J = K^\mathrm{diag}$ where $K$ is a transitive affine groupoid of finite type over $H$ which has
a reductive subgroupoid over $H$.

Let $K$ be a transitive affine groupoid over $H$.
Then by Lemma~\ref{l:tensisoconj} with $f_1 = f_2$ the structural morphism of $H$, we have a group isomorphism
\begin{equation}\label{e:H0Aut}
H^0_H(X,K^\mathrm{diag}) \iso \Aut^\otimes(\omega_{K/H})
\end{equation}
whose composite with the homomorphism that sends a tensor automorphism of $\omega_{K/H}$ to
its component at the representation $\sV$ of $K$ is the homomorphism
\begin{equation}\label{e:H0AutV}
H^0_H(X,K^\mathrm{diag}) \to H^0_H(X,\underline{\Aut}_{\sO_X}(\sV)) = \Aut_H(\sV)
\end{equation}
induced by the action $K^\mathrm{diag} \to \underline{\Aut}_{\sO_X}(\sV)$ of $K^\mathrm{diag}$ on $\sV$.

Suppose that $H^0_H(X,\sO_X)$ is a local $k$\nd algebra with residue field $k$.
If
\[
Q:\Mod_H(X) \to \overline{\Mod_H(X)}
\]
is the projection, then we have a group homomorphism
\begin{equation}\label{e:Autproj}
\Aut^\otimes(\omega_{K/H}) \to \Aut^\otimes(Q\omega_{K/H})
\end{equation}
defined by composition with $Q$.

\begin{lem}\label{l:radKerAut}
Let $K$ be a transitive affine groupoid of finite type over $H$.
Suppose that $X$ is geometrically $H$\nd connected and $H$ is $K$\nd finite.
Then the group of $k$\nd points of ${}^\mathrm{rad}\underline{H}^0_H(X,K^\mathrm{diag})$ is contained
in the kernel of the composite of \eqref{e:Autproj} with \eqref{e:H0Aut}, and coincides with this kernel
if $K$ has a reductive subgroupoid over $H$. 
\end{lem}

\begin{proof}
Let $\sV$ be a representation of $K$.
Then the $k$\nd group of automorphisms
\[
\underline{\Aut}_H(\sV) = \underline{H}^0_H(X,\underline{\Aut}_{\sO_X}(\sV))
\]
exists, and it is the $k$\nd group of units of the finite $k$\nd algebra
\[
\End_H(\sV) = H^0_H(X,\underline{\End}_{\sO_X}(\sV)),
\]
because \eqref{e:HsVtensV} with $\underline{\End}_{\sO_X}(\sV)$ for $\sV$ and a commutative $k$\nd algebra for $V$ is an isomorphism.
The action of $K^\mathrm{diag}$ on $\sV$ defines a $k$\nd homomorphism
\[
h_{\sV}:
\underline{H}^0_H(X,K^\mathrm{diag}) \to \underline{H}^0_H(X,\underline{\Aut}_{\sO_X}(\sV)),
\]
which induces on $k$\nd points the homomorphism \eqref{e:H0AutV}. 
By Cayley--Hamilton, the kernel $\Ker(Q_{\sV,\sV})$ of $Q$ on $\End_H(\sV)$ is a nilideal,
because the traces of the exterior powers of its elements 
lie in the maximal ideal of $H^0_H(X,\sO_X)$, and hence are nilpotent.
It follows that $\underline{\Aut}_H(\sV)$ has a unipotent $k$\nd subgroup $U_{\sV}$ such that
\[
\Lie(U_{\sV}) = \Ker(Q_{\sV,\sV}) = {}^\mathrm{rad}H^0_H(X,\underline{\End}_{\sO_X}(\sV))
\]
and such that the $k$\nd points of $U_{\sV}$ are those $H$\nd automorphism of $\sV$ lying above the identity 
in $\overline{\Mod_H(X)}$.
The kernel of the composite of \eqref{e:Autproj} with \eqref{e:H0Aut} then consists of those $k$\nd points
of $\underline{H}^0_H(X,K^\mathrm{diag})$ which lie in $h_{\sV}{}\!^{-1}(U_{\sV})$ for every $\sV$.

Denote by $\rho_\sV:\Lie(K^\mathrm{diag}) \to \underline{\End}_{\sO_X}(\sV)$ the action of 
$\Lie(K^\mathrm{diag})$ on $\sV$. 
Then
\[
\eta_{\sV} = H^0_H(X,\rho_\sV):H^0_H(X,\Lie(K^\mathrm{diag})) \to H^0_H(X,\underline{\End}_{\sO_X}(\sV))
\]
is the differential of $h_{\sV}$.
By functoriality of ${}^\mathrm{rad}H^0_H(X,-)$ we have 
\begin{equation}\label{e:Lieincl}
{}^\mathrm{rad}H^0_H(X,\Lie(K^\mathrm{diag})) \subset \eta_{\sV}{}\!^{-1}({}^\mathrm{rad}H^0_H(X,\underline{\End}_{\sO_X}(\sV))).
\end{equation}
Since ${}^\mathrm{rad}\underline{H}^0_H(X,K^\mathrm{diag})$ is unipotent and hence connected, this implies that
\begin{equation}\label{e:grpincl}
{}^\mathrm{rad}\underline{H}^0_H(X,K^\mathrm{diag}) \subset h_{\sV}{}\!^{-1}(U_{\sV}).
\end{equation}
The required inclusion of $k$\nd points follows.

Suppose now that $K$ has a reductive subgroupoid $K_0$ over $H$.
Let $\sV$ be a faithful representation of $K$.
Then $\rho_\sV$ is a monomorphism in $\Mod_K(X)$, and hence by semisimplicity of $\Mod_{K_0}(X)$ 
has a left inverse in $\Mod_H(X)$.
Thus $\Hom_H(\rho_\sV,\sO_X)$ is surjective,
so that by \eqref{e:Hpairingsquare} applied to $\rho_\sV$ we have equality in \eqref{e:Lieincl}.
Since $h_{\sV}$ is an embedding, $h_{\sV}{}\!^{-1}(U_{\sV})$ is unipotent and hence connected.
Thus we have equality in \eqref{e:grpincl}.
The required equality of $k$\nd points follows.
\end{proof}

\begin{thm}\label{t:semidirect}
Let $K$ be a minimal reductive groupoid over $H$ and $J$ be an affine $K$\nd group of finite type.
Suppose that $X$ is geometrically $H$\nd connected and $H$ is $K$\nd finite.
Then $\underline{H}^0_H(X,J)$ is the semidirect product of its $k$\nd subgroup $\underline{H}^0_K(X,J)$ 
with its normal $k$\nd subgroup ${}^\mathrm{rad}\underline{H}^0_H(X,J)$.
\end{thm}

\begin{proof}
Replacing $K$ by a quotient through which it acts on $J$, we may suppose that $K$ is of finite type.
It is enough to prove the particular case where $J = K'{}^\mathrm{diag}$ for a transitive affine groupoid
$K'$ over $K$ of finite type:
the general case follows by taking $J \rtimes_X K$ and $K$ for  $K'$, 
because the semidirect product of $K$\nd groups $J \rtimes_X K^\mathrm{diag}$ is preserved
by each of $\underline{H}^0_H(X,-)$, $\underline{H}^0_K(X,-)$ and ${}^\mathrm{rad}\underline{H}^0_H(X,-)$.

After extending the scalars, we may assume that $k$ is algebraically closed.
It is then enough to show that we have a semidirect product decomposition on groups of $k$\nd points.
We have a commutative diagram
\[
\xymatrix{
H^0_K(X,K'{}^\mathrm{diag}) \ar[d]^{\wr} \ar[r] &   H^0_H(X,K'{}^\mathrm{diag}) \ar[d]^{\wr} \ar[r] &  
\Aut^\otimes(Q\omega_{K'/H}) \ar@{=}[d]\\
\Aut^{\otimes}(\omega_{K'/K}) \ar[r] &   \Aut^{\otimes}(\omega_{K'/H}) \ar[r] &   \Aut^\otimes(Q\omega_{K'/H})
}
\]
where the vertical isomorphisms are of the form \eqref{e:H0Aut}, the top left arrow is the embedding,
the top right arrow is the composite of Lemma~\ref{l:radKerAut} with $K'$ for $K$, 
the bottom left arrow sends $\varphi$ to $\omega_{K/H}\varphi$, and the bottom right arrow is as in \eqref{e:Autproj}.
The composite of the bottom and hence of the top arrows is an isomorphism,
because $Q\omega_{K/H}$ is fully faithful by Theorem~\ref{t:main}\ref{i:maincrit}.
Since $K'$ has the reductive subgroupoid $K$ over $H$, the required
semidirect product decomposition now follows from Lemma~\ref{l:radKerAut}.
\end{proof}

Let $K$ be a minimal reductive groupoid over $H$.
Suppose that $X$ is geometrically $H$\nd connected and $H$ is $K$\nd finite.
Then for any transitive affine groupoid $K'$ over $H$ with $\Hom_H(K,K')$ non-empty,
it follows from Theorem~\ref{t:semidirect} with $J = K'{}^\mathrm{diag}$ and Corollary~\ref{c:minhomconj}\ref{i:minhomconj}
that the group of $k$\nd points of 
${}^\mathrm{rad}\underline{H}^0_H(X,K'{}^\mathrm{diag})$ acts by conjugation simply transitively on $\Hom_H(K,K')$.

\begin{exmp}
Suppose that $X = \Spec(k)$ and $H$ is an affine $k$\nd group.
Then a transitive affine groupoid over $H$ is an affine $k$\nd group $G$ equipped with a $k$\nd homomorphism
$f:H \to G$.
We have $G^\mathrm{diag} = G$, 
and $\underline{H}^0_H(k,G)$ is the centraliser $Z(f)$ of $f(H)$ in $G$.
When $G$ is of finite type and $f$ factors through a reductive $k$\nd group, ${}^\mathrm{rad}\underline{H}^0_H(k,G)$ is the $k$\nd subgroup $Z_u(f)$
of $Z(f)$ defined in \cite[2.1.12]{O10}.
The case of Theorem~\ref{t:semidirect} where $J = K'{}^\mathrm{diag}$ for a transitive affine groupoid $K'$ over $K$
reduces to \cite[2.3.8]{O10}, and the case where further $K' = K$
to \cite[20.1.3.c)]{AndKah}.
The case where $J = K'{}^\mathrm{diag}$ with $K'$ over $K$ reductive, and $H \to K$
is the embedding $\bG_a \to SL_2$,
reduces to a result of Kostant \cite[3.6]{Kos59}. 
\end{exmp}

\begin{exmp}\label{e:GPAut}
Let $G$ be an affine $k$\nd group, $P$ be a principal $(H,G)$\nd bundle, $G_1$ be a $k$\nd subgroup of $G$, 
and $P_1$ be a principal $(H,G_1)$\nd subbundle of $P$.
Then $G_1$ acts on $G$ by conjugation,
and if $K$ is $\underline{\Iso}_{G_1}(P_1)$ and $J$ is the $K$\nd group $\underline{\Aut}_G(P) = P_1 \times_k^{G_1} G$,
it follows from the analogue of \eqref{e:assocvecbun} for affine $G_1$\nd schemes that 
$\underline{H}^0_K(X,J)$ is the centraliser $G^{G_1}$ of $G_1$ in $G$.
Thus if $G$ is of finite type and $\underline{\Iso}_{G_1}(P_1)$ is minimally reductive over $H$, then by  Theorem~\ref{t:semidirect}
the $k$\nd group of $(H,G)$\nd automorphisms of $P$ is the semidirect product of $G^{G_1}$
with a normal unipotent $k$\nd subgroup.
\end{exmp}

\begin{exmp}
Let $J$ be a smooth affine group scheme over $X$ with (connected) semisimple fibres.
Then the groupoid  over $X$ of isomorphisms $\underline{\Iso}_X(J)$ exists and is affine and of finite type with
reductive fibres above the diagonal \cite[XXIV, 7.3.1]{SGA3}.
Locally in the \'etale topology, $J$ is isomorphic to $G \times_k X$ for a split semisimple $k$\nd group $G$ 
\cite[XXIV, 4.1.5, 4.1.6]{SGA3}.
Thus if $X$ is connected, then $\underline{\Iso}_X(J)$ is transitive, and hence reductive.
Suppose $X$ is complete and geometrically connected.
Then by Proposition~\ref{p:LieRu} and Theorem~\ref{t:semidirect}, 
$\underline{H}^0_K(X,J)$ is a Levi $k$\nd subgroup of $\underline{H}^0_X(X,J)$
for any minimal reductive subgroupoid $K$ of $\underline{\Iso}_X(J)$ over $X$.
\end{exmp}

\section{Pullback of universal and minimal reductive groupoids}\label{s:pullback}

\emph{In this section $k$ is a field of characteristic $0$, $X$ is a non-empty $k$\nd scheme, and
$H$ is a pregroupoid over $X$.}

\medskip

In this section we study the behaviour of a universal or minimal reductive groupoid $K$ over $H$ 
under pullback, and under formation of groupoids $K \times_X X'$ for a $K$\nd scheme $X'$.
Proposition~\ref{p:etunpull}, which deals with the case where $X'$ is a pro\'etale $K$\nd scheme,
will be used in Section~\ref{s:curves1}

\begin{prop}
Let $K$ be a transitive affine groupoid over $H$, and $X'$ be a non-empty scheme over $X$.
Suppose that either $H$ is a transitive groupoid over $X$ or the structural morphism of $X'$ is $fpqc$ covering.
Then $K$ is minimally (resp.\ universally) reductive over $H$ if and only if $K \times_{[X]} [X']$ is minimally
(resp.\ universally) reductive over $H \times_{[X]} [X']$.
\end{prop}

\begin{proof}
Immediate from Lemmas~\ref{l:morphpull}\ref{i:morphpullgpd} and \ref{l:relgrpdequiv}.
\end{proof}

If $p:X' \to X$ is a quasi-compact quasi-separated morphism with $p_*\sO_{X'} = \sO_X$,
then the pullback of $p$ along any flat morphism is schematically dominant.
For such a $p$, the condition on $H$ in Proposition~\ref{p:pull} below is thus satisfied by any
$H$ over $X$ with either of $d_0,d_1:H_{[1]} \to X$ flat.

\begin{prop}\label{p:pull}
Let $K$ be a transitive affine groupoid over $H$,
and $X'$ be an $H$\nd scheme with structural morphism $p:X' \to X$.
Suppose that $p_*\sO_{X'} = \sO_X$,
and that the projection from $H_{[1]} \times_X X'$
to $H_{[1]}$ factors though no closed subscheme of $H_{[1]}$ other than $H_{[1]}$.
Then pullback along $p$ induces a bijection from the set of reductive subgroupoids of $K$ over $H$
to the set of reductive subgroupoids of $K \times_{[X]} [X']$ over $H \times_X X'$.
\end{prop}

\begin{proof}
Write $H'$ for $H \times_X X'$.
We first show that for $Z$ an affine $H$\nd scheme with pullback $Z'$ along $H' \to H$, pullback along 
$H' \to H$ defines a bijection
\begin{equation}\label{e:pbij}
\Hom_H(X,Z) \iso \Hom_{H'}(X',Z').
\end{equation}
Using the identification of $H'$\nd schemes as $H$\nd schemes equipped with an $H$\nd morphism to $X'$,
we may identify the target of \eqref{e:pbij} with $\Hom_H(X',Z)$, and \eqref{e:pbij} then sends $s$ to $s \circ p$.
Since $Z$ is affine over $X$ and $p_*\sO_{X'} = \sO_X$, composition with $p$ defines a bijection from $\Hom_X(X,Z)$
to $\Hom_X(X',Z)$.
Thus \eqref{e:pbij} will hold provided that $s$ is an $H$\nd morphism when $s \circ p$ is an $H$\nd morphism.
This follows from the diagram
\[
\xymatrix@C+2em@R+0.25cm{
H_{[1]} \times_X X' \ar[d] \ar[r]^{H_{[1]} \times_X p}  & H_{[1]} \times_X X \ar[d] \ar[r]^{H_{[1]} \times_X s}
& H_{[1]} \times_X Z \ar[d] \\
X' \ar[r]^{p} & X \ar[r]^{s} & Z
}
\]with vertical arrows the actions, where the left square and outer rectangle commute because $p$ and $s \circ p$
are $H$\nd morphisms, and hence the right square commutes because its two legs are morphisms from
$H_{[1]} = H_{[1]} \times_X X$ to $Z$ over $X$ whose equaliser
is a closed subscheme of $H_{[1]}$ through which $H_{[1]} \times_X p$ factors.

Let $K''$ be a reductive subgroupoid of $K' = K \times_{[X]} [X']$ over $H'$.
By Lemma~\ref{l:affquot}, $K'/K''$ exists and is affine over $X'$.
If $s'$ is the base cross-section
of $K'/K''$, there is
by Lemma~\ref{l:prereppull} and \eqref{e:pbij} a unique pair $(Z,s)$ up to isomorphism
with $Z$ a transitive affine 
$K$\nd scheme and $s:X \to Z$ an $H$\nd morphism, whose pullback along $p$ is isomorphic to $(K'/K'',s')$. 
The stabiliser of $s$ is then the unique reductive subgroupoid of $K$ over $H$ with pullback $K''$ along $p$.
\end{proof}

\begin{cor}
Let $K$ and $X'$ be as in Proposition~\textnormal{\ref{p:pull}}.
Then $K$ is minimally reductive over $H$ if and only if $K \times_{[X]} [X']$ is minimally reductive
over $H \times_X X'$.
\end{cor}

\begin{proof}
Immediate from Proposition~\ref{p:pull}.
\end{proof}

Let $K$ be a reductive groupoid over $H$.
Call $K$ \emph{almost minimally reductive over $H$} if every reductive subgroupoid of $K$ over $H$ 
contains the identity component $K^\mathrm{con}$ of $K^\mathrm{diag}$.
Then $K$ is almost minimally reductive over $H$ if and only if every quotient of $K$ of finite type
is almost minimally reductive over $H$.
Further $K$ is minimally reductive over $H$ if and only if $K$ is almost minimally reductive over $H$ and
$K_{\mathrm{\acute{e}t}} = K/K^\mathrm{con}$ is minimally reductive over $H$.
If $X$ is geometrically $H$\nd connected and $K'$ is
a reductive subgroupoid of $K$ over $H$ containing $K^\mathrm{con}$, 
then $K$ is almost minimally reductive over $H$ if and only if $K'$ is: we have 
$K' = K \times_{K_{\mathrm{\acute{e}t}}} K'{}\!_{\mathrm{\acute{e}t}}$, while by Proposition~\ref{p:etfunexist},
the intersection of two transitive affine subgroupoids of $K_{\mathrm{\acute{e}t}}$ over $H$ is transitive affine.

\begin{lem}\label{l:KK'almost}
Let $K$ be a reductive groupoid of finite type over $X$ and $K'$ be a reductive subgroupoid of $X$.
Then the following conditions are equivalent.
\begin{enumerate}
\renewcommand{\theenumi}{(\alph{enumi})}
\item\label{i:KK'almostconn}
$K'{}^\mathrm{con} = K^\mathrm{con}$.
\item\label{i:KK'almostfin}
There are only finitely many pairwise non-isomorphic irreducible representations $\sV$ of $K$ for which 
$H^0_{K'}(X,\sV) \neq 0$.
\end{enumerate}
\end{lem}

\begin{proof}
By Lemmas~\ref{l:prereppull} and \ref{l:extquot}, we may after extension of scalars and pullback
assume that $k$ is algebraically closed and $X = \Spec(k)$,
with $K$ a reductive $k$\nd group $G$ of finite type, $K'$ a reductive $k$\nd subgroup $G'$ of $G$,
and $H^0_{K'}(X,-) = (-)^{G'}$.

The $k$\nd group $G$ acts by left and right translation on the $k$\nd algebra $k[G]$.
Both \ref{i:KK'almostconn} and \ref{i:KK'almostfin} are equivalent to the finiteness of the $k$\nd algebra $k[G]^{G'}$
of invariants under $G'$ by right translation: for \ref{i:KK'almostconn} because $G/G'$ is affine
with $k[G/G'] = k[G]^{G'}$,
and for \ref{i:KK'almostfin} by the canonical decomposition of $k[G]$ as a $(G,G)$\nd bimodule.
\end{proof}

\begin{prop}\label{p:almin}
let $K$ be a transitive affine groupoid over $H$
and $X'$ be a non-empty finite locally free  $H$\nd scheme.
Suppose that $H^0_H(X,\sO_X)$ is as in Theorem~\textup{\ref{t:main}}.
Then $K$ is almost minimally reductive over $H$ if and only if  
$K \times_{[X]} [X']$ is almost minimally reductive over $H \times_X X'$.
\end{prop}

\begin{proof}
The ``if'' is clear.
To prove the ``only if'', we may suppose that $K$ is of finite type.
Replacing $K$ by the inverse image of an appropriate subgroupoid
of $K_\mathrm{\acute{e}t}$ along the projection $K \to K_\mathrm{\acute{e}t}$,
we may suppose further that $K$ is minimally reductive over $H$.
Write $p:X' \to X$ for the structural morphism of $X'$.
If $K_0$ is a universal reductive groupoid over $H$,
there is a surjective morphism of groupoids $K_0 \to K$ over $H$,
and by Corollary~\ref{c:unequiv} the representation $p_*\sO_{X'}$ of $H$ is
isomorphic to a representation of $K_0$.
Replacing $K$ by a sufficiently large quotient of finite type of $K_0$ through which $K_0 \to K$ factors, 
we may suppose that for some representation $\sW$ of $K$ there is an $H$\nd isomorphism
\[
\sW \iso p_*\sO_{X'}.
\] 
Write $H'$ and $K'$ for $H \times_X X'$ and $K \times_{[X]} [X']$, and $\sW'$ for the pullback of $\sW$ along $K' \to K$.

Let $K''$ be a reductive subgroupoid of $K'$ over $H'$.
We show that if $\sV'$ is an irreducible representation of $K'$ with
\[
H^0_{K''}(X',\sV') \neq 0,
\]
then $\sV'$ is a direct summand of the representation
\[
\sW' \otimes_{\sO_{X'}} \sW'{}^\vee
\]
of $K'$.
It will follow that \ref{i:KK'almostfin} and hence \ref{i:KK'almostconn} of Lemma~\ref{l:KK'almost} is satisfied with $K'$ and $K''$ 
for $K$ and $K'$, so that $K'$ is almost minimally reductive over $H'$.

By Lemma~\ref{l:prereppull}, $K$ has a representation $\sV$ with pullback along $K' \to K$ isomorphic to $\sV'$.
By hypothesis, $\sV'$ has as a representation of $K''$ and hence as a representation of $H'$ 
the direct summand $\sO_{X'}$.
Thus by \eqref{e:pOXV} the representation $p_*\sO_{X'} \otimes_{\sO_X} \sV$ of $H$ has the direct summand $p_*\sO_{X'}$.
By Krull--Schmidt for $\Mod_H(X)$ and $\Mod_K(X)$ and Corollaries~\ref{c:minequiv} and 
\ref{c:minhomconj}\ref{i:minhomconjcoh}, 
the representation $\sW \otimes_{\sO_X} \sV$ of $K$ has thus the direct summand $\sW$
and hence  the representation $\sW' \otimes_{\sO_{X'}} \sV'$ of $K'$ has the direct summand $\sW'$.
Since $\sV'$ is irreducible, it is therefore a direct summand of $\sW' \otimes_{\sO_{X'}} \sW'{}^\vee$.
\end{proof}

Suppose $H^0_H(X,\sO_X)$ is henselian local with residue field $k$, and let $X'$ be a non-empty geometrically 
$H$\nd connected $H$\nd scheme.
If $X'$ is finite and locally free over $X$, then $H^0_{H \times_X X'}(X',\sO_{X'})$ is henselian local
with residue field $k$, because by \eqref{e:Hpushequ} and Cayley--Hamilton, each of its elements is finite over $H^0_H(X,\sO_X)$.
The same holds if $X$ is quasi-compact and $X'$ is pro\'etale, by Lemma~\ref{l:colimHomR} with 
$\sV = \sV' = \sO_X$.

Let $K$ be a transitive affine groupoid over $X$ and $X'$ be a pro\'etale $K$\nd scheme.
By Lemma~\ref{l:transproet}\ref{i:transproetcon}, $X'$ is a transitive $K$\nd scheme if and only if 
it is non-empty and geometrically $K$\nd connected.
When these conditions are satisfied, $K$ is reductive if and only if $K \times_X X'$ is reductive.

Suppose that $X$ is geometrically $H$\nd connected.
Let $K$ be a minimal reductive groupoid over $H$ and $X'$ be a pro\'etale $K$\nd scheme. 
Then $X'$ geometrically $K$\nd connected if and only if it is geometrically $H$\nd connected:
the ``only if'' follows from  Corollary~\ref{c:etfunact} and Proposition~\ref{p:funscalext} after replacing $K$ first by 
$K_\mathrm{\acute{e}t}$ and then, using  
Proposition~\ref{p:etfunexist}, by an initial object in the category of pro\'etale groupoids over $H$.

\begin{prop}\label{p:etminpull}
Let $K$ be a transitive affine groupoid over $H$
and $X'$ be a non-empty, geometrically $H$\nd connected, pro\'etale $K$\nd scheme.
Suppose that $H^0_H(X,\sO_X)$ is as in Theorem~\textup{\ref{t:main}}, 
and that either $X'$ is finite over $X$ or
$X$ is quasi-compact.
Then $K$ is minimally reductive over $H$ if and only if 
$K \times_X X'$ is minimally reductive over $H \times_X X'$.
\end{prop}

\begin{proof}
Write $H'$ for $H \times_X X'$ and $K'$ for $K \times_X X'$.
Suppose that $K'$ is minimally reductive over $H'$.
If $K_1$ is a reductive subgroupoid of $K$ over $H$, then $K_1 \times_X X'$ is a reductive subgroupoid of 
$K'$ over $H'$.
Hence $K_1 \times_X X' = K'$, so that $K_1 = K$ because $X' \to X$ is faithfully flat.
Thus $K$ is minimally reductive over $H$.

Conversely suppose that $K$ is minimally reductive over $H$.
Since $K$ acts on $X'$ through $K_{\mathrm{\acute{e}t}}$, we have 
$K'{}\!_{\mathrm{\acute{e}t}} = K_{\mathrm{\acute{e}t}} \times_X X'$.
Thus if $X'$ is finite over $X$ then $K'$ is minimally reductive over $H'$,
because $K \times_{[X]} [X']$ and hence $K'$
is almost minimally reductive over $H'$ by Proposition~\ref{p:almin} and $K'{}\!_{\mathrm{\acute{e}t}}$ is minimally reductive over $H'$ by Lemma~\ref{l:funact}.
Suppose now that $X$ is quasi-compact,
and write $X'$ as the limit of a filtered inverse system $(X_\lambda)$ of finite \'etale $K$\nd schemes with faithfully flat projections.
By the finite case, $K_\lambda = K \times_X X_\lambda$ is minimally reductive over 
$H_\lambda = H \times_X X_\lambda$ for each $\lambda$.
By Lemmas~\ref{l:HRrepfp} and \ref{l:prereppull}, every representation of $K'$ is isomorphic to 
the pullback onto $X'$ of a representation of some $K_\lambda$.
Thus by Lemma~\ref{l:colimHomR}, \ref{i:minequivindec} of Corollary~\ref{c:minequiv} holds with $H'$ and $K'$ 
for $H$ and $K$ because
for each $\lambda$ it holds with $H_\lambda$ and $K_\lambda$ for $H$ and $K$.
\end{proof}

Suppose that $X$ is geometrically $H$\nd connected, and let $K$ be a universal reductive groupoid over $H$.
Then $K_\mathrm{\acute{e}t}$ is initial in the category of pro\'etale groupoids over $H$,
because it is initial up to conjugacy and an initial object exists by Proposition~\ref{p:etfunexist}.
Thus by Corollary~\ref{c:etfunact}, every pro\'etale $H$\nd scheme has a unique structure of pro\'etale $K$\nd scheme.

\begin{prop}\label{p:etunpull}
Let $K$ be a transitive affine groupoid over $H$
and $X'$ be a non-empty, geometrically $H$\nd connected, pro\'etale $K$\nd scheme.
Suppose that $H^0_H(X,\sO_X)$ is as in Theorem~\textup{\ref{t:main}}, 
and that either of the following conditions holds:
\begin{enumerate}
\renewcommand{\theenumi}{(\alph{enumi})}
\item\label{i:etunpullfin}
$X'$ is finite over X;
\item\label{i:etunpullqc}
$X$ is quasi-compact and quasi-separated and
$H_{[1]}$ is quasi-compact.
\end{enumerate}
Then $K$ is universally reductive over $H$ if and only if $K \times_X X'$ is universally reductive over $H \times_X X'$.
\end{prop}

\begin{proof}
By Lemma~\ref{l:Rsummand} when \ref{i:etunpullfin} holds, and hence by Lemma~\ref{l:HRrepfp} when \ref{i:etunpullqc} holds,
every representation of $H \times_X X'$ is a direct summand of the pullback onto $X'$ of a representation of $H$.
The ``only if'' thus follows from Corollary~\ref{c:unequiv} and  Proposition~\ref{p:etminpull}.

By Theorem~\ref{t:main}, a universal reductive groupoid $K_0$ over $H$ exists, and by the ``only if'',
$K_0 \times_X X'$ is universally reductive over $H \times_X X'$.
The ``if'' follows,
because $K_0 \to K$ is an isomorphism if $K_0 \times_X X' \to K \times_X X'$ is.
\end{proof}

\section{Curves of genus 0}\label{s:curves0}

\emph{In this section $k$ is a field of characteristic $0$ and $\overline{k}$ is an algebraic closure of $k$.}

\medskip

In this section we describe explicitly principal bundles with reductive structure group
over a smooth projective curve $X$ over $k$ of genus $0$.
The study of such bundles is reduced by Theorem~\ref{t:lineuniv} below to the study of the universal
groupoid $\pi_\mathrm{mult}(X)$ of multiplicative type.
The main result is Theorem~\ref{t:gen0rep}.
Theorem~\ref{t:linerep} is the characteristic $0$ case of the algebraic version 
(\cite{Har68}, \cite{Gil02}, \cite{Meh02}) of Grothendieck's original
classification \cite{Gro57} over the Riemann sphere.

Throughout this section, we suppose that $X$ is a $k$\nd scheme which satisfies the following conditions:
it is \emph{reduced, quasi-compact and quasi-separated, with}
\[
H^0(X,\sO_X) = k.
\]
These conditions are stable under extension of scalars.

A scheme over $X$ which is a pro\'etale $X$\nd scheme 
in the sense of Definition~\ref{d:Hproetale} will also be called a \emph{pro\'etale cover of $X$}.
The above conditions on $X$ are satisfied with $X$ replaced by a non-empty geometrically connected pro\'etale cover:
by Lemma~\ref{l:colimH0HV}, we reduce to the case of a finite \'etale scheme over $X$.

We begin by developing the properties of the groupoid $\pi_\mathrm{mult}(X)$ that will be required.
It will be convenient to consider at the same time the groupoids $\pi_\mathrm{\acute{e}t}(X)$ and $\pi_\mathrm{\acute{e}tm}(X)$,
which will be required for the next section.

By Proposition~\ref{p:etfunexist} (resp.\ \ref{p:multinit}, resp.\ \ref{p:etmultinit}) with $H = X$,
an initial object exists in the category of pro\'etale groupoids (resp.\ groupoids of multiplicative type,
resp.\ groupoids of pro\'etale by multiplicative type) over $X$,
which we write as $\pi_{\mathrm{\acute{e}t}}(X)$
(resp.\ $\pi_\mathrm{mult}(X)$, resp.\ $\pi_\mathrm{\acute{e}tm}(X)$).
Each of $\pi_{\mathrm{\acute{e}t}}(X)$, $\pi_\mathrm{mult}(X)$ and $\pi_\mathrm{\acute{e}tm}(X)$ is functorial in 
$X$ in the sense that given $f:X \to X'$, there is for example a unique $f_*$ with
\[
(f,f_*):(X,\pi_{\mathrm{\acute{e}t}}(X)) \to (X',\pi_{\mathrm{\acute{e}t}}(X'))
\] 
a morphism of groupoids in $k$\nd schemes,
by the universal property of $\pi_{\mathrm{\acute{e}t}}(X)$
applied to the pullback of $\pi_{\mathrm{\acute{e}t}}(X')$ onto $X$. 
Further $\pi_{\mathrm{\acute{e}t}}(X)$ coincides with $(\pi_\mathrm{\acute{e}tm}(X))_{\mathrm{\acute{e}t}}$
and $\pi_\mathrm{mult}(X)$ with $(\pi_\mathrm{\acute{e}tm}(X))_\mathrm{ab}$.
A universal reductive groupoid $K$ over $X$ exists by Theorem~\ref{t:main}\ref{i:mainexist}, 
and since $\pi_{\mathrm{\acute{e}t}}(X)$, $\pi_\mathrm{mult}(X)$ 
and $\pi_\mathrm{\acute{e}tm}(X)$ are initial up to conjugacy,
they coincide respectively with $K_{\mathrm{\acute{e}t}}$, $K_\mathrm{ab}$ 
and $K/(K^\mathrm{con})^\mathrm{der}$.

By Proposition~\ref{p:funscalext}, formation of $\pi_{\mathrm{\acute{e}t}}(X)$, $\pi_\mathrm{mult}(X)$ 
and $\pi_\mathrm{\acute{e}tm}(X)$ commutes with algebraic extension of scalars.

Let $X'$ be a pro\'etale cover of $X$.
By Corollary~\ref{c:etfunact}, $X'$ has a unique structure  of $\pi_\mathrm{\acute{e}tm}(X)$\nd scheme.
If $X'$ is non-empty and geometrically connected, then it is a transitive $\pi_\mathrm{\acute{e}tm}(X)$\nd scheme
by Lemma~\ref{l:transproet}\ref{i:transproetcon}, and
\begin{equation}\label{e:pietcov}
\pi_\mathrm{\acute{e}tm}(X') = \pi_\mathrm{\acute{e}tm}(X) \times_X X'
\end{equation}
by Proposition~\ref{p:funpull}.

Let $M$ be a Galois module, i.e.\ an abelian group with a continuous action of $\Gal(\overline{k}/k)$.
As in Section~\ref{s:fundgrpd}, we denote by $D(M)$ the $k$\nd group of multiplicative type 
with Galois module of characters $M$, and by
\[
\chi_\mu:D(M)_{\overline{k}} \to (\bG_m)_{\overline{k}}
\]
the character corresponding to $\mu \in M$.
By \eqref{e:multPic}, we have an equality 
\begin{equation}\label{e:diagPicX}
\pi_\mathrm{mult}(X)^\mathrm{diag} = D(\Pic(X_{\overline{k}})) \times_k X
\end{equation}
of constant $\pi_\mathrm{mult}(X)$\nd schemes, where $(\chi_\mu)  \times_{\overline{k}} {X_{\overline{k}}}$ 
for $\mu$ the class of $\overline{\sL}$
is the restriction to the diagonal of the unique morphism
from $\pi_\mathrm{mult}(X)_{\overline{k}}$ to $\underline{\Iso}_{X_{\overline{k}}}(\overline{\sL})$ over $X_{\overline{k}}$.
A $k$\nd morphism $f:X \to X'$ then induces the morphism
\begin{equation}\label{e:diagPicf}
D((f_{\overline{k}})^*) \times_k f:D(\Pic(X_{\overline{k}})) \times_k X \to 
D(\Pic(X'{}\!_{\overline{k}})) \times_k {X'}
\end{equation}
by restriction of $f_*$ to the diagonals.

When $k$ is algebraically closed, Corollary~\ref{c:IsoHequivac} with $H = X$ and $\overline{k} = k$ shows
that $\pi_\mathrm{mult}(X)$ is isomorphic over $X$ to $\underline{\Iso}_G(P)$ for some $G$ and $P$,
and that $G$ and the class of $P$ represent the functor $H^1(X,-)$ on $k$\nd groups of multiplicative type.

Let $M$ be a Galois submodule of $\Pic(X_{\overline{k}})$.
Since $H^1(X_{\overline{k}},-)$ on $\overline{k}$\nd groups of multiplicative type is representable
and $\Hom_{\overline{k}}(-,(\bG_m)_{\overline{k}})$
from $\overline{k}$\nd groups of multiplicative type to abelian groups is fully faithful,
there is a unique
\[
\tau \in H^1(X_{\overline{k}},D(M)_{\overline{k}})
\]
such that $\chi_\mu$ sends $\tau$ to $\mu$ in $\Pic(X_{\overline{k}}) = H^1(X_{\overline{k}},\bG_m{}_{\overline{k}})$ 
for every $\mu$ in $M$.
We call $\tau$ the \emph{tautological element}.
The functor $H^1(X_{\overline{k}},-)$ on $\overline{k}$\nd groups of multiplicative type is represented
by $D(\Pic(X_{\overline{k}}))_{\overline{k}}$, with universal element the tautological element.
For $M'$ a Galois submodule of $M$, the projection from $D(M)_{\overline{k}}$ onto $D(M')_{\overline{k}}$ respects
the tautological elements.
The tautological element is fixed by the action of $\Gal(\overline{k}/k)$ on $H^1(X_{\overline{k}},D(M)_{\overline{k}})$ through 
$\overline{k}$.

\begin{thm}\label{t:lineuniv}
Let $X$ be a non-empty smooth geometrically connected projective curve over $k$ of genus $0$.
Then $\pi_\mathrm{mult}(X)$ is universally reductive over $X$.
\end{thm}

\begin{proof}
By Corollary~\ref{c:unscalext}, we may suppose that $k$ is algebraically closed.
The indecomposable representations of $\pi_\mathrm{mult}(X)$ are then those of rank $1$, 
and by \eqref{e:PicmultH} passage to the underlying line bundle
defines an isomorphism from the group of isomorphism classes such representations to $\Pic(X)$.
Condition \ref{i:unequivbij} of Corollary~\ref{c:unequiv} is thus satisfied with $K = \pi_\mathrm{mult}(X)$,
because the indecomposable vector bundles over $X$ are those of rank $1$.
\end{proof}

Let $x$ be a $k$\nd point of $X$.
Given an affine $k$\nd group $G$, we denote by
\[
H^1(X,x,G)
\]
the set of isomorphism classes of pairs $(P,z)$ with $P$ a principal $G$\nd bundle over $X$ and
$z$ a $k$\nd point of $P$ above $x$.
It is pointed set which is contravariant in $(X,x)$ and covariant in $G$.
The functor $H^1(X,x,-)$ does not in general factor through the category of affine $k$\nd groups up to conjugacy:
the action of $G(k)$ on $H^1(X,x,G)$ through its action by conjugation on $G$, or equivalently by shifting the $k$\nd point 
$z$ of $P$, is in general non-trivial.
Discarding $z$ defines a bijection
\begin{equation}\label{e:H1H1tilde}
H^1(X,x,G)/G(k) \iso \widetilde{H}^1(X,x,G)
\end{equation}
onto the subset of $H^1(X,G)$ consisting of the classes of the principal $G$\nd bundles trivial above $x$.
When $G$ is commutative, $G(k)$ acts trivially and $H^1(X,x,G)$ and $\widetilde{H}^1(X,x,G)$ coincide.
Since $H^1(k,\bG_m)$ is trivial, we have
\begin{equation}\label{e:Hilb90}
H^1(X,x,\bG_m) = \widetilde{H}^1(X,x,\bG_m) = H^1(X,\bG_m),
\end{equation}
and similarly with $\bG_m$ replaced by $\bG_a$ because $H^1(k,\bG_a)$ is trivial by \eqref{e:VViso}.

We write $\pi_{\mathrm{\acute{e}t}}(X,x)$, $\pi_\mathrm{mult}(X,x)$ and 
$\pi_\mathrm{\acute{e}tm}(X,x)$
for the respective fibres $\pi_{\mathrm{\acute{e}t}}(X)_{x,x}$, $\pi_\mathrm{mult}(X)_{x,x}$ 
and $\pi_\mathrm{\acute{e}tm}(X)_{x,x}$ above the diagonal.
The $k$\nd groups $\pi_{\mathrm{\acute{e}t}}(X,x)$,
$\pi_\mathrm{mult}(X,x)$ and $\pi_\mathrm{\acute{e}tm}(X,x)$ are functorial in $(X,x)$,
because $\pi_{\mathrm{\acute{e}t}}(X)$, $\pi_\mathrm{mult}(X)$ and 
$\pi_\mathrm{\acute{e}tm}(X)$ are functorial in $X$.
Since $\pi_\mathrm{mult}(X)$ is a commutative groupoid, its diagonal is constant, so that $\pi_\mathrm{mult}(X,x)$
is independent of $x$.
Formation of $\pi_{\mathrm{\acute{e}t}}(X,x)$, $\pi_\mathrm{mult}(X,x)$ and 
$\pi_\mathrm{\acute{e}tm}(X,x)$ commutes with algebraic extension of scalars.

By the equivalence \eqref{e:GPpequiv} and the isomorphism \eqref{e:KKKiso}, the triple
\[
(\pi_\mathrm{\acute{e}tm}(X,x),\pi_\mathrm{\acute{e}tm}(X)_{-,x},1_x)
\]
is initial in the category of $k$\nd pointed principal bundles over $(X,x)$ under a $k$\nd group of pro\'etale by multiplicative type.
Thus  the functor $H^1(X,x,-)$ on $k$\nd groups of pro\'etale by multiplicative type 
is represented by $\pi_\mathrm{\acute{e}tm}(X,x)$,
with the universal element in $H^1(X,x,\pi_\mathrm{\acute{e}tm}(X,x))$ the class of $(\pi_\mathrm{\acute{e}tm}(X)_{-,x},1_x)$.

The morphisms of $k$\nd groups $\pi_\mathrm{\acute{e}tm}(X,x)$ defined by functoriality are those
compatible with the universal elements.
Given a morphism $f$ from $(X.x)$ to $(X',x')$, we thus have a commutative square
\begin{equation}\label{e:H1funct}
\begin{gathered}
\xymatrix{
\Hom_k(\pi_\mathrm{\acute{e}tm}(X',x'),G) \ar[d]_{\Hom_k(f_*,G)} \ar[r]^-{\sim} &  H^1(X',x',G) \ar[d]^{f^*} \\
\Hom_k(\pi_\mathrm{\acute{e}tm}(X,x),G)  \ar[r]^-{\sim} &  H^1(X,x,G)
}
\end{gathered}
\end{equation}
for $G$ of pro\'etale by multiplicative type, where the horizontal arrows are the natural bijections
defined by the universal elements.

There are similar universal properties for $\pi_{\mathrm{\acute{e}t}}(X,x)$ and $\pi_\mathrm{mult}(X,x)$.
The universal elements in $H^1(X,x,\pi_{\mathrm{\acute{e}t}}(X,x))$ and $H^1(X,x,\pi_\mathrm{mult}(X,x))$ are the images
of the universal element in $H^1(X,x,\pi_\mathrm{\acute{e}tm}(X,x))$, and the analogues
\eqref{e:H1funct} hold.

The $k$\nd pointed geometric universal cover $(\widetilde{X},\widetilde{x})$ of $(X,x)$
(i.e.\ the $k$\nd pointed geometric universal $X$\nd cover of $(X,x)$ in the sense of
Section~\ref{s:fundgrpd}) exists and
\begin{equation}\label{e:uncovpiet}
(\widetilde{X},\widetilde{x}) = (\pi_{\mathrm{\acute{e}t}}(X)_{-,x},1_x),
\end{equation}
because $\pi_{\mathrm{\acute{e}t}}(X)_{-,x}$ is a simply transitive $\pi_{\mathrm{\acute{e}t}}(X)$\nd scheme.
It is functorial in $(X,x)$, and its formation commutes with algebraic extension of scalars.
Taking $\widetilde{X}$ for $X'$ in \eqref{e:pietcov} shows that we have a short exact sequence of $k$\nd groups
\begin{equation}\label{e:multetmet}
1 \to \pi_\mathrm{\acute{e}tm}(\widetilde{X},\widetilde{x}) \to \pi_\mathrm{\acute{e}tm}(X,x) \to 
\pi_{\mathrm{\acute{e}t}}(X,x) \to 1,
\end{equation}
functorial in $(X,x)$, with 
$\pi_\mathrm{\acute{e}tm}(\widetilde{X},\widetilde{x}) = \pi_\mathrm{mult}(\widetilde{X},\widetilde{x})$
connected and commutative.

For $G$ a $k$\nd group of pro\'etale by multiplicative type,
we have by the universal property of $\pi_\mathrm{\acute{e}tm}(X,x)$ and the compatibility of $\pi_\mathrm{\acute{e}tm}(X)$
with algebraic extension of scalars a canonical bijection
\[
H^1(X,x,G) \iso H^1(X_{\overline{k}},\overline{x},G_{\overline{k}})^{\Gal(\overline{k}/k)}
\]
with $\overline{x}$ the $\overline{k}$\nd point of $X_{\overline{k}}$ defined by $x$.
When $G = \pi_\mathrm{\acute{e}tm}(X,x)$, this bijection respects the universal elements.
For $G$ of multiplicative type, the bijection becomes
\[
H^1(X,x,G) \iso H^1(X_{\overline{k}},G_{\overline{k}})^{\Gal(\overline{k}/k)}
\]
by \eqref{e:H1H1tilde} and Proposition~\ref{p:printriv}.
In particular, taking $G = D(M)$ with $M$ a Galois submodule of $\Pic(X_{\overline{k}})$ shows that there
is a unique element of $H^1(X,x,D(M))$ with image in $H^1(X_{\overline{k}},D(M)_{\overline{k}})$ the tautological element.

If $\sL$ is a line bundle over $X$, then $\underline{\Iso}_X(\sL)_{x,x} = \bG_m$ and $\underline{\Iso}_X(\sL)_{-,x}$ is 
isomorphic to the principal $\bG_m$\nd bundle associated to $\sL$.
Thus by \eqref{e:diagPicX}
\[
\pi_\mathrm{mult}(X,x) = D(\Pic(X_{\overline{k}}))
\]
with the universal element in $H^1(X,x,\pi_\mathrm{mult}(X,x))$ the unique element with image in
$H^1(X_{\overline{k}},D(\Pic(X_{\overline{k}}))_{\overline{k}})$ the tautological element.

For projective $n$\nd space $\bP^n$ over $k$ with $n > 0$, the assignment to every line bundle of its degree identifies $\Pic(\bP^n{}\!_{\overline{k}})$ 
with the trivial Galois module $\Z$.
Thus
\[
\pi_\mathrm{mult}(\bP^n,x) = \bG_m
\]
for every $n > 0$ and $k$\nd point $x$ of $\bP^n$.
By \eqref{e:Hilb90}, we have
\[
H^1(\bP^n,x,\bG_m) = \widetilde{H}^1(\bP^n,x,\bG_m) = H^1(\bP^n,\bG_m) = \Z,
\]
with the universal element of degree $1$.

\begin{thm}\label{t:linerep}
Let $x$ be a $k$\nd point of $\bP^1$.
Then the functor $\widetilde{H}^1(\bP^1,x,-)$ on the category of reductive $k$\nd groups up to conjugacy is 
represented by $\bG_m$,
with universal element in
\[
\widetilde{H}^1(\bP^1,x,\bG_m) = H^1(\bP^1,\bG_m)
\]
the class of degree $1$. 
\end{thm}

\begin{proof}
With the identification $\pi_\mathrm{mult}(\bP^1,x) = \bG_m$,
a principal $\bG_m$\nd bundle $P$ over $X$ of degree $1$ is isomorphic to $\pi_\mathrm{mult}(\bP^1)_{-,x}$.
Thus by \eqref{e:KKKiso}, $\underline{\Iso}_{\bG_m}(P)$ is isomorphic to $\pi_\mathrm{mult}(\bP^1)$,
and hence by Theorem~\ref{t:lineuniv} is universally reductive over $\bP^1$.
The result now follows from Lemma~\ref{l:cohreppt}\ref{i:unreppt}.
\end{proof}

\begin{rem}
An equivariant analogue of the classification of vector bundles over $\bP^1$ has been given by Kumar \cite{Kum03}.
Theorem~\ref{t:linerep} can be adapted as follows to give a classification of equivariant principal
bundles over $\bP^1$ under reductive $k$\nd groups.
Let $G_0$ be a reductive $k$\nd group acting on $\bP^1$, and 
\[
1 \to \bG_m \to G_1 \to G_0 \to 1
\]
be the pullback along the action $G_0 \to \underline{\Aut}_k(\bP^1) = PGL_2$ of the extension $GL_2$ of $PGL_2$ by $\bG_m$.
Then $G_1$ acts through the projection $G_1 \to GL_2$ on $\bA^2$, and hence on the complement $Y$ of the origin 
in $\bA^2$.
If we write $H_0$, $H_1$, $H_1{}\!'$ for the groupoids $G_0 \times_k \bP^1$, $G_1 \times_k \bA^2$, $G_1 \times_k Y$
over $\bP^1$, $\bA^2$, $Y$ defined by the actions, then by Example~\ref{ex:equivariant}
principal $(H_0,G)$\nd bundles for example are the same as
$G_0$\nd equivariant principal $G$\nd bundles over $\bP^1$. 
Further $H_1{}\!'$ is the pullback of both $H_0$ along $Y \to \bP^1$ and of $H_1$ along $Y \to \bA^2$.
The principal $(G_1,G_1)$\nd bundle $G_1$ gives by pullback along the structural morphism $a$ of $\bA^2$ 
a principal $(H_1,G_1)$\nd bundle $P_1$. 
By Lemma~\ref{l:torspull} the restriction $P_1{}\!'$ of $P_1$ to $Y$ is the pullback along $Y \to \bP^1$ of a
principal $(H_0,G_1)$\nd bundle $P_0$.
For any $x$ in $\bP^1(k)$, it can be seen as follows that $G_1$ represents $\widetilde{H}^1_{H_0}(\bP^1,x,-)$ 
on reductive $k$\nd groups up to conjugacy with universal element the class of $P_0$.
Pullback onto $Y$ defines an equivalence from $\Mod_{H_0}(\bP^1)$ to $\Mod_{H_1{}\!'}(Y)$ by Lemma~\ref{l:prereppull},
and restriction to $Y$ defines an equivalence from $\Mod_{H_1}(\bA^2)$ to $\Mod_{H_1{}\!'}(Y)$ because every vector bundle on $Y$ extends to a vector bundle on $\bA^2$.
It is thus enough to show that $a^*$ from $\Mod_{G_1}(k)$ to $\Mod_{H_1}(\bA^2)$ is bijective on 
isomorphism classes of objects:
since \ref{i:unequivbijpt} of Corollary~\ref{c:unequivpt} clearly holds with $(G_1,G_1,G_1)$ for 
$(H,G,P)$, it will then hold with each of $(H_1,G_1,P_1)$, $(H_1{}\!',G_1,P_1{}\!')$ and 
$(H_0,G_1,P_0)$ for $(H,G,P)$.
If $e$ is the embedding of the origin into $\bA^2$, applying the exact functor 
$\Hom_{H_1}(\sO_{\bA^2},-) = (a_*-)^{G_1}$ to the epimorphism $\sV \to e_*e^*\sV$ in $\MOD_{H_1}(\bA^2)$
and using \eqref{e:dualiso} shows that $e^*$ from $\Mod_{H_1}(\bA^2)$ to $\Mod_{G_1}(k)$ is full.
Since every endomorphism in the kernel of $e^*$ is nilpotent by Cayley--Hamilton, $e^*$ is also conservative.
Thus $\sV$ in $\Mod_{H_1}(\bA^2)$ is isomorphic to $a^*V$ if and only if $V$ is isomorphic to $e^*\sV$.
\end{rem}

Let $M$ be a topological group and $N$ be a topological group with a continuous action of $M$ (by group automorphisms).
Write $H^0(M,N)$ for the subgroup of $N$ of invariants under $M$, and $H^1(M,N)$ for the pointed set of orbits
under the action by conjugation of $N$ on the set of sections $M \to N \rtimes M$ of topological groups to the projection onto $M$.
By projecting onto $N$, sections $M \to N \rtimes M$ may be identified with continuous $1$\nd cocyles on 
$M$ with values in $N$.
Both $H^0(M,N)$ and $H^1(M,N)$ are functorial in $N$.
When $M$ is profinite and $N$ is discrete, $H^0(M,N)$ is the usual group and $H^1(M,N)$ is the usual pointed set.

Let $N$ and $N'$ be topological groups.
We write $\sE(N,N')$ for the set of isomorphism classes of extensions of $N$ by $N'$ (in the sense of Section~\ref{s:grpdgalext},
i.e.\ topologically locally split over $N$).
Given an extension $E'$ of $N$ by $N'$ and an isomorphism $i:N' \iso N''$ of topological groups, 
a pair consisting of an extension $E''$ of $N$ by $N''$ and an isomorphism of extensions from $E'$ to $E''$
with component $1_N$ at $N$ and $i$ at $N'$ exists and is unique up to unique isomorphism.
Thus $i$ induces a map from $\sE(N,N')$ to $\sE(N,N'')$, which sends the class of $E'$ to the class of $E''$.
If $N'' = N'$ and $i$ is an inner automorphism of $N'$, then the map induced by $i$ is the identity. 
Thus we have a functor $\sE(N,-)$ on the category of topological groups and isomorphisms between them up to conjugacy.

Let $M$ be a topological group and $N$ be a commutative topological group with a continuous action of $M$.
Write $H^2(M,N)$ for the subset of $\sE(M,N)$ consisting of the classes of those extensions
for which the action on $N$ through $M$ by conjugation coincides with the given action.
The set $H^2(M,N)$ is functorial in $N$, with $f:N \to N'$ sending the class of $E$ to the class of its push forward 
along $f$, given by the quotient of $N' \rtimes E$ by $N$ embedded as $n \mapsto f(n)^{-1}n$.
Further $H^2(M,-)$ preserves finite products, so that $H^2(M,N)$ has a structure of abelian group.
When $M$ is profinite and $N$ is discrete, the abelian group $H^2(M,N)$
is the usual one.

Given a short exact sequence (i.e.\ an extension)
\begin{equation}\label{e:sextopgrp}
1 \to N' \to N \to N'' \to 1
\end{equation}
of topological groups with a continuous action of $M$, we have a connecting map
\[
H^0(M,N'') \to H^1(M,N')
\]
which sends $n''$ in $H^0(M,N'')$ to the class in  $H^1(M,N')$ of the section that sends $m$ in $M$ to
$nmn^{-1}$ in $N' \rtimes M \subset N \rtimes M$, where $n$ in $N$ lies above $n''$.
Also $N''$ acts on $\sE(M,N')$ through its action up to conjugacy on $N'$, and we have a map
\[
H^1(M,N'') \to \sE(M,N')/N''
\]
which sends the class of $s:M \to N'' \rtimes M$ to the class of the pullback of
\[
1 \to N' \to N \rtimes M \to N'' \rtimes M \to 1
\] 
along $s$.
We then have a cohomology sequence
\begin{multline}\label{e:lextopgrp}
1 \to H^0(M,N') \to H^0(M,N) \to H^0(M,N'') \to \\
\to H^1(M,N') \to H^1(M,N) \to H^1(M,N'') \to \sE(M,N')/N'',
\end{multline}
with the other maps defined by functoriality.
It is exact in the first five pairs of maps, and for the final pair the image of $H^1(M,N)$ in $H^1(M,N'')$ is the
inverse image of the subset $\sE_0(M,N')/N''$ of $\sE(M,N')/N''$, where $\sE_0(M,N')$ consists of the classes of the split extensions.
The sequence \eqref{e:lextopgrp} is functorial for those morphisms from \eqref{e:sextopgrp}
which are isomorphisms on $N'$.
When $N'$ is central in $N$, the action of $N''$ on $\sE(M,N')$ is trivial,
and $\sE(M,N')/N''$ in \eqref{e:lextopgrp} may be replaced by its subset
$H^2(M,N')$, which $\sE_0(M,N')/N''$ intersects in the $0$ element.
When $M$ is profinite, $N$, $N'$ and $N''$ are discrete, and $N'$ is central, \eqref{e:lextopgrp} with $\sE(M,N')/N''$
replaced by $H^2(M,N')$ is the usual long exact sequence.

The Galois group $\Gal(\overline{k}/k)$ acts on the set $G(\overline{k})$ of $\overline{k}$\nd points of an affine 
$k$\nd group $G$ through its action on $\overline{k}$.
This action is continuous for the Krull topology on $\Gal(\overline{k}/k)$ and
(in the sense of Section~\ref{s:grpdgalext}) on $G(\overline{k})$.
The cohomology sets
\[
H^n(\Gal(\overline{k}/k),G(\overline{k}))
\]
for $n = 0,1$, and when $G$ is commutative for $n=2$, will be understood
to be those for which $\Gal(\overline{k}/k)$ acts on $G(\overline{k})$ through $\overline{k}$.
Similarly when $G$ is commutative, extensions of $\Gal(\overline{k}/k)$ by $G(\overline{k})$ will be understood
to those for which $\Gal(\overline{k}/k)$ acts through $\overline{k}$.
For $H^1$ we have the natural bijection \eqref{e:Hbij}, which for $G$ of finite type reduces to 
the usual description of $H^1(k,G)$ as a Galois cohomology set.

Suppose now that $X$ is a Severi--Brauer variety over $k$ of dimension $n$.
Fix an isomorphism of $\overline{k}$\nd schemes
\[
i:\bP^n{}\!_{\overline{k}} \iso X_{\overline{k}}.
\]
With the identification $PGL_{n+1} = \underline{\Aut}_k(\bP^n)$,
we have a cartesian square
\[
\xymatrix{
F_i \ar[d] \ar[r] & [\overline{k}] \ar[d] \\
 GL_{n+1} \times_k [\overline{k}] \ar[r] & PGL_{n+1} \times_k [\overline{k}]
}
\]
of groupoids over $\overline{k}$, 
where the right arrow sends $(s_1,s_0)$ to $(i_{s_1}{}\!^{-1} \circ i_{s_0},s_1,s_0)$,
the bottom arrow is the constant morphism, and the left arrow is the embedding of a
transitive affine subgroupoid $F_i$ of $GL_{n+1} \times_k {[\overline{k}]}$ over $\overline{k}$ with
\[
F_i{}\!^\mathrm{diag} = \bG_m \times_k \overline{k}.
\]
The homomorphism from $\Gal(\overline{k}/k)$ to $PGL_{n+1}(\overline{k}) \rtimes \Gal(\overline{k}/k)$
induced as in Section~\ref{s:grpdgalext} by the right arrow on $\overline{k}$\nd points over $\overline{k}$
sends $\sigma$ to $(i^{-1} \circ i_{\sigma})\sigma$.
Its class in 
\[
H^1(\Gal(\overline{k}/k),PGL_{n+1}(\overline{k}))
\]
is thus the usual class of the Severi--Brauer variety $X$ over $k$ \cite[p. 168]{Ser68}. 
The class of the extension $F_i(\overline{k})_{\overline{k}}$ of $\Gal(\overline{k}/k)$ by $\bG_m(\overline{k}) = \overline{k}{}^*$ in
$H^2(\Gal(\overline{k}/k),\bG_m(\overline{k}))$
is then the image under connecting map associated to the short exact sequence
\[
1 \to \bG_m(\overline{k}) \to GL_{n+1}(\overline{k}) \to PGL_{n+1}(\overline{k}) \to 1
\]
of the class of $X$ in $H^1(\Gal(\overline{k}/k),PGL_{n+1}(\overline{k}))$.
It is thus the class of $X$ in the Brauer group $H^2(\Gal(\overline{k}/k),\bG_m(\overline{k}))$
of $k$ \cite[p. 166]{Ser68}.

Suppose that $n > 0$.
Define the degree of a line bundle over $X_{\overline{k}}$, or equivalently of a principal
$\bG_m{}_{\overline{k}}$\nd bundle over $X_{\overline{k}}$, as the degree (independent of the choice of $i$)
of its pullback along $i$.
Write $Q$ for the affine space $\bA^{n+1}$ with the origin removed.
There is a canonical projection from $Q$ to $\bP^n$,
and the standard linear action of $GL_{n+1}$ on $\bA^{n+1}$ induces an action on $Q$ above the 
action of $PGL_{n+1}$ on $\bP^n$.
By restricting to $\bG_m$ the corresponding right action of $GL_{n+1}$,
defined using the inverse involution, we obtain a structure on $Q$ of principal 
$\bG_m$\nd bundle over $\bP^n$ of degree $1$. 
Write $P_i$ for $Q_{\overline{k}}$ regarded as a scheme over $X_{\overline{k}}$ with structural morphism 
\[
Q_{\overline{k}} \to \bP^n\!_{\overline{k}} \xrightarrow{i} X_{\overline{k}}.
\]
The action of $GL_{n+1}$ on $Q$ defines an action of $GL_{n+1} \times_k {[\overline{k}]}$ on $Q_{\overline{k}}$.
Restricting to $F_i$, we obtain an action of $F_i$ on $P_i$ above the trivial action on $X_{\overline{k}}$.
The corresponding right action of $F_i$ defines on $P_i$ a structure of principal $F_i$\nd bundle,
and hence of principal $(F_i{}\!^\mathrm{diag},F_i(\overline{k})_{\overline{k}})$\nd bundle, over $X$, with underlying
principal bundle over $X_{\overline{k}}$ under 
$F_i{}\!^\mathrm{diag} = \bG_m{}_{\overline{k}}$ of degree $1$.

\begin{thm}\label{t:gen0rep}
Let $X$ be a non-empty smooth geometrically connected projective curve over $k$ of genus $0$, and $E$ be
a topological extension of $\Gal(\overline{k}/k)$ by $\bG_m(\overline{k})$ with class in the
Brauer group $H^2(\Gal(\overline{k}/k),\bG_m(\overline{k}))$ of $k$ that of $X$.
Then there exists a unique element $\alpha$ of $H^1(X,\bG_m{}_{\overline{k}},E)$ with image under
\[
H^1(X,\bG_m{}_{\overline{k}},E) \to H^1(X_{\overline{k}},\bG_m{}_{\overline{k}})
\]
of degree $1$.
The functor $H^1(X,-,-)$ on the category of reductive Galois extended $\overline{k}$\nd groups up to conjugacy 
is represented by $(\bG_m{}_{\overline{k}},E)$ with universal element $\alpha$. 
\end{thm}

\begin{proof}
We may suppose that $E = F_i(\overline{k})_{\overline{k}}$ for some 
$i:\bP^1{}\!_{\overline{k}} \iso X_{\overline{k}}$.
Taking the class of $P_i$ then shows that an $\alpha$ exists.
Since $i$ defines an isomorphism from $H(X_{\overline{k}},-)$ to 
$H^1(\bP^1{}\!_{\overline{k}},-) = \widetilde{H}^1(\bP^1{}\!_{\overline{k}},x,-)$ for any $\overline{k}$\nd point
$x$ over $\overline{k}$, the representation statement holds for any $\alpha$ by 
Corollary~\ref{c:unscalextac} and Theorem~\ref{t:linerep}.
If both $\alpha$ and $\alpha'$ have image in $H^1(X_{\overline{k}},\bG_m{}_{\overline{k}})$ of degree $1$,
there is thus an endomorphism $(h,l)$ of $(\bG_m{}_{\overline{k}},E)$ with $h$ the identity which sends 
$\alpha$ to $\alpha'$.
The map $e \mapsto l(e)e^{-1}$ factors through a continuous $1$\nd cocycle of $\Gal(\overline{k}/k)$ with values
in $\bG_m(\overline{k})$, which is a $1$\nd coboundary because $H^1(k,\bG_m)$ is trivial.
Thus $l$ is conjugate to the identity, and $\alpha' = \alpha$.
\end{proof}

\begin{exmp}
Suppose that $k = \R$ and $\overline{k} = \C$, 
and let $X$ be as in Theorem~\ref{t:gen0rep} with $X(\R)$ empty.
Then $E$ in Theorem~\ref{t:gen0rep} is the non-split extension of $\Gal(\C/\R) = \Z/2$ by $\bG_m(\C)$.
Explicitly, $E$ is the quotient of $\bG_m(\C) \rtimes (\Z/4)$ by $(-1,2)$, where $\Z/4$ acts through $\Z/2$ 
by complex conjugation.
If $G$ is a reductive $\R$\nd group, then by Theorem~\ref{t:gen0rep} principal $G$\nd bundles over $X$ 
are classified by morphisms from $(\bG_m{}_{\C},E)$ to $(G_{\C},G(\C) \rtimes (\Z/2))$ up to conjugacy.
Such morphisms may be identified with pairs $(h,\gamma)$ with $h$ a $\C$\nd homomorphism from $\bG_m{}_{\C}$
to $G_{\C}$ and $\gamma$ in $G(\C)$ such that $(\gamma,1)^2 = (h(-1),0)$ in $G(\C) \rtimes (\Z/2)$, i.e.\ 
\[
\gamma \overline{\gamma} = h(-1),
\]
and $(\gamma,1)(h(z),0)(\gamma,1)^{-1} = (h(\overline{z}),0)$ for every $z$ in $\C$, i.e.
\[
\gamma\overline{h(z)}\gamma^{-1} = h(\overline{z}).
\]
The conjugate ${}^g(h,\gamma)$ of $(h,\gamma)$ by $g$ in $G(\C)$ is then $({}^g h,g\gamma \overline{g}{}^{-1})$.
Denote by $G_0$ the real form of $\bG_m{}_{\C} \times_{\C} (\Z/2)$ with
$\overline{(z,n)} = ((-1)^n\overline{z},n)$.
It is a non-split extension of $\Z/2$ by $\bG_m$.
Up to isomorphism, there is a unique principal $G_0$\nd bundle $P_0$ over $X$ with $P_0{}_{\C}$ the push forward 
of the principal $\bG_m{}_{\C}$ bundle over $X_{\C}$ of degree $1$ along the embedding 
$h_0:\bG_m{}_{\C} \to G_0{}_{\C}$, corresponding to the conjugacy class of those $(h_0,\gamma_0)$ where 
$\gamma_0 = (z,1)$ with $z\overline{z} = 1$.
In \cite[1.1 and 4.12]{Bis08}, it is asserted that for $G$ connected reductive of finite type over $\R$, any principal $G$\nd bundle $P$ over $X$
is the push forward of $P_0$ along an $\R$\nd homomorphism $G_0 \to G$.
As it stands, this assertion is incorrect: for a counter-example take $G = SO_3$ and for $P$ the twist by a principal 
$(PGL_2,SO_3)$\nd bundle of the push forward along an embedding $\bG_m \to PGL_2$ of a non-trivial principal 
$\bG_m$\nd bundle over $X$.
However the assertion is correct under the additional hypothesis that $G$ be \emph{split}.
To see this it is enough to show that any $(h,\gamma)$ is conjugate to one with $h$ defined over $\R$ and 
$\gamma^2 = 1$, because such a conjugate will factor as $(h_0,\gamma_0)$ with $\gamma_0 = (1,1)$ followed by
the morphism defined by $G_0 \to G$ with $(z,n) \mapsto h(z)\gamma^n$.
By hypothesis, $G$ has a split torus $T$ with $T_{\C}$ maximal in $G_{\C}$.
Since $h$ has a conjugate which factors through $T_{\C}$, we may assume that $h = j_{\R}$ is defined over $\R$.
The centraliser $Z(j)$ of $j(\bG_m)$ in $G$ is then a connected reductive $\R$\nd subgroup of $G$ containing $\gamma$.
We now show that if $G'$ is a connected reductive $\R$\nd group of finite type then for any $\gamma'$
in $G'(\C)$ with $\gamma'\overline{\gamma}{}' = 1$ there is a $g'$ in $G'(\C)$ with 
$(g'\gamma'\overline{g}{}'{}^{-1})^2 = 1$.
Taking $G' = Z(j)/j(\bG_m)$ with $\gamma'$ the image of $\gamma$ and appropriately lifting $g'$ to $Z(j)$ 
will then give what is required. 
Write $\gamma'$ as $\gamma_u\gamma_s = \gamma_s\gamma_u$ with $\gamma_s$ semisimple and $\gamma_u$ unipotent.
Then $\gamma_s\overline{\gamma}_s = 1 = \gamma_u\overline{\gamma}_u$.
There is thus a connected reductive $\R$\nd subgroup $G''$ of $G'$ with $G''{}\!_{\C}$ the 
identity component of the centraliser of $\gamma_s$.
Further $\gamma_s$ lies in $G''(\C)$ because $\gamma_s$ is contained in a maximal torus of $G'{}\!_{\C}$,
and there is an $\R$\nd subgroup $\bG_a$ of $G''$ with $\gamma_u$ in $\bG_a(\C)$.
Replacing $\gamma'$ by $g'\gamma'\overline{g}{}'{}^{-1}$ with $g' = -\gamma_u/2$ in $\bG_a(\C)$, 
we may assume that $\gamma' = \gamma_s$.
Thus $\gamma'$ lies in the centre of $G''$, and hence in any maximal torus $T'$ of $G''$.
Then $(g'\gamma'\overline{g}{}'{}^{-1})^2 = 1$ for $g'$ in $T'(\C)$ 
with $g'{}^{-2} = \gamma'$.
\end{exmp}

\section{Curves of genus 1}\label{s:curves1}

\emph{In this section $k$ is a field of characteristic $0$ and $\overline{k}$ is an algebraic closure of $k$.}

\medskip

In this section we describe principal bundles with reductive structure group over a curve $X$ over $k$ of genus $1$.
The study of such bundles is reduced by Theorem~\ref{t:elluniv} below to the study of the universal
groupoid $\pi_\mathrm{\acute{e}tm}(X)$ of pro\'etale by multiplicative type.
The groupoid $\pi_\mathrm{\acute{e}tm}(X)$ will in fact be determined more generally for any $X$ with
$X_{\overline{k}}$ isomorphic to an abelian variety.
The main applications to curves of genus $1$ are Theorem~\ref{t:ellrep}, which deals with the case where
$X$ has a $k$\nd point, and Theorem~\ref{t:gen1rep}, which deals with the general case.

Throughout this section we use the following notation.
If $M$ is an abelian group, then ${}_nM$ denotes the $n$\nd torsion subgroup 
and ${}_{\mathrm{tors}}M$ the torsion subgroup of $M$.
If $G$ is a commutative group scheme over $k$, then $n_G$ denotes the multiplication by the integer $n$ on $G$,
and ${}_nG$ denotes the kernel of $n_G$.

Let $X$ be an abelian variety over $k$, with base point $x$.
We denote by $(\widetilde{X},\widetilde{x})$ the $k$\nd pointed geometric universal cover of $(X,x)$.
Thus $\widetilde{X}$ is geometrically simply connected, and
$(\widetilde{X},\widetilde{x})$ is initial in the category of  $k$\nd pointed pro\'etale covers of $(X,x)$.
We denote by $\widehat{X}$ the dual abelian variety of $X$, and by $NS(X_{\overline{k}})$ the N\'eron--Severi group of
$X_{\overline{k}}$.
Thus $NS(X_{\overline{k}})$ is a Galois module with finitely generated free underlying abelian group, and the
kernel of the projection from $\Pic(X_{\overline{k}})$ onto $NS(X_{\overline{k}})$ is the Galois submodule
\[
\widehat{X}(\overline{k}) = \Pic^0(X_{\overline{k}})
\]
of $\Pic(X_{\overline{k}})$ consisting of the classes of those line bundles algebraically equivalent to $0$.
For $y$ a $\overline{k}$\nd point of $X$,
we write 
\begin{equation}\label{e:translation}
T_y:X_{\overline{k}} \iso X_{\overline{k}}
\end{equation}
for the translation automorphism of the $\overline{k}$\nd scheme $X_{\overline{k}}$ that sends $x$ to $y$.
Given $\nu$ in $NS(X_{\overline{k}})$, we write
\begin{equation}\label{e:phinu}
\varphi_\nu:X_{\overline{k}} \to \widehat{X}_{\overline{k}}
\end{equation}
for the morphism over $\overline{k}$ that sends $y$ to the class of $(T_y)^*\sL \otimes \sL^\vee$, where $\sL$ is
a line bundle over $X$ with class $\nu$.

\begin{lem}\label{l:H1dim1}
Let $X_0$ be a smooth geometrically connected projective curve over $k$ of genus $1$, and $X$ be a
non-empty geometrically connected pro\'etale
cover of $X_0$. 
Then the canonical homomorphism from $H^1(X_0,\sO_{X_0})$ to $H^1(X,\sO_X)$ is an isomorphism.
\end{lem}

\begin{proof}
We may write $X$ as the filtered limit of finite \'etale schemes $X_\lambda$ over $X_0$, with each $X_\lambda$ 
a smooth geometrically connected projective curve over $k$ of genus $1$.
If $X_\lambda$ is the spectrum of the commutative $\sO_{X_0}$\nd algebra $\sR_\lambda$, then
\[
H^1(X,\sO_X) = H^1(X_0,\colim \sR_\lambda) = \colim H^1(X_0,\sR_\lambda),
\] 
with the interchange of colimits justified by a \v{C}ech calculation.
For each $\lambda$, the embedding $\sO_{X_0} \to \sR_\lambda$ has a left inverse $(1/r)\tr_{\sR_\lambda}$, with $r > 0$
the rank of $\sR_\lambda$. 
Thus the homomorphism from $H^1(X_0,\sO_{X_0})$ to $H^1(X_0,\sR_\lambda)$ is injective, 
and hence an isomorphism
since its source and target are $1$\nd dimensional.
The result follows.
\end{proof}

\begin{lem}\label{l:ellindecsym}
Let $X_0$ be a smooth geometrically connected projective curve over $k$ of genus $1$, and $X$ be a 
non-empty geometrically connected pro\'etale
cover of $X_0$. 
Then non-split extensions of $\sO_X$ by $\sO_X$ exist, and any two of them are isomorphic as $\sO_X$\nd modules.
If $\sE$ is such a non-split extension, then every symmetric power $S^r\sE$ of $\sE$ is indecomposable as an $\sO_X$\nd module.
\end{lem}

\begin{proof}
The existence and uniqueness statements follow from Lemma~\ref{l:H1dim1}.

To prove that $S^r\sE$ is indecomposable for $r=1$,
suppose the contrary.
Then $\sE$ is the direct sum of line bundles $\sE'$ and $\sE''$.
We may assume $H^0(X,\sE'{}^\vee)$ is $0$,
because $\sE^\vee$ is a non-split extension of $\sO_X$ by $\sO_X$ and hence $H^0(X,\sE^\vee)$ is $1$\nd dimensional.
The epimorphism from $\sE$ to $\sO_X$ then induces an epimorphism and hence an isomorphism from $\sE''$ to $\sO_X$.
Thus $\sE$ is isomorphic to $\sO_X \oplus \sO_X$, which is impossible.

To prove that $S^r\sE$ is indecomposable for arbitrary $r$, we may suppose after extending the scalars that 
$k$ is algebraically closed and that $X$ has a $k$\nd point $x$.
If $x$ lies above the $k$\nd point $x_0$ of $X_0$, we may  by Lemma~\ref{l:H1dim1} assume further,
after pullback along the unique morphism over $(X_0,x_0)$ from $(\widetilde{X}_0,\widetilde{x}_0)$ to $(X,x)$,
that $X = \widetilde{X}_0$.
It then suffices to use Example~\ref{ex:tenspower} with $H = X$.
\end{proof}

Let $X$ be an abelian variety over $k$, with base point $x$.
If we write $(X_{(n)},x_{(n)})$ for $(X,x)$ regarded as a $k$\nd pointed scheme over itself with structural morphism $n_X$,
then it can be seen as follows that
\begin{equation}\label{e:ucoverlim}
(\widetilde{X},\widetilde{x}) = \lim_{n \ne 0} (X_{(n)},x_{(n)}).
\end{equation}
Since $X_{(n)}$ is finite \'etale over $X$ for $n \ne 0$, it is enough
to check that every $k$\nd pointed finite \'etale scheme $(X',x')$ over $(X,x)$ is the target of a unique
morphism over $(X,x)$ from the limit in \eqref{e:ucoverlim}.
The uniqueness is clear, because the equaliser of two such morphisms is an open and closed subscheme of 
$\lim_{n \ne 0} X_{(n)}$ containing $(x_{(n)})$. 
For the existence, we may suppose $X'$ connected and $k$ algebraically closed.
Then by a theorem of Lang and Serre \cite[p.167, Theorem]{Mum70} $X'$ has a (unique) structure of abelian variety with base point $x'$,
and for some $n \ne 0$ there exists \cite[p.169, Remark]{Mum70} a morphism from $X_{(n)}$ to $X'$ over $X$.

We have an extension of Galois modules
\[
0 \to \widehat{X}(\overline{k}) \to \Pic(X_{\overline{k}}) \to NS(X_{\overline{k}}) \to 0,
\]
where $n_X$ acts as $n$ on $\widehat{X}(\overline{k})$ and as $n^2$ on $NS(X_{\overline{k}})$
\cite[p.75, (iii) and (iv)]{Mum70}.
By \eqref{e:ucoverlim} together with Lemmas~\ref{l:colimHomR} and \ref{l:HRrepfp},
$\Pic(\widetilde{X}_{\overline{k}})$ is the filtered colimit of copies of $\Pic(X_{\overline{k}})$
with transition homomorphisms induced by morphisms $m_X$.
Thus $\Pic(\widetilde{X}_{\overline{k}})$ is uniquely divisible, 
and $\Pic(X_{\overline{k}}) \to \Pic(\widetilde{X}_{\overline{k}})$ 
induces an isomorphism 
\begin{equation}\label{e:PictensQ}
\Pic(X_{\overline{k}})_\Q \iso \Pic(\widetilde{X}_{\overline{k}})
\end{equation}
of Galois modules.

Suppose that $k$ is algebraically closed and that $X$ is an elliptic curve.
Then \cite[Theorem~5(i)]{Ati57} for $r > 0$ there is an indecomposable vector bundle $\sF_r$ over $X$ of rank $r$ 
and degree $0$
with $H^0(X,\sF_r) \ne 0$, and $\sF_r$ is unique up to isomorphism.
Further \cite[Theorem~5(ii)]{Ati57} any indecomposable vector bundle over $X$ of rank $r$ and degree $0$ is isomorphic to 
$\sL_0 \otimes_{\sO_X} \sF_r$ for a line bundle $\sL_0$ over $X$.
If $\sE_0$ is a non-split extension of $\sO_X$ by $\sO_X$, it follows from Lemma~\ref{l:ellindecsym} that 
$\sF_{r+1}$ is isomorphic to $S^r\sE_0$ for $r \ge 0$, because the $r$th symmetric power of the embedding of $\sO_X$ into $\sE_0$ defines a non-zero element of $H^0(X,S^r\sE_0)$.

\begin{lem}\label{l:elllinesym}
Let $X$ be an elliptic curve over $k$,
and $\sE$ be a non-split extension of $\sO_{\widetilde{X}}$ by $\sO_{\widetilde{X}}$.
Suppose that $k$ is algebraically closed.
Then any indecomposable vector bundle over $\widetilde{X}$ of rank 
$r+1$ is isomorphic to $\sL \otimes_{\sO_{\widetilde{X}}} S^r\sE$ for a line bundle $\sL$ over $\widetilde{X}$,
determined uniquely up to isomorphism.
\end{lem}

\begin{proof}
Let $\sV$ be an indecomposable vector bundle over $\widetilde{X}$ of rank $r+1$.
If a line bundle $\sL$ over $\widetilde{X}$ with 
$\sV$ isomorphic to $\sL \otimes_{\sO_{\widetilde{X}}} S^r\sE$ 
exists, then by \eqref{e:PictensQ} it is unique up to isomorphism, 
because $\sL^{\otimes (r+1)}$ must be isomorphic to
the determinant of $\sV$.
To prove that $\sL$ exists, we may suppose by \eqref{e:PictensQ} that $\sV$ has trivial determinant.
By Lemmas~\ref{l:colimHomR} and \ref{l:HRrepfp}, 
there is a projection $\widetilde{X} \to X_{(n)}$ of \eqref{e:ucoverlim} along which $\sV$ 
is the pullback of a vector bundle $\sV_0$ over $X = X_{(n)}$
and $\sE$ is the pullback of an extension $\sE_0$ of $\sO_X$ by $\sO_X$.
Then $\sV_0$ is indecomposable, with determinant of degree $0$ by \eqref{e:PictensQ}, and $\sE_0$ is non-split.
By the above, $\sV_0$ is isomorphic to $\sL_0 \otimes_{\sO_X} S^r\sE_0$ for a line
bundle $\sL_0$ over $X$.
\end{proof}

Taking $H = X$ and $\sV = \sO_X$ in \eqref{e:VViso} gives an isomorphism
\begin{equation}\label{e:OXGaiso}
H^1(X,\sO_X) \iso H^1(X,\bG_a)
\end{equation}
which sends the class of the extension $\sE$ of $\sO_X$ by $\sO_X$ to the class of the principal $\bG_a$\nd bundle 
$P_0$ over $X$
of splittings $\sO_X \to \sE$ of $\sE$,
where $\bG_a$ acts by restricting along the embedding $\sO_X \to \sE$ the translation of $\sE$.
Fix an $\sE$ and write $P_1$ for the scheme over $X$ of isomorphisms from $\sO_X^2$ to $\sE$.
Composition on the left and right defines on $P_1$
a structure of principal $(\underline{\Iso}_X(\sE),GL_2)$\nd bundle with the action of $\underline{\Iso}_X(\sE)$ simply transitive.
Thus by Lemma~\ref{l:simpleprin} we have an isomorphism
\begin{equation}\label{e:IsoEIsoP}
\underline{\Iso}_X(\sE) \iso \underline{\Iso}_{GL_2}(P_1)
\end{equation}
 of groupoids over $X$.
The morphism $\alpha \mapsto (\varepsilon,\alpha)$ from $P_0$ to $P_1$ with $\varepsilon:\sO_X \to \sE$ the embedding
is compatible with the actions of $\bG_a$ and $GL_2$ and the
upper triangular embedding $\bG_a \to GL_2$.
Thus $P_1$ is the push forward of $P_0$ along $\bG_a \to GL_2$.

For $X_0$ and $X$ as in Lemma~\ref{l:H1dim1}, it follows from \eqref{e:OXGaiso} that the sets of isomorphism 
classes of non-trivial principal $\bG_a$\nd bundles over $X_0$ and over $X$ are non-empty, and that pullback
from $X_0$ to $X$ defines a bijection between them.

\begin{thm}\label{t:elluniv}
Let $X_0$ be a smooth geometrically connected projective curve over $k$ of genus $1$, and 
$X$ be a non-empty geometrically connected pro\'etale
cover of $X_0$. 
Let $P$ be the push forward of a non-trivial principal $\bG_a$\nd bundle
over $X$ along an embedding $\bG_a \to SL_2$.
Then $\pi_\mathrm{\acute{e}tm}(X) \times_{[X]} \Iso_{SL_2}(P)$ is universally reductive over $X$.
\end{thm}

\begin{proof}
By Corollary~\ref{c:minscalext}, we may suppose that $k$ is algebraically closed.
Then $X_0$ has a $k$\nd point $x_0$.
Applying Proposition~\ref{p:etunpull} and \eqref{e:pietcov} to $X \to X_0$ and $\widetilde{X}_0 \to X_0$,
we may suppose further that $X = \widetilde{X}_0$.
Then $\pi_\mathrm{\acute{e}tm}(X)$ coincides with $\pi_\mathrm{mult}(X)$, and it is to be shown that
\[
K = \pi_\mathrm{mult}(X) \times_{[X]} \Iso_{SL_2}(P)
\]
is universally reductive over $X$.

Restricting to $\widetilde{x}_0$ and using Lemma~\ref{l:prereppull}
shows that the indecomposable representations of $K$ are the tensor products of 
those of $\pi_\mathrm{mult}(X)$ with those of $\Iso_{SL_2}(P)$.
The indecomposable representations of $\pi_\mathrm{mult}(X)$ are those of rank $1$, and by \eqref{e:PicmultH},
discarding the action of $\pi_\mathrm{mult}(X)$ defines an isomorphism from the group of isomorphism classes 
such representations to $\Pic(X)$.
If $P$ is the push forward of the principal $\bG_a$\nd bundle $P_0$ over $X$ corresponding under \eqref{e:OXGaiso}
to the non-split extension $\sE$ of $\sO_X$ by $\sO_X$,
then $P_1$ in \eqref{e:IsoEIsoP} is the push forward of $P$ along
the embedding of $SL_2$ into $GL_2$.
Thus $\underline{\Iso}_{SL_2}(P)$ embeds into $\underline{\Iso}_{GL_2}(P_1)$,
and $\sE$ has by \eqref{e:IsoEIsoP} a structure of faithful 
representation of $\underline{\Iso}_{SL_2}(P)$.
Restricting to $\widetilde{x}_0$ shows that the indecomposable representations of $\underline{\Iso}_{SL_2}(P)$ are the 
symmetric powers of $\sE$.
By Lemma~\ref{l:elllinesym}, condition \ref{i:unequivbij} of Corollary~\ref{c:unequiv}
is thus satisfied.
\end{proof}

Let $X$ be an abelian variety over $k$ with base point $x$.
By \eqref{e:ucoverlim}, formation of the  $k$\nd pointed geometric universal cover $(\widetilde{X},\widetilde{x})$
of $(X,x)$ commutes with finite products of abelian varieties.
Since the unique action of $\pi_\mathrm{\acute{e}t}(X)$ on $\widetilde{X}$ is simply transitive, it follows
that the same holds for formation of the groupoid in $k$\nd schemes $(X,\pi_\mathrm{\acute{e}t}(X))$.
By functoriality of $\pi_\mathrm{\acute{e}t}(X)$, there is thus a unique structure
\begin{equation}\label{e:pietgroup}
(X \times_k X, \pi_\mathrm{\acute{e}t}(X) \times_k \pi_\mathrm{\acute{e}t}(X)) \to
(X,\pi_\mathrm{\acute{e}t}(X))
\end{equation}
of commutative group on $(X,\pi_\mathrm{\acute{e}t}(X))$ above that on $X$.
The usual $k$\nd group structure on $\pi_\mathrm{\acute{e}t}(X,x)$ defined by composition necessarily coincides with the one 
defined by restriction of \eqref{e:pietgroup}, and is thus commutative.

By functoriality of $(\widetilde{X},\widetilde{x})$, there is a unique structure of commutative $k$\nd group 
on $\widetilde{X}$ with identity $\widetilde{x}$ such that $\widetilde{X} \to X$ is a morphism of $k$\nd groups.
This structure necessarily coincides with that given by \eqref{e:ucoverlim}, and with that given by restriction of
\eqref{e:pietgroup} using \eqref{e:uncovpiet}.
In particular
\begin{equation}\label{e:pietKern}
\pi_\mathrm{\acute{e}t}(X,x) = \lim_{n \ne 0} {}_nX
\end{equation}
is the kernel of $\widetilde{X} \to X$.
Further the structure of principal $\pi_\mathrm{\acute{e}t}(X,x)$\nd bundle over $X$ on $\pi_\mathrm{\acute{e}t}(X)_{-,x}$
defined by composition coincides with that defined by translation by the $k$\nd subgroup $\pi_\mathrm{\acute{e}t}(X,x)$.

Since $\pi_\mathrm{\acute{e}t}(X,x)$ is commutative, it is the maximal pro\'etale $k$\nd quotient of
$\pi_\mathrm{mult}(X,x) = D(\Pic(X_{\overline{k}}))$.
Thus
\begin{equation}\label{e:DPicpi}
\pi_\mathrm{\acute{e}t}(X,x) = D({}_\mathrm{tors}\widehat{X}(\overline{k})),
\end{equation}
where the universal element in $H^1(X,x,D({}_\mathrm{tors}\widehat{X}(\overline{k})))$ is the unique element with
image in $H^1(X_{\overline{k}},D({}_\mathrm{tors}\widehat{X}(\overline{k}))_{\overline{k}})$ the tautological element.

It follows from \eqref{e:pietKern} and \eqref{e:DPicpi} that for $n \ne 0$
\begin{equation}\label{e:nXDnXhat}
{}_n X = D({}_n \widehat{X}(\overline{k})),
\end{equation}
where $\chi_\lambda:{}_nX_{\overline{k}} \to \bG_m{}_{\overline{k}}$ for $\lambda$ in 
${}_n\widehat{X}(\overline{k}) \subset \Pic(X_{\overline{k}})$ is the unique character with the following property:
if $X_{(n)}$ as in \eqref{e:ucoverlim} is regarded as a principal ${}_nX$\nd bundle over $X$ by translation,
then $\lambda$ is the class of
the push forward of $X_{(n)}{}_{\overline{k}}$ along $\chi_\lambda$.
For $m \ne 0$, \eqref{e:nXDnXhat} is compatible with the projection ${}_{mn}X \to {}_n X$,
and hence, as follows by factoring $m_{{}_{mn} X}$ through ${}_n X$, with the embedding ${}_n X \to {}_{mn}X$.

Given $n \ne 0$, \eqref{e:pietcov} with $X' = X_{(n)}$ shows that
the endomorphism $n_X{}_*$ of $\pi_\mathrm{\acute{e}tm}(X)$ induces by restriction an automorphism
\begin{equation}\label{e:nXpietmx}
n_X{}_*:\pi_\mathrm{\acute{e}tm}(X)_{-,x} \iso \pi_\mathrm{\acute{e}tm}(X)_{-,x}
\end{equation}
of $\pi_\mathrm{\acute{e}tm}(X)_{-,x}$ over $n_X$.
When $\pi_\mathrm{\acute{e}tm}$ is
replaced by $\pi_\mathrm{\acute{e}t}$, \eqref{e:nXpietmx} is the automorphism $n_{\widetilde{X}}$
of $\pi_\mathrm{\acute{e}t}(X)_{-,x} = \widetilde{X}$. 
A point of $\pi_\mathrm{\acute{e}tm}(X)$ above $(x,x)$ with image $y$ under
\[
\pi_\mathrm{\acute{e}tm}(X,x) \to \pi_\mathrm{\acute{e}t}(X,x) \to {}_n X,
\]
is the image under \eqref{e:nXpietmx} of a point above $(y,x)$.

We may identify $\pi_\mathrm{\acute{e}tm}(\widetilde{X},\widetilde{x}) = \pi_\mathrm{mult}(\widetilde{X},\widetilde{x})$ with
$D(\Pic(X_{\overline{k}})_\Q)$ using \eqref{e:PictensQ}.
The embedding of \eqref{e:multetmet} then becomes
\begin{equation}\label{e:DPicQembed}
D(\Pic(X_{\overline{k}})_\Q) \to \pi_\mathrm{\acute{e}tm}(X,x).
\end{equation}
Both \eqref{e:DPicQembed} and the projection 
\begin{equation}\label{e:DPicproject}
\pi_\mathrm{\acute{e}tm}(X,x) \to D(\Pic(X_{\overline{k}}))
\end{equation}
onto $\pi_\mathrm{mult}(X,x)$ are natural in the abelian variety $X$.
The composite of \eqref{e:DPicproject} with \eqref{e:DPicQembed} is 
the $k$\nd homomorphism induced by the canonical homomorphism
from $\Pic(X_{\overline{k}})$ to $\Pic(X_{\overline{k}})_\Q$.
Conjugation by a point of $\pi_\mathrm{\acute{e}tm}(X,x)$ thus induces the identity on
$D(\Pic(X_{\overline{k}})_\Q)$ because it induces the identity on $D(\Pic(X_{\overline{k}}))$.
It follows that \eqref{e:DPicQembed} is the embedding of a central $k$\nd subgroup,
so that the extension \eqref{e:multetmet} is central for $X$ an abelian variety.

Let $M$ be a Galois submodule of $\Pic(X_{\overline{k}})$ containing 
${}_\mathrm{tors}\widehat{X}(\overline{k}) ={}_\mathrm{tors}\Pic(X_{\overline{k}})$, 
and denote by $M'$ the cokernel 
\[
M \to \Pic(X_{\overline{k}})_\Q \to M' \to 0.
\]
We obtain from \eqref{e:multetmet} using \eqref{e:DPicpi} a central extension of $k$\nd groups 
\begin{equation}\label{e:DMDM'ext}
1 \to D(M') \to \pi_\mathrm{\acute{e}tm}(X,x) \to D(M) \to 1,
\end{equation}
where the first arrow is the restriction of \eqref{e:DPicQembed} to $D(M')$, and the second is
the projection onto $D(M)$ composed with \eqref{e:DPicproject}.
The second arrow sends the universal element in $H^1(X,x,\pi_\mathrm{\acute{e}tm}(X,x))$ 
to the unique element of $H^1(X,x,D(M))$
with image in $H^1(X_{\overline{k}},D(M)_{\overline{k}})$ the tautological element.

Taking $M = \widehat{X}(\overline{k})$ and $M' = NS(X_{\overline{k}})_\Q$ in \eqref{e:DMDM'ext},
we obtain a central extension
\begin{equation}\label{e:NSQPic0ext}
1 \to D(NS(X_{\overline{k}})_\Q) \to \pi_\mathrm{\acute{e}tm}(X,x) \to D(\widehat{X}(\overline{k})) \to 1
\end{equation}
of $k$\nd groups, which is functorial in $X$.
The $k$\nd endomorphism $n_{X*}$ of $\pi_\mathrm{\acute{e}tm}(X,x)$
induces the $n$th power on $D(\widehat{X}(\overline{k}))$ and the $n^2$th power on $D(NS(X_{\overline{k}})_\Q)$.

The commutative $k$\nd groups $D(\Pic(X_{\overline{k}})_\Q)$ and $D(NS(X_{\overline{k}})_\Q)$ are uniquely
divisible, in the sense that the $n$th power morphism is an isomorphism for every $n \ne 0$,
because $\Pic(X_{\overline{k}})_\Q$ and $NS(X_{\overline{k}})_\Q$ are $\Q$\nd vector spaces.
Rational powers of any point of these $k$\nd groups are thus uniquely defined.
Similarly for $n \ne 0$ the $k$\nd group $D(\widehat{X}(\overline{k}))$ has no $n$\nd torsion,
in the sense that the $n$th power morphism is a monomorphism, because $\widehat{X}(\overline{k})$ is divisible.

Let $G'$ and $G''$ be commutative affine $k$\nd groups.
We write $\Alt^2(G',G'')$ for the abelian group of those morphisms of $k$\nd schemes
\[
z:G' \times_k G' \to G''
\]
which are bilinear, i.e.\ $z(g',-)$ and $z(-,g')$ are group homomorphisms for every point $g'$ of $G'$,
and alternating, i.e.\ $z(g',g') = 1$ for every $g'$.
To every central extension $G$ of $G'$ by $G''$ is associated its  \emph{commutator morphism} in $\Alt^2(G',G'')$,
given by factoring through $G' \times_k G'$ the commutator of $G$, which sends the point 
$(g,h)$ of $G \times_k G$ to the point
\[
[g,h] = ghg^{-1}h^{-1}
\]
of the $k$\nd subgroup $G''$ of $G$.

Write the commutator morphism associated to the central extension \eqref{e:NSQPic0ext} as
\[
c_X \in \Alt^2(D(\widehat{X}(\overline{k})),D(NS(X_{\overline{k}})_\Q)).
\]
By \eqref{e:DMDM'ext}, the cokernel of the central embedding \eqref{e:DPicQembed} is 
$D({}_\mathrm{tors}\widehat{X}(\overline{k}))$ and 
the kernel $\pi_\mathrm{\acute{e}tm}(X,x)^\mathrm{der}$ of the projection  \eqref{e:DPicproject} onto the maximal commutative $k$\nd quotient is $D(NS(X_{\overline{k}}) \otimes \Q/\Z)$.
Thus $c_X$ factors as
\[
\xymatrix{
D(\widehat{X}(\overline{k})) \times_k D(\widehat{X}(\overline{k})) \ar[d]_{r \times r} \ar[r]^-{c_X} 
& D(NS(X_{\overline{k}})_\Q) \\
D({}_\mathrm{tors}\widehat{X}(\overline{k})) \times_k D({}_\mathrm{tors}\widehat{X}(\overline{k})) \ar[r] 
& D(NS(X_{\overline{k}}) \otimes \Q/\Z) \ar[u]_{v}
}
\]
where $r$ is the projection and $v$ is the embedding.
Further $v$ is the embedding of the smallest $k$\nd subgroup through which $c_X$ factors.

Let $j_1$ and $j_2$ be endomorphisms of the $k$\nd group $D(NS(X_{\overline{k}})_\Q)$ with 
\begin{equation}\label{e:jcX}
j_1 \circ c_X = j_2 \circ c_X.
\end{equation}
Then $j_1 = j_2$.
Indeed  $j_1 \circ v = j_2 \circ v$ because $c_X$ and hence $v$ factors through the equaliser of $j_1$ and $j_2$.
Since the cokernel of $v$ is the $k$\nd torus $D(NS(X_{\overline{k}}))$,
the image in $D(NS(X_{\overline{k}})_\Q)$ of the difference of $j_1$ and $j_2$ is a $k$\nd torus 
with no $n$\nd torsion for $n \ne 0$ and hence trivial.

Let $G'$ and $G''$ be affine $k$\nd groups.
If $G_1$ and $G_2$ are extensions of $G'$ by $G''$,
we write $\Hom_{G'}(G_1,G_2)$ for the set of $k$\nd homomorphisms
from $G_1$ to $G_2$ compatible with the projection onto $G'$ and $\Hom^{G''}_{G'}(G_1,G_2)$ for the subset of those
also compatible with the embedding of $G''$.
We write $\Aut^{G''}_{G'}(G)$ for the group $\Hom^{G''}_{G'}(G,G)$.
If $G$ is a central extension of $G'$ by $G''$, we have an isomorphism 
\[
\Hom_k(G',G'') \iso \Aut^{G''}_{G'}(G)
\]
of abelian groups which sends $f$ to $g \mapsto f(g')g$ with $g'$ the image of $g$ in $G'$.
In particular if $M$ and $M'$ are Galois modules and $G$ is a central extension of $D(M)$ by $D(M')$,
we have an isomorphism
\begin{equation}\label{e:HomAutiso}
\Hom(M',M) \iso \Hom_{\overline{k}}(D(M)_{\overline{k}},D(M')_{\overline{k}})
\iso \Aut^{D(M')_{\overline{k}}}_{D(M)_{\overline{k}}}(G_{\overline{k}})
\end{equation}
of abelian groups, which is compatible with the actions of the Galois group. 

Let $G$ be a central extension of $D(\widehat{X}(\overline{k}))$ by $D(NS(X_{\overline{k}})_\Q)$ with commutator
morphism $c_X$.
Then we have
\begin{equation}\label{e:AutEnd}
\Hom^{D(NS(X_{\overline{k}})_\Q)}_{D(\widehat{X}(\overline{k}))}(\pi_\mathrm{\acute{e}tm}(X,x),G) =
\Hom_{D(\widehat{X}(\overline{k}))}(\pi_\mathrm{\acute{e}tm}(X,x),G)
\end{equation}
because \eqref{e:jcX} holds with $j_1$ the identity and $j_2$ any induced $k$\nd endomorphism.

Given $t$ in $\widetilde{X}(\overline{k})$, define as follows an automorphism
\[
\xi_t:\pi_\mathrm{\acute{e}tm}(X,x)_{\overline{k}} \iso \pi_\mathrm{\acute{e}tm}(X,x)_{\overline{k}}
\]
of the $\overline{k}$\nd group $\pi_\mathrm{\acute{e}tm}(X,x)_{\overline{k}} = 
\pi_\mathrm{\acute{e}tm}(X_{\overline{k}},x)$.
Since \eqref{e:multetmet} is here a central extension, $\pi_\mathrm{\acute{e}tm}(X)$ acts on $\pi_\mathrm{\acute{e}tm}(X)^\mathrm{diag}$ 
through $\pi_\mathrm{\acute{e}t}(X)$. 
The action $\alpha_s$ of the $\overline{k}$\nd point $s$ of $\pi_\mathrm{\acute{e}t}(X)$ is
conjugation by a $\overline{k}$\nd point of $\pi_\mathrm{\acute{e}tm}(X)$ above $s$, 
which exists by Proposition~\ref{p:printriv}.
If $t$ in $\widetilde{X}(\overline{k}) = \pi_\mathrm{\acute{e}t}(X)_{-,x}(\overline{k})$ lies above 
$y$ in $X(\overline{k})$, then $\xi_t$ is the composite
\begin{equation}\label{e:xidef}
\pi_\mathrm{\acute{e}tm}(X_{\overline{k}},x) \xrightarrow{T_y{}_*} \pi_\mathrm{\acute{e}tm}(X_{\overline{k}},y)
\xrightarrow{\alpha_{t^{-1}}} \pi_\mathrm{\acute{e}tm}(X_{\overline{k}},x),
\end{equation}
with $T_y$ the translation \eqref{e:translation}.
Up to conjugacy, $\xi_t$ depends only on $y$.
We have
\[
T_y{}_* \circ \alpha_{t'{}^{-1}} = \alpha_{T_y{}_*(t')^{-1}} \circ T_y{}_*
\]
for $t'$ in $\widetilde{X}(\overline{k})$, by compatibility of $T_y{}_*$ with composition, while
\[
\alpha_{t^{-1}} \circ \alpha_{T_y{}_*(t')^{-1}} = \alpha_{(T_y{}_*(t') \circ t)^{-1}}
=
\alpha_{(t+t')^{-1}},
\]
by \eqref{e:pietgroup} applied to $(t',1_y) \circ (1_x,t) = (t',t)$, because $T_y{}_* = - + 1_y$.
Thus
\[
\xi_{t+t'} = \xi_t \circ \xi_{t'},
\]
so that $t \mapsto \xi_t$ defines an action of $\widetilde{X}(\overline{k})$ on $\pi_\mathrm{\acute{e}tm}(X,x)_{\overline{k}}$,
and an action up to conjugacy of $X(\overline{k})$ on $\pi_\mathrm{\acute{e}tm}(X,x)_{\overline{k}}$.
Also 
\begin{equation}\label{e:xintxit}
\xi_{nt} \circ n_X{}_* = n_X{}_* \circ \xi_t,
\end{equation}
because $T_{ny}{}_* \circ n_X{}_* = n_X{}_* \circ T_y{}_*$ and $n_X{}_*(t) = nt$.

For $y$ in $X(\overline{k})$, define a $\overline{k}$\nd automorphism $\zeta_y$ of 
$\pi_\mathrm{mult}(X,x)_{\overline{k}} = \pi_\mathrm{mult}(X_{\overline{k}},x)$ by
\[
\pi_\mathrm{mult}(X_{\overline{k}},x) \xrightarrow{T_y{}_*} \pi_\mathrm{mult}(X_{\overline{k}},y)
\iso \pi_\mathrm{mult}(X_{\overline{k}},x)
\]
with the second arrow that defined by the $[X]$\nd scheme structure on $\pi_\mathrm{mult}(X)^\mathrm{diag}$.
For $t$ in $\widetilde{X}(\overline{k})$ above $y$, the $\overline{k}$\nd automorphisms
$\xi_t$ and $\zeta_y$ are compatible with the projection from $\pi_\mathrm{\acute{e}tm}(X,x)_{\overline{k}}$
onto $\pi_\mathrm{mult}(X,x)_{\overline{k}}$, because $\pi_\mathrm{\acute{e}tm}(X) \to \pi_\mathrm{mult}(X)$ 
commutes with translation by the universal property of $\pi_\mathrm{\acute{e}tm}$.
With the identification \eqref{e:diagPicX},
the second arrow of the composite defining $\zeta_y$ is the identity.
Hence by the case $k = \overline{k}$ of \eqref{e:diagPicf}, $\zeta_y$ is the $\overline{k}$\nd automorphism of 
$D(\Pic(X_{\overline{k}}))_{\overline{k}}$ defined by $T_y{}\!^*$, or equivalently for every $\mu$ in $\Pic(X_{\overline{k}})$
we have
\begin{equation}\label{e:zetaDT}
\chi_\mu \circ \zeta_y = \chi_{T_y{}\!^*\mu}.
\end{equation}
Since $T_y{}\!^*$ acts on $\widehat{X}(\overline{k})$ and $NS(X_{\overline{k}})$ trivially,
$\zeta_y$ is thus compatible with the projection onto $D(\widehat{X}(\overline{k}))_{\overline{k}}$ and the embedding of
$D(NS(X_{\overline{k}}))_{\overline{k}}$.
By compatibility of $\xi_t$ and $\zeta_y$ with the projection,
it follows that $\xi_t$ is compatible with the projection onto $D(\widehat{X}(\overline{k}))_{\overline{k}}$
and hence by \eqref{e:AutEnd} with the embedding of $D(NS(X_{\overline{k}})_\Q)_{\overline{k}}$.
Thus we have a commutative diagram of abelian groups
\begin{equation}\label{e:XkbarAut}
\begin{gathered}
\xymatrix{
\widetilde{X}(\overline{k}) \ar[d] \ar[r]^-{\xi_-} &
\Aut^{D(NS(X_{\overline{k}})_\Q)_{\overline{k}}}_{D(\widehat{X}(\overline{k}))_{\overline{k}}}(\pi_\mathrm{\acute{e}tm}(X,x)_{\overline{k}}) \ar[d]  \ar[r]^-{\sim} &
\Hom(NS(X_{\overline{k}})_\Q,\widehat{X}(\overline{k})) \ar[d]  \\
X(\overline{k}) \ar[r]^-{\zeta_-} & 
\Aut^{D(NS(X_{\overline{k}}))_{\overline{k}}}_{D(\widehat{X}(\overline{k}))_{\overline{k}}}(D(\Pic(X_{\overline{k}}))_{\overline{k}})  \ar[r]^-{\sim} &
\Hom(NS(X_{\overline{k}}),\widehat{X}(\overline{k}))
}
\end{gathered}
\end{equation}
which is compatible with the actions of the Galois group,
where the isomorphisms are the inverses of those of the form
\eqref{e:HomAutiso}.

Denote by $\widetilde{\xi}_t$ and $\widetilde{\zeta}_y$  the images of $t$ and $y$ under the top 
and bottom rows of \eqref{e:XkbarAut}. 
For $\nu$ in $NS(X_{\overline{k}})$ and $\varphi_\nu$ as in \eqref{e:phinu}, we have by \eqref{e:zetaDT}
\[
\chi_\nu(\zeta_y(g)g^{-1}) = \chi_{T_y{}\!^*\mu - \mu}(g) = \chi_{\varphi_\nu(y)}(g)
\]
for every point $g$ of $D(\Pic(X_{\overline{k}}))_{\overline{k}}$, where $\mu$ is any element 
of $\Pic(X_{\overline{k}})$ above $\nu$.
Thus
\begin{equation}\label{e:zetatilde}
\widetilde{\zeta}_y(\nu) = \varphi_\nu(y).
\end{equation}
Since $\widetilde{X}(\overline{k})$ is uniquely divisible, it follows from \eqref{e:XkbarAut}
and \eqref{e:zetatilde} that
\[
\widetilde{\xi}_t(\nu/n) = \widetilde{\xi}_{t/n}(\nu) = \varphi_\nu((t/n)_1) = \varphi_\nu(t_n)
\]
for $\nu$ in $NS(X_{\overline{k}})$ and $n \ne 0$, where a subscript $n$ denotes the component of a point of $\widetilde{X}$
in the copy $X_{(n)}$
of $X$ in the limit \eqref{e:ucoverlim}.
Thus 
\begin{equation}\label{e:chinuchiphi}
\chi_{\nu/n}(\xi_t(g)g^{-1}) = \chi_{\varphi_\nu(t_n)}(g)
\end{equation}
for every point $g$ of $\pi_\mathrm{\acute{e}tm}(X,x)_{\overline{k}}$, where $\chi_{\varphi_\nu(t_n)}$ is regarded as a character of 
$\pi_\mathrm{\acute{e}tm}(X,x)_{\overline{k}}$ by inflation.

For $t$ a $\overline{k}$\nd point of the kernel $\pi_\mathrm{\acute{e}t}(X,x) = \lim_{n \ne 0} {}_n X$
of $\widetilde{X} \to X$, the $\overline{k}$\nd automorphism $\xi_t$ 
is by definition conjugation by any $\overline{k}$\nd point of $\pi_\mathrm{\acute{e}tm}(X,x)$ above $t^{-1}$.
Thus if $a$ and $a'$ are $\overline{k}$\nd points of $D(\widehat{X}(\overline{k}))$ we have
\[
c_X(a,a') = c_X(a,a'{}^{-1})^{-1} = g'{}^{-1}gg'g^{-1} = \xi_{t'}(g)g^{-1},
\] 
where $g$ and $g'$ are $\overline{k}$\nd points of
$\pi_\mathrm{\acute{e}tm}(X,x)$ above $a$ and $a'$, and $t'$ is the image of $a'$ in $\pi_\mathrm{\acute{e}t}(X,x)$.
Hence using \eqref{e:nXDnXhat} it follows from \eqref{e:chinuchiphi} that
\begin{equation}\label{e:chicom}
\chi_{\nu/n}(c_X(a,a')) = \chi_{\varphi_\nu(a'{}\!_n)}(a_n)
\end{equation}
for $\nu$ in $NS(X_{\overline{k}})$ and $n \ne 0$, where $a_n$ and $a'{}_n$ are the images of $a$ and $a'$ in ${}_nX(\overline{k})$.

It follows from \eqref{e:chicom} that the largest $k$\nd subgroup $G$ of $D(\widehat{X}(\overline{k}))$ for which
$c_X(-,G)$ is trivial is its identity component
\[
D(\widehat{X}(\overline{k})/{}_\mathrm{tors}\widehat{X}(\overline{k})) = D(\widehat{X}(\overline{k})_\Q).
\]
Indeed if $\nu$ the class of an ample line bundle on $X$,
then given $a'$ and $n \ne 0$ with $a'{}\!_n \ne 0$,
we have $\varphi_\nu(a'{}\!_m) \ne 0$ for some multiple $m \ne 0$ of $n$, 
and hence $\chi_{\varphi_\nu(a'{}\!_m)}(a_m) \ne 1$ for some $a$.
The centre of $\pi_\mathrm{\acute{e}tm}(X,x)$ is thus its identity component, 
given by the image of the embedding \eqref{e:DPicQembed}. 

\begin{rem}
The pairing of Galois modules $(a,\lambda) \mapsto \chi_\lambda(a)$ occurring in \eqref{e:chicom} 
and elsewhere is the inverse of the usual Weil pairing
\[
\widetilde{e}_n:{}_nX(\overline{k}) \times {}_n\widehat{X}(\overline{k}) \to \overline{k}{}^*
\]
associated to $X$.
Recall \cite[p.183]{Mum70} that $\widetilde{e}_n(a,\lambda)$ is defined as $\chi(a)$,
where $\chi$ is the unique character of ${}_nX_{\overline{k}}$ with the following property:
if $X_{(n)}$ is $X$
regarded as a principal ${}_nX$\nd bundle over itself using $n_X$,
then $\lambda$ is the class of the line bundle 
$X_{(n)}{}_{\overline{k}} \times_{\overline{k}}^{{}_nX_{\overline{k}}} V$ with
$V$ the $1$\nd dimensional representation of ${}_nX_{\overline{k}}$ with character $\chi^{-1}$, or equivalently $\lambda$ is the
class of the push forward of $X_{(n)}{}_{\overline{k}}$ along $\chi^{-1}$. 
That
\[
\widetilde{e}_n(a,\lambda) = \chi_\lambda (a)^{-1}
\]
follows from the fact that by \eqref{e:nXDnXhat} we have $\chi^{-1} = \chi_\lambda$.
Suppose that $X$ is an elliptic curve.
Then we may identify $NS(X_{\overline{k}})$ with $\Z$ by assigning to a line bundle its degree.
The usual identification of $X_{\overline{k}}$ with $\widehat{X}_{\overline{k}}$, 
by assigning to $y$ the class of the divisor $(y) - (x)$, is then given by $\varphi_{-1}$.
With this identification, $\widetilde{e}_n$ becomes a pairing on ${}_nX(\overline{k})$, and by \eqref{e:chicom}
\[
\chi_{1/n}(c_X(a,a')) = \widetilde{e}_n(a_n,a'{}\!_n)
\]
for $\overline{k}$\nd points $a$ and $a'$ of $D(\widehat{X}(\overline{k})) = D(X(\overline{k}))$ with images $a_n$ and $a'{}\!_n$ 
in ${}_nX(\overline{k})$.
\end{rem}

Given commutative affine $k$\nd groups $G'$ and $G''$ and $z$ in $\Alt^2(G',G'')$, we write
\[
C_z
\]
for the affine $k$\nd group with 
underlying $k$\nd scheme $G'' \times_k G'$ and product given by
\[
(g'',g')(h'',h') = (g''h''z(g',h'),g'h').
\]
The embedding $g'' \mapsto (g'',1)$ of $G''$ into $C_z$ and the projection onto $G'$ define on $C_z$
a structure of central extension of $G'$ by $G''$ with commutator morphism $z^2$.
Given also commutative affine $k$\nd groups $G'{}\!_1$ and $G''{}\!_1$
and $z_1$ in $\Alt^2(G'{}\!_1,G''{}\!_1)$, any pair $j':G' \to G'{}\!_1$ and $j'':G'' \to G''{}\!_1$ of $k$\nd homomorphisms
such that
\[
z_1 \circ (j' \times_k j') = j'' \circ z
\]
induces a $k$\nd homomorphism $C_z \to C_{z_1}$ with underlying morphism of $k$\nd schemes $j'' \times_k j'$.
In particular for any integer $n$ we write
\[
n_z:C_z \to C_z
\]
for the endomorphism of $k$\nd groups induced by the pair $n_{G'}$ and $(n^2)_{G''}$.

A commutative affine $k$\nd group $G$ will be called uniquely $2$\nd divisible if $2_G$ is an isomorphism.
Every point $g$ of $G$ has then a unique square root $\sqrt{g}$.

Let $G$ be an affine $k$\nd group equipped with an involution, i.e.\ a $k$\nd automorphism $\iota$ with $\iota^2 = 1_G$.
Suppose that the $k$\nd subgroup $G^\iota$ of invariants is central and uniquely $2$\nd divisible, and
that $G/G^\iota$ is commutative.
Then $[\iota(g),g]$ lies in $G^\iota$ for every point $g$ of $G$,
so that
\[
[\iota(g),g] = \iota([\iota(g),g]) = [g,\iota(g)] = [\iota(g),g]^{-1}
\]
and hence $[\iota(g),g] = 1$.
Thus $\iota(g)$ commutes with $g$, so that $g\iota(g)$ lies in $G^\iota$.
It follows that $g$ can be written uniquely in the form 
\[
g = g_+g_-
\]
with $\iota(g_+) = g_+$ and $\iota(g_-) = (g_-)^{-1}$.
Indeed $g_+ = \sqrt{g\iota(g)}$.
If $c$ in $\Alt^2(G/G^\iota,G^\iota)$ is the commutator and $u$ in $\Alt^2(G/G^\iota,G^\iota)$ is given by
\[
u(a,a') = \sqrt{c(a,a')},
\]  
we then have an isomorphism
\begin{equation}\label{e:grpinviso}
(G,\iota) \iso (C_u,(-1)_u)
\end{equation}
of $k$\nd groups with an involution, which sends $g$ to $(g_+,\overline{g})$ with 
$\overline{g}$ the image of $g$ in $G/G^\iota$.
The isomorphism \eqref{e:grpinviso} is natural in $(G,\iota)$.
It is the unique $k$\nd homomorphism from $G$ to $C_u$ which is compatible with both the involutions and
the structures of extension of $G/G^\iota$ by $G^\iota$.

Let $X$ be an abelian variety over $k$ with base point $x$.
Then there is a unique
\[
u_X \in \Alt^2(D(\widehat{X}(\overline{k})),D(NS(X_{\overline{k}})_\Q))
\]
such that for every $\nu$ in $NS(X_{\overline{k}})$ and $n \ne 0$ we have
\begin{equation}\label{e:chiu}
\chi_{\nu/n}(u_X(a,a')) = \chi_{\varphi_\nu(a'{}\!_{2n})}(a_{2n})
\end{equation}
for every pair of $\overline{k}$\nd points $a$ and $a'$ of $D(\widehat{X}(\overline{k}))$ with respective images 
$a_{2n}$ and $a'{}\!_{2n}$ in $D({}_{2n}\widehat{X}(\overline{k})) = {}_{2n}X$.
Indeed by \eqref{e:chicom}
\[
u_X(a,a') = \sqrt{c_X(a,a')}.
\]
The involution $(-1)_{X*}$ of $\pi_\mathrm{\acute{e}tm}(X,x)$ induces through
the short exact sequence \eqref{e:NSQPic0ext} the identity on $D(NS(X_{\overline{k}})_\Q)$
and the inverse involution on $D(\widehat{X}(\overline{k}))$.
Since $D(\widehat{X}(\overline{k}))$ has no $2$\nd torsion, it follows that
\[
\pi_\mathrm{\acute{e}tm}(X,x)^{(-1)_{X*}} = D(NS(X_{\overline{k}})_\Q).
\]
Thus \eqref{e:grpinviso} with $G = \pi_\mathrm{\acute{e}tm}(X,x)$ and $\iota = (-1)_{X*}$ gives
by \eqref{e:NSQPic0ext} an isomorphism
\begin{equation}\label{e:piCuX}
\pi_\mathrm{\acute{e}tm}(X,x) \iso C_{u_X}
\end{equation}
of $k$\nd groups with involution.
It is natural in $X$, by naturality of \eqref{e:grpinviso} and functoriality of \eqref{e:NSQPic0ext},
and is the unique $k$\nd homomorphism compatible with 
the involutions and the structures of extension of $D(\widehat{X}(\overline{k}))$ by $D(NS(X_{\overline{k}})_\Q)$.

\begin{prop}\label{p:repCuX}
Let $X$ be an abelian variety over $k$ with base point $x$.
Then there is a unique element of $H^1(X,x,C_{u_X})$ which has image under
\[
H^1(X,x,C_{u_X}) \to H^1(X_{\overline{k}},C_{u_X}{}_{\overline{k}}) \to 
H^1(X_{\overline{k}},D(\widehat{X}(\overline{k}))_{\overline{k}})
\]
the tautological element and is fixed by the involution $((-1)_X,(-1)_{u_X})$ of $(X,C_{u_X})$.
For any element $\alpha$ of $H^1(X,x,C_{u_X})$ with image in $H^1(X_{\overline{k}},D(\widehat{X}(\overline{k}))_{\overline{k}})$ 
the tautological
element, the functor $H^1(X,x,-)$ on the category of $k$\nd groups of pro\'etale by multiplicative type
is represented by $C_{u_X}$ with universal element $\alpha$.
\end{prop}

\begin{proof}
Let $j$ be a $k$\nd homomorphism from $\pi_\mathrm{\acute{e}tm}(X,x)$ to $C_{u_X}$.
Then the image $\alpha_j$ of $j$ under the bijection
\[
\Hom_k(\pi_\mathrm{\acute{e}tm}(X,x),C_{u_X}) \iso H^1(X,x,C_{u_X})
\]
defined by the universal property of $\pi_\mathrm{\acute{e}tm}(X,x)$ has image
in $H^1(X_{\overline{k}},D(\widehat{X}(\overline{k}))_{\overline{k}})$ the tautological
element if and only if $j$ is compatible with the
projections onto $D(\widehat{X}(\overline{k}))$.
When this condition holds, $j$ is also compatible with the embeddings of $D(NS(X_{\overline{k}})_\Q)$,
by \eqref{e:AutEnd}.
Since $\alpha_j$ is universal if and only if $j$ is an isomorphism, the second statement is clear.
By \eqref{e:H1funct} with $f = (-1)_X$ and $G = C_{u_X}$, the element $\alpha_j$ is fixed by the involution
$((-1)_X,(-1)_{u_X})$ if and only if $j$ is compatible with the involutions $(-1)_{X*}$ and $(-1)_{u_X}$.
Thus \eqref{e:piCuX} is the unique $j$ for which $\alpha_j$ has image the tautological element and is fixed
by $((-1)_X,(-1)_{u_X})$.
\end{proof}

We may identify $\pi_\mathrm{\acute{e}tm}(X,x)$ with $C_{u_X}$ using \eqref{e:piCuX}.
With this identification, the universal element in 
\[
H^1(X,x,C_{u_X})
\]
is the unique element of Proposition~\ref{p:repCuX} which is fixed by
$((-1)_X,(-1)_{u_X})$ and which has image in $H^1(X_{\overline{k}},D(\widehat{X}(\overline{k}))_{\overline{k}})$
the tautological element, because by \eqref{e:H1funct} the universal element in $H^1(X,x,\pi_\mathrm{\acute{e}tm}(X,x))$
is fixed by $((-1)_X,(-1)_{X*})$.
Similarly, by naturality of \eqref{e:piCuX},
the $k$\nd homomorphism $f_*:C_{u_X} \to C_{u_{X'}}$
defined by a commutative diagram of the form \eqref{e:H1funct} for a morphism $f:(X,x) \to (X',x')$ of abelian varieties
coincides with that defined by functoriality of $C_{u_X}$.
In particular
\[
n_{X*} = n_{u_X}
\]
for any $n$.

With the identification of $\pi_\mathrm{\acute{e}tm}(X,x)$ and $C_{u_X}$, we have a commutative diagram
\begin{equation}\label{e:CuXdiagram}
\begin{gathered}
\xymatrix{
D(\Pic(X_{\overline{k}})_\Q) \ar[r]^-{e} &  C_{u_X} \ar@{=}[d] \ar[r]^-{p} & D(\Pic(X_{\overline{k}})) \ar[d] \\
D(NS(X_{\overline{k}})_\Q) \ar[u] \ar[r] & C_{u_X} \ar[r] & D(\widehat{X}(\overline{k}))
}
\end{gathered}
\end{equation}
where $e$ is \eqref{e:DPicQembed}, $p$ is \eqref{e:DPicproject}, the bottom row is \eqref{e:NSQPic0ext},
the left vertical arrow is the embedding and the right vertical arrow is the projection.
It is functorial in $X$.
By compatibility of \eqref{e:piCuX} with the extensions, the bottom row of \eqref{e:CuXdiagram} is the 
canonical structure of extension defined by $C_{u_X}$.
The $k$\nd homomorphism $p \circ e$ is that  defined by the canonical homomorphism from 
$\Pic(X_{\overline{k}})$ to $\Pic(X_{\overline{k}})_\Q$.

When applied with $G = D(\Pic(X_{\overline{k}})_\Q)$ and $\iota$ defined by $(-1)_X$, \eqref{e:grpinviso} is
the decomposition of $D(\Pic(X_{\overline{k}})_\Q)$ as the product of $k$\nd subgroups
\[
D(NS(X_{\overline{k}})_\Q) \times_k D(\widehat{X}(\overline{k})_\Q),
\]
where $\widehat{X}(\overline{k})_\Q$ is identified with the quotient of
$\Pic(X_{\overline{k}})_\Q$ by its invariants under $(-1)_X$.
Since $e$ in \eqref{e:CuXdiagram} is compatible with the involutions, it is uniquely determined
by the induced $k$\nd homomorphisms from $D(NS(X_{\overline{k}})_\Q)$ to $D(NS(X_{\overline{k}})_\Q)$ and
from $D(\widehat{X}(\overline{k})_\Q)$ to $D(\widehat{X}(\overline{k}))$.
The first is the identity, by the left square of \eqref{e:CuXdiagram}, and the
second is defined by the projection of $\widehat{X}(\overline{k})$ onto $\widehat{X}(\overline{k})_\Q$, by the right square.

Since the $k$\nd subgroup of invariants of $D(\Pic(X_{\overline{k}}))$ under the involution induced by $(-1)_X$ is
the torus $D(NS(X_{\overline{k}}))$, and hence not uniquely $2$\nd divisible, 
it is less trivial to determine $p$ in \eqref{e:CuXdiagram} explicitly.
To do so, note first that by \eqref{e:CuXdiagram}, the restriction of $p$ to the $k$\nd subgroup $D(NS(X_{\overline{k}})_\Q)$
of $C_{u_X}$ is defined by the canonical homomorphism from $\Pic(X_{\overline{k}})$ to $NS(X_{\overline{k}})_\Q$.
It thus suffices to determine the restriction
\[
p_-:D(\widehat{X}(\overline{k})) \to D(\Pic(X_{\overline{k}}))
\]
of $p$ to $D(\widehat{X}(\overline{k}))$ embedded as a $k$\nd subscheme of $C_{u_X}$ by $g \mapsto (1,g)$.
This $k$\nd subscheme is not a $k$\nd subgroup of $C_{u_X}$ for $X$ of dimension $> 0$,
and $p_-$ is not in general a $k$\nd homomorphism:
by the definition of the product of $C_{u_X}$ and \eqref{e:chiu},
\begin{equation}\label{e:pbilin}
\chi_\lambda(p_-(ss')p_-(s)^{-1}p_-(s')^{-1}) = \chi_\nu(u_X(s,s')) = \chi_{\varphi_\nu(y')}(y)
\end{equation}
for $\lambda$ in $\Pic(X_{\overline{k}})$ above $\nu$ in $NS(X_{\overline{k}})$ 
and $\overline{k}$\nd points $s$ and $s'$ of $D(\widehat{X}(\overline{k}))$ above
the $\overline{k}$\nd points $y$ and $y'$ of $D({}_2\widehat{X}(\overline{k})) = {}_2X$.

To determine $p_-$, we use the decomposition 
\[
D(\Pic(X_{\overline{k}})) = D(\Pic(X_{\overline{k}})^{(-1)_X}) \times_{D({}_2\widehat{X}(\overline{k}))} D(\widehat{X}(\overline{k})),
\]
with the projections defined by the embeddings into $\Pic(X_{\overline{k}})$.
This decomposition holds because $\widehat{X}(\overline{k})$ is $2$\nd divisible, 
so that above every element of $NS(X_{\overline{k}})$ there exists
an element of $\Pic(X_{\overline{k}})$ fixed by $(-1)_X$.
By the right square of \eqref{e:CuXdiagram}, the component of $p_-$ at $D(\widehat{X}(\overline{k}))$ is the identity.
It remains to determine the component
\[
p_{+-}:D(\widehat{X}(\overline{k})) \to D(\Pic(X_{\overline{k}})^{(-1)_X})
\]
of $p_-$ at $D(\Pic(X_{\overline{k}})^{(-1)_X})$.

To determine $p_{+-}$, define for each $\lambda$ in $\Pic(X_{\overline{k}})^{(-1)_X}$ a map
\[
\varepsilon_\lambda:{}_2X(\overline{k}) \to \{\pm 1\} \subset \overline{k}{}^*,
\]
as follows.
Let $\sL$ be a line bundle over $X_{\overline{k}}$ with class $\lambda$.
Then there is a unique involution $i$ of $\sL$ above $(-1)_{X_{\overline{k}}}$ which acts as $+1$ on $\sL_x$.
Now ${}_2X(\overline{k})$ is the fixed point set of $X(\overline{k})$ under $(-1)_X$.
We define $\varepsilon_\lambda(y)$ as the action of $i$ on $\sL_y$.
In general, $\varepsilon_\lambda$ is not a group homomorphism.
Since $\varepsilon_\lambda(y)$ is additive in $\lambda$ for given $y$
and is compatible with the action of the Galois group on $\lambda$ and $y$,
there is a unique morphism of $k$\nd schemes
\[
\varepsilon:{}_2X \to D(\Pic(X_{\overline{k}})^{(-1)_X})
\]
such that
\[
\chi_\lambda(\varepsilon(y)) = \varepsilon_\lambda(y)
\]
for each $y$ in ${}_2X(\overline{k})$ and $\lambda$ in $\Pic(X_{\overline{k}})^{(-1)_X}$.
We now show that $p_{+-}$ factors as
\[
D(\widehat{X}(\overline{k})) \to {}_2 X \xrightarrow{\varepsilon} D(\Pic(X_{\overline{k}})^{(-1)_X}),
\]
where the first arrow is the projection onto $D({}_2\widehat{X}(\overline{k})) = {}_2X$.

To see that $p_{+-}$ factors as above, we may suppose that $k = \overline{k}$.
Let $s$ be a $k$\nd point of $\pi_\mathrm{\acute{e}tm}(X,x) = C_{u_X}$ which lies in the $k$\nd subscheme $D(\widehat{X}(k))$.
If $s$ has image $y$ in ${}_2 X$,
it is to be shown that for every $\lambda$ in $\Pic(X)^{(-1)_X}$ we have
\begin{equation}\label{e:chipepsilon}
\chi_\lambda(p_{+-}(s)) = \varepsilon_\lambda(y).
\end{equation}
Let $\sL$ and $i$ be as above.
Then there is a unique morphism of groupoids over $X$
\[
\rho:\pi_\mathrm{\acute{e}tm}(X) \to \underline{\Iso}_X(\sL).
\]
It factors uniquely through a morphism from $\pi_\mathrm{mult}(X)$ to $\underline{\Iso}_X(\sL)$
whose restriction to the diagonal is $\chi_\lambda \times_k X$ by \eqref{e:diagPicX}, 
where $\chi_\lambda$ is regarded as a character of $D(\Pic(X))$.
With $s$ regarded as a $k$\nd point of $\pi_\mathrm{\acute{e}tm}(X)$ above $(x,x)$, we thus have
\[
\chi_\lambda(p_{+-}(s)) = \rho(s).
\]
We have an involution $(-1)_X{}_*$ of $\pi_\mathrm{\acute{e}tm}(X)$ above $(-1)_X$,
and an involution of $\underline{\Iso}_X(\sL)$ above $(-1)_X$ which sends an isomorphism from $\sL_{z_0}$ to $\sL_{z_1}$
to the unique isomorphism from $\sL_{-z_0}$ to $\sL_{-z_1}$ compatible with it and $i_{z_0}$ and $i_{z_1}$.
By the universal property of $\pi_\mathrm{\acute{e}tm}(X)$, the morphism $\rho$ is compatible with these involutions.
By \eqref{e:nXpietmx} with $n = 2$, there is a unique $k$\nd point $t$ of $\pi_\mathrm{\acute{e}tm}(X)_{-,x}$
such that
\[
2_X{}_*(t) = s.
\]
Further $t$ lies above $(y,x)$.
Now $n_X{}_*(s) = s^n$, because $n_X{}_*$ is $n_{u_X}$ on the fibre $C_{u_X}$ above $(x,x)$.
Hence
\[
2_X{}_*((-1)_X{}_*(t) \circ s) = (-1)_X{}_*(2_X{}_*(t))s^2
= (-1)_X{}_*(s)s^2 = s.
\]
By uniqueness of $t$, it follows that
\[
(-1)_X{}_*(t) \circ s = t.
\]
Since $\rho((-1)_X{}_*(t)) = \varepsilon_\lambda(y) \rho(t)$ by compatibility of $\rho$ with the involutions,
applying $\rho$ shows that $\rho(s) = \varepsilon_\lambda(y)$, so that \eqref{e:chipepsilon} holds as required.

It follows from \eqref{e:pbilin} and  \eqref{e:chipepsilon} that for
$\lambda$ in $\Pic(X_{\overline{k}})^{(-1)_X}$ above $\nu$ in $NS(X_{\overline{k}})$ the map
$\varepsilon_\lambda$ is quadratic with associated bilinear form $(y,y') \mapsto \chi_{\varphi_\nu(y')}(y)$.

\begin{rem}
Let $Y$ be a smooth projective curve over $k$ of genus $g \ge 1$ and $X$ be the scheme of divisor classes of degree $g-1$ on $Y$.
The involution $\iota$ of $X$ defined by $\Omega_Y \otimes_{\sO_Y} (-)^\vee$ sends the theta divisor $\Theta$ on $X$ to itself.
There is thus a unique involution $i$ of $\sO_X(\Theta)$ above $\iota$ with the section of 
$\sO_X(\Theta)$ defining $\Theta$ compatible with $\iota$ and $i$.
If $y$ is a fixed $\overline{k}$\nd point of $\iota$, considering the action of $i$ on the ideal $\sO_X(-\Theta)$ of $\sO_X$
shows that $i$ acts on $\sO_X(\Theta)_y$ as $(-1)^{m(y)}$ with $m(y)$ the multiplicity of $\Theta$ at $y$.
Now the fixed $\overline{k}$\nd points of $\iota$ are the theta characteristics of $Y$, and $(-1)^{m(y)}$ is the parity of the theta characteristic $y$.
If $x$ is an even theta characteristic of $Y$, then $X$ with base point $x$ is an abelian variety with $(-1)_X = \iota$, 
and the parity of any theta characteristic $y$ of $Y$ is $\varepsilon_\lambda(y)$ with $\lambda$ the class of $\sO_X(\Theta)_{\overline{k}}$.
\end{rem}

Let $X$ be an elliptic curve over $k$.
Then $NS(X_{\overline{k}})$ is the trivial Galois module $\Z$.
Thus $D(NS(X_{\overline{k}}) \otimes \Q)$ is $D(\Q)$, and $u_X$ is an element of $\Alt^2(D(\widehat{X}(\overline{k})),D(\Q))$.

\begin{thm}\label{t:ellrep}
Let $X$ be an elliptic curve over $k$ with base point $x$.
Then there exists an element of $\widetilde{H}^1(X,x,C_{u_X})$ with image under
\[
\widetilde{H}^1(X,x,C_{u_X}) \to H^1(X_{\overline{k}},(C_{u_X})_{\overline{k}}) \to 
H^1(X_{\overline{k}},D(\widehat{X}(\overline{k}))_{\overline{k}})
\]
the tautological element.
If $\alpha$ is such an element, and if $\beta$ in 
$\widetilde{H}^1(X,x,SL_2)$
is the image of a non-zero element of 
$\widetilde{H}^1(X,x,\bG_a) = H^1(X,\bG_a)$
under an embedding $\bG_a \to SL_2$, 
then the functor $\widetilde{H}^1(X,x,-)$ on the category of reductive $k$\nd groups
up to conjugacy is represented by 
\[
C_{u_X} \times_k SL_2
\]
with universal element $(\alpha,\beta)$. 
\end{thm}

\begin{proof}
The existence statement is clear, because the image $\alpha_0$ in $\widetilde{H}^1(X,x,C_{u_X})$ 
of the universal element in $H^1(X,x,C_{u_X})$
has image in $H^1(X_{\overline{k}},D(\widehat{X}(\overline{k}))_{\overline{k}})$ the tautological element.
If also $\alpha$ has image the tautological element, then by \eqref{e:H1H1tilde} and Proposition~\ref{p:repCuX}
there is a $k$\nd automorphism of $C_{u_X}$ which sends $\alpha_0$ to $\alpha$.
To prove the representation statement, we may thus suppose that $\alpha = \alpha_0$.
Then $\alpha$ is the class of $\pi_\mathrm{\acute{e}tm}(X)_{-,x}$.
Let $P$ be a principal $SL_2$\nd bundle over $X$ with class $\beta$.
Then  \eqref{e:KKKiso} and 
Theorem~\ref{t:elluniv} show that
\[
\underline{\Iso}_{C_{u_X} \times_k SL_2}(\pi_\mathrm{\acute{e}tm}(X)_{-,x} \times_X P) = 
\underline{\Iso}_{C_{u_X}}(\pi_\mathrm{\acute{e}tm}(X)_{-,x}) \times_{[X]} \underline{\Iso}_{SL_2}(P)
\]
is universally reductive over $X$.
The representation statement now follows from Lemma~\ref{l:cohreppt}\ref{i:unreppt}.
\end{proof}

Let $X_0$ be an abelian variety over $k$ with base point $x_0$.
We identify $\pi_\mathrm{\acute{e}tm}(X_0,x_0)$ with $C_{u_{X_0}}$ using \eqref{e:piCuX} with $X_0$ for $X$.
For every $n \ne 0$ we have a projection from $C_{u_{X_0}}$ to $D({}_n\widehat{X}_0(\overline{k})) = {}_nX_0$,
and  a short exact sequence of affine $k$\nd groups
\[
1 \to C_{u_{X_0}} \xrightarrow{n_{u_{X_0}}} C_{u_{X_0}} \to {}_n X_0 \to 1.
\]
Taking points in $\overline{k}$ with the Krull topology gives a short exact sequence 
\begin{equation}\label{e:sexC}
1 \to C_{u_{X_0}}(\overline{k}) \xrightarrow{n_{u_{X_0}}} C_{u_{X_0}}(\overline{k}) \to {}_nX_0(\overline{k}) \to 1
\end{equation}
of topological groups with a continuous action of $\Gal(\overline{k}/k)$.
For every $m \ne 0$ we have a morphism
\begin{equation}\label{e:sexCnm}
\begin{gathered}
\xymatrix@C+1.6em@R+0.2cm{
1 \ar[r] & C_{u_{X_0}}(\overline{k}) \ar[r]^{(nm)_{u_{X_0}}} & C_{u_{X_0}}(\overline{k}) 
\ar[r] & {}_{nm}X_0(\overline{k}) \ar[r] & 1 \\
1 \ar[r] & C_{u_{X_0}}(\overline{k}) \ar@{=}[u] \ar[r]^{n_{u_{X_0}}} & C_{u_{X_0}}(\overline{k}) 
\ar[u]_{m_{u_{X_0}}} \ar[r] & {}_nX_0(\overline{k}) \ar[u] \ar[r] & 1
}
\end{gathered}
\end{equation}
of short exact sequences of topological groups with an action of $\Gal(\overline{k}/k)$.

With $\xi_t$ defined as in \eqref{e:xidef}, we have an action $t \mapsto \xi_t$ of $\widetilde{X}_0(\overline{k})$ on 
the $\overline{k}$\nd group $(C_{u_{X_0}})_{\overline{k}}$, and hence on the topological group $C_{u_{X_0}}(\overline{k})$.
It factors through an action up to conjugacy of $X_0(\overline{k})$ on $C_{u_{X_0}}(\overline{k})$.
By \eqref{e:xintxit}, the restriction of this action along the negative of the embedding
${}_nX_0(\overline{k}) \to X_0(\overline{k})$ coincides with that defined by \eqref{e:sexC}.
Taking \eqref{e:sexC} for \eqref{e:sextopgrp}, we thus obtain from \eqref{e:lextopgrp} a map
\begin{equation}\label{e:nXconnect}
H^1(\Gal(\overline{k}/k),{}_n X_0(\overline{k})) \to \sE(\Gal(\overline{k}/k),C_{u_{X_0}}(\overline{k}))/{}_nX_0(\overline{k})
\end{equation}
for each $n \ne 0$.
By \eqref{e:sexCnm} and functoriality of \eqref{e:lextopgrp}, these maps are
compatible with the embeddings ${}_nX_0(\overline{k}) \to {}_{nm}X_0(\overline{k})$.
Considering the appropriate involution of \eqref{e:sexC} shows that \eqref{e:nXconnect} is compatible with the inverse
involution of the source and the involution defined by $(-1)_{u_{X_0}}$ of the target.

For $n \ne 0$ the embedding of ${}_n X_0(\overline{k})$ into $X_0(\overline{k})$ defines a homomorphism
\begin{equation}\label{e:WCntors}
H^1(\Gal(\overline{k}/k),{}_n X_0(\overline{k})) \to H^1(\Gal(\overline{k}/k),X_0(\overline{k}))
\end{equation}
to the Weil--Ch\^atelet group of $X_0$, with image the $n$\nd torsion subgroup.
Let $\kappa$ be an element of its kernel.
If $\psi$ is a $1$\nd cocycle with class $\kappa$, then for some $y$ in $X_0(\overline{k})$
\[
\psi(\sigma) = {}^\sigma y - y
\]
for every $\sigma$.
Since $n\psi = 0$, the $\overline{k}$\nd point $y$ lies in the subgroup
\[
{}_{(n)}X_0(\overline{k}) \subset X_0(\overline{k})
\]
of the $a$ in $X_0(\overline{k})$ with $na$ in $X_0(k)$.
Choose a $\overline{k}$\nd point $t$ of $\widetilde{X}_0 = \pi_{\mathrm{\acute{e}t}}(X_0)_{-,x_0}$
above $y$.
By Proposition~\ref{p:printriv}, the principal $C_{u_{X_0}}$\nd bundle 
$\pi_{\mathrm{\acute{e}tm}}(X_0)_{ny,x_0}$ over $k$ has a $\overline{k}$\nd point $v$ above 
the $\overline{k}$\nd point $nt$ of $\pi_{\mathrm{\acute{e}t}}(X_0)_{ny,x_0}$.
Then we have a diagram
\[
\xymatrix{
1 \ar[r] & C_{u_{X_0}}(\overline{k}) \ar[d]_{\xi_t} \ar[r]  & C_{u_{X_0}}(\overline{k}) \rtimes \Gal(\overline{k}/k) 
\ar[d] \ar[r] & {}_nX_0(\overline{k}) \rtimes \Gal(\overline{k}/k) \ar[d] \ar[r] & 1 \\
1 \ar[r] & C_{u_{X_0}}(\overline{k}) \ar[r] & C_{u_{X_0}}(\overline{k}) \rtimes \Gal(\overline{k}/k) 
\ar[r] & {}_nX_0(\overline{k}) \rtimes \Gal(\overline{k}/k) \ar[r] & 1
}
\]
with rows the semidirect product of $\Gal(\overline{k}/k)$ by \eqref{e:sexC}, 
where the middle arrow  sends $(g,\sigma)$ to $(\xi_{nt}(g)(v^{-1} \circ {}^\sigma v),\sigma)$ 
and the right arrow sends $(a,\sigma)$ to $(a + \psi(\sigma),\sigma)$.
Writing $\pi_{\mathrm{\acute{e}tm}}(X_0)$ as the limit of its quotients of finite type shows that
$\sigma \mapsto v^{-1} \circ {}^\sigma v$ and hence the middle arrow is continuous,
and \eqref{e:xidef} shows that conjugation of $(C_{u_{X_0}})_{\overline{k}}$ by $v^{-1} \circ {}^\sigma v$ coincides with $\xi_{nt - n{}^\sigma t}$ and hence that the middle arrow is a group homomorphism.
The left square commutes by \eqref{e:xintxit}, and the right square commutes
because the $\xi$ are by \eqref{e:XkbarAut} compatible with the projection onto ${}_nX_0(\overline{k})$  
and the image of $n{}^\sigma t - nt$ in ${}_nX_0(\overline{k})$ is ${}^\sigma y -y$.
Considering pullbacks of the top row of the diagram along sections of ${}_nX_0(\overline{k}) \rtimes \Gal(\overline{k}/k)$
shows that translation by $\kappa$ of the source of the map \eqref{e:nXconnect} corresponds to the bijection of the 
target defined by $y \in  {}_{(n)}X_0(\overline{k})$,
because $y$ acts through the action of $\xi_t$ on $C_{u_{X_0}}(\overline{k})$.

It follows from the above that for each $n \ne 0$, \eqref{e:nXconnect} defines a map
\begin{equation}\label{e:nWCconnect}
{}_nH^1(\Gal(\overline{k}/k),X_0(\overline{k})) \to 
\sE(\Gal(\overline{k}/k),C_{u_{X_0}}(\overline{k}))/{}_{(n)}X_0(\overline{k}).
\end{equation}
The maps so defined are compatible with the embeddings of the $n$\nd torsion subgroup into the $nm$\nd torsion subgroup
and of ${}_{(n)}X_0(\overline{k})$ into ${}_{(nm)}X_0(\overline{k})$.
Since the Weil--Ch\^atelet group is a torsion group, we thus obtain a map
\[
H^1(\Gal(\overline{k}/k),X_0(\overline{k})) \to 
\sE(\Gal(\overline{k}/k),C_{u_{X_0}}(\overline{k}))/{}_{(\mathrm{tors})}X_0(\overline{k})
\]
where ${}_{(\mathrm{tors})}X_0(\overline{k})$ denotes the subgroup of the $y$ in $X_0(\overline{k})$
with $ny$ in $X_0(k)$ for some $n \ne 0$.
For Theorem~\ref{t:gen1rep} below we require only its composite
\begin{equation}\label{e:WCconnect}
H^1(\Gal(\overline{k}/k),X_0(\overline{k})) \to 
\sE(\Gal(\overline{k}/k),C_{u_{X_0}}(\overline{k}))/X_0(\overline{k})
\end{equation}
with the projection onto the coarser quotient by the full group $X_0(\overline{k})$.
All three of the above maps are compatible with the involutions.

Let $X$ be a principal homogeneous space under $X_0$, i.e.\ a $k$\nd scheme equipped with an action of the commutative $k$\nd group
$X_0$ such that $X_{\overline{k}}$ is isomorphic to $X_0{}_{\overline{k}}$ with the action of $X_0{}_{\overline{k}}$
by translation.
If $x$ and $x'$ are $\overline{k}$\nd points of $X$, we write $x' - x$ for the $\overline{k}$\nd point of $X_0$ that sends $x$ to $x'$.
We write ${}_{(n)}X(\overline{k})$ for the set of $x$ in $X(\overline{k})$ with $n({}^\sigma x - x) = 0$ for every
$\sigma$ in $\Gal(\overline{k}/k)$, and ${}_{(\mathrm{tors})}X(\overline{k})$ for the set of $x$ in $X(\overline{k})$ with 
${}^\sigma x - x$ in the torsion subgroup of $X_0(\overline{k})$ for every $\sigma$.  
The subgroup ${}_{(n)}X_0(\overline{k})$ (resp.\ ${}_{(\mathrm{tors})}X_0(\overline{k})$) of $X_0(\overline{k})$
acts simply transitively on the subset ${}_{(n)}X(\overline{k})$ (resp.\ ${}_{(\mathrm{tors})}X(\overline{k})$) of $X(\overline{k})$
when that subset is non-empty,
and coincides with it when $X = X_0$.

Denote by $\delta_X$ the class of $X$ in the Weil--Ch\^atelet group of $X_0$.
The $1$\nd cocycles with class $\delta_X$ are exactly those of the form $\sigma \mapsto {}^\sigma x - x$
for some $x$ in $X(\overline{k})$.
Since the image of \eqref{e:WCntors} is the $n$\nd torsion subgroup, it follows that $n\delta_X = 0$ if and only if
${}_{(n)}X(\overline{k})$ is non-empty.
In particular ${}_{(\mathrm{tors})}X(\overline{k})$ is always non-empty,
because $n\delta_X = 0$ for some $n \ne 0$. 

Let $x$ be a $\overline{k}$\nd point of $X$.
Then there is a unique isomorphism of $\overline{k}$\nd schemes
\[
i_x:X_{\overline{k}} \iso X_0{}_{\overline{k}}
\]
which sends $x$ to $x_0$ and is compatible with the actions of $X_0$.
The pullback
\[
F_x = \pi_\mathrm{\acute{e}tm}(X)_{x \times x}
\]
of $\pi_\mathrm{\acute{e}tm}(X)$ along $x$ is a transitive affine groupoid
over $\overline{k}$ with diagonal $\pi_\mathrm{\acute{e}tm}(X_{\overline{k}},x)$.

We have an isomorphism of $\overline{k}$\nd groups
\[
i_x{}_*:F_x{}\!^\mathrm{diag} = \pi_\mathrm{\acute{e}tm}(X_{\overline{k}},x) \iso 
\pi_\mathrm{\acute{e}tm}(X_0{}_{\overline{k}},x_0) =  (C_{u_{X_0}})_{\overline{k}}.
\]
Denote by $E_x$ the push forward along $i_x{}_*$ of the extension $F_x(\overline{k})_{\overline{k}}$ associated to $F_x$.
Then there is an isomorphism of Galois extended $\overline{k}$\nd groups 
\begin{equation}\label{e:FxExiso}
(F_x{}\!^\mathrm{diag},F_x(\overline{k})_{\overline{k}}) \iso ((C_{u_{X_0}})_{\overline{k}},E_x)
\end{equation}
with underlying isomorphism of $\overline{k}$\nd groups $i_x{}_*$.
Let $x'$ be a $\overline{k}$\nd point of $X$ and $w$ be a $\overline{k}$\nd point of $\pi_\mathrm{\acute{e}tm}(X)$ above $(x',x)$.
Conjugation by $w$ as in \eqref{e:Kconjiso} defines an isomorphism $j_w$ from $F_x$ to $F_{x'}$.
Since $i_{x'} \circ i_x{}\!^{-1}$ is $T_{x - x'}$, we have
\[
i_{x'}{}_* \circ j_w{}\!^\mathrm{diag} \circ i_x{}_*{}\!^{-1} = \xi_t
\]
where $t$ above $(x-x',x_0)$ is the image in $\pi_\mathrm{\acute{e}t}(X_0)$ of $i_{x'}{}_*(w^{-1})$.
Modulo isomorphisms of the form \eqref{e:FxExiso}, $j_w$ thus defines an isomorphism
\[
((C_{u_{X_0}})_{\overline{k}},E_x) \iso ((C_{u_{X_0}})_{\overline{k}},E_{x'})
\]
of Galois extended $\overline{k}$\nd groups with underlying isomorphism of $\overline{k}$\nd groups
$\xi_t$.
It follows that the classes of the $E_x$ in $\sE(\Gal(\overline{k}/k),C_{u_{X_0}}(\overline{k}))$ 
for $x$ in $X(\overline{k})$ form a single orbit under the action of $X_0(\overline{k})$, and similarly with $X(\overline{k})$ 
and $X_0(\overline{k})$ replaced by ${}_{(n)}X(\overline{k})$ and ${}_{(n)}X_0(\overline{k})$ or by 
${}_{(\mathrm{tors})}X(\overline{k})$ and ${}_{(\mathrm{tors})}X_0(\overline{k})$.

\begin{lem}\label{l:WCclass}
Let $X_0$ be an abelian variety over $k$ and $X$ be a principal homogeneous space under $X_0$.
Denote by $\delta_X$ the class of $X$ in the Weil--Ch\^atelet group of $X_0$,
and let $n \ne 0$ be an integer with $n\delta_X = 0$. 
Then the image of $-\delta_X$ under \eqref{e:nWCconnect} is the class of the extensions
$E_x$ for $x$ in ${}_{(n)}X(\overline{k})$.
\end{lem}

\begin{proof}
Fix an $x$ in ${}_{(n)}X(\overline{k})$.
The map $X(\overline{k}) \to X_0(\overline{k})$ induced by the $\overline{k}$\nd morphism
$n_{X_0}{}_{\overline{k}} \circ i_x:X_{\overline{k}} \to X_0{}_{\overline{k}}$ sends the translate of $x$ by $x'{}\!_0$
to $nx'{}\!_0$, and hence commutes with the actions of $\Gal(\overline{k}/k)$.
There is thus a morphism
\[
p:X \to X_0
\]
of $k$\nd schemes such that $p_{\overline{k}} = n_{X_0}{}_{\overline{k}} \circ i_x$.
It sends $x$ to the base point $x_0$ of $X_0$, and $(n_{X_0},p)$ is compatible with the actions of $X_0$.

The morphism $p_*:\pi_\mathrm{\acute{e}tm}(X) \to \pi_\mathrm{\acute{e}tm}(X_0)$ over $p$
induces a morphism
\[
F_x \to \pi_\mathrm{\acute{e}tm}(X_0,x_0) \times_k [\overline{k}] = C_{u_{X_0}} \times_k [\overline{k}]
\]
of groupoids over $\overline{k}$.
Taking $\overline{k}$\nd points over $\overline{k}$ gives the middle vertical arrow of the diagram
of short exact sequences of topological groups
\[
\xymatrix{
1 \ar[r] & F_x{}\!^\mathrm{diag}(\overline{k})_{\overline{k}} \ar[d]^{i_x{}_*} \ar[r] & F_x(\overline{k})_{\overline{k}} 
\ar[d] \ar[r]  & \Gal(\overline{k}/k) \ar[d]^{s} \ar[r] & 1\\
1 \ar[r] & C_{u_{X_0}}(\overline{k}) \ar[r] & C_{u_{X_0}}(\overline{k}) \rtimes \Gal(\overline{k}/k) 
\ar[r] & {}_nX_0(\overline{k}) \rtimes \Gal(\overline{k}/k) \ar[r] & 1
}
\]
where the top row is the extension associated to $F_x$, 
the bottom row is the semidirect product of $\Gal(\overline{k}/k)$ by \eqref{e:sexC},
and  $s(\sigma) = (x - {}^\sigma x,\sigma)$. 
The left square commutes because $p_{\overline{k}} = n_{X_0}{}_{\overline{k}} \circ i_x$.
To see that the right square commutes, consider the unique action of $\pi_\mathrm{\acute{e}tm}(X_0)$
on $X$, regarded as a scheme over $X_0$ using $p$.
By \eqref{e:pietcov}, $p_*$ factors uniquely through an isomorphism from $\pi_\mathrm{\acute{e}tm}(X)$
to $\pi_\mathrm{\acute{e}tm}(X_0) \times_{X_0} X$ over $X$, so that we may identify $F_x(\overline{k})_{\overline{k}}$
with the set of  pairs
$(v,{}^\sigma 1_{\overline{k}})$ with $v$ in $C_{u_{X_0}}(\overline{k})$ and $\sigma$ in $\Gal(\overline{k}/k)$
such that $v{}^\sigma x = x$.
The image of $(v,{}^\sigma 1_{\overline{k}})$ under the middle arrow is then
$(v,\sigma)$ and under the top right arrow is $\sigma$.
Now $C_{u_{X_0}}(\overline{k})$ acts on $X(\overline{k})$ through the action 
of ${}_n X_0(\overline{k})$ by translation, because $i_x$ is compatible with the actions of $X_0$.
Thus the right square commutes.

Since $-\delta_X$ is the image under \eqref{e:WCntors} of the class of $s$,
composing the above diagram
with the isomorphism of extensions from $E_x$ to $F_x(\overline{k})_{\overline{k}}$ defined by the inverse of \eqref{e:FxExiso} shows that the image of $-\delta_X$ under \eqref{e:nWCconnect} is the class of $E_x$.
\end{proof}

Denote by $\Pic^0(X_{\overline{k}})$ the Galois submodule of $\Pic(X_{\overline{k}})$ consisting of the classes
of line bundles algebraically equivalent to zero.
Since translation acts trivially on $\Pic^0(X_0{}_{\overline{k}})$, pullback along $i_x{}\!^{-1}$ for any $x$ in $X(\overline{k})$
defines an isomorphism
\[
\Pic^0(X_{\overline{k}}) \iso \Pic^0(X_0{}_{\overline{k}}) = \widehat{X}_0(\overline{k})
\]
of Galois modules which is independent of $x$.
Thus we have a canonical isomorphism
\[
D(\widehat{X}_0(\overline{k}))_{\overline{k}} \iso D(\Pic^0(X_{\overline{k}}))_{\overline{k}}
\]
which for any $x$ sends the tautological element in $H^1(X_0{}_{\overline{k}},D(\widehat{X}_0(\overline{k}))_{\overline{k}})$ to
the pullback along $i_x{}^{-1}$ of the tautological element in 
$H^1(X_{\overline{k}},D(\Pic^0(X_{\overline{k}}))_{\overline{k}})$. 
By \eqref{e:diagPicf}, it is compatible with $(i_x{}\!^{-1})_*$ and the projections from 
$(C_{u_{X_0}})_{\overline{k}}$ and $F_x{}\!^\mathrm{diag}$.

For any $x$ in $X(\overline{k})$, the right action by composition of $F_x = \pi_\mathrm{\acute{e}tm}(X)_{x \times x}$ on the scheme $\pi_\mathrm{\acute{e}tm}(X)_{- \times x}$ over $X_{\overline{k}}$
defines on  $\pi_\mathrm{\acute{e}tm}(X)_{- \times x}$ a structure of principal $F_x$\nd bundle, and hence of principal 
$(F_x{}\!^\mathrm{diag},F_x(\overline{k})_{\overline{k}})$\nd bundle, over $X$.
Pushing forward along \eqref{e:FxExiso}, we obtain a principal $((C_{u_{X_0}})_{\overline{k}},E_x)$\nd bundle $P_x$ over $X$.
The push forward of the underlying principal $(C_{u_{X_0}})_{\overline{k}}$\nd bundle over $X_{\overline{k}}$ of $P_x$ along 
\[
(C_{u_{X_0}})_{\overline{k}} \to D(\widehat{X}_0(\overline{k}))_{\overline{k}} \iso
D(\Pic^0(X_{\overline{k}}))_{\overline{k}}
\]
has class in $H^1(X_{\overline{k}},D(\Pic^0(X_{\overline{k}}))_{\overline{k}})$ the tautological element.

In Theorem~\ref{t:gen1rep}, we apply the above with $X$ a smooth projective curve over $k$ of genus $1$ and $X_0$
the Jacobian of $X$.
In that case $NS(X_0{}_{\overline{k}}) = \Z$, and by \eqref{e:zetatilde} the horizontal arrows of 
\eqref{e:XkbarAut} with $X_0$ for $X$ are isomorphisms.
By \eqref{e:AutEnd}, the group of those automorphisms up to conjugacy of the $\overline{k}$\nd group
$\pi_\mathrm{\acute{e}tm}(X_0,x_0)_{\overline{k}} = (C_{u_{X_0}})_{\overline{k}}$ that lie above the identity of
$D(\widehat{X}_0(\overline{k}))_{\overline{k}}$ thus coincides with $X_0(\overline{k})$.

\begin{thm}\label{t:gen1rep}
Let $X$ be a smooth geometrically connected projective curve over $k$ of genus $1$.
Denote by $X_0$ the Jacobian of $X$ and by $\delta_X$ the class of $X$ in the Weil--Ch\^atelet group of $X_0$.
Let $E$ be a topological extension of $\Gal(\overline{k}/k)$ by 
$C_{u_{X_0}}(\overline{k})$ whose class in the target of \eqref{e:WCconnect} is the image of $-\delta_X$.
Then $((C_{u_{X_0}})_{\overline{k}},E)$ is a Galois extended $\overline{k}$\nd group, and there exists an element
in $H^1(X,(C_{u_{X_0}})_{\overline{k}},E)$ with image under
\begin{multline*}
H^1(X,(C_{u_{X_0}})_{\overline{k}},E) \to H^1(X_{\overline{k}},(C_{u_{X_0}})_{\overline{k}}) \to \\
\to H^1(X_{\overline{k}},D(\widehat{X}_0(\overline{k}))_{\overline{k}}) \iso
H^1(X_{\overline{k}},D(\Pic^0(X_{\overline{k}}))_{\overline{k}})
\end{multline*}
the tautological element.
If $\alpha$ is such an element, and if $\beta$ in 
\[
H^1(X,SL_2) = H^1(X,SL_2{}_{\overline{k}},SL_2(\overline{k}) \rtimes \Gal(\overline{k}/k))
\] 
is the image of a non-zero element of $H^1(X,\bG_a)$ under an embedding $\bG_a \to SL_2$, then the functor $H^1(X,-,-)$ 
on the category of reductive Galois extended $\overline{k}$\nd groups up to conjugacy is represented by
\[
((C_{u_{X_0}})_{\overline{k}},E) \times (SL_2{}_{\overline{k}},SL_2(\overline{k}) \rtimes \Gal(\overline{k}/k))
\]
with universal element $(\alpha,\beta)$.
\end{thm}

\begin{proof}
By Lemma~\ref{l:WCclass} we may suppose that $E = E_x$ for some $x$ in $X(\overline{k})$.
Then $((C_{u_{X_0}})_{\overline{k}},E)$ is a Galois extended $\overline{k}$\nd group by \eqref{e:FxExiso},
and  the class of $P_x$ in $H^1(X,(C_{u_{X_0}})_{\overline{k}},E)$ has image in  
$H^1(X_{\overline{k}},D(\Pic^0(X_{\overline{k}}))_{\overline{k}})$ the tautological element.
Since pullback along $i_x{}\!^{-1}$ defines an isomorphism from $H^1(X_{\overline{k}},-)$ to
$H^1(X_0{}_{\overline{k}},-) = \widetilde{H}^1(X_0{}_{\overline{k}},x_0,-)$, 
the representation statement follows from Corollary~\ref{c:unscalextac} and Theorem~\ref{t:ellrep}.
\end{proof}

\providecommand{\bysame}{\leavevmode\hbox to3em{\hrulefill}\thinspace}
\providecommand{\MR}{\relax\ifhmode\unskip\space\fi MR }
\providecommand{\MRhref}[2]{%
  \href{http://www.ams.org/mathscinet-getitem?mr=#1}{#2}
}
\providecommand{\href}[2]{#2}

\end{document}